\documentclass{bmcart}

\usepackage[utf8]{inputenc} 

\usepackage{amsmath,amsfonts,amssymb,amsthm,stmaryrd,mathrsfs,mathdots,amsbsy,float,cmll,xcolor,mathpartir,colonequals}
\usepackage{hyperref}
\usepackage{bm}
\usepackage{tikz}
\usetikzlibrary{tikzmark}
\usepackage[small]{diagrams}
\usepackage{float}
\usetikzlibrary{positioning,calc}
\usepackage{tikz-cd}
\usepackage{bussproof}

\theoremstyle{definition}
\newtheorem{definition}{Definition}[section]
\newtheorem{example}[definition]{Example}
\theoremstyle{theorem}
\newtheorem{theorem}[definition]{Theorem}
\newtheorem{proposition}[definition]{Proposition}
\newtheorem{lemma}[definition]{Lemma}
\newtheorem{corollary}[definition]{Corollary}
\theoremstyle{remark}
\newtheorem*{remark}{Remark}
\newtheorem*{notation}{Notation}
\newtheorem*{convention}{Convention}

\newarrow{Corresponds}<--->
\newarrow{TeXto}----{->}

\def\includegraphics{}

\startlocaldefs
\endlocaldefs

\begin{document}

\begin{frontmatter}

\begin{fmbox}
\dochead{Research}


\title{Game semantics of Martin-L\"of type theory, part~III: its consistency with Church's thesis}


\author[
   addressref={aff1},                   
   corref={aff1},                       
   email={yamad041@umn.edu}   
]{\inits{NY}\fnm{Norihiro} \snm{Yamada}} \\
This is a preprint submitted to \emph{Research in the Mathematical Sciences (RMS).}


\address[id=aff1]{
  \orgname{University of Minnesota}, 
  \city{Minneapolis},                              
  \postcode{MN 55455},                                
  \cny{USA}                                    
}

\end{fmbox}


\begin{abstractbox}

\begin{abstract} 
We prove \emph{consistency of intensional Martin-L\"{o}f type theory (MLTT) with formal Church's thesis (CT)}, which was open for at least fifteen years.  
The difficulty in proving the consistency is that a standard method of \emph{realizability \`{a} la Kleene} does not work for the consistency, though it validates CT, as it does not model MLTT; specifically, the realizability does not validate MLTT's congruence rule on pi-types (known as the $\xi$-rule).
We overcome this point and prove the consistency by novel realizability \`{a} la \emph{game semantics}, which is based on the author's previous work.
\end{abstract}




\end{abstractbox}
%

\end{frontmatter}




\tableofcontents

\begin{convention}
To clear ambiguity, we call a mere assignment of semantic objects to syntactic objects an \emph{interpretation} or \emph{semantics}, and an interpretation that respects target syntactic properties (e.g., existence of a derivation of equality between terms) by semantic properties (e.g., agreement of the assigned semantic objects) a \emph{model}. 
(I.e., a model means a \emph{sound} interpretation.)
We say that an interpretation \emph{models} a type theory if it is a model of the type theory, and it \emph{models} or \emph{validates} (resp. \emph{refutes}) an axiom or rule if it respects (resp. does not respect) the axiom or rule. 
\end{convention}

\section{Introduction}
Ishihara et al.~\cite[Section~8]{ishihara2018consistency} conclude by stating that
\begin{quote}
Unfortunately, the consistency of Church's thesis with full Martin-L\"{o}f's is still open, and it is presumably quite difficult to answer this question.
\end{quote}
The present work answers the question \emph{in the affirmative}, i.e., we prove \emph{consistency of intensional Martin-L\"of type theory with formal Church's thesis}, which was open for at least fifteen years. 
Our motivation comes from the view that the consistency problem is interesting as a mathematical problem in its own right but also from its consequences on foundations of constructive mathematics. 

Intensional Martin-L\"of type theory, or \emph{MLTT} for short, is a formal system for constructive mathematics, which is shown to be \emph{consistent} (i.e., it does not derive \emph{falsity}). 
On the other hand, formal Church's thesis is a logical formula expressible in MLTT and states that (total) maps on the set $\mathbb{N}$ of all natural numbers are all `computable' or \emph{recursive}. 
Then, our result, i.e., the targeted consistency, means that MLTT equipped with formal Church's thesis as an additional axiom is consistent. 

Our consistency proof is based on a mathematical model of MLTT, called \emph{game semantics}, and in a novel manner takes advantage of its distinguishing features: the \emph{distinction} and the \emph{asymmetry} between \emph{Player}, who `executes (the interpretation of) a term,' and \emph{Opponent}, who plays the role of the `environment' or `rebutter' for Player in game semantics.
More precisely, our modified game semantics, which we call \emph{realizability \`{a} la game semantics}, forces Opponent to \emph{play as `total,' `effective' computation} when he plays as an input for Player and moreover \emph{exhibit a realizer} for his input computation before its execution even if the calculation of the realizer is not `effective,'\footnote{It is not a problem because game semantics treats Opponent as an `oracle' endowed with an unlimited computational power. See \cite[Section~1]{yamada2019game} for more on this point.} while in contrast it does not impose this heavy task on Player or assign to her a choice of a realizer for the interpretation of a term. 
Consequently, realizability \`{a} la game semantics trivially validates formal Church's thesis because Opponent discloses a realizer for his input computation $\mathbb{N} \rightarrow \mathbb{N}$ (so that it suffices for validating formal Church's thesis to just let Player `copy-cat' the realizer), but also it models MLTT since it is as abstract as the existing game semantics of MLTT. 

This novel `game-semantic art' explains why our method, unlike existing ones, succeeds in solving the consistency problem, and moreover it opens up a new way of applying game semantics to the study of constructive mathematics.

In the rest of this introduction, we explain the backgrounds, our motivations, the problem to solve, our solution for the problem, our contribution and related work in such a way that it would suffice in subsequent sections to just fill in the details.

\subsection{Constructive mathematics} 
\emph{Constructive mathematics} \cite{troelstra1988constructivism,troelstra2014constructivism,beeson2012foundations} is a branch of mathematical logic \cite{shoenfield1967mathematical} and foundations of mathematics \cite{kleene1952introduction} that studies `constructive,' `computable' or `effective' objects and reasonings.
One of the major motivations for constructive mathematics is the suspicion against `nonconstructive' objects and reasonings in classical or ordinary mathematics, e.g., the \emph{law of excluded middle} and the \emph{axiom of choice} \cite{troelstra1988constructivism}. 

However, constructive mathematicians have never established a universal consensus on which objects and reasonings in mathematics are `constructive' or how to formulate the informal notion of `constructivity.'
In fact, various different schools of constructive mathematics have arisen and been present in the literature.

\subsection{Type theories}
\emph{Type theories} \cite{girard1989proofs,jacobs1999categorical} are a particular class of \emph{formal systems} \cite{shoenfield1967mathematical}, whose distinguishing feature is that their variables and terms are \emph{typed}.
By the \emph{Curry-Howard isomorphisms (CHIs)} \cite{girard1989proofs,sorensen2006lectures}, type theories serve as a single formalism for both logic and computation; they are not only formal systems but also \emph{programming languages}.

A type theory is \emph{simple} if it prohibits variables from occurring in types, and \emph{dependent} otherwise \cite{jacobs1999categorical,sorensen2006lectures}, where we consider only \emph{term variables} (i.e., no \emph{type variables}) in this article.
The generalization of simple type theories to dependent ones corresponds under the CHIs to that of \emph{propositional logic} to \emph{predicate logic} \cite{shoenfield1967mathematical}.

Type theories are similar to \emph{sequent calculi} \cite{gentzen1935untersuchungen,troelstra2000basic} except that vertices of a formal proof or derivation (tree) in a type theory are not sequents but \emph{judgements}, for which we usually write $\mathcal{J}$ (possibly with subscripts/superscripts).
Hence, a type theory consists of \emph{axioms} $\frac{}{ \ \mathcal{J} \ }$ and \emph{(inference) rules} $\frac{ \ \mathcal{J}_1 \ \mathcal{J}_2 \dots \mathcal{J}_k \ }{ \mathcal{J}_0 }$, which are to make a \emph{conclusion} from \emph{hypotheses} by constructing a derivation.
Most type theories have the following six kinds of judgements (followed by their intended meanings):
\begin{enumerate}

\item $\mathsf{\vdash \Gamma \ ctx}$ \label{FirstJudgement} ($\mathsf{\Gamma}$ is a \emph{context}, i.e., a finite sequence of pairs of a variable and a type); 

\item $\mathsf{\Gamma \vdash A \ type}$ \label{SecondJudgement} ($\mathsf{A}$ is a \emph{type} in the context $\mathsf{\Gamma}$);

\item $\mathsf{\Gamma \vdash a : A}$ \label{ThirdJudgement} ($\mathsf{a}$ is a \emph{term} of the type $\mathsf{A}$ in the context $\mathsf{\Gamma}$);

\item $\mathsf{\vdash \Gamma = \Delta \ ctx}$ \label{FourthJudgement} ($\mathsf{\Gamma}$ and $\mathsf{\Delta}$ are \emph{equal} contexts);

\item $\mathsf{\Gamma \vdash A = B \ type}$ \label{FifthJudgement} ($\mathsf{A}$ and $\mathsf{B}$ are \emph{equal} types in the context $\mathsf{\Gamma}$);

\item $\mathsf{\Gamma \vdash a = a' : A}$ \label{SixthJudgement} ($\mathsf{a}$ and $\mathsf{a'}$ are \emph{equal} terms of the type $\mathsf{A}$ in the context $\mathsf{\Gamma}$),

\end{enumerate}
where the judgements (\ref{FirstJudgement}), (\ref{FourthJudgement}) and (\ref{FifthJudgement}) are trivial and usually omitted in simple type theories.
We often omit the \emph{turnstile} $\mathsf{\vdash}$ in judgements if the context on the LHS is empty.
The CHIs regard contexts, types and terms as \emph{assumptions}, \emph{formulas} and \emph{proofs} in logic, respectively; type theories serve as formal systems in this way \cite{jacobs1999categorical,sorensen2006lectures}.

\subsection{Intensional Martin-L\"{o}f type theory}
\label{IntroMLTT}
\emph{Martin-L\"{o}f type theory (MLTT)} invented by Martin-L\"{o}f \cite{martin1982constructive,martin1984intuitionistic,martin1998intuitionistic} is a prominent dependent type theory meant to serve as a foundation of constructive mathematics, which is comparable to \emph{Zermelo-Fraenkel set theory with the axiom of choice (ZFC)} \cite{enderton1977elements} for classical mathematics. 
Also, MLTT is a subject in computer science as well since it is not only a formal system but also a programming language \cite{nordstrom1990programming}. 

Strictly speaking, there are the \emph{intensional} and the \emph{extensional} variants of MLTT; see Appendix~\ref{MLTT} for the details. 
In this article, we focus on the intensional variant as the extensional one is known to be inconsistent with formal Church's thesis \cite[Proposition~6.2]{maietti2005toward}, and \emph{MLTT} refers to the intensional one unless stated otherwise. 

By its computational nature, more specifically \emph{operational semantics}, of MLTT, one may prove \emph{consistency} of MLTT, i.e., there is no formal proof of empty-type (i.e., the type or formula of \emph{falsity}) in MLTT; e.g., see \cite[Corollary~A.4.6]{hottbook}.

For the rest of this article, let us assume that the reader is familiar with MLTT, especially its syntax, and leave the details to Appendix~\ref{MLTT} or the references \cite{martin1998intuitionistic,hofmann1997syntax}.

\subsection{Formal Church's thesis}
\label{IntroFormalChurch'sThesis}
\emph{(Formal) Church's thesis (CT)} \cite{kreisel1970church} is a logical formula that is expressible in MLTT and argued mostly in constructive mathematics.
Informally, CT states that every (total) function on $\mathbb{N}$ is \emph{recursive} (i.e., `computable' in the standard mathematical sense \cite{rogers1967theory,cutland1980computability}). 
In other words, CT represents the particular school of constructive mathematics that requires \emph{every} mathematical object to be `constructive.'  

\begin{remark}
The name of CT comes from that of \emph{Church-Turing thesis (CTT)} \cite{rogers1967theory,cutland1980computability}, which states that `computable' (in the informal sense) partial functions on $\mathbb{N}$ are precisely recursive ones.
Note, however, that CT and CTT are different statements.
\end{remark}

In contrast with CT, some of the other schools in constructive mathematics such as \emph{computable analysis} \cite{weihrauch2012computable} accept constructions or operations on `nonconstructive' objects such as non-recursive functions and real numbers.  
Also, CT contradicts the classical \emph{recursion theory} \cite{rogers1967theory,cutland1980computability} because the latter shows that there are non-recursive functions on $\mathbb{N}$ such as the one for the \emph{halting problem}.

At this point, let us recall a standard formalization of CT as a logical formula,
\begin{equation}
\label{CTinTroelstraAndVanDalen}
\forall f \in \mathbb{N} \Rightarrow \mathbb{N}, \exists e \in \mathbb{N}, \forall n \in \mathbb{N}, \exists c \in \mathbb{N} . \,  T(e, n, c) \wedge U(c) = f(n) \ \text{\cite[p.~192]{troelstra1988constructivism}},
\end{equation}
where $T$ represents \emph{Kleene's T-predicate}, and $U$ does the \emph{result-extracting function} \cite[p.~133]{troelstra1988constructivism}.
Recall that $T(e, n, c)$ holds if and only if $c$ encodes the computational history of the (necessarily terminating\footnote{The computation must be terminating since its computational history is finite.}) computation encoded by $e$ applied to the input $n$, and $U(c)$ is the output of the computational process encoded by $c$, where an encoding of algorithms by natural numbers is arbitrarily fixed.
Because a function $f : \mathbb{N} \rightarrow \mathbb{N}$ is recursive if and only if there is a natural number $e \in \mathbb{N}$ that encodes an algorithm for $f$ \cite{rogers1967theory,cutland1980computability}, the formula (\ref{CTinTroelstraAndVanDalen}) indeed formalizes the intended meaning of CT.
Following a convention in the field of \emph{realizability} \cite[Section~4.4.1]{troelstra1988constructivism}, we call such a natural number $e$ a \emph{realizer} for $f$, and equivalently say that $e$ \emph{realizes} $f$.

The standard constructive reading or the \emph{BHK-interpretation} of intuitionistic logic \cite[Section~1.3.1]{troelstra1988constructivism} interprets CT as: \emph{There is an algorithm to compute a realizer $e \in \mathbb{N}$ for a given function $f : \mathbb{N} \rightarrow \mathbb{N}$}, where note that it is trivial to constructively validate the remaining part $\forall n \in \mathbb{N}, \exists c \in \mathbb{N} . \,  T(e, n, c) \wedge U(c) = f(n)$ once a realizer $e$ for $f$ is obtained.
However, it is in general not `effective' to compute a realizer for a given function $\mathbb{N} \rightarrow \mathbb{N}$, even if the function is recursive, since otherwise equality on recursive functions $\mathbb{N} \rightarrow \mathbb{N}$ would be decidable, contradicting \emph{Rice's theorem} \cite{rogers1967theory,cutland1980computability,rice1953classes}. 
Therefore, a \emph{constructive} model of CT must require the inputs $\mathbb{N} \rightarrow \mathbb{N}$ not only to be recursive but also to \emph{exhibit their realizers}.

Finally, let us translate under the CHIs the formula (\ref{CTinTroelstraAndVanDalen}) into a formula in MLTT,
\begin{equation}
\label{CTinMLTT}
\mathsf{\Pi_{f : N \Rightarrow N} \Sigma_{x : N} \Pi_{y : N}\Sigma_{z : N} T(x, y, z) \times Id_N(U(z), f(y))},
\end{equation}
where the type $\mathsf{x : N, y : N, z : N \vdash T(x, y, z) \ type}$ represents the T-predicate, and the term $\mathsf{z : N \vdash U(z) : N}$ does the result-extracting function.
Note that they are both \emph{primitive recursive} \cite[Section~3.4.3]{troelstra1988constructivism} and so expressible as formulas of MLTT. 

\subsection{Our goal: to prove consistency of intensional Martin-L\"{o}f type theory with formal Church's thesis}
Although CT has been shown to be consistent with most intuitionistic formalisms in the literature \cite[Section~4.10.2]{troelstra1988constructivism}, MLTT has been an exception: Consistency of MLTT with CT (\ref{CTinMLTT}) has been an open problem though both MLTT and CT were introduced in the 1970s.
In other words, it has been unknown whether or not MLTT is compatible with the perspective that every function $\mathbb{N} \rightarrow \mathbb{N}$ is `constructive.'
Let us add another historical fact that this consistency problem was articulated explicitly fifteen years ago by Maietti and Sambin \cite[Section~6.4]{maietti2005toward}. 

The primary goal of the present work is to solve this long-standing open problem \emph{in the affirmative}. 
Our motivation is \emph{mathematical}; that is, we find the consistency problem as technically challenging yet interesting in its own right, which we explain below.
We are also motivated by the fact that the consistency, if established, would have consequences on foundations of constructive mathematics; for instance, MLTT plus CT would be available as an intensional level foundation for the programme of \emph{a minimalist two-level foundation for constructive mathematics} \cite{maietti2009minimalist}.

\subsection{Why is the consistency difficult to prove?}
\label{WhyDifficult}
As its longevity indicates, the consistency problem poses a technical challenge. 
Specifically, the challenge is that a standard method of \emph{realizability \`{a} la Kleene} \cite[Section~4.4.1]{troelstra1988constructivism} for showing consistency of an intuitionistic formal system with CT does not work for the case of MLTT \cite[Sections~1 and 8]{ishihara2018consistency}.\footnote{Also, it is unclear how to extend the syntactic method for proving consistency of MLTT mentioned in Section~\ref{IntroMLTT} to MLTT plus CT.}
More concretely, the obstacle is that the realizability does not model MLTT, though it validates CT, since it does not validate MLTT's congruence rule on pi-types or the \emph{$\xi$-rule}:
\begin{equation*}
\AxiomC{$\mathsf{\Gamma, x : A \vdash b = b' : B}$}
\LeftLabel{\textsc{($\xi$)}}
\UnaryInfC{$\mathsf{\Gamma \vdash \lambda x^A . b = \lambda x^A . b' : \Pi_{x : A} B}$}
\DisplayProof
\end{equation*}
The point is that the realizability interprets a term in MLTT by a choice of a \emph{realizer} for it, i.e., a natural number that encodes an algorithm for the term, and inputs of a term, if any,\footnote{A term has inputs precisely when it is open or of a pi-type.} by a choice of realizers for them too.
This highly \emph{intensional} interpretation constructively (in the sense explained in Section~\ref{IntroFormalChurch'sThesis}) yet trivially validates CT since the inputs $\mathbb{N} \rightarrow \mathbb{N}$ already come in the form of realizers. 
However, this interpretation of \emph{open} terms, i.e., terms whose contexts are nonempty, does not validate equality between open terms such as $\mathsf{x : N \vdash x = x+0 :N}$ because the two terms may represent different algorithms. 
Hence, we must take the \emph{quotient} of the interpretation of each open term modulo input-output pairs or \emph{extensions}. 
(N.b., in contrast, we cannot take the quotients of closed terms of pi-types unless we give up validating CT.)
But the quotient makes the interpretation of an open term \emph{more extensional} than that of its $\lambda$-abstraction; consequently, the interpretation refutes the $\xi$-rule.
For a detailed account of this problem, see Ishihara et al.~\cite{ishihara2018consistency}.

On the other hand, another standard \emph{realizability \`{a} la assemblies} \cite{reus1999realizability,streicher2008realizability} models MLTT, including the $\xi$-rule, since unlike realizability \`{a} la Kleene it does not assign realizers as part of its interpretation of terms, i.e., the one \`{a} la assemblies is more extensional than the one \`{a} la Kleene. 
But still, the former requires the \emph{existence} of a realizer for the interpretation of a term, and thus it reads CT \emph{constructively}. 
It also requires the existence of realizers for inputs of the interpretation of a term, if any, but again it does not assign a choice of these realizers. 
Consequently, realizability \`{a} la assemblies does not validate CT since the inputs $\mathbb{N} \rightarrow \mathbb{N}$ in the interpretation are just recursive maps $\mathbb{N} \rightarrow \mathbb{N}$, and so there is no algorithm to calculate realizers for them as remarked in Section~\ref{IntroFormalChurch'sThesis} (but the realizability reads CT constructively).

\begin{remark}
If we assign a choice of realizers only to inputs of the interpretation of an open term and not to the interpretation itself in realizability \`{a} la assemblies then the interpretation would refute the $\xi$-rule just like realizability \`{a} la Kleene does. 
\end{remark}

To summarize, the main difficulty of the consistency problem is that the standard consistency proof of constructing a realizability model of MLTT plus CT suffers from the \emph{dilemma between intensionality and extensionality}: An interpretation must be intensional enough to validate CT but also extensional enough to model open terms.

We conquer this challenge by introducing novel realizability \`{a} la \emph{game semantics}.

\subsection{Game semantics of simple type theories}
\label{IntroGameSemantics}
\emph{Game semantics} \cite{abramsky1997semantics,hyland1997game} refers to a particular class of \emph{denotational (mathematical) semantics of logic and computation} \cite{amadio1998domains} that interprets formulas (or types) and proofs (or terms) as \emph{games} and \emph{strategies}, respectively. 
Game semantics has been highly successful in interpreting a wide range of simple type theories \cite{abramsky2014intensionality,abramsky1999game} and moreover extended to dependent type theories just recently \cite{abramsky2015games,yamada2016game}. 

A distinguishing feature of game semantics is that it models logic and computation in terms of \emph{interactions} between two participants of games, \emph{Player}, who plays by following a strategy, and \emph{Opponent}, who in contrast plays like an `oracle.' 
Our idea to prove consistency of MLTT with CT is to take advantage of these \emph{distinction} and \emph{asymmetry} between Player and Opponent. 
Concretely, our solution is to modify the game semantics of MLTT \cite{yamada2016game} by requiring that \emph{only Opponent} has to play by a `total,' `effective' strategy when he plays as an input for Player and moreover exhibit a realizer for the input strategy at the beginning of each play. 
Note that strategies are a much more intensional concept than (partial) maps; we then take realizers for strategies also as realizers for their \emph{extensions} (i.e., the partial maps that the strategies compute) and implement the T-predicate and the result-extracting function with respect to the resulting realizers for (total) recursive maps $\mathbb{N} \rightarrow \mathbb{N}$.
Consequently, the resulting, modified game semantics \emph{constructively validates CT} since it suffices for validating CT to let Player `copy-cat' the realizer for an input $\mathbb{N} \rightarrow \mathbb{N}$ given by Opponent, but also it \emph{models MLTT} as the interpretations of terms are as abstract as those of the game semantics \cite{yamada2016game}.
The rest of this introduction is devoted mostly to explaining more in detail this solution to the consistency problem.

Let us first recall games and strategies \`{a} la McCusker \cite{abramsky1999game,mccusker1998games} that model simple type theories because our method is based on them.
Our review is largely taken from the author's earlier work \cite[Section~1.4]{yamada2019game}.
A \emph{game}, roughly, is a certain kind of a directed rooted forest whose branches represent possible `developments' or \emph{(valid) positions} in a `game in the usual sense' (such as chess, poker, etc.).
\emph{Moves} in a game are nodes of the game, where some moves are distinguished and called \emph{initial}; only initial moves can be the first element or \emph{occurrence} of a position in the game. 
\emph{Plays} in a game are increasing sequences $\boldsymbol{\epsilon}, m_1, m_1 m_2, \dots$ of positions in the game, where $\boldsymbol{\epsilon}$ is the \emph{empty sequence}. 
For our purpose, it suffices to focus on standard games played by two participants, \emph{Player} (\emph{P}), who represents a `computational agent' or `prover,' and \emph{Opponent} (\emph{O}), who represents a `environment' or `rebutter,' in each of which O always starts a play, and then they separately and alternately make moves allowed by the rules of the game.
Strictly speaking, a position in each game is not just a sequence of moves: Each occurrence $m$ of O's or O- (resp. P's or P-) non-initial move in a position \emph{points} to a previous occurrence $m'$ of P- (resp. O-) move in the position, representing that $m$ is performed specifically as a response to $m'$. 
The pointers are necessary to distinguish similar yet distinct computations \cite{abramsky1999game}, and also they play a crucial role in the \emph{game-semantic model of computation} \cite{yamada2019game}.

On the other hand, a \emph{strategy} $g$ on a game $G$, written $g : G$, is what tells P which move (together with a pointer) she should make at each of her turns in the game.
Technically, $g$ is a set of (selected) even-length positions in $G$ that is
\begin{itemize}

\item Nonempty and \emph{even-prefix-closed} (i.e., $\boldsymbol{s}mn \in g \Rightarrow \boldsymbol{s} \in g$);

\item \emph{Deterministic} (i.e., $\boldsymbol{s}mn, \boldsymbol{s}mn' \in g \Rightarrow \boldsymbol{s}mn = \boldsymbol{s}mn'$),

\end{itemize}
so that it directs P to play by $\boldsymbol{s}m \mapsto n$ (with the pointer from $n$ into $\boldsymbol{s}m$) if and only if there is (necessarily unique) $\boldsymbol{s}mn \in g$ at each odd-length position $\boldsymbol{s}m$ in $G$.

Then, game semantics $\llbracket \_ \rrbracket_{\mathbb{G}}$ of a simple type theory $\mathscr{S}$ interprets a type $\mathsf{A}$ in $\mathscr{S}$ as a game $\llbracket \mathsf{A} \rrbracket_{\mathbb{G}}$ that specifies possible plays between P and O, and a term $\mathsf{M : A}$\footnote{For simplicity, here we focus on \emph{closed} terms, i.e., ones with the \emph{empty context}.} in $\mathscr{S}$ as a strategy $\llbracket \mathsf{M} \rrbracket_{\mathbb{G}} : \llbracket \mathsf{A} \rrbracket_{\mathbb{G}}$ that describes for P how to play in $\llbracket \mathsf{A} \rrbracket_{\mathbb{G}}$. 
An execution of the term $\mathsf{M}$ is then interpreted as a play in $\llbracket \mathsf{A} \rrbracket_{\mathbb{G}}$ in which P follows $\llbracket \mathsf{M} \rrbracket_{\mathbb{G}}$.

Let us consider examples. 
The simplest game is the \emph{unit game} $\boldsymbol{1}$, which has no moves. 
Thus, it has only the trivial position $\boldsymbol{\epsilon}$ and the trivial strategy $\top \colonequals \{ \boldsymbol{\epsilon} \}$.
The unit game $\boldsymbol{1}$ forms a terminal object in the category $\mathbb{G}$ of games and strategies.  
Another simple game is the \emph{empty game} $\boldsymbol{0}$, which has an arbitrary element $q$ as its only move; thus, its positions are $\boldsymbol{\epsilon}$ and $q$, and so it only has the strategy $\bot \colonequals \{ \boldsymbol{\epsilon} \}$.
The game semantics $\llbracket \_ \rrbracket_{\mathbb{G}}$ interprets unit-type $\mathsf{1}$, empty-type $\mathsf{0}$ and the unique \emph{canonical}\footnote{See Appendix~\ref{UnitType} for canonical terms.} term $\mathsf{\top : 1}$ respectively by $\llbracket \mathsf{1} \rrbracket_{\mathbb{G}} \colonequals \boldsymbol{1}$, $\llbracket \mathsf{0} \rrbracket_{\mathbb{G}} \colonequals \boldsymbol{0}$ and $\llbracket \mathsf{\top : 1} \rrbracket_{\mathbb{G}} \colonequals \top : \boldsymbol{1}$.

Yet another example is the game $N$ of natural numbers, which is the rooted tree (infinite in width)
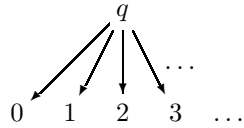
\begin{diagram}
& & q & & \\
& \ldTo(2, 2) \ldTo(1, 2) & \dTo & \rdTo(1, 2) \ \dots & \\
0 & 1 & 2 & 3 & \dots
\end{diagram}
in which a play starts with O's question $q$ (`What is your number?') and ends with P's answer $n \in \mathbb{N}$ (`My number is $n$!'), where $n$ points to $q$ (though this pointer is omitted in the diagram). 
Henceforth, we usually skip drawing arrows that represent edges of a game. 
A strategy $\underline{m} \colonequals \{ \boldsymbol{\epsilon}, qm \}$ on $N$ for each $m \in \mathbb{N}$ can be represented by the map $q \mapsto m$ equipped with a pointer from $m$ to $q$ (though it is the only choice).
In the following, pointers of most strategies are obvious, and thus we often omit them.
The game semantics $\llbracket \_ \rrbracket_{\mathbb{G}}$ interprets natural number type $\mathsf{N}$ and the numeral $\mathsf{\underline{m} : N}$ respectively by $\llbracket \mathsf{N} \rrbracket_{\mathbb{G}} \colonequals N$ and $\llbracket \mathsf{\underline{m} : N} \rrbracket_{\mathbb{G}} \colonequals \underline{m} : N$.

There is a construction $\otimes$ on games, called \emph{tensor}.
Conceptually, a position $\boldsymbol{s}$ in the tensor $A \otimes B$ of games $A$ and $B$ is an interleaving mixture of a position $\boldsymbol{t}$ in $A$ and a position $\boldsymbol{u}$ in $B$ developed `in parallel without communication.' 
More specifically, $\boldsymbol{t}$ (resp. $\boldsymbol{u}$) is the subsequence of $\boldsymbol{s}$ consisting of moves in $A$ (resp. $B$) such that the change of $AB$-parity (i.e., the switch between $\boldsymbol{t}$ and $\boldsymbol{u}$) in $\boldsymbol{s}$ must be made by O. 
The pointer of $\boldsymbol{s}$ is inherited from those of $\boldsymbol{t}$ and $\boldsymbol{u}$ in the obvious manner; this point holds also for other constructions on games and strategies in the rest of the introduction, and therefore we shall not mention it again.
For instance, a maximal position in the tensor $N \otimes N$ is either of the following forms:\footnote{The diagrams are only to make it explicit which component game each move belongs to; the two positions are just finite sequences $q^{[0]} n^{[0]} q^{[1]} m^{[1]}$ and $q^{[1]} m^{[1]} q^{[0]} n^{[0]}$ equipped with the pointers $q^{[i]} \leftarrow n^{[i]}$ and $q^{[i]} \leftarrow m^{[i]}$ ($i = 0, 1$).}
\begin{mathpar}
\begin{tabular}{ccc}
$N^{[0]}$ & $\otimes$ & $N^{[1]}$ \\ \cline{1-3}
\tikzmark{ctensor1} $q^{[0]}$&& \\
\tikzmark{dtensor1} $n^{[0]}$&& \\
&&\tikzmark{ctensor3} $q^{[1]}$ \\
&&\tikzmark{dtensor3} $m^{[1]}$
\end{tabular}
\begin{tikzpicture}[overlay, remember picture, yshift=.25\baselineskip]
\draw [->] ({pic cs:dtensor1}) [bend left] to ({pic cs:ctensor1});
\draw [->] ({pic cs:dtensor3}) [bend left] to ({pic cs:ctensor3});
\end{tikzpicture}
\and
\begin{tabular}{ccc}
$N^{[0]}$ & $\otimes$ & $N^{[1]}$ \\ \cline{1-3}
&&\tikzmark{ctensor2} $q^{[1]}$ \\
&&\tikzmark{dtensor2} $m^{[1]}$ \\
\tikzmark{ctensor4} $q^{[0]}$&& \\
\tikzmark{dtensor4} $n^{[0]}$&&
\end{tabular}
\begin{tikzpicture}[overlay, remember picture, yshift=.25\baselineskip]
\draw [->] ({pic cs:dtensor2}) [bend left] to ({pic cs:ctensor2});
\draw [->] ({pic cs:dtensor4}) [bend left] to ({pic cs:ctensor4});
\end{tikzpicture}
\end{mathpar}
where $n, m \in \mathbb{N}$, and $(\_)^{[i]}$ ($i = 0, 1$) are (arbitrary, unspecified) `tags' to distinguish the two copies of $N$ (but we often omit them if it does not bring confusion), and the arrows represent pointers (n.b., they are not edges of the games).
The corresponding \emph{tensor} $a \otimes b$ of strategies $a : A$ and $b : B$ is the strategy on $A \otimes B$ that plays by $a$ if the last O-move is in A, and by $b$ otherwise. 
For instance, the above two plays in $N \otimes N$ can be seen as the ones where P plays by the tensor $\underline{n} \otimes \underline{m} : N \otimes N$.

Next, a fundamental construction $\oc $ on games, called \emph{exponential}, is basically the countably infinite iteration of tensor $\otimes$, i.e., $\oc A \colonequals A \otimes A \otimes A \cdots$ for each game $A$, where the `tag' for each copy of $A$ is given by $(\_, i)$ such that $i \in \mathbb{N}$.

Another central construction $\multimap$, called \emph{linear implication}, captures the notion of \emph{linear functions}, i.e., functions that consume exactly one input to produce an output. 
A position in the linear implication $A \multimap B$ from $A$ to $B$ is almost like a position in the tensor $A \otimes B$ except the following three points:
\begin{enumerate}

\item The first occurrence in the position must be a move in $B$; 

\item A change of $AB$-parity in the position must be made by P; 

\item Each occurrence of an initial move, called an \emph{initial occurrence}, in $A$ points to an initial occurrence in $B$.

\end{enumerate}
Thus, a typical position in the game $N \multimap N$ is the following:
\begin{mathpar}
\begin{tabular}{ccc}
$N^{[0]}$ & $\multimap$ & $N^{[1]}$ \\ \hline 
&&\tikzmark{cmultimap1000} $q^{[1]}$ \tikzmark{cmultimap3} \\
\tikzmark{cmultimap2} $q^{[0]}$ \tikzmark{dmultimap1000}&& \\
\tikzmark{dmultimap2} $n^{[0]}$&& \\
&&$m^{[1]}$ \tikzmark{dmultimap3}
\end{tabular}
\begin{tikzpicture}[overlay, remember picture, yshift=.25\baselineskip]
\draw [->] ({pic cs:dmultimap1000}) to ({pic cs:cmultimap1000});
\draw [->] ({pic cs:dmultimap2}) [bend left] to ({pic cs:cmultimap2});
\draw [->] ({pic cs:dmultimap3}) [bend right] to ({pic cs:cmultimap3});
\end{tikzpicture}
\end{mathpar}
where $n, m \in \mathbb{N}$, which can be read as follows:
\begin{enumerate}
\item O's question $q^{[1]}$ for an output (`What is your output?');
\item P's question $q^{[0]}$ for an input (`Wait, what is your input?');
\item O's answer, say, $n^{[0]}$, to $q^{[0]}$ (`OK, here is an input $n$.');
\item P's answer, say, $m^{[1]}$, to $q^{[1]}$ (`Alright, the output is then $m$.').
\end{enumerate}
This play corresponds to any linear function that maps $n \mapsto m$.
The strategy $\mathrm{succ}$ (resp. $\mathrm{double}$) on $N \multimap N$ for the successor (resp. doubling) function is represented by the map $q^{[1]} \mapsto q^{[0]}, q^{[1]}q^{[0]}n^{[0]} \mapsto n+1^{[1]}$ (resp. $q^{[1]} \mapsto q^{[0]}, q^{[1]}q^{[0]}n^{[0]} \mapsto 2n^{[1]}$). 
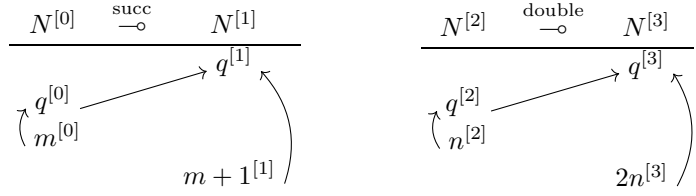
\begin{mathpar}
\begin{tabular}{ccc}
$N^{[0]}$ & $\stackrel{\mathrm{succ}}{\multimap}$ & $N^{[1]}$ \\ \cline{1-3} 
&&\tikzmark{csucc31} $q^{[1]}$ \tikzmark{csucc33} \\
\tikzmark{csucc32} $q^{[0]}$ \tikzmark{dsucc31}&& \\
\tikzmark{dsucc32} $m^{[0]}$ && \\
&&$m+1^{[1]}$ \tikzmark{dsucc33}
\end{tabular}
\begin{tikzpicture}[overlay, remember picture, yshift=.25\baselineskip]
\draw [->] ({pic cs:dsucc31}) to ({pic cs:csucc31});
\draw [->] ({pic cs:dsucc32}) [bend left] to ({pic cs:csucc32});
\draw [->] ({pic cs:dsucc33}) [bend right] to ({pic cs:csucc33});
\end{tikzpicture}
\and
\begin{tabular}{ccc}
$N^{[2]}$ & $\stackrel{\mathrm{double}}{\multimap}$ & $N^{[3]}$ \\ \cline{1-3} 
&&\tikzmark{cdouble31} $q^{[3]}$ \tikzmark{cdouble33} \\
\tikzmark{cdouble32} $q^{[2]}$ \tikzmark{ddouble31} && \\
\tikzmark{ddouble32} $n^{[2]}$&& \\
&&$2 n^{[3]}$ \tikzmark{ddouble33} 
\end{tabular}
\begin{tikzpicture}[overlay, remember picture, yshift=.25\baselineskip]
\draw [->] ({pic cs:ddouble31}) to ({pic cs:cdouble31});
\draw [->] ({pic cs:ddouble32}) [bend left] to ({pic cs:cdouble32});
\draw [->] ({pic cs:ddouble33}) [bend right] to ({pic cs:cdouble33});
\end{tikzpicture}
\end{mathpar}

Let us remark that the following play, which corresponds to a \emph{constant} linear function that maps $x \mapsto m$ for all $x \in \mathbb{N}$, is also possible in $N^{[0]} \multimap N^{[1]}$: $\boldsymbol{\epsilon}, q^{[1]}, q^{[1]}m^{[1]}$.
Hence, strictly speaking, $A \multimap B$ is the game of \emph{affine functions} from $A$ to $B$, but let us stick to the standard convention to call $\multimap$ linear implication. 

Now, let us recall \emph{composition} on strategies, which is given by \emph{internal communication plus hiding}.
For instance, the composition $\mathrm{succ} ; \mathrm{double} : N \multimap N$ of strategies $\mathrm{succ} : N \multimap N$ and $\mathrm{double} : N \multimap N$, is given as follows. 
First, by \emph{internal communication}, we mean that P plays the role of O as well in the intermediate component games $N^{[1]}$ and $N^{[2]}$ by `copy-catting' her last moves, resulting in the following play:
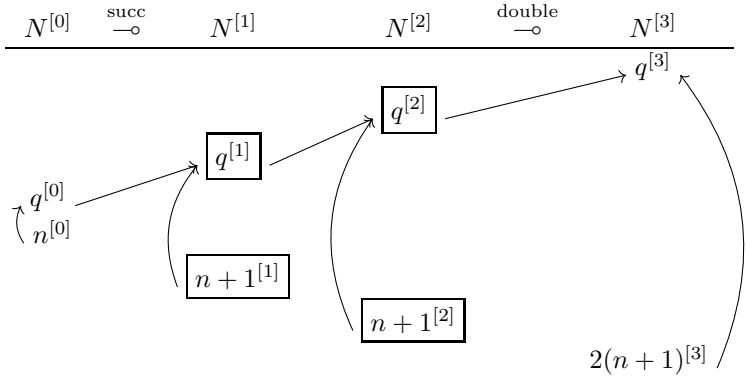
\begin{mathpar}
\begin{tabular}{ccccccc}
$N^{[0]}$ & $\stackrel{\mathrm{succ}}{\multimap}$ & $N^{[1]}$ && $N^{[2]}$ & $\stackrel{\mathrm{double}}{\multimap}$ & $N^{[3]}$ \\ \hline
&&&&&& \tikzmark{cPC2} $q^{[3]}$ \tikzmark{cPC1} \\
&&&&\tikzmark{cPC3} \fbox{$q^{[2]}$} \tikzmark{dPC2} && \\
&& \tikzmark{cPC4} \fbox{$q^{[1]}$} \tikzmark{dPC3} &&&& \\
\tikzmark{cPC5} $q^{[0]}$ \tikzmark{dPC4} &&&&&& \\
\tikzmark{dPC5} $n^{[0]}$ &&&&&& \\
&&\tikzmark{dPC6} \fbox{$n+1^{[1]}$} &&&& \\
&&&&\tikzmark{dPC7} \fbox{$n+1^{[2]}$} && \\
&&&&&&$2 (n+1)^{[3]}$ \tikzmark{dPC1}
\end{tabular}
\begin{tikzpicture}[overlay, remember picture, yshift=.25\baselineskip]
\draw [->] ({pic cs:dPC1}) [bend right] to ({pic cs:cPC1});
\draw [->] ({pic cs:dPC2}) to ({pic cs:cPC2});
\draw [->] ({pic cs:dPC3}) to ({pic cs:cPC3});
\draw [->] ({pic cs:dPC4}) to ({pic cs:cPC4});
\draw [->] ({pic cs:dPC5}) [bend left] to ({pic cs:cPC5});
\draw [->] ({pic cs:dPC6}) [bend left] to ({pic cs:cPC4});
\draw [->] ({pic cs:dPC7}) [bend left] to ({pic cs:cPC3});
\end{tikzpicture}
\end{mathpar}
where each move for internal communication is marked by a square box just for clarity, and the pointer from $q^{[1]}$ to $q^{[2]}$ is added because the move $q^{[1]}$ is no longer initial \cite{yamada2019game,yamada2016dynamic}. 
Importantly, it is assumed that O plays on the game $N^{[0]} \multimap N^{[3]}$, `seeing' only moves in $N^{[0]}$ or $N^{[3]}$.
Thus, the resulting play is to be read as follows:
\begin{enumerate}

\item O's question $q^{[3]}$ for an output in $N^{[0]} \multimap N^{[3]}$ (`What is your output?');

\item P's question \fbox{$q^{[2]}$} by $\mathrm{double}$ for an input in $N^{[2]} \multimap N^{[3]}$ (`What is an input?');

\item \fbox{$q^{[2]}$} triggers the question \fbox{$q^{[1]}$} for an output in $N^{[0]} \multimap N^{[1]}$ (`What is an output?');

\item P's question $q^{[0]}$ by $\mathrm{succ}$ for an input in $N^{[0]} \multimap N^{[1]}$ (`What is an input?');

\item O's answer, say, $n^{[0]}$, to $q^{[0]}$ in $N^{[0]} \multimap N^{[3]}$ (`Here is an input $n$.');

\item P's answer \fbox{$n+1^{[1]}$} to \fbox{$q^{[1]}$} by $\mathrm{succ}$ in $N^{[0]} \multimap N^{[1]}$ (`The output is $n+1$.');

\item \fbox{$n+1^{[1]}$} triggers the answer \fbox{$n+1^{[2]}$} to \fbox{$q^{[2]}$} in $N^{[2]} \multimap N^{[3]}$ (`Here is the input $n+1$.');

\item P's answer $2 (n+1)^{[3]}$ to $q^{[3]}$ by $\mathrm{double}$ in $N^{[2]} \multimap N^{[3]}$ (`The output is then $2 (n+1)$!').

\end{enumerate}
Next, \emph{hiding} means to hide or delete every move with a square box from the play, resulting in the strategy for the (linear) function $n \mapsto 2 (n + 1)$ as expected:
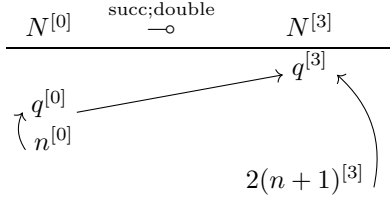
\begin{mathpar}
\begin{tabular}{ccc}
$N^{[0]}$ & $\stackrel{\mathrm{succ} ; \mathrm{double}}{\multimap}$ & $N^{[3]}$ \\ \hline
&&\tikzmark{cPCH1} $q^{[3]}$ \tikzmark{cPCH3} \\
\tikzmark{cPCH2} $q^{[0]}$ \tikzmark{dPCH1}&& \\
\tikzmark{dPCH2} $n^{[0]}$&& \\
&&$2 (n + 1)^{[3]}$ \tikzmark{dPCH3} 
\end{tabular}
\begin{tikzpicture}[overlay, remember picture, yshift=.25\baselineskip]
\draw [->] ({pic cs:dPCH1}) to ({pic cs:cPCH1});
\draw [->] ({pic cs:dPCH2}) [bend left] to ({pic cs:cPCH2});
\draw [->] ({pic cs:dPCH3}) [bend right] to ({pic cs:cPCH3});
\end{tikzpicture}
\end{mathpar}
Note that hiding makes the resulting play a valid one in the game $N^{[0]} \multimap N^{[3]}$.

Another construction $\&$ on games, called \emph{product}, is similar to yet simpler than tensor $\otimes$: A position $\boldsymbol{s}$ in the product $A \mathbin{\&} B$ of $A$ and $B$ is a position $\boldsymbol{t}^{[0]}$ in $A^{[0]}$ or a position $\boldsymbol{u}^{[1]}$ in $B^{[1]}$.
I.e., the set of all positions in $A \mathbin{\&} B$ is the disjoint union of that in $A$ and that in $B$.
It forms product in the category $\mathbb{G}$. 
The corresponding \emph{pairing} $\langle f, g \rangle$ of strategies $f : C \multimap A$ and $g : C \multimap B$, where $C$ is an arbitrary game, is the strategy on $C \multimap A \mathbin{\&} B$ that plays by $f$ if O initiates a play by a move in $A$, and by $g$ otherwise.
The pairing $\langle a, b \rangle : A \mathbin{\&} B$ of strategies $a : A$ and $b : B$ is given by regarding $a$ and $b$ trivially as the ones on $\boldsymbol{1} \multimap A$ and $\boldsymbol{1} \multimap B$, respectively. 

These four constructions $\otimes$, $\oc $, $\multimap$ and $\&$ come from the corresponding ones in \emph{linear logic} \cite{girard1987linear,abramsky1994games}. 
Thus, in particular, the usual \emph{implication} (or the \emph{function space}) $\Rightarrow$ is recovered by \emph{Girard translation} \cite{girard1987linear}: $A \Rightarrow B \colonequals \oc A \multimap B$.

Girard translation makes it explicit that some functions refer to an input \emph{more than once} to produce an output, i.e., there are non-linear functions.
For instance, consider the game $(N \Rightarrow N) \Rightarrow N$, in which the following position is possible: 
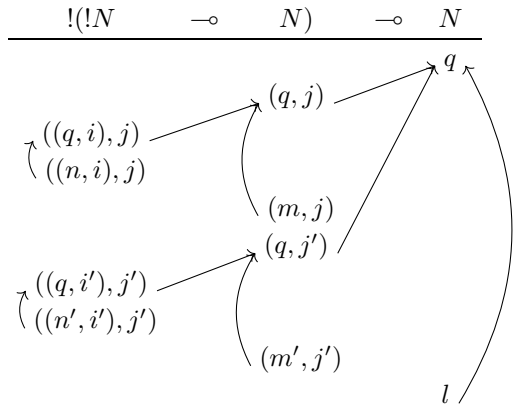
\begin{mathpar}
\begin{tabular}{ccccc}
$\oc (\oc N$ & $\multimap$ & $N)$ & $\multimap$ & $N$ \\ \hline 
&&&& \tikzmark{chigher1} $q$ \tikzmark{chigher9} \\
&&\tikzmark{chigher2} $(q, j)$ \tikzmark{dhigher1} && \\
\tikzmark{chigher3} $((q, i), j)$ \tikzmark{dhigher2} && && \\
\tikzmark{dhigher3} $((n, i), j)$&& && \\
&& \tikzmark{dhigher4} $(m, j)$ && \\
&&\tikzmark{chigher8} $(q, j')$ \tikzmark{dhigher5} && \\
\tikzmark{chigher7} $((q, i'), j')$ \tikzmark{dhigher6} && && \\
\tikzmark{dhigher7} $((n', i'), j')$&& && \\
&& \tikzmark{dhigher8} $(m', j')$ && \\
&&&& $l$ \tikzmark{dhigher9}
\end{tabular}
\begin{tikzpicture}[overlay, remember picture, yshift=.25\baselineskip]
\draw [->] ({pic cs:dhigher1}) to ({pic cs:chigher1});
\draw [->] ({pic cs:dhigher2}) to ({pic cs:chigher2});
\draw [->] ({pic cs:dhigher3}) [bend left] to ({pic cs:chigher3});
\draw [->] ({pic cs:dhigher4}) [bend left] to ({pic cs:chigher2});
\draw [->] ({pic cs:dhigher5}) to ({pic cs:chigher1});
\draw [->] ({pic cs:dhigher6}) to ({pic cs:chigher8});
\draw [->] ({pic cs:dhigher7}) [bend left] to ({pic cs:chigher7});
\draw [->] ({pic cs:dhigher8}) [bend left] to ({pic cs:chigher8});
\draw [->] ({pic cs:dhigher9}) [bend right] to ({pic cs:chigher9});
\end{tikzpicture}
\end{mathpar}
where $n, n', m, m', l, i, i', j, j' \in \mathbb{N}$, $i \neq i'$ and $j \neq j'$, which can be read as follows:
\begin{enumerate}
\item O's question $q$ for an output (`What is your output?');
\item P's question $(q, j)$ for an input function (`Wait, your first output please!');
\item O's question $((q, i), j)$ for an input (`What is your first input then?');
\item P's answer, say, $((n, i), j)$, to $((q, i), j)$ (`Here is my first input $n$.');
\item O's answer, say, $(m, j)$, to $(q, j)$ (`OK, then here is my first output $m$.');
\item P's question $(q, j')$ for an input function (`Your second output please!');
\item O's question $((q, i'), j')$ for an input (`What is your second input then?');
\item P's answer, say, $((n', i'), j')$, to $((q, i'), j')$ (`Here is my second input $n'$.');
\item O's answer, say, $(m', j')$, to $(q, j')$ (`OK, then here is my second output $m'$.');
\item P's answer, say, $l$, to $q$ (`Alright, my output is then $l$.').
\end{enumerate}
In this play, P asks O \emph{twice} about an input strategy $N \Rightarrow N$.
Clearly, such a play is not possible on the linear implication $(N \multimap N) \multimap N$ or $(N \Rightarrow N) \multimap N$.
The strategy $\mathrm{pazo} : (N \Rightarrow N) \Rightarrow N$ that computes the sum $f(0)+f(1)$ for a given function $f : \mathbb{N} \rightarrow \mathbb{N}$, for instance, plays as follows: 
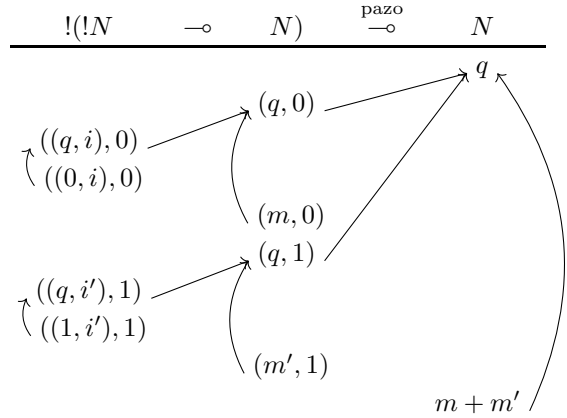
\begin{mathpar}
\begin{tabular}{ccccc}
$\oc (\oc N$ & $\multimap$ & $N)$ & $\stackrel{\mathrm{pazo}}{\multimap}$ & $N$ \\ \hline 
&&&& \tikzmark{chigher11} $q$ \tikzmark{chigher19} \\
&&\tikzmark{chigher12} $(q, 0)$ \tikzmark{dhigher11} && \\
\tikzmark{chigher13} $((q, i), 0)$ \tikzmark{dhigher12} && && \\
\tikzmark{dhigher13} $((0, i), 0)$ && && \\
&& \tikzmark{dhigher14} $(m, 0)$&& \\
&&\tikzmark{chigher18} $(q, 1)$ \tikzmark{dhigher15} && \\
\tikzmark{chigher17} $((q, i'), 1)$ \tikzmark{dhigher16} && && \\
\tikzmark{dhigher17} $((1, i'), 1)$&& && \\
&& \tikzmark{dhigher18} $(m', 1)$ && \\
&&&& $m+m'$ \tikzmark{dhigher19}
\end{tabular}
\begin{tikzpicture}[overlay, remember picture, yshift=.25\baselineskip]
\draw [->] ({pic cs:dhigher11}) to ({pic cs:chigher11});
\draw [->] ({pic cs:dhigher12}) to ({pic cs:chigher12});
\draw [->] ({pic cs:dhigher13}) [bend left] to ({pic cs:chigher13});
\draw [->] ({pic cs:dhigher14}) [bend left] to ({pic cs:chigher12});
\draw [->] ({pic cs:dhigher15}) to ({pic cs:chigher11});
\draw [->] ({pic cs:dhigher16}) to ({pic cs:chigher18});
\draw [->] ({pic cs:dhigher17}) [bend left] to ({pic cs:chigher17});
\draw [->] ({pic cs:dhigher18}) [bend left] to ({pic cs:chigher18});
\draw [->] ({pic cs:dhigher19}) [bend right] to ({pic cs:chigher19});
\end{tikzpicture}
\end{mathpar}
where $j = 0$ and $j' = 1$ are arbitrarily chosen, i.e., any $j, j' \in \mathbb{N}$ with $j \neq j'$ work.
Clearly, this computation is impossible on $(N \multimap N) \multimap N$ or $(N \Rightarrow N) \multimap N$.

Finally, let us recall that any strategy $f$ on the implication $\oc A \multimap B$ induces its \emph{promotion} $f^\dagger : \oc A \multimap \oc B$ such that if $f$ plays, for instance, as 
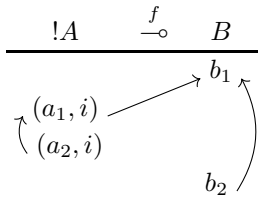
\begin{mathpar}
\begin{tabular}{ccc}
$\oc A$ & $\stackrel{f}{\multimap}$ & $B$ \\ \hline 
&&\tikzmark{cpromotion1} $b_1$ \tikzmark{cpromotion3} \\
\tikzmark{cpromotion2} $(a_1, i)$ \tikzmark{dpromotion1}&& \\
\tikzmark{dpromotion2} $(a_2, i)$ && \\
&&$b_2$ \tikzmark{dpromotion3}
\end{tabular}
\begin{tikzpicture}[overlay, remember picture, yshift=.25\baselineskip]
\draw [->] ({pic cs:dpromotion1}) to ({pic cs:cpromotion1});
\draw [->] ({pic cs:dpromotion2}) [bend left] to ({pic cs:cpromotion2});
\draw [->] ({pic cs:dpromotion3}) [bend right] to ({pic cs:cpromotion3});
\end{tikzpicture}
\end{mathpar}
then, for any $j \in \mathbb{N}$, the promotion $f^\dagger$ plays as
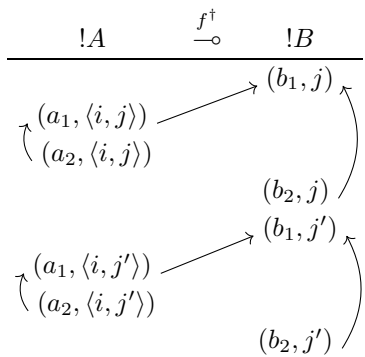
\begin{mathpar}
\begin{tabular}{ccc}
$\oc A$ & $\stackrel{f^\dagger}{\multimap}$ & $\oc B$ \\ \hline 
&&\tikzmark{cpromotion11} $(b_1, j)$ \tikzmark{cpromotion13} \\
\tikzmark{cpromotion12} $(a_1, \langle i, j \rangle)$ \tikzmark{dpromotion11} && \\
\tikzmark{dpromotion12} $(a_2, \langle i, j \rangle)$ && \\
&& $(b_2, j)$ \tikzmark{dpromotion13} \\
&& \tikzmark{cpromotion14} $(b_1, j')$ \tikzmark{cpromotion16} \\
\tikzmark{cpromotion15} $(a_1, \langle i, j' \rangle)$ \tikzmark{dpromotion14} && \\
\tikzmark{dpromotion15} $(a_2, \langle i, j' \rangle)$ && \\
&& $(b_2, j')$ \tikzmark{dpromotion16}
\end{tabular}
\begin{tikzpicture}[overlay, remember picture, yshift=.25\baselineskip]
\draw [->] ({pic cs:dpromotion11}) to ({pic cs:cpromotion11});
\draw [->] ({pic cs:dpromotion12}) [bend left] to ({pic cs:cpromotion12});
\draw [->] ({pic cs:dpromotion13}) [bend right] to ({pic cs:cpromotion13});
\draw [->] ({pic cs:dpromotion14}) to ({pic cs:cpromotion14});
\draw [->] ({pic cs:dpromotion15}) [bend left] to ({pic cs:cpromotion15});
\draw [->] ({pic cs:dpromotion16}) [bend right] to ({pic cs:cpromotion16});
\end{tikzpicture}
\end{mathpar}
where $\langle \_, \_ \rangle : \mathbb{N} \times \mathbb{N} \stackrel{\sim}{\to} \mathbb{N}$ is an arbitrarily fixed bijection, i.e., $f^\dagger$ plays as $f$ for each \emph{thread} of a position in $\oc A \multimap \oc B$ that corresponds to a position in $\oc A \multimap B$. 
The promotion $b^\dagger : \oc B$ of a strategy $b : B$ is given by regarding $b$ trivially as $b : \oc \boldsymbol{1} \multimap B$.

As already indicated, the category $\mathbb{G}$, whose objects are games and morphisms $A \rightarrow B$ are strategies on the implication $A \Rightarrow B$, is \emph{cartesian closed}, where its terminal object, products and exponential objects are given by the unit game $\boldsymbol{1}$, product $\&$ and implication $\Rightarrow$, respectively. 
Consequently, following the standard \emph{categorical} semantics of simple type theories \cite{jacobs1999categorical,lambek1988introduction}, the game semantics $\llbracket \_ \rrbracket_{\mathbb{G}}$ interprets unit, product and function types by $\boldsymbol{1}$, $\&$ and $\Rightarrow$, respectively, and a term $\mathsf{x_1 : A_1, x_2 : A_2, \dots, x_n : A_n \vdash b : B}$ by a strategy $\llbracket \mathsf{b} \rrbracket_{\mathbb{G}}$ on the implication $(\llbracket \mathsf{A_1} \rrbracket_{\mathbb{G}} \mathbin{\&} \llbracket \mathsf{A_2} \rrbracket_{\mathbb{G}} \mathbin{\&} \cdots \mathbin{\&} \llbracket \mathsf{A_n} \rrbracket_{\mathbb{G}}) \Rightarrow \llbracket \mathsf{B} \rrbracket_{\mathbb{G}}$, where we regard $\boldsymbol{1}$ as the $0$-ary product.

Last but not least, there is a constraint on strategies, called \emph{winning}, such that winning strategies correspond to \emph{proofs} in logic.
We leave the details of winning to Section~\ref{PartI} and here only recall one of its axioms: \emph{totality}.
A strategy $g : G$ is \emph{total} if there is (unique) $\boldsymbol{s}mn \in g$ for each odd-length position $\boldsymbol{s}m$ in $G$.
Hence, it is very similar to totality of partial maps. 
Conceptually, winning strategies must be total since intuitively a proof should not get `stuck.' 
For example, the natural number game $N$ has total strategies $\underline{n}$ for all $n \in \mathbb{N}$ and a non-total one $\bot \colonequals \{ \boldsymbol{\epsilon} \}$, and the unique strategy $\top$ on the unit game $\boldsymbol{1}$ is total.
Crucially, the unique strategy $\bot$ on the empty game $\boldsymbol{0}$ is not total (n.b., note the difference between $\boldsymbol{1}$ and $\boldsymbol{0}$), where recall that empty-type $\mathsf{0}$ is the type of \emph{falsity}, and hence there is no proof of $\mathsf{0}$. 
In this way, game semantics can exclude strategies that do not compute as proofs.

\subsection{Game semantics of intensional Martin-L\"{o}f type theory}
\label{IntroPartI}
Game semantics, including the one \`{a} la McCusker just reviewed, was applicable only to simple type theories (and \emph{polymorphic} ones \cite{abramsky1997semantics,abramsky2005game,hughes2000hypergame}), not to dependent ones, for a technical challenge until recently. 
Nevertheless, the present author has established game semantics of MLTT \cite{yamada2016game} based on the one \`{a} la McCusker; we shall take advantage of this game semantics for the present work.
There is another games-based denotational model of MLTT given by Abramsky et al.~\cite{abramsky2015games}, but our consistency proof seems unavailable for it; we shall come back to this point later.  

To convey our main idea, let us sketch how the game semantics \cite{yamada2016game} models sigma- and pi-types.
For simplicity, we consider a dependent type $\mathsf{x : C \vdash D \ type}$ with one variable $\mathsf{x}$.
For convenience, we also write $G$ for the set of all positions in a game $G$.
The game semantics models $\mathsf{x : C \vdash D \ type}$ as a certain family $D = (D(x))_{x : C}$ of games $D(x)$ indexed by strategies $x$ on the game $C$ that models the simple type $\mathsf{C}$. 

Then, in light of product $\&$ on games, which models a particular kind of sigma-types, viz., product types, it seems a natural idea to model the sigma-type $\mathsf{\Sigma_{x : C} D(x)}$ by a game $\Sigma(C, D)$ such that $\Sigma(C, D) \subseteq C \uplus \bigcup_{x : C}D(x)$, where $\uplus$ denotes disjoint union, and strategies on $\Sigma(C, D)$ are precisely the pairings $\langle c, d \rangle$ of strategies $c : C$ and $d : D(c)$.
However, this idea does not work due to the following two problems: 
\begin{enumerate}

\item Each game $G$, by definition, determines the set of all strategies on $G$;

\item It is impossible for P, when playing on such a game $\Sigma(C, D)$ if any, to \emph{fix} a strategy $c : C$, let alone a game $D(c)$ on the RHS, at the beginning of a play. 

\end{enumerate}
As an example of the first problem, consider a dependent type $\mathsf{x : N \vdash N_b \ type}$ such that canonical terms of the simple type $\mathsf{N_b(\underline{k})}$ for each $k \in \mathbb{N}$ are the numerals $\mathsf{\underline{n}}$ such that $n \leqslant k$. 
However, there is no game $G$ such that $G \subseteq N \uplus N$ and $\langle \underline{k}, \underline{n} \rangle : G \Leftrightarrow n \leqslant k$ for all $k, n \in \mathbb{N}$ since if such a game $G$ existed then $\langle \underline{0}, \underline{0} \rangle, \langle \underline{1}, \underline{1} \rangle : G$, which implies $\langle \underline{0}, \underline{1} \rangle : G$ by the definition of strategies on a game, a contradiction. 
Hence, no game can properly model the sigma-type $\mathsf{\Sigma_{x:N}N_b}$.

Let us also give an example of the second problem. 
Let $\mathsf{x : N \vdash List_N \ type}$ be the dependent type such that canonical terms of the simple type $\mathsf{List_N(\underline{k})}$ for each $k \in \mathbb{N}$ are $k$-lists of numerals, and assume that we interpret $\mathsf{List_N}$ as the family $\mathrm{List}_N$ of games such that $\mathrm{List}_N(\bot) \colonequals \mathrm{List}_N(\underline{0}) \colonequals \boldsymbol{1}$ and $\mathrm{List}_N(\underline{n+1}) \colonequals \mathrm{List}_N(\underline{n}) \otimes N$ for each $n \in \mathbb{N}$.
If there were a game that models the sigma-type $\mathsf{\Sigma(N, List_N)}$ then for all $k, n_1, n_2, \dots, n_k \in \mathbb{N}$ the pairings $\langle \underline{k}, \underline{n_1} \otimes \underline{n_2} \otimes \cdots \otimes \underline{n_k} \rangle$ would be total strategies on the game; however, there is no such a game since in this case O may determine $x \in \mathbb{N}$, by his first move in each play, for the $x$-ary tensor $\otimes$ of $N$ on the RHS.
Note that totality matters here since MLTT is a formal system for (intuitionistic) \emph{logic}, and so strategies in game semantics of MLTT must be all total \cite{yamada2016game}. 

To solve these two problems, the previous work \cite{yamada2016game} reformulates strategies as \emph{deterministic games}, i.e., games in which P can play in only one way, called \emph{predicative (p-) strategies}, and then generalizes games to certain sets of p-strategies, called \emph{predicative (p-) games}, in which P first \emph{declares} a p-strategy in her mind before a play with O begins, and then O and P play in the declared p-strategy. 
We say that a p-strategy is \emph{on} a p-game if it is an element of the p-game. 

The point of the reformulation of strategies as p-strategies is that p-strategies are defined independently of p-games essentially by \emph{containing odd-length positions} as well, which in turn enables us to \emph{define p-games in terms of p-strategies}. 
Note that in contrast we cannot define a strategy without specifying its underlying game since a strategy does not define its odd-length positions; also, a game determines the set of all strategies on the game just by definition.
Note also that games and strategies are transformed trivially into p-games and p-strategies, respectively: A strategy $\sigma : G$ is mapped to the p-strategy $P(\sigma)$ whose odd-length positions are those in $G$, and the game $G$ to the p-game $P(G)$ whose elements are the p-strategies $P(\sigma)$. 
Hence, we regard games and strategies as p-games and p-strategies, respectively, as well.
It is also easy to lift constructions on games and strategies to those on p-games and p-strategies. 
Then, the dependence of strategies on games explains why there is no game that models the sigma-type $\mathsf{\Sigma_{x:N}N_b(x)}$, but the p-game $\Sigma(N, N_b) \colonequals \{ \langle \bot, \top \rangle \} \cup \bigcup_{k \in \mathbb{N}} N_k$, where $N_k \colonequals \{ \langle \underline{k}, \underline{n} \rangle \mid k, n \in \mathbb{N}, n \leqslant k \, \}$, \emph{does} (n.b., we include the pairing $\langle \bot, \top \rangle$ into $\Sigma(N, N_b)$ for the \emph{downward completeness} axiom on p-games; see Definition~\ref{DefPredicativeGames}). 
More generally, p-games and p-strategies solve the first problem.
For instance, plays in $\Sigma(N, N_b)$ by the p-strategy $\langle \underline{7}, \underline{3} \rangle \in \Sigma(N, N_b)$ look like
\begin{mathpar}
\begin{tabular}{ccc}
$\Sigma(N,$ & & $N_b)$ \\ \cline{1-3}
& $q_{\Sigma(N, N_b)}$ & \\
& $\langle \underline{7}, \underline{3} \rangle$ & \\
\tikzmark{csigma1} $q$&& \\
\tikzmark{dsigma1} $7$&&
\end{tabular}
\begin{tikzpicture}[overlay, remember picture, yshift=.25\baselineskip]
\draw [->] ({pic cs:dsigma1}) [bend left] to ({pic cs:csigma1});
\end{tikzpicture}
\and
\begin{tabular}{ccc}
$\Sigma(N,$ & & $N_b)$ \\ \cline{1-3}
&$q_{\Sigma(N, N_b)}$& \\
&$\langle \underline{7}, \underline{3} \rangle$& \\
&&\tikzmark{csigma2} $q$ \\
&&\tikzmark{dsigma2} $3$
\end{tabular}
\begin{tikzpicture}[overlay, remember picture, yshift=.25\baselineskip]
\draw [->] ({pic cs:dsigma2}) [bend left] to ({pic cs:csigma2});
\end{tikzpicture}
\end{mathpar}
where \emph{Judge (J)}\footnote{The game semantics of MLTT \cite{yamada2016game} introduces J for a conceptual reason; technically, however, J is not necessary, and we may replace J's or J-moves by O-moves.} first asks P the question $q_{\Sigma(N, N_b)}$ (`What is your p-strategy?') and P answers it by the p-strategy $\langle \underline{7}, \underline{3} \rangle \in \Sigma(N, N_b)$ (`I declare the p-strategy $\langle \underline{7}, \underline{3} \rangle$!'), and then a play in the declared p-strategy $\langle \underline{7}, \underline{3} \rangle$ between P and O follows.
Although the declaration of a p-strategy is not strictly necessary in this case, it is clear why P cannot play by the p-strategy $\langle \underline{0}, \underline{1} \rangle$ on the p-game $\Sigma(N, N_b)$: It is because $\langle \underline{0}, \underline{1} \rangle \notin \Sigma(N, N_b)$ by the definition of $\Sigma(N, N_b)$.
Let us emphasize that the definition of the p-game $\Sigma(N, N_b)$ is made possible by reversing the traditional relation between games and strategies: P-games are defined in terms of p-strategies. 

Next, the declaration of p-strategies in p-games solves the second problem: The p-game $\Sigma(N, \mathrm{List}_N) \colonequals \{ \langle \bot, \top \rangle \} \cup \{ \langle \underline{k}, \underline{n_1} \otimes \underline{n_2} \otimes \cdots \otimes \underline{n_k} \rangle \mid k, n_1, n_2, \dots, n_k \in \mathbb{N} \, \}$ in fact models the sigma-type $\mathsf{\Sigma_{x:N} List_N(x)}$, where we again include the pairing $\langle \bot, \top \rangle$ into $\Sigma(N, \mathrm{List}_N)$ for the axiom on p-games.
Typical plays in $\Sigma(N, \mathrm{List}_N)$ look like
\begin{mathpar}
\begin{tabular}{ccc}
$\Sigma(N,$ & & $\mathrm{List}_N)$  \\ \cline{1-3}
& $q_{\Sigma(N, \mathrm{List}_N)}$ & \\
& $\langle \underline{2}, \underline{1} \otimes \underline{3} \rangle$ & \\
\tikzmark{csigma71} $q$&& \\
\tikzmark{dsigma71} $2$&& \\
&& \\
&& 
\end{tabular}
\begin{tikzpicture}[overlay, remember picture, yshift=.25\baselineskip]
\draw [->] ({pic cs:dsigma71}) [bend left] to ({pic cs:csigma71});
\end{tikzpicture}
\and
\begin{tabular}{ccccc}
$\Sigma(N,$ & & & &$\mathrm{List}_N)$  \\ \cline{1-5} 
$q_{\Sigma(N, \mathrm{List}_N)}$&& \\
$\langle \underline{2}, \underline{1} \otimes \underline{3} \rangle$&&&& \\
&&&&\tikzmark{csigma72} $q$ \\
&&&&\tikzmark{dsigma72} $3$ \\
&&\tikzmark{csigma73} $q$ \\
&&\tikzmark{dsigma73} $1$ 
\end{tabular}
\begin{tikzpicture}[overlay, remember picture, yshift=.25\baselineskip]
\draw [->] ({pic cs:dsigma72}) [bend left] to ({pic cs:csigma72});
\draw [->] ({pic cs:dsigma73}) [bend left] to ({pic cs:csigma73});
\end{tikzpicture}
\end{mathpar}
where the declaration of the p-strategy $\langle \underline{2}, \underline{1} \otimes \underline{3} \rangle \in \Sigma(N, \mathrm{List}_N)$ \emph{fixes} the underlying game on the RHS (n.b., a p-strategy is a game, and so it specifies its odd-length positions as well) so that O must play on the $2$-ary tensor $N \otimes N$ there.  
In this way, the p-strategy $\langle \underline{2}, \underline{1} \otimes \underline{3} \rangle$ is a \emph{total} one on the p-game $\Sigma(N, \mathrm{List}_N)$.
Intuitively, recalling that the generalization of simple type theories to dependent ones (or that of propositional logic to predicate logic) is made by introducing dependent types (or predicates) that refer to \emph{individuals}, i.e., terms (or proofs), we may understand the declaration of p-strategies in p-games as (part of) the game-semantic counterpart of the generalization because it enables P in p-games to refer to p-strategies. 

Moreover, the game semantics \cite{yamada2016game} defines the p-game $\Pi(C, D)$ that models the pi-type $\mathsf{\Pi_{x:C}D(x)}$ as follows: A p-strategy on $\Pi(C, D)$ is the union $\phi \colonequals \bigcup_{c \in C} \Phi_c$ on a family $\Phi = (\Phi_c)_{c \in C}$ of p-strategies $\Phi_c$ indexed by p-strategies $c \in C$ such that
\begin{enumerate}

\item For each $c \in C$, we have $\Phi_c \in c \Rightarrow d$ for some p-strategy $d \in D(c)$, where we regard the game $c \Rightarrow d$ as a p-game (as already remarked);

\item For any pair $c, \tilde{c} \in C$, the components $\Phi_{c}$ and $\Phi_{\tilde{c}}$ compute in the same manner at the same odd-length position $\boldsymbol{s}m \in \Phi_{c} \cap \Phi_{\tilde{c}}$ (i.e., $\Phi_{c}(\boldsymbol{s}m) \simeq \Phi_{\tilde{c}}(\boldsymbol{s}m)$),

\end{enumerate} 
where we write $x \downarrow$ if an element $x$ is defined, and $x \uparrow$ otherwise, and let $\simeq$ denote the \emph{Kleene equality}, i.e., $x \simeq y \stackrel{\mathrm{df. }}{\Leftrightarrow} (x \downarrow \wedge \ y \downarrow \wedge \ x = y) \vee (x \uparrow \wedge \ y \uparrow)$.
The first axiom on the p-strategy $\phi \in \Pi(C, D)$ is to ensure that $\phi$ respects the type dependency of the pi-type $\mathsf{\Pi_{x:C}D(x)}$, i.e., $\phi \circ c^\dagger \in D(c)$ for each $c \in C$, where the composition $\phi \circ c^\dagger$ is given by trivially identifying p-strategies on $C$ with those on $\boldsymbol{1} \Rightarrow C$.  
On the other hand, the second axiom is to guarantee determinacy of $\phi$ so that it is a well-defined p-strategy.
Moreover, the second axiom ensures that $\phi$ may inspect an input $c \in C$ only gradually by a finite interaction between $\phi$ and $c$ (just as in traditional game semantics), which makes the interpretation of pi-types very natural as semantics of computation (e.g., the extension of $\phi$ forms a \emph{continuous} function \cite{yamada2016game}).

To summarize, the game semantics \cite{yamada2016game} models MLTT by reformulating strategies as p-strategies and games as p-games, and then allowing P to \emph{control possible plays for O} by her `initial protocol' with J. 
In particular, the p-games $\Pi (\Sigma (X, Y), Z)$ and $\Pi (X, \Pi (Y, Z))$ for any (families of) p-games $X$, $Y$ and $Z$ such that the sigma and the pi constructions make sense coincide up to `tags' for disjoint union; hence, the game semantics trivially validates the $\xi$-rule.
Let us call p-games that model sigma- and pi-types \emph{sigma p-games} and \emph{pi p-games}, respectively. 
It is easy to see that sigma and pi p-games generalize product and function games, respectively. 

\begin{remark}
In the previous work \cite{yamada2016game}, strategies and p-strategies in the sense given above are rather called \emph{skeletons} and \emph{tree (t-) skeletons}, respectively, and strategies and p-strategies are slightly more abstract concepts, which follows \cite{abramsky2009game}.
We employ this abuse of the terminologies here for simplicity, but we correct it in Section~\ref{PartI}.
\end{remark}

\subsection{Game semantics is not directly applicable to formal Church's thesis}
Let us next apply the game semantics of MLTT \cite{yamada2016game} to CT (\ref{CTinMLTT}).
Recall that it models \emph{empty-type} $\mathsf{0}$ by the p-game $\boldsymbol{0} \colonequals \{ \bot \}$, called the \emph{empty p-game}, where $\bot \colonequals \{ \boldsymbol{\epsilon}, q \}$. 
It then models the T-predicate $\mathsf{x : N, y : N, z : N \vdash T(x, y, z) \ type}$ by the family $\mathcal{T} = (\mathcal{T}(\langle x, y, z \rangle))_{x, y, z \in N}$ of p-games $\mathcal{T}(\langle x, y, z \rangle)$, where $\langle x, y, z \rangle$ is the evident iteration of pairings, given by $\mathcal{T}(\langle x, y, z \rangle) \colonequals \boldsymbol{1}$ if $x = \underline{e}$, $y = \underline{n}$, $z = \underline{c}$ for some $e, n, c \in \mathbb{N}$, and the triple $(e, n, c)$ satisfies the T-predicate relation, and $\mathcal{T}(\langle x, y, z \rangle) \colonequals \boldsymbol{0}$ otherwise.
Also, it models the result-extracting function $\mathsf{z : N \vdash U(z) : N}$ by a p-strategy $\mu \in N \Rightarrow N$ such that the computation $\underline{n} \in N \mapsto \mu \circ \underline{n}^\dagger \in N$ matches the function in the evident sense. 
Finally, it models an \emph{Id-type} $\mathsf{Id_A}$, where let $\mathsf{A}$ be a simple type for simplicity, by the family $\mathrm{Id}_A = (\mathrm{Id}_A(\langle a_1, a_2 \rangle))_{a_1, a_2 \in A}$ of p-games $\mathrm{Id}_A(\langle a_1, a_2 \rangle)$, called \emph{Id p-games} on $A$, where $A$ is the p-game that interprets the simple type $\mathsf{A}$, given by $\mathrm{Id}_A(\langle a_1, a_2 \rangle) \colonequals \boldsymbol{1}$ if $a_1 = a_2$, and $\mathrm{Id}_A(\langle a_1, a_2 \rangle) \colonequals \boldsymbol{0}$ otherwise. 

We are now able to see how the p-game $\mathcal{CT}$ that interprets CT (\ref{CTinMLTT}) by the game semantics \cite{yamada2016game} looks like. 
Clearly, the only nontrivial point in giving a p-strategy on $\mathcal{CT}$ is the LHS of the first occurrence of a sigma p-game, i.e., a p-strategy on the p-game $(N \Rightarrow N) \Rightarrow N$ that outputs a p-strategy $\underline{e} \in N$ from an input p-strategy $\phi \in N \Rightarrow N$ given by O such that $e \in \mathbb{N}$ realizes $\phi$.
Recall that we have defined realizers for the extension $\mathbb{N} \rightarrow \mathbb{N}$ of $\phi$ to be realizers for $\phi$ and implement the T-predicate and the result-extracting function with respect to such realizers; hence, a realizer $e$ for $\phi$, if any, works as a witness of the LHS of the sigma p-game in $\mathcal{CT}$.

However, there is no such a p-strategy on $(N \Rightarrow N) \Rightarrow N$ because O may play on the domain $N \Rightarrow N$ as a \emph{non-recursive} function $f : \mathbb{N} \rightarrow \mathbb{N}$ so that there is no natural number that realizes $f$.
Strictly speaking, we should say that the extension of the play by O on $N \Rightarrow N$ is $f$, but let us keep using this abuse of terminologies in the rest of this introduction. 
Another problem is that O may play as a \emph{partial} function $\mathbb{N} \rightharpoonup \mathbb{N}$ on the domain.
Unfortunately, the game semantics \cite{yamada2016game} refutes CT even if we somehow manage to restrict plays by O to total recursive ones since positions in games are \emph{finite}, i.e., it is in general impossible for P or a p-strategy to completely identify a given p-strategy on $N \Rightarrow N$ by a finite interaction with it.

\subsection{Our solution: realizability \`{a} la game semantics}
\label{GameSemanticRealizability}
Then, how can we model MLTT plus CT?
Our solution is to \emph{limit O's plays as inputs to total recursive ones} and require him to \emph{exhibit a realizer} for his input computation at the very beginning of each play. 
The resulting, modified game semantics, which we call \emph{realizability \`{a} la game semantics}, models MLTT in a way similar to the game semantics \cite{yamada2016game} since both do not assign a choice of a realizer to the interpretation of each term. 
On the other hand, the main difference between the two models of MLTT is that only realizability \`{a} la game semantics validates CT since it requires O to exhibit a realizer for each input $N \Rightarrow N$ of the p-game $\mathcal{CT}$, and so there is trivially a p-strategy on $\mathcal{CT}$ that essentially `copy-cats' the realizer given by O.
Crucially, O may supply the realizers since game semantics treats O as an `oracle' endowed with an unlimited computational power \cite[Section~1]{yamada2019game}. 
Hence, realizability \`{a} la game semantics models MLTT plus CT. 
Also, it refutes empty-type as it requires p-strategies on codomains to be total too, proving consistency of MLTT with CT.

Technically, we implement our solution by modifying p-strategies on the pi p-game $\Pi(C, D)$ (Section~\ref{IntroPartI}) into the \emph{disjoint union} $\uplus \psi \colonequals \uplus_{e \in \mathscr{R}_{\mathrm{wr}}(C)} \psi_e$ on a nonempty family $\psi = (\psi_e)_{e \in \mathscr{R}_{\mathrm{wr}}(C)}$ of p-strategies $\psi_e \in r(e) \Rightarrow r \circ \pi_{\psi}(e)$, where 
\begin{itemize}

\item $\mathscr{R}_{\mathrm{wr}}(C) \subseteq \mathbb{N}$ is the set of all realizers for winning, recursive p-strategies $c \in C$; 

\item $r(e)$ is the (necessarily winning and recursive) p-strategy realized by $e$, and $\pi_{\psi}$ is an `effective' map $e \in \mathscr{R}_{\mathrm{wr}}(C) \mapsto \pi_{\psi}(e) \in \mathscr{R}_{\mathrm{wr}}(D(r(e)))$ assigned to $\psi$,

\end{itemize}
for which we encode all positions by natural numbers,\footnote{Positions encodable by natural numbers suffice for our model of MLTT plus CT.} define \emph{recursive} p-strategies accordingly with respect to the functional computation $\boldsymbol{s}m \in p \mapsto \boldsymbol{s}mn \in p$ of p-strategies $p$ in the standard sense \cite{rogers1967theory,cutland1980computability}, and fix \emph{realizers} for recursive p-strategies.  
Of course, we encode the `tags' for the disjoint union also by natural numbers. 
The `effectivity' of the associated map $\pi_{\psi}$ is to preserve recursiveness of these disjoint unions $\uplus \psi$ under \emph{composition} (Lemma~\ref{LemWellDefinedDoWRWLIs}). 
On the other hand, we simply define the p-strategy $\uplus \psi$ to be the trivial one $\{ \boldsymbol{\epsilon} \}$ if $\psi$ is empty.

\begin{remark}
Strictly speaking, a realizer $e$ in general does not uniquely determine a p-strategy $p$ realized by $e$ since a p-strategy contains odd-length positions as well. 
Therefore, we actually define $r(e)$ to be the \emph{union} of all p-strategies on $C$ realized by $e$; see the remark that comes immediately below Definition~\ref{DefDoWRWLIs}.
\end{remark}

We take the disjoint union of $\psi$ for $\uplus \psi$ since the union of $\psi$ may be indeterministic for the lack of the second axiom on p-strategies on original pi p-games. 
Also, the lack implies that the extension of $\uplus \psi$ may not be continuous. 
However, by the disjoint union, \emph{O must select a component $\psi_e$ and exhibit $e$ at his first move} in the modified $\Pi(C, D)$, where $e$ is a realizer for his play on $C$.
In contrast, this task is not imposed on P.
Let us write $\Pi$ for modified pi p-games in the rest of this introduction. 

On the other hand, $\uplus \psi$ itself is a p-strategy, not a realizer for it, and so modified pi p-games are as abstract as original ones. 
Thus, the isomorphism $\Pi (\Sigma (X, Y), Z) \cong \Pi (X, \Pi (Y, Z))$ holds for modified pi p-games as well. 
This point illustrates why realizability \`{a} la game semantics validates the $\xi$-rule.

As an example, consider the resulting implication $(N \Rightarrow N) \Rightarrow N$, an instance of a modified pi p-game.
First, a p-strategy on $N \Rightarrow N$ is the disjoint union  $\uplus \phi \colonequals \uplus_{e \in \mathscr{R}_{\mathrm{wr}}(N)} \phi_e$ on a family $\phi = (\phi_e)_{e \in \mathscr{R}_{\mathrm{wr}}(N)}$ of p-strategies $\phi_e \in r(e) \Rightarrow r \circ \pi_{\phi}(e)$.
Note that winning, recursive p-strategies on $N \Rightarrow N$ as well as realizers for them are already determined, and hence the implication $(N \Rightarrow N) \Rightarrow N$ is well-defined: A p-strategy on $(N \Rightarrow N) \Rightarrow N$ is the disjoint union $\uplus \psi \colonequals \uplus_{f \in \mathscr{R}_{\mathrm{wr}}(N \Rightarrow N)} \psi_f$ on a family $\psi = (\psi_f)_{f \in \mathscr{R}_{\mathrm{wr}}(N \Rightarrow N)}$ of p-strategies $\psi_f \in r(f) \Rightarrow r \circ \pi_{\psi}(f)$.
We write $\Rightarrow$ for the implication as an instance of the modified pi $\Pi$ in the rest of this introduction. 

Schematically, a typical play by our p-strategy $\uplus \mathrm{ct} \in \mathcal{CT}$ looks like
\begin{mathpar}
\begin{tabular}{ccccc}
$\Pi (\uplus \nu : N \Rightarrow N,$ & $\Sigma (x : N,$ & $\Pi (y : N,$ & $\Sigma (z : N,$ & $\mathcal{T}(\langle x, y, z \rangle) \mathbin{\&} \mathrm{Id}_N(\langle \uplus \mu \circ z^\dagger, \uplus \nu \circ x^\dagger \rangle)))))$ \\ \cline{1-5}
&&&&$q_{\mathcal{CT}}$ \\
&&&&$\uplus \mathrm{ct}$ \\
&\tikzmark{cpromotion100001} $q^{[e]}$&&& \\
&\tikzmark{dpromotion100001} $e^{[e]}$&&&
\end{tabular}
\begin{tikzpicture}[overlay, remember picture, yshift=.25\baselineskip]
\draw [->] ({pic cs:dpromotion100001}) [bend left] to ({pic cs:cpromotion100001});
\end{tikzpicture}
\end{mathpar}
where the superscript $(\_)^{[e]}$ ($e \in \mathbb{N}$) represents the `tag' for the leftmost occurrence of a pi p-game, and so $e$ is a realizer for an input p-strategy $\uplus \nu \in N \Rightarrow N$; we write $\uplus \mu \in N \Rightarrow N$ for the evident modification of $\mu$ with respect to the modification of pi p-games. 
In this play, the p-strategy $\uplus \mathrm{ct}$ never interacts with $\uplus \nu$ but `copy-cats' the realizer $e$ given by O.
Let us emphasize that such a play by $\uplus \mathrm{ct}$ would be impossible in the presence of the second axiom on p-strategies on original pi p-games. 

By the way, the play in the p-game $\mathcal{CT}$ by the p-strategy $\uplus \mathrm{ct}$ just given is far from typical plays in game semantics \cite{abramsky1999game,mccusker1998games} as $\uplus \mathrm{ct}$ computes the P-move $e^{[e]}$ \emph{without interacting with the input} $\uplus \nu$.
This point is the main reason why we call our model \emph{realizability \`{a} la game semantics}, rather than `game semantics \`{a} la realizability.'

Another typical play by the p-strategy $\uplus \mathrm{ct}$ looks like
\begin{mathpar}
\begin{tabular}{ccccc}
$\Pi (\uplus \nu : N \Rightarrow N,$ & $\Sigma (x : N,$ & $\Pi (y : N,$ & $\Sigma (z : N,$ & $\mathcal{T}(\langle x, y, z \rangle) \mathbin{\&}\mathrm{Id}_N(\langle \uplus \mu \circ z^\dagger, \uplus \nu \circ x^\dagger \rangle)))))$ \\ \cline{1-5}
&&&&$q_{\mathcal{CT}}$ \\
&&&&$\uplus \mathrm{ct}$ \\
&&&\tikzmark{cpromotion100002} $q^{[e'], [e]}$ \tikzmark{cpromotion200003}& \\
&&\tikzmark{cpromotion200002} $q^{[e'], [e]}$ \tikzmark{dpromotion100002}&& \\
&&\tikzmark{dpromotion200002} $n^{[e'], [e]}$&& \\
&&&$c^{[e'], [e]}$ \tikzmark{dpromotion200003}&
\end{tabular}
\begin{tikzpicture}[overlay, remember picture, yshift=.25\baselineskip]
\draw [->] ({pic cs:dpromotion100002}) to ({pic cs:cpromotion100002});
\draw [->] ({pic cs:dpromotion200002}) [bend left] to ({pic cs:cpromotion200002});
\draw [->] ({pic cs:dpromotion200003}) [bend right] to ({pic cs:cpromotion200003});
\end{tikzpicture}
\end{mathpar}
where the triple $e, n, c \in \mathbb{N}$ satisfies the T-predicate, and $(\_)^{[e']}$ ($e' \in \mathbb{N}$) is the `tag' for the rightmost occurrence of a pi p-game, and so $e'$ is a realizer for the p-strategy $\underline{n} \in N$ given by O. 
Because the p-games $\mathcal{T}(\langle \underline{e}, \underline{n}, \underline{c} \rangle)$ and $\mathrm{Id}_N(\langle \uplus \mu \circ \underline{c}^\dagger, \uplus \nu \circ \underline{n}^\dagger \rangle)$ both become the unit p-game $\boldsymbol{1}$ in this play, $\uplus \mathrm{ct}$ trivially validates them.

Finally, note that $\uplus \mathrm{ct}$ is a p-strategy, not a realizer. 
This point illustrates the fact that realizability \`{a} la game semantics is as abstract as the game semantics of MLTT \cite{yamada2016game}, and it models MLTT.
On the other hand, $\uplus \mathrm{ct}$ itself is \emph{recursive}, which means that we have established a \emph{constructive} model of MLTT plus CT.

\subsection{Our contribution and related work}
Our main contribution is to solve the long-standing open problem, i.e., consistency of MLTT with CT, by novel realizability \`{a} la game semantics that (even \emph{constructively}) resolves the dilemma between intensionality and extensionality.
Methodologically, the present work opens up a new way of applying game semantics to the study of constructive mathematics, while in the literature game semantics has been applied mostly in the context of \emph{full abstraction/completeness} problems \cite{curien2007definability}. 

Our consistency proof is based on the game semantics of MLTT \cite{yamada2016game}, especially its interpretation of pi-types by pi p-games.
Note in particular that its generalization of games and strategies to p-games and p-strategies, respectively, plays a crucial role for our modified pi p-games.
This technique seems unavailable for another games-based model of MLTT \cite{abramsky2015games} because it employs ordinary games and strategies, and interprets pi-types by induction on the lengths of positions.
Therefore, the previous work \cite{yamada2016game} is an indispensable stepping stone to the consistency proof.

As related work, Ishihara et al.~\cite{ishihara2018consistency} prove consistency of a modification of MLTT, called \emph{mTT}, with CT by realizability \`{a} la Kleene formalized within another formal system. 
The main difference between MLTT and mTT is that MLTT's congruence rules such as the $\xi$-rule are replaced with \emph{explicit substitutions} in mTT so that their realizability models mTT.
On the other hand, their realizability cannot model MLTT since it refutes the $\xi$-rule as already explained.
In other words, they circumvent the consistency problem of MLTT plus CT for the obstacle but instead prove consistency of mTT with CT, which nevertheless suffices for one of the authors' research programme: \emph{a minimalist foundation for constructive mathematics} \cite{maietti2005toward}.

Finally, our consistency result also contributes to the programme of a minimalist foundation for constructive mathematics since it implies that MLTT is available as an intensional level foundation of constructive mathematics for the programme.  

\subsection{Structure of the present article}
The rest of the present article is structured as follows.
We first recall basic definitions given in the previous work \cite{yamada2016game}, in particular p-games and p-strategies, in Section~\ref{PartI} as a technical preparation.
We then establish realizability \`{a} la game semantics that models MLTT plus CT and prove consistency of MLTT with CT as an immediate corollary in Section~\ref{MainConsistencyOfMLTTwithCT}.
Finally, we draw a conclusion and propose future work in Section~\ref{ConclusionAndFutureWork}.
In addition, Appendix~\ref{MLTT} presents the syntax of MLTT.

\begin{notation}
At the end of this introduction, let us introduce the following notations: 
\begin{itemize}

\item We employ bold small letters $\boldsymbol{s}, \boldsymbol{t}, \boldsymbol{u}, \boldsymbol{v}$, $\boldsymbol{w}$, etc. for sequences, and small letters $a, b, c, m, n, x, y$, etc. for elements of sequences;

\item We define $\overline{n} \colonequals \{ 1, 2, \dots, n \, \}$ for each $n \in \mathbb{N}^+ \colonequals \mathbb{N} \setminus \{ 0 \}$, and $\overline{0} \colonequals \emptyset$;

\item We often abbreviate a finite sequence $\boldsymbol{s} = (x_1, x_2, \dots, x_{|\boldsymbol{s}|})$ as $x_1 x_2 \dots x_{|\boldsymbol{s}|}$, where $|\boldsymbol{s}|$ denotes the \emph{length} (i.e., the number of elements) of $\boldsymbol{s}$, and write $\boldsymbol{s}(i)$, where $i \in \overline{|\boldsymbol{s}|}$, as another notation for $x_i$;

\item A \emph{concatenation} of sequences $\boldsymbol{s}$ and $\boldsymbol{t}$ is represented by the juxtaposition $\boldsymbol{s}\boldsymbol{t}$ (or written $\boldsymbol{s} . \boldsymbol{t}$) of them, and we write $a \boldsymbol{s}$, $\boldsymbol{t} b$, $\boldsymbol{u} c \boldsymbol{v}$ for $(a) \boldsymbol{s}$, $\boldsymbol{t} (b)$, $\boldsymbol{u} (c) \boldsymbol{v}$, etc.;


\item We write $\mathrm{Even}(\boldsymbol{s})$ (resp. $\mathrm{Odd}(\boldsymbol{s})$) if $\boldsymbol{s}$ is of even- (resp. odd-) length, and given a set $S$ of sequences and $P \in \{ \mathrm{Even}, \mathrm{Odd} \}$, we define $S^P \colonequals \{ \boldsymbol{s} \in S \mid P(\boldsymbol{s}) \, \}$;

\item We write $\boldsymbol{s} \preceq \boldsymbol{t}$ if $\boldsymbol{s}$ is a \emph{prefix} of $\boldsymbol{t}$, and $\mathrm{Pref}(S)$ for the set of all prefixes of elements in a set $S$ of sequences, i.e., $\mathrm{Pref}(S) \colonequals \{ \boldsymbol{s} \mid \exists \boldsymbol{t} \in S . \, \boldsymbol{s} \preceq \boldsymbol{t} \, \}$;


\item We define $X^* \colonequals \{ x_1 x_2 \dots x_n \mid n \in \mathbb{N}, \forall i \in \overline{n} . \, x_i \in X \, \}$ for each set $X$;

\item Given a map $f : A \to B$ and a subset $S \subseteq A$, we define $f \upharpoonright S : S \to B$ to be the \emph{restriction} of $f$ to $S$, and $f^\ast : A^\ast \to B^\ast$ by $f^\ast(a_1 a_2 \dots a_n) \colonequals f(a_1) f(a_2) \dots f(a_n) \in B^\ast$ for all $a_1 a_2 \dots a_n \in A^\ast$;


\item Given sets $X_1, X_2, \dots, X_n$, and a natural number $i \in \overline{n}$, we write $\pi^{(n)}_i$ or $\pi_i$ for the \emph{$i^{\text{th}}$-projection (map)} $X_1 \times X_2 \times \cdots \times X_n \rightarrow X_i$. 

\end{itemize}
\end{notation}

\section{Review: predicative games and predicative strategies}
\label{PartI}
First, we review the game semantics of MLTT \cite{yamada2016game} since our consistency proof is based on it, where we mostly focus on the basic definitions necessary for the present work. 
See the original article \cite{yamada2016game} for more details and explanations. 

We first recall key preliminary concepts such as \emph{arenas}, \emph{legal positions}, \emph{games} and \emph{tree skeletons} in Section~\ref{ArenasLegalPositionsGamesAndTreeSkeletons}, and \emph{consistency/completeness} of tree skeletons and the \emph{universal identification} in Section~\ref{ConsistencyAndCompletenessOfTreeSkeletonsAndUniversalIdentification}.
Then finally, we recall the central notions of \emph{p-games} and \emph{p-strategies} in Section~\ref{PredicativeGamesAndPredicativeStrategies}. 

\subsection{Arenas, legal positions, games and tree skeletons}
\label{ArenasLegalPositionsGamesAndTreeSkeletons}
We first need to confess our simplification of game semantics in the introduction: To be precise, p-strategies are deterministic games \emph{up to inessential details of `tags'} for disjoint union on sets of moves, and a deterministic game (on the nose) is called a \emph{tree (t-) skeleton} \cite{yamada2016game}. 
More accurately, a p-strategy is the union of all equivalent t-skeletons modulo `tags.' 
Also, what is described as strategies in the introduction is, strictly, \emph{skeletons}, and a strategy is the union of all equivalent skeletons modulo `tags.' 
Such a complication is necessary for game semantics to match the abstraction degree of terms in type theories. 
We henceforth switch to the precise terminologies. 

We then start with recalling \emph{t-skeletons}, for which it makes sense to first introduce a more general concept of \emph{games}.
But before that, we need to recall two preliminary concepts: \emph{arenas} and \emph{legal positions}. 
An arena defines the basic components of a game, which in turn induces legal positions of the arena that specify the basic rules of the game in the sense that each position of the game must be legal. 

\begin{definition}[Moves \cite{yamada2016game}]
\label{DefMoves}
Let us fix, throughout the present work, arbitrary pairwise distinct symbols $\mathsf{O}$, $\mathsf{P}$, $\mathsf{Q}$ and $\mathsf{A}$, and call them \emph{\bfseries labels}.
A \emph{\bfseries move} is any triple $m^{xy} \colonequals (m, x, y)$ such that $x \in \{ \mathsf{O}, \mathsf{P} \}$ and $y \in \{ \mathsf{Q}, \mathsf{A} \}$, for which we often abbreviate $m^{xy}$ as $m$, and instead define $\lambda(m) \colonequals xy$, $\lambda^\mathsf{OP}(m) \colonequals x$ and $\lambda^\mathsf{QA}(m) \colonequals y$.
A move $m$ is called an \emph{\bfseries Opponent (O-) move} if $\lambda^\mathsf{OP}(m) = \mathsf{O}$, a \emph{\bfseries Player (P-) move} if $\lambda^\mathsf{OP}(m) = \mathsf{P}$, a \emph{\bfseries question} if $\lambda^\mathsf{QA}(m) = \mathsf{Q}$, and an \emph{\bfseries answer} if $\lambda^\mathsf{QA}(m) = \mathsf{A}$.
\end{definition}

\begin{definition}[Arenas \cite{abramsky1999game,yamada2016game,mccusker1998games}]
\label{DefArenas}
An \emph{\bfseries arena} is a pair $G = (M_G, \vdash_G)$ such that
\begin{itemize}
\item $M_G$ is a set of moves;

\item $\vdash_G$ is a subset of $(\{ \star \} \cup M_G) \times M_G$, where $\star$ (or represented more precisely by $\star_G$) is an arbitrarily fixed element such that $\star \not \in M_G$, called the \emph{\bfseries enabling relation}, that satisfies
\begin{itemize}

\item \textsc{(E1)} If $\star \vdash_G m$ then $\lambda (m) = \mathsf{OQ}$;

\item \textsc{(E2)} If $m \vdash_G n$ and $\lambda^\mathsf{QA} (n) = \mathsf{A}$ then $\lambda^\mathsf{QA} (m) = \mathsf{Q}$;

\item \textsc{(E3)} If $m \vdash_G n$ and $m \neq \star$ then $\lambda^\mathsf{OP} (m) \neq \lambda^\mathsf{OP} (n)$.

\end{itemize}

\end{itemize}

A move $m \in M_G$ of $G$ is called \emph{\bfseries initial} if $\star \vdash_G m$, and \emph{\bfseries non-initial} otherwise.
We define the subset $M_G^{\mathrm{Init}} \colonequals \{ m \in M_G \mid \star \vdash_G m \, \} \subseteq M_G$
\end{definition}

That is, an arena $G$ is to specify moves in a game, each of which is O's/P's question/answer, and which move $n$ can be performed for each move $m$ during a play in the game in terms of the relation $m \vdash_G n$ (see Definition~\ref{DefJSequences} for more on this point), where $\star \vdash_G m$ means that O can initiate a play by $m$ in the game.

The axioms E1, E2 and E3 are then to be read as follows:
\begin{itemize}

\item E1 sets the convention that an initial move must be O's question;

\item E2 states that an answer must be performed for a question;

\item E3 says that an O-move must be performed for a P-move, and vice versa.

\end{itemize}

We shall later focus on \emph{well-founded} arenas:
\begin{definition}[Well-founded arenas \cite{clairambault2010totality}]
\label{DefWellFoundedArenas}
An arena $G$ is \emph{\bfseries well-founded} if so is the enabling relation $\vdash_G$ \emph{downwards}, i.e., there is no countably infinite sequence $(m_i)_{i \in \mathbb{N}}$ of moves $m_i \in M_G$ such that $\star \vdash_G m_0$ and $m_i \vdash_G m_{i+1}$ for all $i \in \mathbb{N}$.
\end{definition}

Let us proceed to review \emph{legal positions}, for which recall first that a legal position is a certain finite sequence of moves equipped with a \emph{pointer} from later occurrences to earlier ones of the sequence. 
The idea is that each non-initial occurrence in a legal position must be performed for a specific previous occurrence, and such a pair of occurrences is specified by a pointer. 
Technically, pointers are introduced in order to distinguish similar yet different computations \cite{abramsky1999game,curien2006notes}. 

We call a finite sequence of moves together with a pointer a \emph{justified (j-) sequence}; a legal position is a particular kind of a j-sequence.
\begin{definition}[Justified sequences \cite{abramsky1999game,yamada2016game,mccusker1998games}]
\label{DefJSequences}
An \emph{\bfseries occurrence} in a finite sequence $\boldsymbol{s}$ is a pair $(\boldsymbol{s}(i), i)$ such that $i \in \overline{|\boldsymbol{s}|}$. 
A \emph{\bfseries justified (j-) sequence} is a pair $\boldsymbol{s} = (\boldsymbol{s}, \mathcal{J}_{\boldsymbol{s}})$ of a finite sequence $\boldsymbol{s}$ of moves and a map $\mathcal{J}_{\boldsymbol{s}} : \overline{|\boldsymbol{s}|} \rightarrow \{ 0 \} \cup \overline{|\boldsymbol{s}|-1}$ such that $0 \leqslant \mathcal{J}_{\boldsymbol{s}}(i) < i$ for all $i \in \overline{|\boldsymbol{s}|}$, called the \emph{\bfseries pointer} of the j-sequence. 
Each occurrence $(\boldsymbol{s}(i), i)$ is \emph{\bfseries initial} (resp. \emph{\bfseries non-initial}) in $\boldsymbol{s}$ if $\mathcal{J}_{\boldsymbol{s}}(i) = 0$ (resp. otherwise).

A \emph{\bfseries justified (j-) sequence in an arena} $G$ is a j-sequence $\boldsymbol{s}$ such that its elements are moves in $G$, and its pointer respects the enabling relation of $G$, i.e., it satisfies $\boldsymbol{s} \in M_G^\ast$ and $\forall i \in \overline{|\boldsymbol{s}|} . \, \big(\mathcal{J}_{\boldsymbol{s}}(i) = 0 \Rightarrow \star \vdash_G \boldsymbol{s}(i)\big) \wedge \big(\mathcal{J}_{\boldsymbol{s}}(i) \neq 0 \Rightarrow \boldsymbol{s}({\mathcal{J}_{\boldsymbol{s}}(i)}) \vdash_G \boldsymbol{s}(i)\big)$.
We write $\mathscr{J}_G$ for the set of all j-sequences in an arena $G$.
\end{definition}

\begin{convention}
We say that the occurrence $(\boldsymbol{s}({\mathcal{J}_{\boldsymbol{s}}(i)}), \mathcal{J}_{\boldsymbol{s}}(i))$ is the \emph{\bfseries justifier} of a non-initial one $(\boldsymbol{s}(i), i)$ in a j-sequence $\boldsymbol{s}$, and $(\boldsymbol{s}(i), i)$ is \emph{\bfseries justified} by $(\boldsymbol{s}({\mathcal{J}_{\boldsymbol{s}}(i)}), \mathcal{J}_{\boldsymbol{s}}(i))$. 
\end{convention}

\begin{definition}[Justified subsequences \cite{yamada2016game}]
\label{DefJSubsequences}
A \emph{\bfseries justified (j-) subsequence} of a j-sequence $\boldsymbol{s}$ is a j-sequence $\boldsymbol{t}$ such that $\boldsymbol{t}$ is a subsequence of $\boldsymbol{s}$, and for all $i, j \in \mathbb{N}$ $\mathcal{J}_{\boldsymbol{t}}(i) = j$ if and only if $\mathcal{J}_{\boldsymbol{s}}^n(i) = j$ for some $n \in \mathbb{N}$.
\end{definition}

\begin{convention}
We are henceforth casual about the distinction between moves and occurrences; by abuse of notation, we often keep the pointer $\mathcal{J}_{\boldsymbol{s}}$ of each j-sequence $\boldsymbol{s} = (\boldsymbol{s}, \mathcal{J}_{\boldsymbol{s}})$ implicit and abbreviate occurrences $(\boldsymbol{s}(i), i)$ in $\boldsymbol{s}$ as $\boldsymbol{s}(i)$.
Moreover, we often write $\mathcal{J}_{\boldsymbol{s}}(\boldsymbol{s}(i)) = \boldsymbol{s}(j)$ if $\mathcal{J}_{\boldsymbol{s}}(i) = j$ for all $i, j \in \mathbb{N}$.
\end{convention}

Next, we recall the `relevant part' or \emph{view} of the previous occurrences of each occurrence in a j-sequence, which is also fundamental for legal positions.
\begin{definition}[Views \cite{abramsky1999game,mccusker1998games,hyland2000full}] 
\label{DefViews}
The \emph{\bfseries Player (P-) view} $\lceil \boldsymbol{s} \rceil$ and the \emph{\bfseries Opponent (O-) view} $\lfloor \boldsymbol{s} \rfloor$ of a j-sequence $\boldsymbol{s}$ are respectively the j-subsequences of $\boldsymbol{s}$ defined by the following induction on the length $|\boldsymbol{s}|$ of $\boldsymbol{s}$: 
\begin{itemize}

\item $\lceil \boldsymbol{\epsilon} \rceil \colonequals \boldsymbol{\epsilon}$; 

\item $\lceil \boldsymbol{s} m \rceil \colonequals \lceil \boldsymbol{s} \rceil . m$ if $m$ is a P-move; 

\item $\lceil \boldsymbol{s} m \rceil \colonequals m$ if $m$ is initial;

\item $\lceil \boldsymbol{s} m \boldsymbol{t} n \rceil \colonequals \lceil \boldsymbol{s} \rceil . m n$ if $n$ is an O-move such that $m$ justifies $n$; 

\item $\lfloor \boldsymbol{\epsilon} \rfloor \colonequals \boldsymbol{\epsilon}$;

\item $\lfloor \boldsymbol{s} m \rfloor \colonequals \lfloor \boldsymbol{s} \rfloor . m$ if $m$ is an O-move; 

\item $\lfloor \boldsymbol{s} m \boldsymbol{t} n \rfloor \colonequals \lfloor \boldsymbol{s} \rfloor . m n$ if $n$ is a P-move such that $m$ justifies $n$. 

\end{itemize}

A \emph{\bfseries Player (P-) view} (resp. a \emph{\bfseries Opponent (O-) view}) refers to that of some j-sequence, and a \emph{\bfseries view} (of a j-sequence) to a P- or O-view (of the j-sequence).
\end{definition}

The idea behind the notion of views is as follows.
Given a nonempty j-sequence $\boldsymbol{s} m$ such that $m$ is a P- (resp. O-) move, the P-view $\lceil \boldsymbol{s} \rceil$ (resp. O-view $\lfloor \boldsymbol{s} \rfloor$) is intended to be the currently `relevant part' of the previous occurrences in $\boldsymbol{s}$ for P (resp. O). 
I.e., P (resp. O) is concerned only with the last occurrence of an O- (resp. P-) move, its justifier and that justifier's P- (resp. O-) view, which then recursively proceeds.
See \cite{curien2006notes,curien1998abstract} for an explanation of views in terms of their counterparts in syntax.

We are now ready to recall \emph{legal positions}:
\begin{definition}[Legal positions \cite{abramsky1999game,yamada2016game}]
\label{DefLegalPositions}
A \emph{\bfseries legal position} is a j-sequence $\boldsymbol{s}$ that satisfies
\begin{itemize}

\item \textsc{(Alternation)} If $\boldsymbol{s} = \boldsymbol{s_1} m n \boldsymbol{s_2}$, then $\lambda^\mathsf{OP} (m) \neq \lambda^\mathsf{OP} (n)$;


\item \textsc{(Visibility)} If $\boldsymbol{s} = \boldsymbol{t} m \boldsymbol{u}$ with $m$ non-initial, then $\mathcal{J}_{\boldsymbol{s}}(m)$ occurs in $\lceil \boldsymbol{t} \rceil$ if $m$ is a P-move, and in $\lfloor \boldsymbol{t} \rfloor$ otherwise.

\end{itemize}

A \emph{\bfseries legal position in an arena} $G$ is a legal position that is a j-sequence in $G$.
We write $\mathscr{L}_G$ for the set of all legal positions in $G$.
\end{definition}


As already stated, legal positions are to specify the basic rules of a game in the sense that each position in the game must be legal so that
\begin{itemize}

\item During a play in the game, O makes the first move by a question, and then P and O alternately perform moves (by alternation), where each non-initial move is performed for a specific previous occurrence, viz., its justifier; 


\item The justifier of each non-initial occurrence belongs to the `relevant part' or view of the previous occurrences (by visibility). 
\end{itemize}

Having reviewed arenas and legal positions, we are now able to recall \emph{games} and deterministic games called \emph{t-skeletons}: 
\begin{definition}[Games \cite{abramsky1999game,yamada2016game,mccusker1998games}]
A \emph{\bfseries game} is a set $G$ of legal positions, called \emph{\bfseries (valid) positions} in $G$, that satisfies 
\begin{itemize}

\item \textsc{(Tree)} The set $G$ is nonempty and \emph{prefix-closed} (i.e., $\boldsymbol{s}m \in G \Rightarrow \boldsymbol{s} \in G$);

\item \textsc{(Wfoud)} The arena $\mathrm{Arn}(G) \colonequals (M_G, \vdash_G)$ is well-founded, 

\end{itemize}
where $M_G \colonequals \{ \boldsymbol{s}(i) \mid \boldsymbol{s} \in G, i \in \overline{|\boldsymbol{s}|} \, \}$ and $\vdash_G \, \colonequals \{ (\star, \boldsymbol{s}(j)) \mid \boldsymbol{s} \in G, \mathcal{J}_{\boldsymbol{s}}(j) = 0 \, \}\cup \\ \{ (\boldsymbol{s}(i), \boldsymbol{s}(j)) \mid \boldsymbol{s} \in S, \mathcal{J}_{\boldsymbol{s}}(j)= i \, \}$.
A \emph{\bfseries subgame} of $G$ is a game $H$ such that $H \subseteq G$.
\end{definition}

Nonemptiness and prefix-closure of a game $G$ formulates the natural phenomenon that each nonempty `moment' or position has a previous `moment.'
The underlying arena $\mathrm{Arn}(G)$ of $G$ is well-founded so that we can impose \emph{winning}, more specifically \emph{noetherianity} (Definition~\ref{DefConstraintsOnTSkeletons}), on identities in the categories of games; see \cite{yamada2016game} for the details.
Also, note that every position in $G$ is a legal position in $\mathrm{Arn}(G)$.

\begin{definition}[Tree skeletons \cite{yamada2016game}]
\label{DefTreeSkeletons}
A \emph{\bfseries tree (t-) skeleton} is a game $\sigma$ that is \emph{deterministic}: $\boldsymbol{s}mn, \boldsymbol{s}mn' \in \sigma^{\mathrm{Even}} \Rightarrow \boldsymbol{s}mn = \boldsymbol{s}mn$.
A t-skeleton $\sigma$ is \emph{\bfseries on} a game $G$, written $\sigma :: G$, if it satisfies $\sigma \subseteq G$ and $(\boldsymbol{s}m \in G^{\mathrm{Odd}} \wedge \boldsymbol{s} \in \sigma) \Rightarrow \boldsymbol{s}m \in \sigma$.
\end{definition}

In other words, a t-skeleton on a game $G$ is a deterministic subgame $\sigma \subseteq G$ such that possible plays by O in $\sigma$ coincide precisely with those in $G$.
Therefore, such a t-skeleton $\sigma :: G$ describes for P \emph{how to play in $G$}.

Clearly, \emph{skeletons}, i.e., what is called strategies in Section~\ref{IntroGameSemantics}, on $G$ correspond bijectively to t-skeletons on $G$.
The main difference between the two is, however, that a skeleton needs its underlying game, but a t-skeleton does not. 

\begin{example}
\label{ExamplesOfGames}
The simplest game is the \emph{\bfseries unit game} $\boldsymbol{1} \colonequals \{ \boldsymbol{\epsilon} \}$.
There is only the trivial t-skeleton $\top \colonequals \{ \boldsymbol{\epsilon} \}$ on $\boldsymbol{1}$.

Another simple game is the \emph{\bfseries empty game} $\boldsymbol{0} \colonequals \{ \boldsymbol{\epsilon}, q^{\mathsf{OQ}} \}$, where $q$ is an arbitrarily fixed element.
There is only the unique t-skeleton $\bot \colonequals \{ \boldsymbol{\epsilon}, q \, \}$ on $\boldsymbol{0}$.

The \emph{\bfseries natural number game} $N$ is given by $N \colonequals \mathrm{Pref}(\{ q^{\mathsf{OQ}}n^{\mathsf{PA}} \mid n \in \mathbb{N} \, \})$, where $q$ justifies $n$.
T-skeletons on $N$ are $\bot \colonequals \{ \boldsymbol{\epsilon}, q \}$ and $\underline{n} \colonequals \{ \boldsymbol{\epsilon}, q, q n \, \}$ for each $n \in \mathbb{N}$.
\end{example}

At this point, recall that not every t-skeleton corresponds to a \emph{proof} \cite{abramsky1999game,yamada2016game}.
For instance, the empty-game $\boldsymbol{0}$ models the \emph{empty-type} or \emph{falsity} $\mathsf{0}$ \cite{yamada2016game}, and therefore the t-skeleton $\bot :: \boldsymbol{0}$ cannot be an interpretation of a proof.
This point matters for the present work since our consistency proof relies on a model of MLTT plus CT that \emph{does not inhabit} the empty-type; that is, we need to carve out t-skeletons that compute as proofs in such a way that the t-skeleton $\bot :: \boldsymbol{0}$ is excluded. 

The previous work \cite{yamada2016game} characterizes such t-skeletons for proofs as \emph{winning} ones:
\begin{definition}[Constraints on tree skeletons \cite{abramsky1999game,yamada2016game,hyland2000full,clairambault2010totality}]
\label{DefConstraintsOnTSkeletons}
A t-skeleton $\sigma$ is
\begin{itemize}

\item \emph{\bfseries Total} if it always responds: $\forall \boldsymbol{s} m \in \sigma^{\mathrm{Odd}} . \, \exists \boldsymbol{s} m n \in \sigma$;

\item \emph{\bfseries Innocent} if its computation depends only on P-views: $\forall \boldsymbol{s} m n \in \sigma^{\mathrm{Even}}, \boldsymbol{\tilde{s}} \tilde{m} \in \sigma^{\mathrm{Odd}} . \, \lceil \boldsymbol{s} m \rceil = \lceil \boldsymbol{\tilde{s}} \tilde{m} \rceil \Rightarrow \exists \boldsymbol{\tilde{s}} \tilde{m} \tilde{n} \in \sigma^{\mathrm{Even}} . \, \lceil \boldsymbol{s} m n \rceil = \lceil \boldsymbol{\tilde{s}} \tilde{m} \tilde{n} \rceil$;


\item \emph{\bfseries Noetherian} if there is no strictly increasing infinite sequence of elements in the set $\lceil \sigma \rceil \colonequals \{ \lceil \boldsymbol{s} \rceil \mid \boldsymbol{s} \in \sigma \, \}$ of all P-views in $\sigma$;

\item \emph{\bfseries Winning} if it is total, innocent and noetherian.

\end{itemize}
\end{definition}

\begin{example}
The t-skeletons $\top :: \boldsymbol{1}$ and $\underline{n} :: N$ for each $n \in \mathbb{N}$ are winning, while the t-skeleton $\bot :: \boldsymbol{0}$ is not even total, let alone winning, as desired. 
\end{example}

Intuitively, we regard winning t-skeletons as proofs (in classical logic) as follows.\footnote{We may further impose \emph{well-bracketing} \cite{abramsky1999game,yamada2016game,hyland2000full} on winning t-skeletons so that they would correspond to proofs in \emph{intuitionistic logic}. Nevertheless, it is not necessary for the present work, and therefore let us skip it for brevity.}
First, a proof should not get `stuck,' and so t-skeletons for proofs must be \emph{total}.
Next, recall that imposing \emph{innocence} on t-skeletons corresponds to excluding \emph{stateful} terms \cite{abramsky1999game}.
Since logic is concerned with truths of formulas, which are invariant with respect to `passage of time,' proofs should not depended on `states of arguments.' 
Thus, we impose innocence on t-skeletons for proofs.
Also, we impose \emph{noetherianity} on t-skeletons for proofs to handle infinite plays: If a play by an innocent, noetherian t-skeleton keeps growing infinitely then it cannot be P's `intention,' and so the play must be a `valid argument' or \emph{win} for P.
Technically, we need noetherianity since total t-skeletons are not closed under \emph{composition} but winning ones are \cite{abramsky1997semantics,yamada2016game,clairambault2010totality}.
Various \emph{full completeness} results in the literature \cite{abramsky1999game,abramsky2015games} indicate that winning is not only necessary but also sufficient as a characterization of t-skeletons for proofs. 


Finally, recall that we have already sketched constructions on games and skeletons \cite{abramsky1999game,mccusker1998games} in Section~\ref{IntroGameSemantics}, and constructions on t-skeletons are essentially the same as those on skeletons \cite{yamada2016game}.
Let us next present their mathematical formalizations:

\begin{convention}
For brevity, we omit `tags' for disjoint union $\uplus$ of sets of moves except the ones for exponential $\oc$. 
For instance, we write $x \in A \uplus B$ if $x \in A$ or $x \in B$; also, given relations $R_A \subseteq A \times A$ and $R_B \subseteq B \times B$, we write $R_A \uplus R_B$ for the relation on the disjoint union $A \uplus B$ such that $(x, y) \in R_A \uplus R_B \stackrel{\mathrm{df. }}{\Leftrightarrow} (x, y) \in R_A \vee (x, y) \in R_B$. 
\end{convention}

\begin{definition}[Constructions on arenas \cite{abramsky1999game,mccusker1998games,hyland2000full}]
Given arenas $A$ and $B$, we define arenas 
\begin{itemize}

\item $A \uplus B \colonequals (M_A \uplus M_B, \vdash_A \uplus \vdash_B)$; 

\item $A \multimap B \colonequals (\{ a^{\overline{x}y} \mid a^{xy} \in M_A \, \} \uplus M_B, \vdash_{A \multimap B})$, where $\overline{\mathsf{O}} \colonequals \mathsf{P}$, $\overline{\mathsf{P}} \colonequals \mathsf{O}$, $\star \vdash_{A \multimap B} m :\Leftrightarrow \star \vdash_B m$ and $m \vdash_{A \multimap B} n :\Leftrightarrow m \vdash_A n \vee m \vdash_B n \vee (\star \vdash_B m \wedge \star \vdash_A n)$; 

\item $\oc A \colonequals (\{ (a, i)^{x y} \mid a^{xy} \in M_A, i \in \mathbb{N} \, \}, \vdash_{\oc A})$, where  $\star \vdash_{\oc A} (a, i) :\Leftrightarrow \star \vdash_A a$ and $(a, i) \vdash_{\oc A} (a', i') :\Leftrightarrow i = i' \wedge a \vdash_A a'$.

\end{itemize}
\end{definition}

\begin{definition}[Tensor on games \cite{abramsky1999game,mccusker1998games}]
\label{DefTensorOfGames}
The \emph{\bfseries tensor} of games $G$ and $H$ is the game 
\begin{equation*}
G \otimes H \colonequals \{ \, \boldsymbol{s} \in \mathscr{L}_{\mathrm{Arn}(G) \uplus \mathrm{Arn}(H)} \mid \forall X \in \{ G, H \} . \, \boldsymbol{s} \upharpoonright X \in X \, \}, 
\end{equation*}
where $\boldsymbol{s} \upharpoonright X$ is the j-subsequence of $\boldsymbol{s}$ that consists of moves in $X$.
\end{definition}

\begin{definition}[Linear implication on games \cite{abramsky1999game,mccusker1998games}]
\label{DefLinearImplication}
The \emph{\bfseries linear implication} between games $G$ and $H$ is the game 
\begin{equation*}
G \multimap H \colonequals \{ \, \boldsymbol{s} \in \mathscr{L}_{\mathrm{Arn}(G) \multimap \mathrm{Arn}(H)} \mid \forall X \in \{ G, H \} . \, \boldsymbol{s} \upharpoonright X \in X \, \}.
\end{equation*}
\end{definition}

By the alternation axiom on legal positions (Definition~\ref{DefLegalPositions}), it is easy to see that only O (resp. P) may switch between games $G$ and $H$ during a play in the tensor $G \otimes H$ (resp. the linear implication $G \multimap H$) \cite{abramsky1997semantics}, which matches the description of tensor (resp. linear implication) on games given in Section~\ref{IntroGameSemantics}.

\begin{definition}[Product on games \cite{abramsky1999game,mccusker1998games}]
\label{DefProduct}
The \emph{\bfseries product} of games $G$ and $H$ is the game 
\begin{align*}
G \mathbin{\&} H &\colonequals \{ \, \boldsymbol{s} \in \mathscr{L}_{\mathrm{Arn}(G) \mathbin{\uplus} \mathrm{Arn}(H)} \mid (\boldsymbol{s} \upharpoonright G \in G \wedge \boldsymbol{s} \upharpoonright H = \boldsymbol{\epsilon}) \vee (\boldsymbol{s} \upharpoonright G = \boldsymbol{\epsilon} \wedge \boldsymbol{s} \upharpoonright H \in H) \, \}.
\end{align*}
\end{definition}

\begin{definition}[Exponential on games \cite{mccusker1998games}]
\label{DefExponential}
The \emph{\bfseries exponential} of a game $G$ is the game 
\begin{equation*}
\oc G \colonequals \{ \boldsymbol{s} \in \mathscr{L}_{\oc \mathrm{Arn}(G)} \mid \forall i \in \mathbb{N} . \, \boldsymbol{s} \upharpoonright i \in G \, \}, 
\end{equation*}
where $\boldsymbol{s} \upharpoonright i$ is the j-subsequence of $\boldsymbol{s}$ that consists of moves $(a, i)$, where $a \in M_G$, changed into $a$.
\end{definition}

\begin{definition}[Constructions on tree skeletons between games \cite{yamada2016game}]
\label{DefConstructionsOnT-Skeletons}
Given t-skeletons $\phi :: A \multimap B$, $\sigma :: C \multimap D$, $\tau :: A \multimap C$, $\psi :: B \multimap C$ and $\varphi :: \oc A \multimap B$ between games, we define
\begin{itemize}

\item The \emph{\bfseries tensor} $\phi \otimes \sigma :: A \otimes C \multimap B \otimes D$ of $\phi$ and $\sigma$ by
\begin{equation*}
\phi \otimes \sigma \colonequals \{ \, \boldsymbol{s} \in \mathscr{L}_{\mathrm{Arn}(A \otimes C \multimap B \otimes D)} \mid \boldsymbol{s} \upharpoonright A, B \in \phi, \boldsymbol{s} \upharpoonright C, D \in \sigma \, \},
\end{equation*}
where $\boldsymbol{s} \upharpoonright A, B$ (resp. $\boldsymbol{s} \upharpoonright C, D$) is the j-subsequence of $\boldsymbol{s}$ that consists of moves in $A$ or $B$ (resp. $C$ or $D$);

\item The \emph{\bfseries pairing} $\langle \phi, \tau \rangle :: A \multimap B \mathbin{\&} C$ of $\sigma$ and $\tau$ by
\begin{equation*}
\langle \phi, \tau \rangle \colonequals \{ \, \boldsymbol{s} \in \mathscr{L}_{\mathrm{Arn}(A \multimap B \mathbin{\&} C)} \mid (\boldsymbol{s} \upharpoonright A, B \in \phi \wedge \boldsymbol{s} \upharpoonright C = \boldsymbol{\epsilon}) \vee (\boldsymbol{s} \upharpoonright A, C \in \tau \wedge \boldsymbol{s} \upharpoonright B = \boldsymbol{\epsilon}) \, \};
\end{equation*}

\item The \emph{\bfseries composition} $\phi ; \psi :: A \multimap C$ (also written $\psi \circ \phi$) of $\phi$ and $\psi$ by
\begin{equation*}
\phi ; \psi \colonequals \{ \, \boldsymbol{s} \upharpoonright A, C \mid \boldsymbol{s} \in \phi \parallel \psi \, \},
\end{equation*}
where $\phi \parallel \psi \colonequals \{ \, \boldsymbol{s} \in \mathscr{J}_{\mathrm{Arn}(((A \multimap B^{[0]}) \multimap B^{[1]}) \multimap C)} \mid \boldsymbol{s} \upharpoonright A, B^{[0]} \in \phi, \boldsymbol{s} \upharpoonright B^{[1]}, C \in \psi, \boldsymbol{s} \upharpoonright B^{[0]}, B^{[1]} \in \mathrm{cp}_B \, \}$, the superscripts $(\_)^{[i]}$ are to distinguish the two copies of $B$, and $\mathrm{cp}_B \colonequals \{ \, \boldsymbol{t} \in B^{[0]} \multimap B^{[1]} \mid \forall \boldsymbol{u} \preceq \boldsymbol{t} . \, \mathrm{Even}(\boldsymbol{u}) \Rightarrow \boldsymbol{u} \upharpoonright B^{[0]} = \boldsymbol{u} \upharpoonright B^{[1]} \, \}$;

\item The \emph{\bfseries promotion} $\varphi^\dagger :: \oc A \multimap \oc B$ of $\varphi$ by
\begin{equation*}
\varphi^\dagger \colonequals \{ \, \boldsymbol{s} \in \mathscr{L}_{\mathrm{Arn}(\oc A \multimap \oc B)} \mid \forall i \in \mathbb{N} . \ \! \boldsymbol{s} \upharpoonright i \in \varphi \, \},
\end{equation*}
where $\boldsymbol{s} \upharpoonright i$ is the j-subsequence of $\boldsymbol{s}$ that consists of moves of the form $(b, i)$ such that $b \in M_B$ and $i \in \mathbb{N}$, or $(a, \langle i, j \rangle)$ such that $a \in M_A$ and $i, j \in \mathbb{N}$, changed into $b$ and $(a, j)$, respectively, and $\langle \_, \_ \rangle : \mathbb{N} \times \mathbb{N} \stackrel{\sim}{\rightarrow} \mathbb{N}$ is an arbitrary bijection fixed throughout this article. 
\end{itemize}
\end{definition}

\begin{notation}
\label{NotationOnSkeletons}
We employ the following notations:
\begin{itemize}

\item Given a t-skeleton $\sigma :: G$, we write $\sigma^{\boldsymbol{1}} :: \boldsymbol{1} \multimap G$ and $\sigma^{\oc \boldsymbol{1}} :: \boldsymbol{1} \Rightarrow G$ for the t-skeletons both of which coincide with $\sigma$ up to `tags.'

\item Given t-skeletons $\phi :: \boldsymbol{1} \multimap G$ and $\phi' :: \boldsymbol{1} \Rightarrow G$, we write $\phi_{\boldsymbol{1}}, \phi'_{\oc \boldsymbol{1}} :: G$ for the t-skeletons that coincide with $\phi$ and $\phi'$ up to `tags,' respectively.

\item Given t-skeletons $\psi :: A \multimap B$ and $\alpha :: A$, we define $\psi \circ \alpha  \colonequals (\psi \circ \alpha^{\boldsymbol{1}})_{\boldsymbol{1}} :: B$.

\item Given t-skeletons $\alpha :: A$ and $\beta :: B$, we define $\alpha \otimes \beta \colonequals ((\alpha^{\boldsymbol{1}} \otimes \beta^{\boldsymbol{1}}) \circ \Delta)_{\boldsymbol{1}} :: A \otimes B$, where $\Delta$ is the unique t-skeleton on $\boldsymbol{1} \multimap \boldsymbol{1} \otimes \boldsymbol{1}$, and $\langle \alpha, \beta \rangle \colonequals \langle \alpha^{\boldsymbol{1}}, \beta^{\boldsymbol{1}} \rangle_{\boldsymbol{1}} :: A \mathbin{\&} B$.

\item Given a t-skeleton $\alpha :: A$, we define $\alpha^\dagger \colonequals ((\alpha^{\oc \boldsymbol{1}})^\dagger)_{\oc \boldsymbol{1}} :: \oc A$.

\item Given an innocent t-skeleton $\theta :: \oc G$, we write $\theta^\ddagger :: G$ for the unique (and necessarily innocent) t-skeleton that satisfies $(\theta^\ddagger)^\dagger = \theta$.

\end{itemize}
\end{notation}

\subsection{Consistency/completeness of tree skeletons and universal identification}
\label{ConsistencyAndCompletenessOfTreeSkeletonsAndUniversalIdentification}
Let us next recall a few more preliminary concepts for p-games: \emph{consistency} and \emph{completeness} of t-skeletons, and the \emph{universal identification} on j-sequences. 

Roughly, a nonempty set $\mathcal{S}$ of t-skeletons is \emph{consistent} if there is a game $G$ such that every element of $\mathcal{S}$ is a t-skeleton on $G$, or equivalently:
\begin{definition}[Consistency of tree skeletons \cite{yamada2016game}]
\label{DefConsistencyOfTSkeletons}
A nonempty set $\mathcal{S}$ of t-skeletons is \emph{\bfseries consistent} if
\begin{enumerate}

\item The arena $(\bigcup_{\sigma \in \mathcal{S}}M_\sigma, \bigcup_{\sigma \in \mathcal{S}}\vdash_\sigma)$ is well-founded;

\item $\forall \sigma, \tau \in \mathcal{S}, \boldsymbol{s}m \in (\sigma \cup \tau)^{\mathrm{Odd}} . \, \boldsymbol{s} \in (\sigma \cap \tau) \Rightarrow \boldsymbol{s}m \in (\sigma \cap \tau)$.

\end{enumerate}

We write $\sigma \asymp \tilde{\sigma}$ and say that t-skeletons $\sigma$ and $\tilde{\sigma}$ are \emph{\bfseries consistent} if the two-element set $\{ \sigma, \tilde{\sigma} \, \}$ is consistent. 
\end{definition}

The union $\bigcup \mathcal{S}$ of a consistent set $\mathcal{S}$ of t-skeletons forms a game such that each element of $\mathcal{S}$ is a t-skeleton on $\bigcup \mathcal{S}$ (but not necessarily vice versa), and conversely the set of all t-skeletons on a game $G$ is consistent. 
Hence, Definition~\ref{DefConsistencyOfTSkeletons} in fact formulates the intended meaning of consistency on t-skeletons. 
E.g., given a set $\mathcal{P}$ of t-skeletons, each consistent subset $\mathcal{S} \subseteq \mathcal{P}$ (not necessarily faithfully) identifies a game $\bigcup \mathcal{S}$ contained in $\mathcal{P}$; we shall utilize consistency on t-skeletons in this way. 

We may further impose \emph{completeness} on consistent sets of t-skeletons such that complete sets of t-skeletons correspond \emph{bijectively} to games:
\begin{definition}[Completeness of tree skeletons \cite{yamada2016game}]
A consistent set $\mathcal{S}$ of t-skeletons is \emph{\bfseries complete} if any subset $\mathcal{A} \subseteq \bigcup \mathcal{S}$ is an element of $\mathcal{S}$ whenever it is a t-skeleton on the game $\bigcup \mathcal{S}$.
\end{definition}

By a bijection $G \stackrel{\sim}{\mapsto} \{ \, \sigma \mid \sigma :: G \, \}$ between games $G$ and complete sets $\{ \, \sigma \mid \sigma :: G \, \}$ of t-skeletons $\sigma$, we may identify games with complete sets of t-skeletons as the previous work \cite{yamada2016game} does. 
We encourage the reader to see that the first problem in modeling sigma-types by games sketched in Section~\ref{IntroPartI}, i.e., there is no game that models the sigma-type $\mathsf{\Sigma (N, N_b)}$, is precisely due to completeness of games. 

As explained in Section~\ref{IntroPartI}, the main idea of p-games is to replace games with complete sets of t-skeletons and then discard completeness, and even consistency, for which we suggest the reader to observe that the second problem in modeling sigma-types by games sketched in Section~\ref{IntroPartI}, i.e., there is no game that models the sigma-type $\mathsf{\Sigma (N, List_N)}$, is due to consistency of games.

Next, let us review the \emph{universal identification} on j-sequences.
Recall first that traditionally each game $G$ comes together with a certain equivalence relation $\simeq_G$ on its positions, called the \emph{identification (of positions)}  \cite{abramsky2000full,mccusker1998games}. 
The identification $\simeq_G$ is to identify positions in $G$ \emph{up to inessential details of `tags' on moves} so that the resulting game semantics matches the abstraction degree of terms in type theories. 

Nevertheless, the previous work \cite{yamada2016game} observes that the identifications proposed in the literature identify positions always in the same way and shows that we may therefore replace the identifications equipped on games with a single equivalence relation $\simeq_{\mathscr{U}}$ on j-sequences, called the \emph{universal identification}. 
By this unification of identifications, t-skeletons do not need an identification either, which facilitates key concepts in \cite{yamada2016game} such as \emph{p-strategies} (Definition~\ref{DefPredicativeStrategies}). 

The technical detail of the universal identification $\simeq_{\mathscr{U}}$ is slightly involved (see \cite{yamada2016game} for its precise definition), but the idea is straightforward: 
\begin{definition}[Universal identification, informally \cite{yamada2016game}]
\label{DefUniversalIdentification}
The \emph{\bfseries universal identification} $\simeq_{\mathscr{U}}$ holds between positions $\boldsymbol{s}, \boldsymbol{t} \in G$ in a game $G$, for which we write $\boldsymbol{s} \simeq_{\mathscr{U}} \boldsymbol{t}$, if $\boldsymbol{s}$ and $\boldsymbol{t}$ are the same j-sequence up to permutation of natural numbers $i \in \mathbb{N}$ serving as `tags' $(\_, i)$ on the same occurrence of exponential $\oc$ in $G$.
\end{definition}

For example, by any bijection $f : \mathbb{N} \stackrel{\sim}{\rightarrow} \mathbb{N}$ that satisfies $f(0) = 1$, we have $q . (q, 0) . (n, 0) . m \simeq_{\mathscr{U}} q . (q, 1) . (n, 1) . m$ for $q . (q, 0) . (n, 0) . m, q . (q, 1) . (n, 1) . m \in N \Rightarrow N$.

Finally, we henceforth focus on games and t-skeletons whose moves are \emph{standard} so that we can implement the universal identification $\simeq_{\mathscr{U}}$ on them.
Also, for technical convenience, we require that they are \emph{saturated} with respect to $\simeq_{\mathscr{U}}$, which leads to:
\begin{definition}[Standard games and t-skeletons, informally \cite{yamada2016game}]
\label{DefStandardMovesAndTSkeletons}
Fix a set $\mathscr{M}$ of moves, called \emph{\bfseries standard moves}, that allows us to accommodate all the games and the constructions on them given in \cite{abramsky1999game,yamada2016game} but also to implement the universal identification $\simeq_{\mathscr{U}}$ (see \cite{yamada2016game} for concrete implementations of $\mathscr{M}$ and $\simeq_{\mathscr{U}}$).
We write $\mathscr{J}(\mathscr{M})$ for the set of all j-sequences whose elements are standard moves.

A game $G$ is \emph{\bfseries standard} if $\{ \, \boldsymbol{s} \in \mathscr{J}(\mathscr{M}) \mid \exists \boldsymbol{t} \in G . \, \boldsymbol{s} \simeq_{\mathscr{U}} \boldsymbol{t} \, \} \subseteq G \subseteq \mathscr{J}(\mathscr{M})$, and a t-skeleton $\sigma$ is \emph{\bfseries standard} if $\sigma \subseteq \mathscr{J}(\mathscr{M})$ and $(\boldsymbol{s}m \simeq_{\mathscr{U}} \boldsymbol{s}m' \wedge \boldsymbol{s}m \in \sigma^{\mathrm{Odd}}) \Rightarrow \boldsymbol{s}m' \in \sigma$.
\end{definition}

It is straightforward to see that standard games and standard t-skeletons are closed under the constructions reviewed in this section; see \cite{yamada2016game} for the details.

\subsection{Predicative games and predicative strategies}
\label{PredicativeGamesAndPredicativeStrategies}
We are now ready to recall a central concept in the game semantics of MLTT \cite{yamada2016game}:
\begin{definition}[Predicative games \cite{yamada2016game}]
\label{DefPredicativeGames}
A \emph{\bfseries predicative (p-) game} is a nonempty set $G$ of standard t-skeletons, called \emph{\bfseries tree (t-) skeletons on $\boldsymbol{G}$}, such that the union $P_G \colonequals \bigcup G$ forms a standard game that satisfies
\begin{enumerate}

\item \textsc{(Det-j completeness)} The set $G$ is \emph{\bfseries deterministic-join (det-j) complete}: If a  consistent subset $\mathcal{S} \subseteq G$ is \emph{deterministic}, i.e., $\boldsymbol{s}mn, \boldsymbol{s}mn' \in \bigcup \mathcal{S}^{\mathrm{Even}}$ implies $\boldsymbol{s}mn = \boldsymbol{s}mn'$, then $\bigcup \mathcal{S} \in G$;

\item \textsc{(Downward completeness)} The set $G$ is \emph{\bfseries downward complete}: If standard t-skeletons $\sigma \in G$ and $\tilde{\sigma} \subseteq P_G$ satisfy $\tilde{\sigma} \leqslant_G \sigma$ then $\tilde{\sigma} \in G$, where $\leqslant_G$ is the partial order on t-skeletons $\sigma, \tilde{\sigma} \subseteq P_G$ defined by
\begin{equation*}
\sigma \leqslant_G \tilde{\sigma} \stackrel{\mathrm{df. }}{\Leftrightarrow} \sigma \asymp \tilde{\sigma} \wedge \sigma \subseteq \tilde{\sigma};
\end{equation*}

\item \textsc{(Horizontal completeness)} The set $G$ is \emph{\bfseries horizontally complete}: If standard t-skeletons $\sigma \in G$ and $\tilde{\sigma} \subseteq P_G$ satisfy $\sigma \simeq_G \tilde{\sigma}$ then $\tilde{\sigma} \in G$, where $\simeq_G$, called the \emph{\bfseries identification of t-skeletons} on $G$, is the symmetric closure of the preorder $\lesssim_G$ on standard t-skeletons $\sigma, \tilde{\sigma} \subseteq P_G$ defined by
\begin{equation*}
\sigma \lesssim_G \tilde{\sigma} \stackrel{\mathrm{df. }}{\Leftrightarrow} \sigma \asymp \tilde{\sigma} \wedge \forall \boldsymbol{s}mn \in \sigma^{\mathrm{Even}}, \boldsymbol{\tilde{s}}\tilde{m} \in \tilde{\sigma} . \, \boldsymbol{s}m \simeq_G \boldsymbol{\tilde{s}}\tilde{m} \Rightarrow \exists \boldsymbol{\tilde{s}}\tilde{m}\tilde{n} \in \tilde{\sigma} . \, \boldsymbol{s}mn \simeq_{\mathscr{U}} \boldsymbol{\tilde{s}}\tilde{m}\tilde{n}.
\end{equation*}

\end{enumerate}

We write $\sigma :: G$ for $\sigma \in G$. 
A \emph{\bfseries (valid) position} in $G$ is a prefix of a sequence $q_G \sigma \boldsymbol{s}$ such that $\sigma :: G$ and $\boldsymbol{s} \in \sigma$, where $q_G$ is any distinguished element. 
\end{definition}

Conceptually, a play of a p-game $G$ proceeds as follows.
At the beginning, \emph{\bfseries Judge (J)} asks P a question $q_G$ `What is your t-skeleton?,' and P answers it by some $\sigma :: G$; then, an `actual play' between O and P in the `declared' t-skeleton $\sigma$ follows.

The main point of a p-game $G$ is that $G$ may be incomplete so that it is more general than a game.
Moreover, $G$ may be even inconsistent so that $\sigma :: G$ ranges over t-skeletons \emph{on different games}; P \emph{fixes} a game when she answers J. 
In this way, p-games solve the problems in game semantics of sigma-types (raised in Section~\ref{IntroPartI}) and indeed model sigma-types; see the original article \cite{yamada2016game} for the details. 

For the three axioms on p-games, let us briefly mention that det-j completeness is vital for the domain-theoretic nature of p-games, downward completeness for linear implication between p-games, and horizontal completeness for p-strategies; see \cite{yamada2016game}. 

\begin{example}
\label{ExamplesOfPredicativeGames}
Given a set $S$, define the p-game $\mathrm{flat}(S) \colonequals \{ \, \underline{x} \mid x \in S \, \} \cup \{ \bot \}$, where $\underline{x} \colonequals \mathrm{Pref}(\{ q^{\mathsf{OQ}} x^{\mathsf{PA}} \})$ and $\bot \colonequals \mathrm{Pref}(\{ q^{\mathsf{OQ}} \})$.
Let us then call the p-games $\boldsymbol{0} \colonequals \mathrm{flat}(\emptyset)$ and $N \colonequals \mathrm{flat}(\mathbb{N})$, respectively, the \emph{\bfseries empty predicative (p-) game} and the \emph{\bfseries natural number predicative (p-) game}.
In addition, define another simple p-game $\boldsymbol{1} \colonequals \{ \top \}$, where $\top \colonequals \{ \boldsymbol{\epsilon} \}$, called the \emph{\bfseries unit predicative (p-) game}. 

Some maximal positions of these p-games are depicted in the following diagram:
\begin{mathpar}
\begin{tabular}{ccc}
&$N$ \\ \cline{2-2}
&$q_{N}$ \\
&$\underline{0}$ \\
& \tikzmark{chundred1} $q$ \\
& \tikzmark{dhundred1} $0$
\end{tabular}
\begin{tikzpicture}[overlay, remember picture, yshift=.25\baselineskip]
\draw [->] ({pic cs:dhundred1}) [bend left] to ({pic cs:chundred1});
\end{tikzpicture}
\and
\begin{tabular}{ccc}
&$N$ \\ \cline{2-2}
&$q_{N}$ \\
&$\underline{10}$ \\
& \tikzmark{chundred2} $q$ \\
& \tikzmark{dhundred2} $10$
\end{tabular}
\begin{tikzpicture}[overlay, remember picture, yshift=.25\baselineskip]
\draw [->] ({pic cs:dhundred2}) [bend left] to ({pic cs:chundred2});
\end{tikzpicture}
\and
\begin{tabular}{ccc}
&$N$ \\ \cline{2-2}
&$q_{N}$ \\
&$\underline{100}$ \\
& \tikzmark{chundred3} $q$ \\
& \tikzmark{dhundred3} $100$
\end{tabular}
\begin{tikzpicture}[overlay, remember picture, yshift=.25\baselineskip]
\draw [->] ({pic cs:dhundred3}) [bend left] to ({pic cs:chundred3});
\end{tikzpicture}
\and
\begin{tabular}{ccc}
& $\boldsymbol{1}$ \\ \cline{2-2}
& $q_{\boldsymbol{1}}$ \\
& $\top$ \\ 
& \\
&
\end{tabular}
\and
\begin{tabular}{ccc}
& $\boldsymbol{0}$ \\ \cline{2-2}
& $q_{\boldsymbol{0}}$  \\
& $\bot$ \\
& $q$ \\
&
\end{tabular}
\end{mathpar}
They clearly correspond to the natural number game, the unit game and the empty game given in Example~\ref{ExamplesOfGames}, respectively, as the abuse of the notations indicates.

\end{example}

Finally, since we identify positions up to the universal identification $\simeq_{\mathscr{U}}$, it makes sense to identify t-skeletons on a p-game $G$ up to the identification $\simeq_{G}$ as well. 
In fact, such an identification of equivalent t-skeletons matches the abstraction degree of terms in type theories \cite[Section~3.6]{mccusker1998games}.
For this point, it is technically more convenient to take the union of all equivalent t-skeletons on $G$ up to $\simeq_{G}$ than the quotient \cite{abramsky2009game}, which leads us to another central concept in the previous work \cite{yamada2016game}:
\begin{definition}[Predicative strategies \cite{yamada2016game}]
\label{DefPredicativeStrategies}
A \emph{\bfseries predicative (p-) strategy} on a p-game $G$ is the \emph{\bfseries saturation} $\mathrm{Sat}(\sigma) \colonequals \bigcup \{ \, \sigma' :: G \mid \sigma \simeq_G \sigma' \, \} \subseteq P_G$, written $\mathrm{Sat}(\sigma) : G$, of an arbitrary t-skeleton $\sigma :: G$ that satisfies \emph{\bfseries validity}: $\sigma \simeq_G \sigma$. 
Given $\sigma' :: G$, we write $\sigma' \propto \mathrm{Sat}(\sigma)$ and say that $\sigma'$ \emph{\bfseries implements} $\mathrm{Sat}(\sigma)$ if $\sigma \simeq_G \sigma'$.
\end{definition}

\begin{remark}
The original definitions of saturations and p-strategies \cite{yamada2016game} are different from yet equivalent to the ones given above. 
The modifications are just for brevity. 
\end{remark}

\begin{example}
The t-skeleton $\mathrm{der}_N^{(i)} \colonequals \mathrm{Pref}(\{ \, q . (q, i) . (n, i) . n \mid n \in \mathbb{N} \, \}) :: N \Rightarrow N$ for any $i \in \mathbb{N}$ is \emph{ad-hoc} or \emph{too low-level} since it chooses $i$ for the `tag' $(\_, i)$ on the domain $\oc N$. 
The p-strategy $\mathrm{Sat}(\mathrm{der}_N^{(i)}) = \bigcup_{i \in \mathbb{N}} \mathrm{der}_N^{(i)} : N \Rightarrow N$ fixes this problem.
\end{example}

For each p-game $G$, validity of a t-skeleton $\sigma :: G$ ensures nonemptiness of the p-strategy $\mathrm{Sat}(\sigma) : G$.
Since $\mathrm{Sat}(\sigma) = \mathrm{Sat}(\sigma')$ if and only if $\sigma \simeq_G \sigma'$ for all $\sigma, \sigma' :: G$ \cite{yamada2016game}, a p-strategy $\mathrm{Sat}(\sigma) : G$ is invariant with respect to the representative valid t-skeleton $\sigma :: G$ up to $\simeq_G$, which justify the arbitrary choice of $\sigma$ in Definition~\ref{DefPredicativeStrategies}.

\section{Consistency of intensional Martin-L\"{o}f type theory with formal Church's thesis}
\label{MainConsistencyOfMLTTwithCT}
We have reviewed all the necessary preliminaries, and therefore let us now turn to the main content of this article: consistency of MLTT with CT.

In this section, we prove the consistency by a constructive model of MLTT plus CT as follows.
We first define \emph{recursive} t-skeletons and their \emph{realizers}, and fix a choice of \emph{canonical} realizers for technical convenience in Section~\ref{RealizableTSkeletonsAndCanonicalRealizers}.
Next, we introduce another, `nonstandard' variant of p-games, called \emph{np-games}, to accommodate \emph{disjoint unions of winning-realizer-wise linear implications (DoWRWLIs)} in Section~\ref{RealizerwisePStrategies}. 
DoWRWLIs are t-skeletons on modified linear implication, in which O must play on the domain by a winning, recursive t-skeleton and exhibit the canonical realizer for it at his first move. 
As explained in Section~\ref{GameSemanticRealizability}, DoWRWLIs implement our idea on how to validate CT.
Next, we show that a CCC $\mathbb{NPG}_{\mathrm{wrw}}^{\mathrm{wo}}$ of \emph{well-opened} np-games and winning, recursive DoWRWLIs gives rise to a CwF equipped with semantic type formers for unit-, empty-, N-, pi-, sigma- and Id-types in Section~\ref{RealizabilityModelOfMLTTAlaGameSemantics}, which establishes a model of MLTT equipped with these types.
Finally, we prove that the model of MLTT in $\mathbb{NPG}_{\mathrm{wrw}}^{\mathrm{wo}}$ (even \emph{constructively}) validates CT and refutes empty-type, establishing consistency of MLTT with CT, in Section~\ref{MainResult}.

\subsection{Recursive tree skeletons and canonical realizers}
\label{RealizableTSkeletonsAndCanonicalRealizers}
In this section, we define \emph{recursive} t-skeletons and their \emph{realizers}, which are based on \emph{recursion theory} \cite{rogers1967theory,cutland1980computability} similarly to Section~5.6 of the classic \emph{HO-games} \cite{hyland2000full}. 
Also, for technical convenience, we fix a \emph{canonical} realizer for each recursive t-skeleton.

First, let us arbitrarily encode standard moves (Definition~\ref{DefStandardMovesAndTSkeletons}) by finite sequences of natural numbers, which is clearly possible \cite{yamada2016game}. 
Next, let us recall that the pointer of each j-sequence (Definition~\ref{DefJSequences}) is a finite function on natural numbers.
It then follows from these two points that there is a recursive bijection $\mathscr{J}(\mathscr{M}) \stackrel{\sim}{\rightarrow} \mathbb{N}$ \cite{rogers1967theory,cutland1980computability}.

\begin{definition}[Coding of tree skeletons]
Let us fix once and for all a recursive bijection $\mathscr{C} : \mathscr{J}(\mathscr{M}) \stackrel{\sim}{\rightarrow} \mathbb{N}$ and call $\mathscr{C}(\sigma) \subseteq \mathbb{N}$ the \emph{\bfseries coding} of t-skeletons $\sigma \subseteq \mathscr{J}(\mathscr{M})$. 
\end{definition}

The coding $\mathscr{C}$ of t-skeletons enables us to employ recursion theory for defining `constructive' or \emph{recursive} t-skeletons similarly to HO-games \cite[Section~5.6]{hyland2000full}:

\begin{convention}
Let us fix once and for all an enumeration or \emph{G\"{o}del numbering} $\mathscr{G}_{\mathrm{PRF}} : \mathbb{N} \twoheadrightarrow \mathrm{PRF}$ on the set $\mathrm{PRF}$ of all partial recursive functions $\mathbb{N} \rightharpoonup \mathbb{N}$. 
\end{convention}

\begin{definition}[Functional representation]
The \emph{\bfseries functional representation} of a t-skeleton $\sigma \subseteq \mathscr{J}(\mathscr{M})$ is the partial map $\mathrm{fun}(\sigma) : \mathscr{C}(\sigma^{\mathrm{Odd}}) \rightharpoonup \mathscr{C}(\sigma^{\mathrm{Even}})$ given by
\begin{equation*}
\mathrm{fun}(\sigma)(\mathscr{C}(\boldsymbol{s}m)) \colonequals \begin{cases} \mathscr{C}(\boldsymbol{s}mn) &\text{if there is $\boldsymbol{s}mn \in \sigma^{\mathrm{Even}}$;} \\ \uparrow &\text{otherwise} \end{cases} \quad (\boldsymbol{s}m \in \sigma^{\mathrm{Odd}}).
\end{equation*}
\end{definition}

\begin{definition}[Recursive tree skeletons]
\label{DefRecursiveTreeSkeletons}
A \emph{\bfseries realizer} for a t-skeleton $\sigma \subseteq \mathscr{J}(\mathscr{M})$ is a natural number $e \in \mathbb{N}$ that \emph{\bfseries realizes} (the functional representation of) $\sigma$: 
\begin{equation*}
\forall \boldsymbol{s}m \in \sigma^{\mathrm{Odd}} . \, \mathscr{G}_{\mathrm{PRF}}(e)(\mathscr{C}(\boldsymbol{s}m)) \simeq \mathrm{fun}(\sigma)(\mathscr{C}(\boldsymbol{s}m)),
\end{equation*}
and a t-skeleton is \emph{\bfseries recursive} if there is a realizer for it.
\end{definition}

Our definition of recursive t-skeletons is the same as that of \emph{recursive strategies} \cite[Section~5.6]{hyland2000full} except that the latter focuses on innocent strategies and their P-views.
Our choice is just for simplicity; we could follow precisely the latter formulation. 

Note that a realizer $e$ for a recursive t-skeleton $\sigma$ is only concerned with positions in $\sigma$, and hence $\mathscr{G}_{\mathrm{PRF}}(e)$ may not coincide with $\mathrm{fun}(\sigma)$ as a partial map $\mathbb{N} \rightharpoonup \mathbb{N}$ (i.e., they may compute differently on natural numbers that do not code positions in $\sigma$).
Note also that $\mathscr{G}_{\mathrm{PRF}}(e)$ does not (even partially) decide odd-length positions in $\sigma$, and it \emph{only computes even-length ones from odd-length ones}.
These `compromises' play a crucial role for our validation of CT (essentially because the set of all total recursive functions is not recursively enumerable \cite{rogers1967theory,cutland1980computability}), as we shall see shortly.

\begin{notation}
Let $G$ be \emph{any} set of t-skeletons (so that $G$ can be an \emph{np-game} (Definition~\ref{DefNonstandardPredicativeGames}) as well as a p-game).
We write $\mathcal{TS}_{\mathrm{r}}(G)$ (resp. $\mathcal{TS}_{\mathrm{wr}}(G)$) for the set of all recursive (resp. winning and recursive) t-skeletons in $G$, and $\mathscr{R}_{\mathrm{r}}(G)$ (resp. $\mathscr{R}_{\mathrm{wr}}(G)$) for the set of all realizers for recursive (resp. winning and recursive) ones in $G$.
\end{notation}

Let us then fix an arbitrary choice of a realizer for each recursive t-skeleton and call it the \emph{canonical} one:
\begin{definition}[Canonical realizers]
\label{DefCanonicalRealizers}
Let us fix once and for all a function $\mathscr{G}_G : \mathcal{TS}_{\mathrm{r}}(G) \rightarrow \mathscr{R}_{\mathrm{r}}(G)$ for each set $G$ of t-skeletons such that $\mathscr{G}_G(\sigma) \in \mathscr{R}_{\mathrm{r}}(G)$ realizes each $\sigma \in \mathcal{TS}_{\mathrm{r}}(G)$, and call $\mathscr{G}_G$ the \emph{\bfseries canonical G\"{o}del numbering} on $G$.
We assume that canonical G\"{o}del numberings satisfy the condition at the end of Definition~\ref{DefConstructionsOnDoWRWLIs}.

We call the realizer $\mathscr{G}_G(\sigma)$ and the pair $(\sigma, \mathscr{G}_G(\sigma))$ the \emph{\bfseries canonical realizer} and the \emph{\bfseries canonical pair} for $\sigma$, respectively, and define 
\begin{equation*}
\mathscr{R}^{\mathrm{cp}}_{\mathrm{wr}}(G) \colonequals \{ \, (\sigma, \mathscr{G}_G(\sigma)) \mid \sigma \in \mathcal{TS}_{\mathrm{wr}}(G) \, \}. 
\end{equation*}
\end{definition}

For technical convenience (e.g., for Definitions~\ref{DefWRWDerelictionTreeSkeletons} and \ref{DefConstructionsOnDoWRWLIs}), we shall henceforth employ canonical pairs frequently. 
Note that, although canonical realizers in general do not have complete information about t-skeletons, canonical pairs (trivially) do.

\subsection{Cartesian closed category of well-opened nonstandard predicative games and winning, recursive, winning-realizer-wise tree skeletons}
\label{RealizerwisePStrategies}
In this section, we modify a certain class of t-skeletons, called \emph{unions of pointwise linear implications (UoPLIs)}, on the \emph{linear implication $G \multimap H$ between p-games} $G$ and $H$ \cite{yamada2016game}, and based on it define a modified linear implication $G \rightarrowtriangle H$. 
In $G \rightarrowtriangle H$, O must play by a \emph{winning, recursive} t-skeleton $\sigma$ on the domain $G$ and \emph{exhibit the canonical realizer for $\sigma$} at his first move.
This modification of linear implication implements our idea on how to model MLTT plus CT sketched in Section~\ref{GameSemanticRealizability}.

However, there arises a problem: P-games $G$ are \emph{not} closed under the modified linear implication $\rightarrowtriangle$ because it does not preserve \emph{saturation} of the set $P_G$ of all positions.
Note that we cannot simply discard the saturation axiom since if we do so then a t-skeleton $\sigma :: G$ may generate a p-strategy $\mathrm{Sat}(\sigma)$ that is not even a subset of $P_G$, i.e., p-strategies would be no longer `strategies on p-games' \cite{yamada2016game}.

On the other hand, recall that McCusker \cite{abramsky1999game,mccusker1998games} dispenses with identification of positions and strategies by employing much simpler exponential $\hat{\oc}$ \emph{without the `tags'} $(\_, i)$ such that $i \in \mathbb{N}$, which we call \emph{simplified exponential}, and \emph{well-opened} games for which simplified exponential $\hat{\oc}$ works.
As confessed in \cite[p.~48]{mccusker1998games}, however, this simpler approach to cartesian closure of games is mathematically \emph{ad-hoc}, and it is why the previous work \cite{yamada2016game} employs identifications of positions and p-strategies. 

Nevertheless, the goal of the present work is to prove consistency of MLTT with CT, and mathematical elegance of the employed method is secondly. 
Therefore, we adopt, as a solution to the above problem, McCusker's simplified exponential $\hat{\oc}$ and well-opened games.
Let us first recall these concepts plus \emph{thread-closed} games:

\begin{definition}[Thread-closed games \cite{abramsky1999game,mccusker1998games}]
A game $A$ is \emph{\bfseries thread-closed} if $\boldsymbol{s} \upharpoonright \mathscr{I} \in A$ for any position $\boldsymbol{s} \in A$ and set $\mathscr{I}$ of initial occurrences in $\boldsymbol{s}$, where $\boldsymbol{s} \upharpoonright \mathscr{I}$ is the j-subsequence of $\boldsymbol{s}$ that consists of occurrences in $\boldsymbol{s}$ \emph{hereditarily justified}\footnote{An initial occurrence $m$ in a j-sequence $\boldsymbol{s}$ \emph{hereditarily justifies} an occurrence $n$ in $\boldsymbol{s}$ if a finite iteration of the pointer $\mathcal{J}_{\boldsymbol{s}}$ applied to $n$ goes back to $m$ \cite{abramsky1999game,mccusker1998games}.} by initial occurrences in $\mathscr{I}$, called the \emph{\bfseries thread} of $\boldsymbol{s}$ with respect to $\mathscr{I}$.
\end{definition}

That is, a game is thread-closed if its positions are closed under taking threads. 

\begin{remark}
The terminology \emph{thread-closed} is not used in the original articles \cite{abramsky1999game,mccusker1998games}.
\end{remark}

\begin{definition}[Simplified exponential on games \cite{abramsky1999game,mccusker1998games}]
The \emph{\bfseries simplified exponential} on a game $A$ is the game $\hat{\oc} A \colonequals \{ \, \boldsymbol{s} \in \mathscr{L}_{\mathrm{Arn}(A)} \mid \forall m \in M_A^{\mathrm{Init}} . \, \boldsymbol{s} \upharpoonright \{ m \} \in A \, \}$, where $\{ m \}$ ranges over the singleton set of each initial occurrence in $\boldsymbol{s}$ whose move is $m$ if there is such an initial occurrence in $\boldsymbol{s}$, and the empty set $\emptyset$ otherwise. 
\end{definition}

We want the relation $A \subseteq \hat{\oc} A$ for every game $A$ since we employ $\hat{\oc}$ as an alternative to exponential $\oc$ (Section~\ref{IntroGameSemantics}).
Although it does not hold for games in general, it does for thread-closed ones; it is the point of the thread-closing constraint \cite[p.~41]{mccusker1998games}.

As mentioned in \cite[pp.~42-43]{mccusker1998games}, however, identities (or identity morphisms) in the CCC of games, called \emph{derelictions}, are not well-defined with respect to simplified exponential $\hat{\oc}$. 
To remedy this problem, we have to further focus on:
\begin{definition}[Well-opened games \cite{abramsky1999game,mccusker1998games}]
\label{DefWellOpenedGames}
A game $A$ is \emph{\bfseries well-opened} if the conjunction of $\boldsymbol{s} m \boldsymbol{t} \in A$ and $m \in M_A^{\mathrm{Init}}$ implies $\boldsymbol{s} = \boldsymbol{\epsilon}$.
\end{definition}

In other words, a game is well-opened if its position contains at most one initial occurrence. 
Note that well-opened games are trivially thread-closed. 
Note also that well-opened games are not closed under (simplified) exponential, but it does not matter for the present work as what we need is the implication $\Rightarrow$, not (simplified) exponential itself, and well-opened games are closed under implication \cite[p.~43]{mccusker1998games}. 

Because simplified exponential $\hat{\oc}$ dispenses with the `tags' $(\_, i)$ such that $i \in \mathbb{N}$ on exponential $\oc$, we no longer need the universal identification $\simeq_{\mathscr{U}}$ (Definition~\ref{DefUniversalIdentification}) or relevant concepts. 
Let us therefore introduce modified, `nonstandard' p-games without the nonempty, saturation or three completenesses axioms:
\begin{definition}[Nonstandard predicative games]
\label{DefNonstandardPredicativeGames}
A \emph{\bfseries nonstandard predicative (np-) game} is a set of t-skeletons $\sigma$ such that $\sigma \subseteq \mathscr{J}(\mathscr{M})$. 
An np-game is \emph{\bfseries well-opened} if so are all its elements (Definition~\ref{DefWellOpenedGames}).
\end{definition}

\begin{remark}
We could instead call np-games \emph{generalized p-games} since np-games are obtained from p-games by discarding all the axioms. 
Nevertheless, we would like to emphasize that np-games are undoubtedly `nonstandard' as they do not have the domain-theoretic structures of game semantics, e.g., an np-game may be the empty set, which motivates the term \emph{np-games}.
Technically, we discard in Definition~\ref{DefNonstandardPredicativeGames} the nonempty axiom on p-games for the modified linear implication $\rightarrowtriangle$ between np-games (Lemma~\ref{LemWellDefinedWRWLI}), and the saturation and the three completenesses axioms since we no longer care for identifications of positions or orders between t-skeletons (n.b., recall that mathematical elegance of our mathematical structures is secondary).  
\end{remark}

\begin{example}
\label{ExamplesOfWellOpenedNonstandardPredicativeGames}
The p-games in Example~\ref{ExamplesOfPredicativeGames} are all well-opened np-games. 
\end{example}

\begin{convention}
We apply the notations/conventions for p-games to np-games as well.
\end{convention}

Let us next introduce the modified linear implication $\rightarrowtriangle$ between np-games as announced above.
Because np-games are defined in terms of t-skeletons (just like p-games), we define the modified linear implication in terms of t-skeletons as follows:
\begin{definition}[FoWRWLIs]
\label{DefFoWRWLIs}
A \emph{\bfseries family of winning-realizer-wise linear implications (FoWRWLI)} between np-games $G$ and $H$ is a family 
\begin{equation*}
\phi = (\phi_{(\gamma, e)})_{(\gamma, e) \in \mathscr{R}^{\mathrm{cp}}_{\mathrm{wr}}(G)}
\end{equation*}
of t-skeletons $\phi_{(\gamma, e)} :: \gamma \multimap \mathrm{cod}_\phi(\gamma, e)$ that satisfies
\begin{enumerate}

\item $\mathrm{cod}_\phi(\gamma, e) \in \mathcal{TS}_{\mathrm{wr}}(H)$;

\item $\forall (\gamma, e), (\gamma', e) \in  \mathscr{R}^{\mathrm{cp}}_{\mathrm{wr}}(G), \boldsymbol{s}mn \in \phi_{(\gamma, e)}^{\mathrm{Even}}, \boldsymbol{s}mn' \in \phi_{(\gamma', e)}^{\mathrm{Even}} . \, \boldsymbol{s}mn = \boldsymbol{s}mn'$;

\item There is a natural number $f \in \mathbb{N}$ such that $\mathscr{G}_{\mathrm{PRF}}(f)(e) = \pi^\gamma_\phi(e)$ for all $(\gamma, e) \in \mathscr{R}^{\mathrm{cp}}_{\mathrm{wr}}(G)$, where $\pi^\gamma_\phi : e \mapsto \mathscr{G}_H (\mathrm{cod}_\phi(\gamma, e))$ (called the \emph{\bfseries realizer-map} of $\phi$ at $\gamma$).

\end{enumerate}

Given an np-game $K$ and an FoWRWLI $\psi$ between $H$ and $K$, the \emph{\bfseries composition} of $\phi$ and $\psi$ is the FoWRWLI $\psi \circ \phi$ between $G$ and $K$ defined by 
\begin{mathpar}
(\psi \circ \phi)_{(\gamma, e)} \colonequals \psi_{(\mathrm{cod}_\phi(\gamma, e), \pi_\phi^\gamma (e))} \circ \phi_{(\gamma, e)} \quad ((\gamma, e) \in \mathscr{R}^{\mathrm{cp}}_{\mathrm{wr}}(G)).
\end{mathpar}

We write $\mathscr{F_{\mathrm{wrw}}}(G, H)$ for the set of all FoWRWLIs between $G$ and $H$.
\end{definition}

Each realizer-map $\pi^\gamma_\phi$ is not required to be recursive (or there is $f \in \mathbb{N}$ such that $\mathscr{G}_{\mathrm{PRF}}(f)(n) \simeq \pi_\phi^\gamma(n)$ \emph{for all $n \in \mathbb{N}$}) since otherwise we could not validate CT (again due to the computational hardness of the set of all total recursive functions), as we shall see shortly.
Nevertheless, even in such a weakened sense, the `effectivity' of realizer-maps is necessary for \emph{recursive DoWRWLIs} (Definition~\ref{DefDoWRWLIs}) to be \emph{closed under composition} (the second clause of Lemma~\ref{LemWellDefinedDoWRWLIs}). 
Note also that the realizer-map $\pi_\phi^\gamma$ does not actually depend on the t-skeleton $\gamma \in \pi_1(\mathscr{R}^{\mathrm{cp}}_{\mathrm{wr}}(G))$.

\begin{convention}
We henceforth omit the superscript $(\_)^\gamma$ on realizer-maps $\pi_\phi^\gamma$. 
We write $\emptyset$ for the empty FoWRWLI, and $\{ \varphi \}$ for a singleton FoWRWLI of a t-skeleton $\varphi$.
\end{convention}

Next, the second axiom on FoWRWLIs prohibits their components from depending on the indexing t-skeletons, which is crucial for the first clause of Lemma~\ref{LemWellDefinedDoWRWLIs}.

On the other hand, the components may depend on the indexing realizers, which will play a crucial role for our validation of CT, as we shall see. 
However, it prohibits us from taking the union on an FoWRWLI since such a union may not be a well-defined t-skeleton. 
Specifically, the union may be \emph{nondeterministic} (Definition~\ref{DefTreeSkeletons}). 

Therefore, we instead take the \emph{disjoint union} of each FoWRWLI: 
\begin{definition}[DoWRWLIs]
\label{DefDoWRWLIs}
A \emph{\bfseries disjoint union of winning-realizer-wise linear implications (DoWRWLI)} between np-games $G$ and $H$ is the t-skeleton 
\begin{equation*}
\uplus \phi \colonequals \begin{cases} \top \, (= \{ \boldsymbol{\epsilon} \}) & \text{if $\phi = \emptyset$;} \\ \bigcup_{(\gamma, e) \in \mathscr{R}^{\mathrm{cp}}_{\mathrm{wr}}(G)} \phi_{(\gamma, e)}^{[e]} &\text{otherwise,} \end{cases}
\end{equation*}
for an FoWRWLI $\phi$ between $G$ and $H$, where 
\begin{equation*}
\phi_{(\gamma, e)}^{[e]} \colonequals \{ \, (\boldsymbol{s}(1), e) . (\boldsymbol{s}(2), e) \dots (\boldsymbol{s}(|\boldsymbol{s}|), e) \mid \boldsymbol{s} \in \phi_{(\gamma, e)} \, \}, 
\end{equation*}
and the `tags' $(\_, e)$ are implemented within the formalization of standard moves.

Given an np-game $K$ and a DoWRWLI $\uplus \psi$ between $H$ and $K$, the \emph{\bfseries composition} of $\uplus \phi$ and $\uplus \psi$ is the DoWRWLI $\uplus \psi \circ \uplus \phi$ between $G$ and $K$ defined by 
\begin{equation*}
\uplus \psi \circ \uplus \phi \colonequals \uplus (\psi \circ \phi).
\end{equation*}
\end{definition}

\begin{remark}
Alternatively, a DoWRWLI $\uplus \phi$ between np-games $G$ and $H$ can be defined by $\uplus \phi \colonequals \top$ if $\phi = \emptyset$, and $\uplus \phi \colonequals \bigcup_{e \in \pi_2(\mathscr{R}^{\mathrm{cp}}_{\mathrm{wr}}(G))} \phi_{e}^{[e]}$, where $\phi_e \colonequals \bigcup_{(\gamma, e) \in \mathscr{R}^{\mathrm{cp}}_{\mathrm{wr}}(G)} \phi_{(\gamma, e)}$, otherwise.
This point on DoWRWLIs, which are morphisms in our category, justifies the simplification of the semantics of pi-types made in Section~\ref{GameSemanticRealizability}, as we shall see. 
\end{remark}

\begin{convention}
Henceforth, we shall frequently reason about DoWRWLIs, for which we need case analyses accordingly to how DoWRWLIs are defined. 
It is, however, mostly rather trivial to handle the first case (i.e., when the underlying FoWRWLI $\phi$ of a given DoWRWLI $\uplus \phi$ is empty), and so we shall often skip it.
\end{convention}

\begin{lemma}[Well-defined DoWRWLIs]
\label{LemWellDefinedDoWRWLIs}
Let $G$, $H$ and $K$ be np-games.
\begin{enumerate}

\item Any DoWRWLI between $G$ and $H$ is a t-skeleton contained in $\mathscr{J}(\mathscr{M})$;

\item DoWRWLIs (resp. winning, recursive ones) are closed under composition. 

\end{enumerate}
\end{lemma}
\begin{proof}
For the first clause, it suffices to remark that each DoWRWLI is deterministic thanks to the second axiom on FoWRWLIs.

For the second clause, note that each DoWRWLI uniquely determines the underlying FoWRWLI so that the composition is well-defined on DoWRWLIs (resp. winning ones). 
Finally, the third axiom on FoWRWLIs or `effectivity' of realizer-maps enables us to `effectively' obtain the `intermediate' realizer $\pi_\phi(e)$ on the component $(\psi \circ \phi)_e$ of the composition $\uplus \psi \circ \uplus \phi$ of given composable DoWRWLIs $\uplus \phi$ and $\uplus \psi$ so that the composition evidently preserves recursiveness of DoWRWLIs (by the same argument for closure of recursive strategies under composition in \cite[p.~355]{hyland2000full}).
\end{proof}

For comparison with DoWRWLIs, let us recall \emph{UoPLIs} \cite{yamada2016game} between p-games:
\begin{definition}[UoPLIs \cite{yamada2016game}]
\label{DefUoPLIs}
A \emph{\bfseries family of pointwise linear implications (FoPLI)} between p-games $G$ and $H$ is a family $\varphi = (\varphi_\sigma)_{\sigma :: G}$ of standard t-skeletons $\varphi_\sigma$ that satisfies
\begin{enumerate}


\item $\forall \sigma :: G . \, \exists \tau :: H . \, \varphi_\sigma :: \sigma \multimap \tau$;


\item $\forall \sigma, \tilde{\sigma} :: G, \boldsymbol{s} m n \in \varphi_\sigma^{\mathrm{Even}}, \boldsymbol{s} m \tilde{n} \in \varphi_{\tilde{\sigma}}^{\mathrm{Even}} . \, \boldsymbol{s} m n = \boldsymbol{s} m \tilde{n}$.


\end{enumerate}

We write $\mathscr{F}(G, H)$ for the set of all FoPLIs between $G$ and $H$, and call the union $\bigcup \varphi \colonequals \bigcup_{\sigma :: G}\varphi_\sigma$ the \emph{\bfseries union of PLIs (UoPLI)} between $G$ and $H$ (on $\varphi$).
\end{definition}

The first axiom on an FoPLI specifies the induced UoPLI in the \emph{pointwise} fashion.
On the other hand, the second axiom brings UoPLIs \emph{determinacy} so that they are well-defined t-skeletons. 
Moreover, the second axiom gives rise to the \emph{computational nature} or \emph{domain-theoretic structures} of game semantics, e.g., thanks to the axiom, the input-output pairs (or extension) of a UoPLI induce \emph{continuous} functions \cite{yamada2016game}.

Crucially, DoWRWLIs differ from UoPLIs in the following two points:
\begin{enumerate}

\item O in each DoWRWLI must play on the domain by a \emph{winning, recursive} t-skeleton and \emph{exhibit the canonical realizer} for it by the `tag' at his first move, while in each UoPLI O may play by any t-skeleton on the domain and does not have to exhibit anything other than ordinary O-moves during a play;

\item A play by a DoWRWLI may \emph{depend on the realizer} that O supplies, while it is not the case for a play by a UoPLI due to its second axiom.

\end{enumerate}
Consequently, DoWRWLIs implement our idea on how to validate CT, which is described in Section~\ref{GameSemanticRealizability}, while UoPLIs cannot.

Next, recall that the previous work \cite{yamada2016game} defines the \emph{linear implication} $G \multimap H$ between p-games $G$ and $H$ by $G \multimap H \colonequals \{ \, \bigcup \Phi \mid \Phi \in \mathscr{F}(G, H) \, \}$.
Similarly, we define our modified linear implication $\rightarrowtriangle$ between np-games by
\begin{definition}[Winning-realizer-wise linear implication]
\label{DefWRWLinearImplication}
The \emph{\bfseries winning-realizer-wise (w.r.w.) linear implication} between np-games $G$ and $H$ is the np-game 
\begin{mathpar}
G \rightarrowtriangle H \colonequals \{ \, \uplus \phi \mid \phi \in \mathscr{F}_{\mathrm{wrw}}(G, H) \, \}.
\end{mathpar}
\end{definition}

Note that $G \rightarrowtriangle H = \emptyset$ if $\mathscr{R}^{\mathrm{cp}}_{\mathrm{wr}}(G) \neq \emptyset$ and $\mathscr{R}^{\mathrm{cp}}_{\mathrm{wr}}(H) = \emptyset$. 
It is why we allow an np-game to be the empty set (Definition~\ref{DefNonstandardPredicativeGames}).

\begin{lemma}[Well-defined winning-realizer-wise linear implication]
\label{LemWellDefinedWRWLI}
Np-games are closed under w.r.w. linear implication $\rightarrowtriangle$.
\end{lemma}
\begin{proof}
Immediate from the first clause of Lemma~\ref{LemWellDefinedDoWRWLIs}.
\end{proof}

On the other hand, p-games are not closed under w.r.w. linear implication $\rightarrowtriangle$.
To see this point, note that the `tags' for DoWRWLIs are embedded into standard moves, and thus \emph{certain forms of standard moves determine O's play on the domain} of a DoWRWLI.
Consequently, p-games are not closed under w.r.w. linear implication since it does not preserve \emph{saturation} of positions. 
For instance, consider the w.r.w. linear implication $\oc (\oc N \rightarrowtriangle N) \rightarrowtriangle N$, where $N$ and $\oc N$ are the natural number game and its exponential (Section~\ref{IntroGameSemantics}) regarded as p-games. 
A play in it looks like
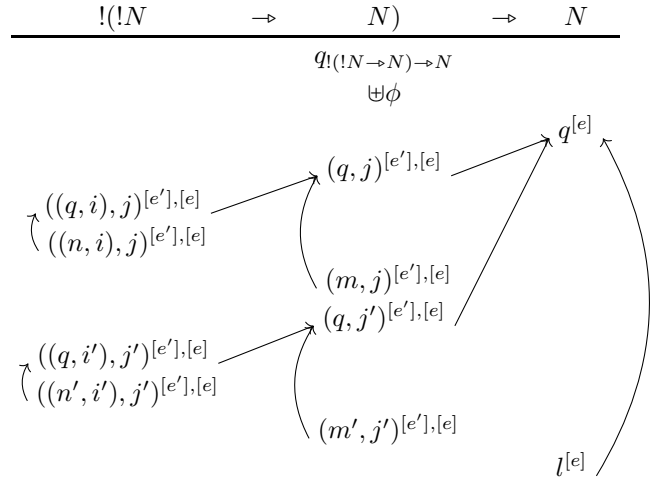
\begin{mathpar}
\begin{tabular}{ccccc}
$\oc (\oc N$ & $\rightarrowtriangle$ & $N)$ & $\rightarrowtriangle$ & $N$ \\ \hline 
&&$q_{\oc (\oc N \rightarrowtriangle N) \rightarrowtriangle N}$&& \\
&&$\uplus \phi$&& \\
&&&& \tikzmark{chigher01} $q^{[e]}$ \tikzmark{chigher09} \\
&&\tikzmark{chigher02} $(q, j)^{[e'], [e]}$ \tikzmark{dhigher01} && \\
\tikzmark{chigher03} $((q, i), j)^{[e'], [e]}$ \tikzmark{dhigher02} && && \\
\tikzmark{dhigher03} $((n, i), j)^{[e'], [e]}$&& && \\
&& \tikzmark{dhigher04} $(m, j)^{[e'], [e]}$ && \\
&&\tikzmark{chigher08} $(q, j')^{[e'], [e]}$ \tikzmark{dhigher05} && \\
\tikzmark{chigher07} $((q, i'), j')^{[e'], [e]}$ \tikzmark{dhigher06} && && \\
\tikzmark{dhigher07} $((n', i'), j')^{[e'], [e]}$&& && \\
&& \tikzmark{dhigher08} $(m', j')^{[e'], [e]}$ && \\
&&&& $l^{[e]}$ \tikzmark{dhigher09}
\end{tabular}
\begin{tikzpicture}[overlay, remember picture, yshift=.25\baselineskip]
\draw [->] ({pic cs:dhigher01}) to ({pic cs:chigher01});
\draw [->] ({pic cs:dhigher02}) to ({pic cs:chigher02});
\draw [->] ({pic cs:dhigher03}) [bend left] to ({pic cs:chigher03});
\draw [->] ({pic cs:dhigher04}) [bend left] to ({pic cs:chigher02});
\draw [->] ({pic cs:dhigher05}) to ({pic cs:chigher01});
\draw [->] ({pic cs:dhigher06}) to ({pic cs:chigher08});
\draw [->] ({pic cs:dhigher07}) [bend left] to ({pic cs:chigher07});
\draw [->] ({pic cs:dhigher08}) [bend left] to ({pic cs:chigher08});
\draw [->] ({pic cs:dhigher09}) [bend right] to ({pic cs:chigher09});
\end{tikzpicture}
\end{mathpar}
Because $e$ realizes O's play on the domain $\oc N \rightarrowtriangle N$, his choice $i \in \mathbb{N}$ is determined by $e$, where the `tag' $(\_)^{[e]}$ constitutes the moves in the diagram. 
It then follows that the set $P_{\oc (\oc N \rightarrowtriangle N) \rightarrowtriangle N}$ is not saturated (as $i$ does not range over all natural numbers unless we change $e$), showing that $\oc (\oc N \rightarrowtriangle N) \rightarrowtriangle N$ is not a p-game. 

We circumvent this problem by discarding identifications and the saturation axiom on p-games, and then adopting np-games and \emph{simplified exponential} $\hat{\oc}$ on them:
\begin{definition}[Simplified exponential on nonstandard predicative games]
The \emph{\bfseries simplified exponential} of an np-game $G$ is the np-game $\hat{\oc} G \colonequals \{ \, \hat{\oc} \sigma \mid \sigma :: G \, \}$.
\end{definition}

\begin{notation}
We adopt the notations for exponential $\oc$ on games given at the end of Section~\ref{ArenasLegalPositionsGamesAndTreeSkeletons} also for simplified exponential $\hat{\oc}$ on np-games in the evident way.
It in particular implies that $\hat{\oc} \sigma = \sigma^\dagger :: \hat{\oc} G$ holds for any np-game $G$ and t-skeleton $\sigma :: G$.
\end{notation}

\begin{remark}
The previous work \cite{yamada2016game} necessitates the \emph{countably infinite tensor} on t-skeletons $(\sigma_i)_{i \in \mathbb{N}}$ for exponential $\oc$ on p-games in order to systematically accommodate \emph{non-innocent} t-skeletons on exponentials. 
In contrast, we shall use simplified exponential $\hat{\oc}$ on np-games exclusively on the domains of w.r.w. linear implications, where O must play by \emph{innocent} t-skeletons which are all \emph{promotions}.
Hence, it suffices to define simplified exponential on np-games as above. 
\end{remark}

Let us proceed to define other constructions on np-games:
\begin{definition}[Winning-realizer-wise implication]
\label{DefWRWImplication}
The \emph{\bfseries winning-realizer-wise (w.r.w.) implication} between np-games $G$ and $H$ is the np-game $\hat{\oc} G \rightarrowtriangle H$.
\end{definition}

\begin{definition}[Product on nonstandard predicative games]
The \emph{\bfseries product} of np-games $G$ and $H$ is the np-game $G \mathbin{\&} H \colonequals \{ \langle \sigma, \tau \rangle \mid \sigma ::, \tau :: H \, \}$.
\end{definition}

Recall once again that McCusker \cite{abramsky1999game,mccusker1998games} has to focus on well-opened games for the use of simplified exponential $\hat{\oc}$ on games.
For the same reason, we shall henceforth focus on well-opened np-games, which the following theorem supports:
\begin{theorem}[Well-defined constructions on nonstandard predicative games]
\label{ThmWellDefinedConstructionsOnNonstandardPredicativeGames}
Np-games (resp. well-opened ones) are closed under product and w.r.w. implication.
\end{theorem}
\begin{proof}
Closure of np-games (resp. well-opened ones) under product is obvious, and that under implication follows immediately from Lemma~\ref{LemWellDefinedWRWLI}.
\end{proof}

At this point, we can sketch how we shall validate CT as follows.
Consider the np-game $\hat{\oc} (\hat{\oc} N \rightarrowtriangle N) \rightarrowtriangle N$, which interprets the type $\mathsf{(N \Rightarrow N) \Rightarrow N}$.
We then define a t-skeleton $\uplus \mathrm{ctl} : \hat{\oc} (\hat{\oc} N \rightarrowtriangle N) \rightarrowtriangle N$ that plays as in the following diagram: 
\begin{mathpar}
\begin{tabular}{ccccc}
$\hat{\oc} (\hat{\oc} N$ & $\rightarrowtriangle$ & $N)$ & $\stackrel{\uplus \mathrm{ctl}}{\rightarrowtriangle}$ & $N$ \\ \hline 
&&&$q_{\hat{\oc} (\hat{\oc} N \rightarrowtriangle N) \rightarrowtriangle N}$& \\
&&&$\uplus \mathrm{ctl}$& \\
&&&&\tikzmark{cmultimap1} $q^{[e]}$ \tikzmark{cmultimap3000} \\
&&&&$(e^\ddagger)^{[e]}$ \tikzmark{dmultimap3000}
\end{tabular}
\begin{tikzpicture}[overlay, remember picture, yshift=.25\baselineskip]
\draw [->] ({pic cs:dmultimap3000}) [bend right] to ({pic cs:cmultimap3000});
\end{tikzpicture}
\end{mathpar}
Note that O must perform the first move $q^{[e]}$ in $\hat{\oc} (\hat{\oc} N \rightarrowtriangle N) \rightarrowtriangle N$, which specifies the canonical realizer $e \in \pi_2(\mathscr{R}^{\mathrm{cp}}_{\mathrm{wr}}(\hat{\oc} (\hat{\oc} N \rightarrowtriangle N)))$ for his intended computation on the domain $\hat{\oc} (\hat{\oc} N \rightarrowtriangle N)$.
The computation must be by a winning, recursive t-skeleton on $\hat{\oc} (\hat{\oc} N \rightarrowtriangle N)$. 
Then, the t-skeleton $\uplus \mathrm{ctl}$ directs P to `copy-cat' the given realizer $e$ as the P-move $(e^\ddagger)^{[e]}$, where $e^\ddagger$ realizes the winning, recursive t-skeleton on $\hat{\oc} N \rightarrowtriangle N$ that essentially coincides with the one realized by $e$.
In this way, the t-skeleton $\uplus \mathrm{ctl}$ models the nontrivial part of CT (\ref{CTinMLTT}), which is the main idea of the present work. 

Moreover, it is straightforward to see that the t-skeleton $\uplus \mathrm{ctl}$ is \emph{recursive} so that we may \emph{constructively} model MLTT plus CT. 
Note that the position $q^{[e]} \in \uplus \mathrm{ctl}$ is not even partially decidable since otherwise it would contradict the well-known fact that the set of all total recursive functions is not recursively enumerable; for this point, we have carefully `compromised' our notions of (the `effectivity' of) recursive t-skeletons and realizer-maps, as already remarked, so that $\uplus \mathrm{ctl}$ is recursive. 

Also, let us emphasize the point that the t-skeleton $\uplus \mathrm{ctl}$ resolves the \emph{dilemma between intensionality and extensionality} explained in Section~\ref{WhyDifficult}. 
In fact, O must exhibit the canonical realizer for his winning, recursive t-skeleton on the domain \emph{on the nose} at his first move in a play in $\uplus\mathrm{ctl}$, which is very intensional like realizability \`{a} la Kleene, while $\uplus\mathrm{ctl}$ itself is just as abstract as morphisms in the game semantics \cite{yamada2016game} or realizability \`{a} la assemblies, i.e., extensional enough to model MLTT.

The rest of this article is dedicated to verifying this solution in detail.
Towards this end, we first modify constructions on t-skeletons on linear implication $\multimap$ \cite{yamada2016game} into appropriate ones on w.r.w. linear implication $\rightarrowtriangle$:

\begin{definition}[Winning-realizer-wise copy-cats]
The \emph{\bfseries winning-realizer-wise (w.r.w.) copy-cat} on an np-game $G$ is the DoWRWLI $\uplus \mathrm{cp}_G :: G \rightarrowtriangle G$ whose component $(\mathrm{cp}_G)_{(\gamma, e)}$ for each $(\gamma, e) \in \mathscr{R}^{\mathrm{cp}}_{\mathrm{wr}}(G)$ is given by
\begin{equation*}
(\mathrm{cp}_G)_{(\gamma, e)} \colonequals \{ \, \boldsymbol{s} \in \gamma^{[0]} \multimap \gamma^{[1]} \mid \forall \boldsymbol{t} \preceq{\boldsymbol{s}}. \ \mathrm{Even}(\boldsymbol{t}) \Rightarrow \boldsymbol{t} \upharpoonright \gamma^{[0]} = \boldsymbol{t} \upharpoonright \gamma^{[1]} \, \},
\end{equation*}
where the superscripts $(\_)^{[i]}$ ($i = 0, 1$) are to distinguish the two copies of $\gamma$.
\end{definition}

\begin{definition}[Winning-realizer-wise derelictions]
\label{DefWRWDerelictionTreeSkeletons}
The \emph{\bfseries winning-realizer-wise (w.r.w.) dereliction} on an np-game $G$ is the DoWRWLI $\uplus \mathrm{der}_G :: \hat{\oc} G \rightarrowtriangle G$ whose component $(\mathrm{der}_G)_{(\gamma^\dagger, e')}$ for each $(\gamma^\dagger, e') \in \mathscr{R}^{\mathrm{cp}}_{\mathrm{wr}}(\hat{\oc} G)$ is given by
\begin{equation*}
(\mathrm{der}_G)_{(\gamma^\dagger, e')} \colonequals \{ \, \boldsymbol{s} \in \gamma^\dagger \multimap \gamma \mid \forall \boldsymbol{t} \preceq{\boldsymbol{s}}. \ \mathrm{Even}(\boldsymbol{t}) \Rightarrow \boldsymbol{t} \upharpoonright \gamma^\dagger = \boldsymbol{t} \upharpoonright \gamma \, \}.
\end{equation*}
\end{definition}

\begin{lemma}[Well-defined winning-realizer-wise copy-cats and derelictions]
\label{LemWellDefinedCopyCatAndDerelictionTreeSkeletons}
Let $G$ be an np-game.
\begin{enumerate}

\item The w.r.w. copy-cat $\uplus\mathrm{cp}_G$ is a winning, recursive t-skeleton on the np-game $G \rightarrowtriangle G$, and it is the unit for composition on DoWRWLIs; 

\item The w.r.w. dereliction $\uplus\mathrm{der}_G$ is a winning, recursive t-skeleton on the np-game $\hat{\oc} G \rightarrowtriangle G$ if $G$ is well-opened. 

\end{enumerate}
\end{lemma}
\begin{proof}
Similar to the case of copy-cats and derelictions given in \cite{yamada2016game}, where $\uplus\mathrm{cp}_G$ and $\uplus\mathrm{der}_G$ are clearly recursive, and they are the trivial t-skeleton $\top$ if $\mathscr{R}_{\mathrm{wr}}^{\mathrm{cp}}(G) = \emptyset$. 
\end{proof}

\begin{definition}[Constructions on DoWRWLIs]
\label{DefConstructionsOnDoWRWLIs}
Given np-games $A$, $B$, $C$ and $D$, and DoWRWLIs $\uplus \phi :: A \rightarrowtriangle B$, $\uplus \rho :: A \rightarrowtriangle C$ and $\uplus \varphi :: \hat{\oc} A \rightarrowtriangle B$, we define
\begin{itemize}

\item The \emph{\bfseries pairing} $\langle \uplus \phi, \uplus \rho \rangle \colonequals \uplus \langle \phi, \rho \rangle :: A \rightarrowtriangle B \mathbin{\&} C$ of $\uplus \phi$ and $\uplus \rho$ by
\begin{mathpar}
\langle \phi, \rho \rangle_{(\gamma, e)} \colonequals \langle \phi_{(\gamma, e)}, \rho_{(\gamma, e)} \rangle
\end{mathpar}
for each $(\gamma, e) \in \mathscr{R}^{\mathrm{cp}}_{\mathrm{wr}}(A)$;

\item The \emph{\bfseries promotion} $(\uplus \varphi)^\dagger \colonequals \uplus (\varphi^\dagger) :: \hat{\oc} A \rightarrowtriangle \hat{\oc} B$ of $\uplus \varphi$ by
\begin{mathpar}
(\varphi^\dagger)_{(\gamma^\dagger, e')} \colonequals (\varphi_{(\gamma^\dagger, e')})^\dagger
\end{mathpar}
for each $(\gamma^\dagger, e') \in \mathscr{R}^{\mathrm{cp}}_{\mathrm{wr}}(\hat{\oc} A)$ (which allows us to write $\uplus \varphi^\dagger$ for $(\uplus \varphi)^\dagger$),
\end{itemize}
where we (can and do) assume that the corresponding constructions on realizer-maps are `effective' by an appropriate choice of canonical realizers (Definition~\ref{DefCanonicalRealizers}). 
\end{definition}

\begin{remark}
We fix \emph{canonical} realizers since they enable us to concisely define these constructions on DoWRWLIs in such a way that Lemma~\ref{LemWRWPromotionLemma} below holds.
\end{remark}

\begin{lemma}[Well-defined constructions on DoWRWLIs]
\label{LemWellDefinedConstructionsOnDoWRWLIs}
DoWRWLIs (resp. winning, recursive ones) are closed under pairing and promotion. 
\end{lemma}
\begin{proof}
Similar to the corresponding constructions on UoPLIs \cite{yamada2016game}, where recursive DoWRWLIs are clearly closed under pairing and promotion.
\end{proof}

\begin{lemma}[Winning-realizer-wise promotion lemma]
\label{LemWRWPromotionLemma}
Let $A$, $B$, $C$ and $D$ be np-games, and $\uplus \phi :: \hat{\oc} A \rightarrowtriangle B$, $\uplus \psi :: \hat{\oc} B \rightarrowtriangle C$ and $\uplus \varphi :: \hat{\oc} C \rightarrowtriangle D$ DoWRWLIs. 
\begin{enumerate}

\item $\uplus \mathrm{der}_A^\dagger = \uplus \mathrm{cp}_{\hat{\oc} A} :: \hat{\oc} A \rightarrowtriangle \hat{\oc} A$ and $\uplus \mathrm{der}_B \circ \uplus \phi^\dagger = \uplus \phi :: \hat{\oc} A \rightarrowtriangle B$;

\item $\uplus \varphi \circ (\uplus \psi \circ \uplus \phi^\dagger)^\dagger = (\uplus \varphi \circ \uplus \psi^\dagger) \circ \uplus \phi^\dagger :: \hat{\oc} A \rightarrowtriangle D$.

\end{enumerate}
\end{lemma}
\begin{proof}
Similar to the case of UoPLIs \cite{yamada2016game}, where it is crucial for the first and the second clauses to fix a choice of canonical realizers (Definition~\ref{DefCanonicalRealizers}). 
\end{proof}

We are now ready to define:
\begin{definition}[Category of realizability \`{a} la game semantics]
\label{DefCategoryWRWPG}
The category $\mathbb{NPG}_{\mathrm{wrw}}^{\mathrm{wo}}$ consists of 
\begin{itemize}

\item Well-opened np-games as objects;

\item Winning, recursive t-skeletons on $\hat{\oc} A \rightarrowtriangle B$ as morphisms $A \rightarrow B$;

\item The composition $\uplus \psi \bullet \uplus \phi : A \rightarrow C$ of morphisms $\uplus \phi : A \rightarrow B$ and $\uplus \psi : B \rightarrow C$ given by $\uplus \psi \bullet \uplus \phi \colonequals \uplus \psi \circ \uplus \phi^\dagger$;

\item W.r.w. derelictions $\uplus \mathrm{der}_A : \hat{\oc} A \rightarrowtriangle A$ as identities $A \rightarrow A$. 

\end{itemize}
\end{definition}

The category $\mathbb{NPG}_{\mathrm{wrw}}^{\mathrm{wo}}$ is well-defined by Lemmata~\ref{LemWellDefinedDoWRWLIs}, \ref{LemWellDefinedCopyCatAndDerelictionTreeSkeletons}, \ref{LemWellDefinedConstructionsOnDoWRWLIs} and \ref{LemWRWPromotionLemma}, and moreover \emph{cartesian closed} (which is an immediate corollary of the theorem given later that $\mathbb{NPG}_{\mathrm{wrw}}^{\mathrm{wo}}$ models MLTT equipped with unit-, pi- and sigma-types via a CwF and semantic type formers, and so the following proof is just a sketch):
\begin{proposition}[Cartesian closure]
The category $\mathbb{NPG}_{\mathrm{wrw}}^{\mathrm{wo}}$ is cartesian closed, where a terminal object, products and exponential objects are given by the unit np-game $\boldsymbol{1}$, product $\&$ and w.r.w. implication $\hat{\oc} (\_) \rightarrowtriangle (\_)$, respectively.
\end{proposition}
\begin{proof}[Proof (sketch)]
By Theorem~\ref{ThmWellDefinedConstructionsOnNonstandardPredicativeGames} and Lemmata~\ref{LemWellDefinedConstructionsOnDoWRWLIs} and \ref{LemWRWPromotionLemma}, where note that t-skeletons in $\mathbb{NPG}_{\mathrm{wrw}}^{\mathrm{wo}}$ are all winning and recursive so that $\mathbb{NPG}_{\mathrm{wrw}}^{\mathrm{wo}}$ is closed. 
\end{proof}

\begin{notation}
Given $G \in \mathbb{NPG}_{\mathrm{wrw}}^{\mathrm{wo}}$, we write $\mathbb{NPG}_{\mathrm{wrw}}^{\mathrm{wo}}(G)$ for the hom-set $\mathbb{NPG}_{\mathrm{wrw}}^{\mathrm{wo}}(\boldsymbol{1}, G)$. 
Hence, each element of $\mathbb{NPG}_{\mathrm{wrw}}^{\mathrm{wo}}(G)$ is of the form $\uplus \{ \gamma_0^{\hat{\oc} \boldsymbol{1}} \}$ for some $\gamma_0 \in \mathcal{TS}_{\mathrm{wr}}(G)$ (n.b., $\gamma_0 :: \Gamma$ implies $\gamma_0^{\hat{\oc} \boldsymbol{1}} :: \top^\dagger \rightarrowtriangle \Gamma$); we usually write $\uplus \gamma$ for it (i.e., $\gamma = \{ \gamma_0^{\hat{\oc} \boldsymbol{1}} \}$).

Just for convenience, we also write $\mathbb{NPG}_{\mathrm{wrw}}(H)$ for the set of all winning, recursive t-skeletons on $\hat{\oc} \boldsymbol{1} \rightarrowtriangle H$ for any (not necessarily well-opened) np-game $H$.
\end{notation}

Similarly to the game semantics of simple type theories and the game semantics of MLTT \cite{yamada2016game}, the w.r.w. implications $\hat{\oc} (A \mathbin{\&} B) \rightarrowtriangle C$ and $\hat{\oc} A \rightarrowtriangle (\hat{\oc} B \rightarrowtriangle C)$ coincide up to `tags' on moves for any np-games $A$, $B$ and $C$ (n.b., t-skeletons on the w.r.w. implications are only the trivial one $\top$ if $\mathscr{R}_{\mathrm{wr}}^{\mathrm{cp}}(A) = \emptyset$ or $\mathscr{R}_{\mathrm{wr}}^{\mathrm{cp}}(B) = \emptyset$), and therefore currying is modeled in the category $\mathbb{NPG}_{\mathrm{wrw}}^{\mathrm{wo}}$.
In particular, it explains essentially why our model of MLTT in $\mathbb{NPG}_{\mathrm{wrw}}^{\mathrm{wo}}$ validates the $\xi$-rule, as we shall see shortly.

\subsection{Realizability model of intensional Martin-L\"{o}f type theory equipped with formal Church's thesis \`{a} la game semantics}
\label{RealizabilityModelOfMLTTAlaGameSemantics}
In this section, we lift the CCC $\mathbb{NPG}_{\mathrm{wrw}}^{\mathrm{wo}}$ to a \emph{category with families (CwF)} \cite{hofmann1997syntax,dybjer1996internal}, an abstract model of MLTT, equipped with \emph{semantic type formers} \cite{hofmann1997syntax} for unit-, empty-, N-, pi-, sigma- and Id-types so that $\mathbb{NPG}_{\mathrm{wrw}}^{\mathrm{wo}}$ is shown to model MLTT together with these types by \emph{soundness} of the semantics of MLTT in CwFs \cite{hofmann1997syntax}. 

Towards this end, let us first introduce our interpretation of (dependent) types:

\begin{definition}[Dependent nonstandard predicative games]
\label{DefDependentNonstandardPredicativeGames}
A \emph{\bfseries dependent nonstandard predicative (np-) game} on an np-game $\Gamma$ is a family 
\begin{equation*}
A = (A(\uplus \gamma))_{\uplus \gamma \in \mathbb{NPG}_{\mathrm{wrw}}(\Gamma)}
\end{equation*}
of np-games $A(\uplus \gamma)$ such that the union $\bigcup_{\uplus \gamma \in \mathbb{NPG}_{\mathrm{wrw}}(\Gamma)} \mathbin{\vdash_{A(\uplus \gamma)}}$ is well-founded.

A dependent np-game is \emph{\bfseries well-opened} if so are all its components.
We write $\mathscr{D}\mathbb{NPG}_{\mathrm{wrw}}^{\mathrm{wo}}(\Gamma)$ for the set of all well-opened dependent np-games on $\Gamma$.
\end{definition}

Unlike the game semantics of MLTT \cite{yamada2016game}, the indexing t-skeletons of a dependent np-game $A$ on an np-game $\Gamma$ are on $\Gamma$, not $\hat{\oc} \Gamma$, but it does not lose any generality for our interpretation since the domain and the codomain t-skeletons of DoWRWLIs are all \emph{innocent} so that we may assume that those on $\hat{\oc} \Gamma$ are all \emph{promotions}.  

Let us then generalize w.r.w. implication (Definition~\ref{DefWRWImplication}) in such a way that it captures type dependency of pi-types (Definition~\ref{DefWRWPi}): 
\begin{definition}[Integration on dependent nonstandard predicative games]
The \emph{\bfseries integration} of a dependent np-game $A$ on an np-game $\Gamma$ is the np-game
\begin{mathpar}
\int_\Gamma A \colonequals \bigcup_{\uplus \gamma \in \mathbb{NPG}_{\mathrm{wrw}}(\Gamma)}A(\uplus \gamma) \quad \text{(n.b., we often abbreviate $\int_\Gamma$ as $\int$)}.
\end{mathpar}
\end{definition}

\begin{remark}
The integration on dependent np-games is simpler than the corresponding one on p-games in the previous work \cite{yamada2016game} as ours does not have to be det-j complete. 
\end{remark}

\begin{definition}[FoDWRWLIs]
\label{DefFoDWRWLIs}
A \emph{\bfseries family of dependently winning-realizer-wise linear implications (FoDWRWLI)} from an np-game $\Gamma$ to a dependent np-game $A$ on $\Gamma$ is an FoWRWLI $\phi$ from $\Gamma$ to $\int_\Gamma A$ that satisfies
\begin{equation*}
\forall (\gamma, e) \in \mathscr{R}^{\mathrm{cp}}_{\mathrm{wr}}(\Gamma) . \, \mathrm{cod}_\phi(\gamma, e) :: A(\uplus \{ \gamma^{\hat{\oc} \boldsymbol{1}} \}) \quad \text{(n.b., $\gamma :: \Gamma$ implies $\gamma^{\hat{\oc} \boldsymbol{1}} :: \top^\dagger \rightarrowtriangle \Gamma$)}.
\end{equation*}

We write $\mathscr{F_{\mathrm{wrw}}}(\Gamma, A)$ for the set of all FoDWRWLIs from $\Gamma$ to $A$.
\end{definition}

\begin{definition}[Winning-realizer-wise linear-pi and pi]
\label{DefWRWPi}
The \emph{\bfseries winning-realizer-wise (w.r.w.) linear-pi} from an np-game $\Gamma$ to a dependent np-game $A$ on $\Gamma$ is the np-game 
\begin{equation*}
\ell\Pi_{\mathrm{wrw}}(\Gamma, A) \colonequals \{ \, \uplus \phi \mid \phi \in \mathscr{F}_{\mathrm{wrw}}(\Gamma, A) \,  \}, 
\end{equation*}
and the \emph{\bfseries winning-realizer-wise (w.r.w.) pi} from $\Gamma$ to $A$ is the np-game 
\begin{equation*}
\Pi_{\mathrm{wrw}}(\Gamma, A) \colonequals \ell\Pi_{\mathrm{wrw}}(\hat{\oc} \Gamma, A^\ddagger),
\end{equation*}
where $A^\ddagger$ is the dependent np-game on $\hat{\oc} \Gamma$ defined by $A^\ddagger(\uplus \gamma^\dagger) \colonequals A(\uplus \gamma)$ for each $\uplus \gamma^\dagger \in \mathbb{NPG}_{\mathrm{wrw}}(\hat{\oc} \Gamma)$. 
\end{definition}

\begin{lemma}[Well-defined winning-realizer-wise linear-pi and pi]
The w.r.w. linear-pi $\ell\Pi_{\mathrm{wrw}}(\Gamma, A)$ and the w.r.w. pi $\Pi_{\mathrm{wrw}}(\Gamma, A)$ are (well-opened) np-games for any np-game $\Gamma$ and (well-opened) dependent np-game $A$ on $\Gamma$. 
\end{lemma}
\begin{proof}
Immediate from Lemma~\ref{LemWellDefinedWRWLI} and Theorem~\ref{ThmWellDefinedConstructionsOnNonstandardPredicativeGames}.
\end{proof}

Let us similarly define the w.r.w. variant of sigma:
\begin{definition}[Winning-realizer-wise sigma]
The \emph{\bfseries winning-realizer-wise (w.r.w.) sigma} of an np-game $\Gamma$ and a dependent np-game $A$ on $\Gamma$ is the np-game
\begin{equation*}
\Sigma_{\mathrm{wrw}}(\Gamma, A) \colonequals \{ \langle \gamma, \alpha \rangle :: \Gamma \mathbin{\&} \int_\Gamma A \mid \gamma \in \mathcal{TS}_{\mathrm{wr}}(\Gamma) \Rightarrow \alpha :: A (\uplus \{ \gamma^{\hat{\oc} \boldsymbol{1}} \}) \, \}.
\end{equation*}
\end{definition}

\begin{lemma}[Well-defined winning-realizer-wise sigma]
The w.r.w. sigma $\Sigma_{\mathrm{wrw}}(\Gamma, A)$ is a (well-opened) np-game for any (well-opened) np-game $\Gamma$ and (well-opened) dependent np-game $A$ on $\Gamma$. 
\label{LemWellDefinedWRWSigma}
\end{lemma}
\begin{proof}
Immediate from Theorem~\ref{ThmWellDefinedConstructionsOnNonstandardPredicativeGames}.
\end{proof}

Finally, let us introduce our interpretation of Id-types:
\begin{definition}[Identity nonstandard predicative games]
\label{DefIdentityNonstandardPredicativeGames}
Given an np-game $\Gamma$ and t-skeletons $\gamma, \gamma' :: \Gamma$, the \emph{\bfseries identity (Id)} between $\gamma$ and $\gamma'$ is the np-game 
\begin{equation*}
\mathrm{Id}_\Gamma(\gamma, \gamma') \colonequals \begin{cases} \boldsymbol{1} &\text{if $\gamma = \gamma'$;} \\ \boldsymbol{0} &\text{otherwise.} \end{cases}
\end{equation*}
\end{definition}

\begin{lemma}[Well-defined Id on nonstandard predicative games]
Given an np-game $\Gamma$, the Id $\mathrm{Id}_\Gamma(\gamma, \gamma')$ between any $\gamma, \gamma' :: \Gamma$ is a well-opened np-game.
\end{lemma}
\begin{proof}
Obvious. 
\end{proof}

We are now ready to lift the CCC $\mathbb{NPG}_{\mathrm{wrw}}^{\mathrm{wo}}$ to a CwF.
Let us first recall the general definition of CwFs \cite{dybjer1996internal}, where our presentation is based on \cite{hofmann1997syntax}:

\begin{definition}[Categories with families \cite{hofmann1997syntax,dybjer1996internal}]
A \emph{\bfseries category with families (CwF)} is a tuple $\mathcal{C} = (\mathcal{C}, \mathrm{Ty}, \mathrm{Tm}, \_\{\_\}, T, \_.\_, \mathrm{p}, \mathrm{v}, \langle\_,\_\rangle_\_)$,
where
\begin{itemize}

\item $\mathcal{C}$ is a category with a terminal object $T \in \mathcal{C}$;

\item $\mathrm{Ty}$ assigns, to each object $\Gamma \in \mathcal{C}$, a set $\mathrm{Ty}(\Gamma)$, called the set of all \emph{\bfseries types} in the \emph{\bfseries context} $\Gamma$;

\item $\mathrm{Tm}$ assigns, to each pair $(\Gamma, A)$ of an object $\Gamma \in \mathcal{C}$ and a type $A \in \mathrm{Ty}(\Gamma)$, a set $\mathrm{Tm}(\Gamma, A)$, called the set of all \emph{\bfseries terms} of type $A$ in the context $\Gamma$;

\item To each $f : \Delta \to \Gamma$ in $\mathcal{C}$, $\_\{\_\}$ assigns a map $\_\{f\} : \mathrm{Ty}(\Gamma) \to \mathrm{Ty}(\Delta)$, called the \emph{\bfseries substitution on types}, and a family $(\_\{f\}_A)_{A \in \mathrm{Ty}(\Gamma)}$ of maps $\_\{f\}_A : \mathrm{Tm}(\Gamma, A) \to \mathrm{Tm}(\Delta, A\{f\})$, called the \emph{\bfseries substitutions on terms};


\item $\_ . \_$ assigns, to each pair $(\Gamma, A)$ of a context $\Gamma \in \mathcal{C}$ and a type $A \in \mathrm{Ty}(\Gamma)$, a context $\Gamma . A \in \mathcal{C}$, called the \emph{\bfseries comprehension} of $A$;

\item $\mathrm{p}$ (resp. $\mathrm{v}$) associates each pair $(\Gamma, A)$ of a context $\Gamma \in \mathcal{C}$ and a type $A \in \mathrm{Ty}(\Gamma)$ with a morphism $\mathrm{p}(A) : \Gamma . A \to \Gamma$ in $\mathcal{C}$ (resp. a term $\mathrm{v}_A \in \mathrm{Tm}(\Gamma . A, A\{\mathrm{p}(A)\})$), called the \emph{\bfseries first projection} on $A$ (resp. the \emph{\bfseries second projection} on $A$);

\item $\langle \_, \_ \rangle_\_$ assigns, to each triple $(f, A, g)$ of a morphism $f : \Delta \to \Gamma$ in $\mathcal{C}$, a type $A \in \mathrm{Ty}(\Gamma)$ and a term $g \in \mathrm{Tm}(\Delta, A\{f\})$, a morphism $\langle f, g \rangle_A : \Delta \to \Gamma . A$
in $\mathcal{C}$, called the \emph{\bfseries extension} of $f$ by $g$

\end{itemize}
that satisfies
\begin{itemize}

\item \textsc{(Ty-Id)} $A \{ \mathrm{id}_\Gamma \} = A$;

\item \textsc{(Ty-Comp)} $A \{ f \circ e \} = A \{ f \} \{ e \}$;

\item \textsc{(Tm-Id)} $h \{ \mathrm{id}_\Gamma \}_A = h$;

\item \textsc{(Tm-Comp)} $h \{ f \circ e \}_A = h \{ f \}_A \{ e \}_{A\{f\}}$;

\item \textsc{(Cons-L)} $\mathrm{p}(A) \circ \langle f, g \rangle_A = f$;

\item \textsc{(Cons-R)} $\mathrm{v}_A \{ \langle f, g \rangle_A \} = g$;

\item \textsc{(Cons-Nat)} $\langle f, g \rangle_A \circ e = \langle f \circ e, g \{ e \}_{A\{f\}} \rangle_A$;

\item \textsc{(Cons-Id)} $\langle \mathrm{p}(A), \mathrm{v}_A \rangle_A = \mathrm{id}_{\Gamma . A}$

\end{itemize}
for any $\Gamma, \Delta, \Theta \in \mathcal{C}$, $A \in \mathrm{Ty}(\Gamma)$, $f : \Delta \to \Gamma$, $e : \Theta \to \Delta$, $h \in \mathrm{Tm}(\Gamma, A)$ and $g \in \mathrm{Tm}(\Delta, A\{ f \})$.
\end{definition}

Roughly, judgements of MLTT are modeled in a CwF $\mathcal{C}$ by
\begin{align*}
\mathsf{\vdash \Gamma \ ctx} &\mapsto \llbracket \mathsf{\Gamma} \rrbracket \in \mathcal{C}; \\
\mathsf{\Gamma \vdash A \ type} &\mapsto \llbracket \mathsf{A} \rrbracket \in \mathrm{Ty}(\llbracket \mathsf{\Gamma} \rrbracket); \\
\mathsf{\Gamma \vdash a : A} &\mapsto \llbracket \mathsf{a} \rrbracket \in \mathrm{Tm}(\llbracket \mathsf{\Gamma} \rrbracket, \llbracket \mathsf{A} \rrbracket); \\
\mathsf{\vdash \Gamma = \Delta \ ctx} &\Rightarrow \llbracket \mathsf{\Gamma} \rrbracket = \llbracket \mathsf{\Delta} \rrbracket \in \mathcal{C}; \\
\mathsf{\Gamma \vdash A = B \ type} &\Rightarrow \llbracket \mathsf{A} \rrbracket = \llbracket \mathsf{B} \rrbracket \in \mathrm{Ty}(\llbracket \mathsf{\Gamma} \rrbracket); \\ 
\mathsf{\Gamma \vdash a = a' : A} &\Rightarrow \llbracket \mathsf{a} \rrbracket = \llbracket \mathsf{a'} \rrbracket \in \mathrm{Tm}(\llbracket \mathsf{\Gamma} \rrbracket, \llbracket \mathsf{A} \rrbracket),
\end{align*}
where $\llbracket \_ \rrbracket$ denotes the \emph{semantic map} or \emph{interpretation} \cite{hofmann1997syntax}.
Strictly speaking, the first three maps define an interpretation $\llbracket \_ \rrbracket$ of MLTT in $\mathcal{C}$, while the last three logical implications are \emph{soundness} of the interpretation $\llbracket \_ \rrbracket$. 
See \cite{hofmann1997syntax} for the details.

Let us now turn to introducing our CwF of \emph{realizability \`{a} la game semantics}:
\begin{definition}[CwF of realizability \`{a} la game semantics]
\label{DefCwFMuWPG}
The CwF $\mathbb{NPG}_{\mathrm{wrw}}^{\mathrm{wo}}$ is the tuple $(\mathbb{NPG}_{\mathrm{wrw}}^{\mathrm{wo}}, \mathrm{Ty}, \mathrm{Tm}, \_\{\_\}, \boldsymbol{1}, \_.\_, \mathrm{p}, \mathrm{v}, \langle\_,\_\rangle_\_)$,
where
\begin{itemize}

\item The category $\mathbb{NPG}_{\mathrm{wrw}}^{\mathrm{wo}}$ is the one defined in Definition~\ref{DefCategoryWRWPG}, and $\boldsymbol{1} \in \mathbb{NPG}_{\mathrm{wrw}}^{\mathrm{wo}}$ is the unit np-game (Example~\ref{ExamplesOfWellOpenedNonstandardPredicativeGames});


\item Given $\Gamma \in \mathbb{NPG}_{\mathrm{wrw}}^{\mathrm{wo}}$ and $A \in \mathscr{D}\mathbb{NPG}_{\mathrm{wrw}}^{\mathrm{wo}}(\Gamma)$, we define $\mathrm{Ty}(\Gamma) \colonequals \mathscr{D}\mathbb{NPG}_{\mathrm{wrw}}^{\mathrm{wo}}(\Gamma)$ and $\mathrm{Tm}(\Gamma, A) \colonequals \mathcal{TS}_{\mathrm{wr}}(\Pi_{\mathrm{wrw}}(\Gamma, A))$; 

\item Given $\uplus \phi : \Delta \rightarrow \Gamma$ in $\mathbb{NPG}_{\mathrm{wrw}}^{\mathrm{wo}}$, we define the map $\_\{\uplus \phi\} : \mathrm{Ty}(\Gamma) \to \mathrm{Ty}(\Delta)$ by
\begin{equation*}
A \{ \uplus \phi \} \colonequals (A(\uplus \phi \bullet \uplus \delta))_{\uplus \delta \in \mathbb{NPG}_{\mathrm{wrw}}^{\mathrm{wo}}(\Delta)} \quad (A \in \mathrm{Ty}(\Gamma)),
\end{equation*}
and the map $\_\{\uplus \phi\}_A : \mathrm{Tm}(\Gamma, A) \to \mathrm{Tm}(\Delta, A\{\uplus \phi\})$ by 
\begin{equation*}
\uplus \alpha \{ \uplus \phi \}_A \colonequals \uplus \alpha \bullet \uplus \phi \quad (\uplus \alpha \in \mathrm{Tm}(\Gamma, A));
\end{equation*}

\item The comprehension $\Gamma.A$ is given by $\Gamma.A \colonequals \Sigma_{\mathrm{wrw}} (\Gamma, A)$, the first projection $\mathrm{p}(A) \colonequals \mathrm{fst}_{\Sigma_{\mathrm{wrw}}(\Gamma, A)} : \Sigma_{\mathrm{wrw}} (\Gamma, A) \to \Gamma$ is the w.r.w. dereliction $\uplus \mathrm{der}_\Gamma$ up to `tags,' the second projection $\mathrm{v}_A \colonequals \mathrm{snd}_{\Sigma_{\mathrm{wrw}}(\Gamma, A)} \in \mathrm{Tm}(\Sigma_{\mathrm{wrw}}(\Gamma, A), A\{\mathrm{p}(A)\})$ is the w.r.w. dereliction $\uplus \mathrm{der}_{\int_\Gamma A}$ up to `tags,' and the extension $\langle \uplus \phi, \uplus \alpha \rangle_A : \Delta \rightarrow \Sigma_{\mathrm{wrw}} (\Gamma, A)$ is the pairing $\langle \uplus \phi, \uplus \alpha \rangle$ of the DoWRWLIs $\uplus \phi$ and $\uplus \alpha$.




\end{itemize}

\end{definition}

\begin{convention}
We frequently omit the subscript $(\_)_A$ on $\_\{ \_ \}_A$ and $\langle \_, \_ \rangle_A$, and the one $(\_)_{\Sigma_{\mathrm{wrw}}(\Gamma, A)}$ on $\mathrm{fst}_{\Sigma_{\mathrm{wrw}}(\Gamma, A)}$ and $\mathrm{snd}_{\Sigma_{\mathrm{wrw}}(\Gamma, A)}$. 
\end{convention}

\begin{theorem}[Well-defined CwF of realizability \`{a} la game semantics]
\label{ThmWellDefinedWPG}
The structure $\mathbb{NPG}_{\mathrm{wrw}}^{\mathrm{wo}}$ given in Definition~\ref{DefCwFMuWPG} forms a well-defined CwF.
\end{theorem}
\begin{proof}
We focus on substitutions on terms and extensions as other verifications are straightforward. 
Let $\Gamma, \Delta \in \mathbb{NPG}_{\mathrm{wrw}}^{\mathrm{wo}}$, $A \in \mathscr{D}\mathbb{NPG}_{\mathrm{wrw}}^{\mathrm{wo}}(\Gamma)$, $\uplus \phi \in \mathbb{NPG}_{\mathrm{wrw}}^{\mathrm{wo}}(\Delta, \Gamma)$, $\uplus \alpha \in \mathcal{TS}_{\mathrm{wr}}(\Pi_{\mathrm{wrw}}(\Gamma, A))$ and $\uplus \tilde{\alpha} \in \mathcal{TS}_{\mathrm{wr}}(\Pi_{\mathrm{wrw}}(\Delta, A\{ \uplus \phi \}))$, and assume that $\mathscr{R}_{\mathrm{wr}}^{\mathrm{cp}}(\Gamma) \neq \emptyset$ and $\mathscr{R}_{\mathrm{wr}}^{\mathrm{cp}}(\Delta) \neq \emptyset$, following our convention, since the other cases are trivial. 

By Lemma~\ref{LemWellDefinedDoWRWLIs}, $\uplus \alpha \{ \uplus \phi \} = \uplus \alpha \bullet \uplus \phi$ is a winning, recursive t-skeleton on $\hat{\oc} \Delta \rightarrowtriangle \int_\Gamma A$. 
Take any $\uplus \delta \in \mathbb{NPG}_{\mathrm{wrw}}^{\mathrm{wo}}(\Delta)$; for proving $\uplus \alpha \{ \uplus \phi \} \in \mathcal{TS}_{\mathrm{wr}}(\Pi_{\mathrm{wrw}}(\Delta, A\{ \uplus \phi \}))$, it remains to show $(\uplus \alpha \bullet \uplus \phi) \bullet \uplus \delta :: \oc \boldsymbol{1} \rightarrowtriangle A\{ \uplus \phi \}(\uplus \delta)$.
Then, we calculate
\begin{align*}
(\uplus \alpha \bullet \uplus \phi) \bullet \uplus \delta &= (\uplus \alpha \circ \uplus \phi^\dagger) \circ \uplus \delta^\dagger \\
&= \uplus \alpha \circ (\uplus \phi \circ \uplus \delta^\dagger)^\dagger \quad \text{(by Lemma~\ref{LemWRWPromotionLemma})} \\
&= \uplus \alpha \circ (\uplus \phi \bullet \uplus \delta)^\dagger :: \oc \boldsymbol{1} \rightarrowtriangle A^\ddagger((\uplus \phi \bullet \uplus \delta)^\dagger),
\end{align*}
where
\begin{align*}
A^\ddagger((\uplus \phi \bullet \uplus \delta)^\dagger) &= A(\uplus \phi \bullet \uplus \delta) \\
&= A\{ \uplus \phi \}(\uplus \delta).
\end{align*}
Hence, we have shown that $\uplus \alpha \{ \uplus \phi \} \in \mathcal{TS}_{\mathrm{wr}}(\Pi_{\mathrm{wrw}}(\Delta, A\{ \uplus \phi \}))$ holds. 

Similarly, the extension $\langle \uplus \phi, \uplus \tilde{\alpha} \rangle$ is a winning, recursive t-skeleton on $\hat{\oc} \Delta \rightarrowtriangle \Gamma \mathbin{\&} (\int_\Gamma A)$ by Lemma~\ref{LemWellDefinedConstructionsOnDoWRWLIs}.
Hence, for proving $\langle \uplus \phi, \uplus \tilde{\alpha} \rangle \in \mathbb{NPG}_{\mathrm{wrw}}^{\mathrm{wo}}(\Delta, \Sigma_{\mathrm{wrw}}(\Gamma, A))$, it suffices to show $\uplus \tilde{\alpha} \bullet \uplus \delta :: \oc \boldsymbol{1} \rightarrowtriangle A(\uplus \phi \bullet \uplus \delta)$ since we have $\langle \uplus \phi, \uplus \tilde{\alpha} \rangle \bullet \uplus \delta = \langle \uplus \phi \bullet \uplus \delta, \uplus \tilde{\alpha} \bullet \uplus \delta \rangle$ with $\uplus \phi \bullet \uplus \delta$ winning and recursive.
Then, we calculate
\begin{equation*}
\uplus \tilde{\alpha} \bullet \uplus \delta :: \oc \boldsymbol{1} \rightarrowtriangle A\{ \uplus \phi \}(\uplus \delta) = \oc \boldsymbol{1} \rightarrowtriangle A(\uplus \phi \bullet \uplus \delta).
\end{equation*}

Finally, let us verify the required equations:
\begin{itemize}

\item \textsc{(Ty-Id)} $A\{\mathrm{id}_\Gamma\} = (A(\uplus \mathrm{der}_\Gamma \bullet \uplus \gamma))_{\uplus \gamma \in \mathbb{NPG}_{\mathrm{wrw}}^{\mathrm{wo}}(\Gamma)} = (A(\uplus \gamma))_{\uplus \gamma \in \mathbb{NPG}_{\mathrm{wrw}}^{\mathrm{wo}}(\Gamma)} = A$ by Lemma~\ref{LemWRWPromotionLemma};

\item \textsc{(Ty-Comp)} Given $\Theta \in \mathbb{NPG}_{\mathrm{wrw}}^{\mathrm{wo}}$ and $\uplus \psi : \Gamma \rightarrow \Theta$ in $\mathbb{NPG}_{\mathrm{wrw}}^{\mathrm{wo}}$, we calculate
\begin{align*}
A\{ \uplus \psi \bullet \uplus \phi \} &= (A((\uplus \psi \bullet \uplus \phi) \bullet \uplus \delta))_{\uplus \delta \in \mathbb{NPG}_{\mathrm{wrw}}^{\mathrm{wo}}(\Delta)} \\
&= (A(\uplus \psi \bullet (\uplus \phi \bullet \uplus \delta)))_{\uplus \delta \in \mathbb{NPG}_{\mathrm{wrw}}^{\mathrm{wo}}(\Delta)} \quad \text{(by Lemma~\ref{LemWRWPromotionLemma})} \\
&= (A \{ \uplus \psi \} (\uplus \phi \bullet \uplus \delta))_{\uplus \delta \in \mathbb{NPG}_{\mathrm{wrw}}^{\mathrm{wo}}(\Delta)} \\
&= (A \{ \uplus \psi \} \{ \uplus \phi \} (\uplus \delta))_{\uplus \delta \in \mathbb{NPG}_{\mathrm{wrw}}^{\mathrm{wo}}(\Delta)} \\
&= A \{ \uplus \psi \} \{ \uplus \phi \};
\end{align*}

\item \textsc{(Tm-Id)} $\uplus \alpha \{ \mathrm{id}_\Gamma \} = \uplus \alpha \bullet \uplus \mathrm{der}_\Gamma = \uplus \alpha \circ \uplus \mathrm{cp}_{\hat{\oc} \Gamma} = \uplus \alpha$ by Lemmata~\ref{LemWellDefinedCopyCatAndDerelictionTreeSkeletons} and~\ref{LemWRWPromotionLemma};

\item \textsc{(Tm-Comp)} $\uplus \alpha \{ \uplus \psi \bullet \uplus \phi \} = \uplus \alpha \bullet (\uplus \psi \bullet \uplus \phi) = (\uplus \alpha \bullet \uplus \psi) \bullet \uplus \phi = \uplus \alpha \{ \uplus \psi \} \bullet \uplus \phi = \uplus \alpha \{ \uplus \psi \} \{ \uplus \phi \}$ by Lemma~\ref{LemWRWPromotionLemma};

\item \textsc{(Cons-L)} $\mathrm{p}(A) \bullet \langle \uplus \psi, \uplus \tilde{\alpha} \rangle = \mathrm{fst} \circ \langle \uplus \psi, \uplus \tilde{\alpha} \rangle^\dagger = \mathrm{fst} \circ \langle \uplus \psi^\dagger, \uplus \tilde{\alpha}^\dagger \rangle = \uplus \psi$;

\item \textsc{(Cons-R)} $\mathrm{v}_A \{ \langle \uplus \psi, \uplus \tilde{\alpha} \rangle \} = \mathrm{snd} \circ \langle \uplus \psi, \uplus \tilde{\alpha} \rangle^\dagger = \mathrm{snd} \circ \langle \uplus \psi^\dagger, \uplus \tilde{\alpha}^\dagger \rangle = \uplus \tilde{\alpha}$;

\item \textsc{(Cons-Nat)} $\langle \uplus \psi, \uplus \tilde{\alpha} \rangle \bullet \uplus \phi = \langle \uplus \psi \bullet \uplus \phi, \uplus \tilde{\alpha} \bullet \uplus \phi \rangle  = \langle \uplus \psi \bullet \uplus \phi, \uplus \tilde{\alpha} \{ \uplus \phi \} \rangle$;

\item \textsc{(Cons-Id)} $\langle \mathrm{p}(A), \mathrm{v}_A \rangle = \langle \mathrm{fst}_{\Sigma_{\mathrm{wrw}}(\Gamma, A)}, \mathrm{snd}_{\Sigma_{\mathrm{wrw}}(\Gamma, A)} \rangle = \uplus \mathrm{der}_{\Sigma_{\mathrm{wrw}}(\Gamma, A)} = \mathrm{id}_{\Gamma . A}$,

\end{itemize}
which completes the proof. 
\end{proof}

Let us then proceed to equip the CwF $\mathbb{NPG}_{\mathrm{wrw}}^{\mathrm{wo}}$ with semantic type formers for unit-, empty-, N-, pi-, sigma- and Id-types.
We begin with \emph{pi-types} (Appendix~\ref{DependentFunctionTypes}).
Recall first the semantic type former for pi-types in an arbitrary CwF:
\begin{definition}[CwFs with pi-types \cite{hofmann1997syntax}]
A CwF $\mathcal{C}$ \emph{\bfseries supports pi} if
\begin{itemize}

\item \textsc{($\Pi$-Form)} Given $\Gamma \in \mathcal{C}$, $A \in \mathrm{Ty}(\Gamma)$ and $B \in \mathrm{Ty}(\Gamma . A)$, there is a type $\Pi (A, B) \in \mathrm{Ty}(\Gamma)$;

\item \textsc{($\Pi$-Intro)} Given $b \in \mathrm{Tm}(\Gamma . A, B)$, there is a term $\lambda_{A, B} (b) \in \mathrm{Tm}(\Gamma, \Pi (A, B))$;

\item \textsc{($\Pi$-Elim)} Given $k \in \mathrm{Tm}(\Gamma, \Pi (A, B))$ and $a \in \mathrm{Tm}(\Gamma, A)$, there is a term $\mathrm{App}_{A, B} (k, a) \in \mathrm{Tm}(\Gamma, B\{ \overline{a} \})$, where $\overline{a} \colonequals \langle \mathrm{id}_\Gamma, a \rangle_A : \Gamma \to \Gamma . A$;

\item \textsc{($\Pi$-Comp)} $\mathrm{App}_{A, B} (\lambda_{A, B} (b) , a) = b \{ \overline{a} \}$;

\item \textsc{($\Pi$-Subst)} Given $\Delta \in \mathcal{C}$ and $f : \Delta \to \Gamma$ in $\mathcal{C}$, $\Pi (A, B) \{ f \} = \Pi (A\{f\}, B\{f^+\})$, where $f^+ \colonequals \langle f \circ \mathrm{p}(A\{ f \}), \mathrm{v}_{A\{ f \}} \rangle_A : \Delta . A\{ f \} \to \Gamma . A$;

\item \textsc{($\lambda$-Subst)} $\lambda_{A, B} (b) \{ f \} = \lambda_{A\{f\}, B\{f^+\}} (b \{ f^+ \}) \in \mathrm{Tm} (\Delta, \Pi (A\{f\}, B\{ f^+\}))$;

\item \textsc{(App-Subst)} $\mathrm{App}_{A, B} (k, a) \{ f \} = \mathrm{App}_{A\{f\}, B\{f^+\}} (k \{ f \}, a \{ f \}) \in \mathrm{Tm} (\Delta, B\{ \overline{a} \} \{ f \})$. 
\end{itemize}

Furthermore, $\mathcal{C}$ \emph{\bfseries supports pi in the strict sense} if it additionally satisfies
\begin{itemize}
\item \textsc{($\lambda$-Uniq)} $\lambda_{A, B} \circ \mathrm{App}_{A\{\mathrm{p}(A)\}, B\{\mathrm{p}(A)^+\}}(k\{ \mathrm{p}(A) \}, \mathrm{v}_A) = k$.

\end{itemize}
\end{definition}

Pi-types (with the $\eta$-rule) are modeled in a CwF that supports pi (in the strict sense); see Appendix~\ref{DependentFunctionTypes} for the details.

Let us now show that the CwF $\mathbb{NPG}_{\mathrm{wrw}}^{\mathrm{wo}}$ supports pi in the strict sense:

\begin{lemma}[Winning-realizer-wise currying lemma]
\label{LemWRWCurryingLemma}
Given $\Gamma \in \mathbb{NPG}_{\mathrm{wrw}}^{\mathrm{wo}}$, $A \in \mathscr{D}\mathbb{NPG}_{\mathrm{wrw}}^{\mathrm{wo}}(\Gamma)$ and $B \in \mathscr{D}\mathbb{NPG}_{\mathrm{wrw}}^{\mathrm{wo}}(\Sigma_{\mathrm{wrw}}(\Gamma, A))$, there is a bijection 
\begin{equation*}
\lambda_{A, B} : \mathcal{TS}_{\mathrm{wr}}(\Pi_{\mathrm{wrw}}(\Sigma_{\mathrm{wrw}}(\Gamma, A), B)) \stackrel{\sim}{\rightarrow} \mathcal{TS}_{\mathrm{wr}}(\Pi_{\mathrm{wrw}}(\Gamma, \Pi_{\mathrm{wrw}}(A, B))),
\end{equation*}
where $\Pi_{\mathrm{wrw}}(A, B) \colonequals (\Pi_{\mathrm{wrw}}(A(\uplus \gamma), B_{\uplus \gamma}))_{\uplus \gamma \in \mathbb{NPG}_{\mathrm{wrw}}^{\mathrm{wo}}(\Gamma)}$ and $B_{\uplus \gamma} \colonequals (B(\langle \uplus \gamma, \uplus \alpha \rangle)_{\uplus \alpha \in \mathbb{NPG}_{\mathrm{wrw}}^{\mathrm{wo}}(A(\uplus \gamma))}$.
\end{lemma}
\begin{proof}
It suffices to show that there is a bijection between (not necessarily winning or recursive) t-skeletons on $\Pi_{\mathrm{wrw}}(\Sigma_{\mathrm{wrw}}(\Gamma, A), B)$ and those on $\Pi_{\mathrm{wrw}}(\Gamma, \Pi_{\mathrm{wrw}}(A, B))$ up to finitary `tags' since such a bijection preserves winning and recursiveness. 

Recall that a t-skeleton $\uplus \phi :: \Pi_{\mathrm{wrw}}(\Sigma_{\mathrm{wrw}}(\Gamma, A), B)$ is the disjoint union on a family $\phi = (\phi_{(\sigma^\dagger, e)})_{(\sigma^\dagger, e) \in \mathscr{R}^{\mathrm{cp}}_{\mathrm{wr}}(\hat{\oc} \Sigma_{\mathrm{wrw}}(\Gamma, A))}$ of t-skeletons $\phi_{(\sigma^\dagger, e)} :: \sigma^\dagger \rightarrowtriangle \mathrm{cod}_\phi(\sigma^\dagger, e)$ such that $\mathrm{cod}_\phi(\sigma^\dagger, e) \in \mathcal{TS}_{\mathrm{wr}}(B(\uplus \{\sigma^{\hat{\oc}\boldsymbol{1}} \}))$, each $\phi_{(\sigma^\dagger, e)}$ does not depend on $\sigma^\dagger$, and the realizer-map $\pi_\phi$ is `effective' (Definitions~\ref{DefFoWRWLIs}, \ref{DefDoWRWLIs}, \ref{DefFoDWRWLIs} and \ref{DefWRWPi}).

Similarly, a t-skeleton $\uplus \psi :: \Pi_{\mathrm{wrw}}(\Gamma, \Pi_{\mathrm{wrw}} (A, B))$ is the disjoint union on a family $\psi = (\psi_{(\gamma^\dagger, f)})_{(\gamma^\dagger, f) \in \mathscr{R}^{\mathrm{cp}}_{\mathrm{wr}}(\hat{\oc} \Gamma)}$ of t-skeletons $\psi_{(\gamma^\dagger, f)} :: \gamma^\dagger \rightarrowtriangle \mathrm{cod}_\psi(\gamma^\dagger, f)$ such that $\mathrm{cod}_\psi(\gamma^\dagger, f) \in \mathcal{TS}_{\mathrm{wr}}(\Pi_{\mathrm{wrw}}(A, B)(\uplus \{\gamma^{\hat{\oc}\boldsymbol{1}} \}))$, each $\psi_{(\gamma^\dagger, f)}$ does not depend on $\gamma^\dagger$, and the realizer-map $\pi_\psi$ is `effective.' 
Note that for each $\sigma_0^\dagger \in \mathcal{TS}_{\mathrm{wr}}(\hat{\oc}\Sigma(\Gamma, A))$ we may write $\sigma_0^\dagger = \langle \gamma_0^\dagger, \alpha_0^\dagger \rangle$ for some $\gamma_0 \in \mathcal{TS}_{\mathrm{wr}}(\Gamma)$ and $\alpha_0 \in \mathcal{TS}_{\mathrm{wr}}(A (\uplus \{ \gamma_0^{\hat{\oc}\boldsymbol{1}} \}))$.

Then, by applying the \emph{currying} on skeletons in game semantics \cite{abramsky1999game,mccusker1998games}, which simply adjusts the finitary `tags' for product and implication appropriately, to each component $\phi_{(\sigma^\dagger, e)}$ of a t-skeleton $\uplus \phi :: \Pi_{\mathrm{wrw}}(\Sigma_{\mathrm{wrw}}(\Gamma, A), B)$ in the evident way, it is easy to see that we get a t-skeleton $\lambda_{A, B}(\uplus \phi) :: \Pi_{\mathrm{wrw}}(\Gamma, \Pi_{\mathrm{wrw}}(A, B))$.
Finally, this operation on t-skeletons has the evident inverse, which completes the proof. 
\end{proof}

\begin{remark}
The proof of Lemma~\ref{LemWRWCurryingLemma} is essentially the same as the corresponding one given in the previous work \cite{yamada2016game}, and thus it explains why our realizability \`{a} la game semantics validates the $\xi$-rule just like the existing game semantics does. 
\end{remark}

\begin{theorem}[Realizability model of pi-types \`{a} la game semantics]
The CwF $\mathbb{NPG}_{\mathrm{wrw}}^{\mathrm{wo}}$ supports pi in the strict sense.
\end{theorem}
\begin{proof}
Let $\Gamma \in \mathbb{NPG}_{\mathrm{wrw}}^{\mathrm{wo}}$, $A \in \mathscr{D}\mathbb{NPG}_{\mathrm{wrw}}^{\mathrm{wo}}(\Gamma)$, $B \in \mathscr{D}\mathbb{NPG}_{\mathrm{wrw}}^{\mathrm{wo}}(\Sigma_{\mathrm{wrw}}(\Gamma, A))$ and $\uplus \beta \in \mathcal{TS}_{\mathrm{wr}}(\Pi_{\mathrm{wrw}}(\Sigma_{\mathrm{wrw}}(\Gamma, A), B))$.
\begin{itemize}

\item \textsc{($\Pi$-Form)} Let us define $\Pi(A, B) \colonequals (\Pi_{\mathrm{wrw}}(A(\uplus \gamma), B_{\uplus \gamma}))_{\uplus \gamma \in \mathbb{NPG}_{\mathrm{wrw}}^{\mathrm{wo}}(\Gamma)}$, where $B_{\uplus \gamma} \colonequals (B(\langle \uplus \gamma, \uplus \tilde{\alpha} \rangle)_{\uplus \tilde{\alpha} \in \mathbb{NPG}_{\mathrm{wrw}}^{\mathrm{wo}}(A(\uplus \gamma))} \in \mathscr{D}\mathbb{NPG}_{\mathrm{wrw}}^{\mathrm{wo}}(A(\uplus \gamma))$.
To distinguish it from the game semantics of pi-types in \cite{yamada2016game}, we write $\Pi_{\mathrm{wrw}}(A, B)$ for $\Pi(A, B)$.
 
\item \textsc{($\Pi$-Intro)} We get $\lambda_{A, B} (\uplus \beta) \in \mathcal{TS}_{\mathrm{wr}}(\Pi_{\mathrm{wrw}}(\Gamma, \Pi_{\mathrm{wrw}}(A, B)))$ by Lemma~\ref{LemWRWCurryingLemma}, where recall that $\lambda_{A, B}$ and $\lambda_{A, B}^{-1}$ simply `adjust finitary tags' on standard moves. 
We often omit the subscripts $(\_)_{A, B}$ on $\lambda_{A, B}$ and $\lambda_{A, B}^{-1}$.

\item \textsc{($\Pi$-Elim)} Given $\uplus \kappa \in \mathcal{TS}_{\mathrm{wr}}(\Pi_{\mathrm{wrw}} (\Gamma, \Pi_{\mathrm{wrw}}(A, B)))$ and $\uplus \alpha \in \mathcal{TS}_{\mathrm{wr}}(\Pi_{\mathrm{wrw}} (\Gamma, A))$, we define $\mathrm{App}_{A, B} (\uplus \kappa, \uplus \alpha) \colonequals  \lambda_{A, B}^{-1}(\uplus \kappa) \bullet \overline{\uplus \alpha}$. 
As in the proof of Theorem~\ref{ThmWellDefinedWPG}, $\lambda_{A, B}^{-1}(\uplus \kappa) \bullet \overline{\uplus \alpha} :: \Pi_{\mathrm{wrw}} (\Gamma, B\{ \overline{\uplus \alpha} \})$, and so $\mathrm{App}_{A, B} (\uplus \kappa, \uplus \alpha) = \lambda_{A, B}^{-1}(\uplus \kappa) \bullet \overline{\uplus \alpha} \in \mathcal{TS}_{\mathrm{wr}}(\Pi_{\mathrm{wrw}}(\Gamma, B\{ \overline{\uplus \alpha} \}))$. 
We often omit the subscripts $A, B$ on $\mathrm{App}_{A, B}$.

\item \textsc{($\Pi$-Comp)} $\mathrm{App} (\lambda (\uplus \beta) , \uplus \alpha) = \lambda_{A, B}^{-1}(\lambda_{A, B}(\uplus \beta)) \bullet \overline{\uplus \alpha} = \uplus \beta \bullet \overline{\uplus \alpha} = \uplus \beta \{ \overline{\uplus \alpha} \}$.

\item \textsc{($\Pi$-Subst)} Given $\Delta \in \mathbb{NPG}_{\mathrm{wrw}}^{\mathrm{wo}}$ and $\uplus \phi \in \mathbb{NPG}_{\mathrm{wrw}}^{\mathrm{wo}}(\Delta, \Gamma)$, we calculate 
\begin{align*}
\textstyle \Pi_{\mathrm{wrw}} (A, B) \{ \uplus \phi \} &= (\Pi_{\mathrm{wrw}} (A(\uplus \gamma), B_{\uplus \gamma}))_{\uplus \gamma \in \mathbb{NPG}_{\mathrm{wrw}}^{\mathrm{wo}}(\Gamma)} \{ \uplus \phi \} \\
&= (\Pi_{\mathrm{wrw}} (A(\uplus \phi \bullet \uplus \delta), B_{\uplus \phi \bullet \uplus \delta}))_{\uplus \delta \in \mathbb{NPG}_{\mathrm{wrw}}^{\mathrm{wo}}(\Delta)} \\
&= (\Pi_{\mathrm{wrw}} (A\{ \uplus \phi \}(\uplus \delta), B\{ \uplus \phi^+ \}_{\uplus \delta}))_{\uplus \delta \in \mathbb{NPG}_{\mathrm{wrw}}^{\mathrm{wo}}(\Delta)} \\
&= \Pi_{\mathrm{wrw}}(A\{ \uplus \phi \}, B\{ \uplus \phi^+ \}),
\end{align*}
where $B\{\uplus \phi^+\}_{\uplus \delta} = B_{\uplus \phi \bullet \uplus \delta}$ since for all $\uplus \alpha' \in \mathbb{NPG}_{\mathrm{wrw}}^{\mathrm{wo}}(A(\uplus \phi \bullet \uplus \delta))$ we have
\begin{align*}
B\{\uplus \phi^+\}_{\uplus \delta} (\uplus \alpha') &= B\{\uplus \phi^+\} (\langle \uplus \delta, \uplus \alpha' \rangle) \\ 
&= B (\langle \uplus \phi \bullet \mathrm{p}(A \{ \uplus \phi \}), \mathrm{v}_{A\{ \uplus \phi \}} \rangle \bullet \langle \uplus \delta, \uplus \alpha' \rangle) \\
&= B(\langle \uplus \phi \bullet \uplus \delta, \uplus \alpha' \rangle) \\
&= B_{\uplus \phi \bullet \uplus \delta} (\uplus \alpha').
\end{align*}

\item \textsc{($\lambda$-Subst)} We calculate 
\begin{align*}
\lambda (\uplus \beta) \{ \uplus \phi \} &= \lambda_{A, B} (\uplus \beta) \bullet \uplus \phi \\
&= \lambda_{A\{\uplus \phi\}, B\{\uplus \phi^+\}} (\uplus \beta \bullet \langle \uplus \phi \bullet \mathrm{fst}, \mathrm{snd} \rangle) \quad \text{(by the definition of $\lambda$)} \\
&= \lambda_{A\{\uplus \phi\}, B\{\uplus \phi^+\}} (\uplus \beta \bullet \langle \uplus \phi \bullet \mathrm{p}(A \{ \uplus \phi \}), \mathrm{v}_{A\{\uplus \phi\}} \rangle) \\
&= \lambda_{A\{\uplus \phi\}, B\{\uplus \phi^+\}} (\uplus \beta \{ \uplus \phi^+ \}).
\end{align*}

\item \textsc{(App-Subst)} We calculate
\begin{align*}
\mathrm{App} (\uplus \kappa, \uplus \alpha) \{ \uplus \phi \} &= (\lambda_{A, B}^{-1} (\uplus \kappa) \bullet \langle \uplus \mathrm{der}_\Gamma, \uplus \alpha \rangle) \bullet \uplus \phi \\
&= \lambda_{A, B}^{-1} (\uplus \kappa) \bullet (\langle \uplus \mathrm{der}_\Gamma, \uplus \alpha \rangle \bullet \uplus \phi) \\
&= \lambda_{A, B}^{-1} (\uplus \kappa) \bullet \langle \uplus \phi, \uplus \alpha \bullet \uplus \phi \rangle \\
&= \lambda_{A, B}^{-1} (\uplus \kappa) \bullet (\langle \uplus \phi \bullet \mathrm{p}(A \{ \uplus \phi \}), \mathrm{v}_{A\{\uplus \phi\}} \rangle \bullet \langle \uplus \mathrm{der}_\Delta, \uplus \alpha \bullet \uplus \phi \rangle) \\
&= (\lambda_{A, B}^{-1} (\uplus \kappa) \bullet \uplus \phi^+) \bullet \overline{(\uplus \alpha \bullet \uplus \phi)} \\
&= \lambda_{A\{ \uplus \phi \}, B\{ \uplus \phi^+ \}}^{-1} (\uplus \kappa \bullet \uplus \phi) \bullet \overline{(\uplus \alpha \bullet \uplus \phi)} \quad \text{(by $\lambda$-Subst)} \\
&= \lambda_{A\{ \uplus \phi \}, B\{ \uplus \phi^+ \}}^{-1} (\uplus \kappa \bullet \uplus \phi) \{ \overline{(\uplus \alpha \bullet \uplus \phi)} \} \\
&= \mathrm{App}_{A\{ \uplus \phi \}, B\{ \uplus \phi^+ \}} (\uplus \kappa \bullet \uplus \phi, \uplus \alpha \bullet \uplus \phi) \quad \text{(by $\Pi$-Comp)} \\
&= \mathrm{App}_{A\{ \uplus \phi \}, B\{ \uplus \phi^+ \}} (\uplus \kappa \{ \uplus \phi \}, \uplus \alpha \{ \uplus \phi \}).
\end{align*}

\item \textsc{($\lambda$-Uniq)} Finally, we calculate
\begin{align*} 
\lambda(\mathrm{App}(\uplus \kappa\{ \mathrm{p}(A) \}, \mathrm{v}_A)) &= \lambda_{A, B}(\lambda^{-1}_{A\{\mathrm{p}(A)\}, B\{ \mathrm{p}(A)^+ \}}(\uplus \kappa\{ \mathrm{p}(A) \}) \bullet \overline{\mathrm{v}_A}) \\
&= \lambda_{A, B}((\lambda^{-1}_{A, B}(\uplus \kappa) \bullet \mathrm{p}(A)^+) \bullet \overline{\mathrm{v}_A}) \quad \text{(by $\lambda$-Subst)} \\
&= \lambda_{A, B}(\lambda^{-1}_{A, B}(\uplus \kappa) \bullet (\mathrm{p}(A)^+ \bullet \overline{\mathrm{v}_A})) \\
&= \lambda_{A, B}(\lambda^{-1}_{A, B}(\uplus \kappa) \bullet \uplus \mathrm{der}_{\Sigma_{\mathrm{wrw}}(\Gamma, A)}) \\
&= \lambda_{A, B}(\lambda^{-1}_{A, B}(\uplus \kappa)) \\
&= \uplus \kappa,
\end{align*}
\end{itemize}
which completes the proof.
\end{proof}

Next, we consider \emph{sigma-types} (Appendix~\ref{DependentPairTypes}). 
Again, we first recall the semantic type former for sigma-types in an arbitrary CwF:
\begin{definition}[CwFs with sigma-types \cite{hofmann1997syntax}]
A CwF $\mathcal{C}$ \emph{\bfseries supports sigma} if
\begin{itemize}

\item \textsc{($\Sigma$-Form)} Given $\Gamma \in \mathcal{C}$, $A \in \mathrm{Ty}(\Gamma)$ and $B \in \mathrm{Ty}(\Gamma . A)$, there is a type $\Sigma (A, B) \in \mathrm{Ty}(\Gamma)$;

\item \textsc{($\Sigma$-Intro)} There is a morphism $\mathrm{Pair}_{A, B} : \Gamma . A . B \to \Gamma . \Sigma (A, B)$ in $\mathcal{C}$;

\item \textsc{($\Sigma$-Elim)} Given $P \in \mathrm{Ty}(\Gamma . \Sigma (A, B))$ and $p \in \mathrm{Tm}(\Gamma . A . B, P \{ \mathrm{Pair}_{A, B} \})$, there is a term $\mathcal{R}^{\Sigma}_{A, B, P}(p) \in \mathrm{Tm}(\Gamma . \Sigma (A, B), P)$;

\item \textsc{($\Sigma$-Comp)} $\mathcal{R}^{\Sigma}_{A, B, P}(p) \{ \mathrm{Pair}_{A, B}\} = p$;

\item \textsc{($\Sigma$-Subst)} Given $\Delta \in \mathcal{C}$ and $f : \Delta \to \Gamma$ in $\mathcal{C}$, $\Sigma (A, B) \{ f \} = \Sigma (A\{f\}, B\{f^+\})$, where $f^+ \colonequals \langle f \circ \mathrm{p}(A\{ f \}), \mathrm{v}_{A\{f\}} \rangle_A : \Delta . A\{f\} \to \Gamma . A$;

\item \textsc{(Pair-Subst)} $\mathrm{p}(\Sigma (A, B)) \circ \mathrm{Pair}_{A, B} = \mathrm{p}(A) \circ \mathrm{p}(B)$ and $f^\star \circ \mathrm{Pair}_{A\{f\}, B\{f^+\}} = \mathrm{Pair}_{A, B} \circ f^{++}$, where 
\begin{mathpar}
f^\star \colonequals \langle f \circ \mathrm{p}(\Sigma(A,B)\{f\}), \mathrm{v}_{\Sigma(A,B)\{f\}} \rangle_{\Sigma(A,B)} : \Delta . \Sigma(A,B)\{f\} \to \Gamma . \Sigma(A,B); \and
f^{++} \colonequals \langle f^+ \circ \mathrm{p}(B\{f^+\}), \mathrm{v}_{B\{f^+\}} \rangle_B : \Delta . A\{f\} . B\{f^+\} \to \Gamma . A . B;
\end{mathpar}

\item \textsc{($\mathcal{R}^{\Sigma}$-Subst)} $\mathcal{R}^{\Sigma}_{A, B, P}(p) \{f^\star\} = \mathcal{R}^{\Sigma}_{A\{f\}, B\{f^+\}, P\{f^\star\}} (p \{ f^{++} \})$.

\end{itemize}

Moreover, $\mathcal{C}$ \emph{\bfseries supports sigma in the strict sense} if it also satisfies
\begin{itemize}
\item \textsc{($\mathcal{R}^{\Sigma}$-Uniq)} If $p \in \mathrm{Tm}(\Gamma . A . B, P \{ \mathrm{Pair}_{A, B} \})$, $q \in \mathrm{Tm}(\Gamma . \Sigma(A, B), P)$ and $q \{ \mathrm{Pair}_{A, B} \} = p$, then $q = \mathcal{R}^{\Sigma}_{A, B, P}(p)$.
\end{itemize}
\end{definition}

Sigma-types (with the $\eta$-rule) are modeled in a CwF that supports sigma (in the strict sense); see Appendix~\ref{DependentPairTypes} for the details.

Now, let us describe our interpretation of sigma-types:
\begin{theorem}[Realizability model of sigma-types \`{a} la game semantics]
The CwF $\mathbb{NPG}_{\mathrm{wrw}}^{\mathrm{wo}}$ supports sigma in the strict sense.
\end{theorem}
\begin{proof}
Let $\Gamma, \Delta \in \mathbb{NPG}_{\mathrm{wrw}}^{\mathrm{wo}}$, $\uplus \phi \in \mathbb{NPG}_{\mathrm{wrw}}^{\mathrm{wo}}(\Delta, \Gamma)$, $A \in \mathscr{D}\mathbb{NPG}_{\mathrm{wrw}}^{\mathrm{wo}}(\Gamma)$ and $B \in \mathscr{D}\mathbb{NPG}_{\mathrm{wrw}}^{\mathrm{wo}}(\Sigma_{\mathrm{wrw}}(\Gamma, A))$. 

\begin{itemize}

\item \textsc{($\Sigma$-Form)} Similarly to pi, let $\Sigma (A, B) \colonequals (\Sigma_{\mathrm{wrw}} (A(\uplus \gamma), B_{\uplus \gamma}))_{\uplus \gamma \in \mathbb{NPG}_{\mathrm{wrw}}^{\mathrm{wo}}(\Gamma)}$.
Again, to distinguish it from the game semantics of sigma-types given in the previous work \cite{yamada2016game}, we write $\Sigma_{\mathrm{wrw}}(A, B)$ for $\Sigma(A, B)$.

\item \textsc{($\Sigma$-Intro)} By the isomorphism $\Sigma_{\mathrm{wrw}}(\Sigma_{\mathrm{wrw}}(\Gamma,A),B) \cong \Sigma_{\mathrm{wrw}}(\Gamma, \Sigma_{\mathrm{wrw}}(A,B))$, which is similar to (and simpler than) the one in Lemma~\ref{LemWRWCurryingLemma} and left to the reader, let $\mathrm{Pair}_{A, B} : \Sigma_{\mathrm{wrw}}(\Sigma_{\mathrm{wrw}}(\Gamma,A),B) \rightarrow \Sigma_{\mathrm{wrw}}(\Gamma, \Sigma_{\mathrm{wrw}}(A,B))$ be the evident w.r.w. dereliction up to `tags,' or $\mathrm{Pair}_{A, B} \colonequals \langle \mathrm{fst} \bullet \mathrm{fst}, \langle \mathrm{snd} \bullet \mathrm{fst}, \mathrm{snd} \rangle \rangle$.
Note that there is the evident inverse $\mathrm{Pair}^{-1}_{A, B} = \langle \langle \mathrm{fst}, \mathrm{fst} \bullet \mathrm{snd} \rangle, \mathrm{snd} \bullet \mathrm{snd} \rangle$.

\item \textsc{($\Sigma$-Elim)} We define $\mathcal{R}^{\Sigma}_{A, B, P}(\uplus \rho) \in \mathcal{TS}_{\mathrm{wr}}(\Pi_{\mathrm{wrw}}(\Sigma_{\mathrm{wrw}}(\Gamma, \Sigma_{\mathrm{wrw}} (A, B)), P))$ to be the composition $\uplus \rho \bullet \mathrm{Pair}_{A, B}^{-1}$ for any $P \in \mathscr{D}\mathbb{NPG}_{\mathrm{wrw}}^{\mathrm{wo}}(\Sigma_{\mathrm{wrw}}(\Gamma, \Sigma_{\mathrm{wrw}}(A,B)))$ and  $\uplus \rho \in \mathcal{TS}_{\mathrm{wr}}(\Pi_{\mathrm{wrw}}(\Sigma_{\mathrm{wrw}}(\Sigma_{\mathrm{wrw}}(\Gamma, A), B), P\{\mathrm{Pair}_{A,B}\}))$.

\item \textsc{($\Sigma$-Comp)} We calculate
\begin{align*}
\mathcal{R}^{\Sigma}_{A, B, P}(\uplus \rho) \{ \mathrm{Pair}_{A, B}\} &= \mathcal{R}^{\Sigma}_{A, B, P}(\uplus \rho) \bullet \mathrm{Pair}_{A, B} \\
&= (\uplus \rho \bullet \mathrm{Pair}_{A, B}^{-1}) \bullet \mathrm{Pair}_{A, B} \\
&= \uplus \rho \bullet (\mathrm{Pair}_{A, B}^{-1} \bullet \mathrm{Pair}_{A, B}) \\
&= \uplus \rho \bullet \uplus \mathrm{der}_{\Sigma_{\mathrm{wrw}}(\Sigma_{\mathrm{wrw}}(\Gamma, A), B)} \\
&= \uplus \rho.
\end{align*}

\item \textsc{($\Sigma$-Subst)} Similar to the case of pi.

\item \textsc{(Pair-Subst)} We calculate
\begin{align*}
\mathrm{p}(\Sigma_{\mathrm{wrw}} (A, B)) \bullet \mathrm{Pair}_{A, B} &= \mathrm{fst} \bullet \langle \mathrm{fst} \bullet \mathrm{fst}, \langle \mathrm{snd} \bullet \mathrm{fst}, \mathrm{snd} \rangle \rangle \\ 
&= \mathrm{fst} \bullet \mathrm{fst} \\
&= \mathrm{p}(A) \bullet \mathrm{p}(B),
\end{align*}
and
\begin{align*}
& {\uplus \phi}^\star \bullet \mathrm{Pair}_{A\{\uplus \phi\}, B\{\uplus \phi^+\}} \\
=& \ \langle \uplus \phi \bullet \mathrm{p}(\Sigma_{\mathrm{wrw}}(A,B)\{\uplus \phi\}), \mathrm{v}_{\Sigma_{\mathrm{wrw}}(A,B)\{\uplus \phi\}} \rangle \bullet \mathrm{Pair}_{A\{\uplus \phi\}, B\{\uplus \phi^+\}} \\
=& \ \langle \uplus \phi \bullet \mathrm{p}(\Sigma_{\mathrm{wrw}}(A\{\uplus \phi\},B\{\uplus \phi^+\})) \bullet \mathrm{Pair}_{A\{\uplus \phi\}, B\{\uplus \phi^+\}}, \mathrm{v}_{\Sigma_{\mathrm{wrw}}(A,B)\{\uplus \phi\}} \bullet \mathrm{Pair}_{A\{ \uplus \phi \}, B\{ \uplus \phi^+ \}} \rangle \\
=& \ \langle \uplus \phi \bullet \mathrm{p}(A\{ \uplus \phi \}) \bullet \mathrm{p}(B\{ \uplus \phi^+ \}), \mathrm{v}_{\Sigma_{\mathrm{wrw}}(A\{ \uplus \phi \},B\{ \uplus \phi^+ \})} \bullet \mathrm{Pair}_{A\{ \uplus \phi \}, B\{ \uplus \phi^+ \}} \rangle \\ 
&\text{(by the above equations)} \\
=& \ \langle \uplus \phi \bullet \mathrm{fst} \bullet \mathrm{fst}, \mathrm{snd} \bullet \langle \mathrm{fst} \bullet \mathrm{fst}, \langle \mathrm{snd} \bullet \mathrm{fst}, \mathrm{snd} \rangle \rangle \rangle \\
=& \ \langle \uplus \phi \bullet \mathrm{fst} \bullet \mathrm{fst}, \langle \mathrm{snd} \bullet \mathrm{fst}, \mathrm{snd} \rangle \rangle \\
=& \ \langle \mathrm{fst} \bullet \mathrm{fst}, \langle \mathrm{snd} \bullet \mathrm{fst}, \mathrm{snd} \rangle \rangle \bullet \langle \langle \uplus \phi \bullet \mathrm{fst} \bullet \mathrm{fst}, \mathrm{snd} \bullet \mathrm{fst} \rangle, \mathrm{snd} \rangle \\
=& \ \langle \mathrm{fst} \bullet \mathrm{fst}, \langle \mathrm{snd} \bullet \mathrm{fst}, \mathrm{snd} \rangle \rangle \bullet \langle \langle \uplus \phi \bullet \mathrm{fst}, \mathrm{snd} \rangle \bullet \mathrm{fst}, \mathrm{snd} \rangle \\
=& \ \mathrm{Pair}_{A, B} \bullet \langle \langle \uplus \phi \bullet \mathrm{p}(A\{ \uplus \phi \}), \mathrm{v}_{A\{ \uplus \phi \}} \rangle \bullet \mathrm{p}(B\{ \uplus \phi^+ \}), \mathrm{v}_{B\{ \uplus \phi^+ \}} \rangle \\
=& \ \mathrm{Pair}_{A, B} \bullet \langle \uplus \phi^+ \bullet \mathrm{p}(B\{ \uplus \phi^+ \}), \mathrm{v}_{B\{ \uplus \phi^+ \}} \rangle \\
=& \ \mathrm{Pair}_{A, B} \bullet \uplus \phi^{++}.
\end{align*}

\item \textsc{($\mathcal{R}^{\Sigma}$-Subst)} We calculate
\begin{align*}
\mathcal{R}^{\Sigma}_{A, B, P}(\uplus \rho) \{ \uplus \phi^\star \} &= \uplus \rho \bullet \mathrm{Pair}_{A, B}^{-1} \bullet \langle \uplus \phi \bullet \mathrm{p}(\Sigma_{\mathrm{wrw}}(A,B)\{ \uplus \phi \}, \mathrm{v}_{\Sigma_{\mathrm{wrw}}(A,B)\{ \uplus \phi \}} \rangle \\
&= \uplus \rho \bullet \langle \langle \mathrm{fst}, \mathrm{fst} \bullet \mathrm{snd} \rangle, \mathrm{snd} \bullet \mathrm{snd} \rangle \bullet \langle \uplus \phi \bullet \mathrm{fst}, \mathrm{snd} \rangle \\
&= \uplus \rho \bullet \langle \langle \uplus \phi \bullet \mathrm{fst}, \mathrm{fst} \bullet \mathrm{snd} \rangle, \mathrm{snd} \bullet \mathrm{snd} \rangle \\
&= \uplus \rho \bullet \langle \langle \uplus \phi \bullet \mathrm{fst}, \mathrm{snd} \rangle \bullet \mathrm{fst}, \mathrm{snd} \rangle \bullet \langle \langle \mathrm{fst}, \mathrm{fst} \bullet \mathrm{snd} \rangle, \mathrm{snd} \bullet \mathrm{snd} \rangle \\
&= \uplus \rho \bullet \langle \uplus \phi^+ \bullet \mathrm{fst}, \mathrm{snd} \rangle \bullet \langle \langle \mathrm{fst}, \mathrm{fst} \bullet \mathrm{snd} \rangle, \mathrm{snd} \bullet \mathrm{snd} \rangle \\
&= \uplus \rho \bullet \langle \uplus \phi^+ \bullet \mathrm{p}(B\{ \uplus \phi^+ \}), \mathrm{v}_{B\{ \uplus \phi^+ \}} \rangle \bullet \mathrm{Pair}_{A\{ \uplus \phi \}, B\{ \uplus \phi^+ \}}^{-1} \\
&= \mathcal{R}^{\Sigma}_{A\{ \uplus \phi \}, B\{ \uplus \phi^+ \}, P\{ \uplus \phi^\star \}} (\uplus \rho \bullet \langle \uplus \phi^+ \bullet \mathrm{p}(B\{ \uplus \phi^+ \}), \mathrm{v}_{B\{ \uplus \phi^+ \}} \rangle ) \\
&= \mathcal{R}^{\Sigma}_{A\{ \uplus \phi \}, B\{ \uplus \phi^+ \}, P\{ \uplus \phi^\star \}} (\uplus \rho \{ \uplus \phi^{++} \}).
\end{align*}

\item \textsc{($\mathcal{R}^{\Sigma}$-Uniq)} If a given t-skeleton $\uplus \varrho \in \mathcal{TS}_{\mathrm{wr}}(\Pi_{\mathrm{wrw}}(\Sigma_{\mathrm{wrw}}(\Gamma, \Sigma_{\mathrm{wrw}}(A, B)), P))$ satisfies the equation $\uplus \varrho \{ \mathrm{Pair}_{A, B} \} = \uplus \rho$ then we have
\begin{equation*}
\uplus \varrho = \uplus \rho \bullet \mathrm{Pair}_{A, B}^{-1} = \mathcal{R}^{\Sigma}_{A, B, P}(\uplus \rho),
\end{equation*}
\end{itemize}
which completes the proof.
\end{proof}

We proceed to model N-type (Appendix~\ref{NaturalNumberType}).
Again, let us first recall the semantic type former for N-type in an arbitrary CwF:
\begin{definition}[CwFs with N-type \cite{hofmann1997syntax}\footnote{The definition is left to the reader in \cite{hofmann1997syntax}, and thus Definition~\ref{DefSemanticTypeFormersForNTypes} is the present author's solution, which may be shown to be \emph{sound} in the same manner as the case of pi- and sigma-types \cite{hofmann1997syntax}. This point is applied to empty-type given below as well.}]
\label{DefSemanticTypeFormersForNTypes}
A CwF $\mathcal{C}$ \emph{\bfseries supports N} or \emph{\bfseries natural numbers} if
\begin{itemize}

\item \textsc{($N$-Form)} Given $\Gamma \in \mathcal{C}$, there is a type $N^{\Gamma} \in \mathrm{Ty}(\Gamma)$, called \emph{\bfseries natural number (N-) type} in $\Gamma$, which we often abbreviate as $N$;

\item \textsc{($N$-Intro)} There are a term and a morphism in $\mathcal{C}$
\begin{mathpar}
\underline{0}_{\Gamma} \in \mathrm{Tm}(\Gamma, N) \and
\mathrm{succ}_{\Gamma} : \Gamma . N \to \Gamma . N
\end{mathpar}
that satisfy the equations
\begin{mathpar}
\underline{0}_{\Gamma} \{ f \} = \underline{0}_\Delta \in \mathrm{Tm}(\Delta, N) 
\and
\mathrm{p}(N) \circ \mathrm{succ}_{\Gamma} = \mathrm{p}(N) : \Gamma . N \to \Gamma 
\and
\mathrm{succ}_{\Gamma} \circ \langle g, \mathrm{v}_{N} \rangle_{N} = \langle g, \mathrm{v}_{N}\{\mathrm{succ}_{\Delta}\} \rangle_{N} : \Delta . N \to \Gamma . N 
\end{mathpar}
for any morphisms $f : \Delta \to \Gamma$ and $g : \Delta . N \to \Gamma$ in $\mathcal{C}$;

\begin{notation}
Let us define $\mathrm{zero}_\Gamma \colonequals \langle \mathrm{id}_\Gamma, \underline{0}_\Gamma \rangle_N : \Gamma \to \Gamma . N$ for each $\Gamma \in \mathcal{C}$; it then satisfies $\mathrm{zero}_\Gamma \circ f = \langle f, \underline{0}_\Delta \rangle_N = \langle f, \mathrm{v}_N \{ \mathrm{zero}_\Delta \} \rangle_N : \Delta \to \Gamma . N$ for any morphism $f : \Delta \to \Gamma$ in $\mathcal{C}$.
We often omit the subscript $(\_)_\Gamma$ on $\underline{0}_\Gamma$, $\mathrm{zero}_{\Gamma}$ and $\mathrm{succ}_{\Gamma}$.
Also, for each $n \in \mathbb{N}$, we define $\underline{n}_\Gamma \in \mathrm{Tm}(\Gamma, N)$ by
\begin{itemize}

\item $\underline{0}_\Gamma$ is already given;

\item $\underline{n+1}_\Gamma \colonequals \mathrm{v}_N \{ \mathrm{succ}_\Gamma \circ \langle \mathrm{id}_\Gamma, \underline{n}_\Gamma \rangle \}$;

\end{itemize}
\end{notation}

\item \textsc{($N$-Elim)} Given a type $P \in \mathrm{Ty}(\Gamma . N)$, and terms $z \in \mathrm{Tm}(\Gamma, P\{\mathrm{zero}\})$ and $s \in \mathrm{Tm}(\Gamma . N . P, P\{ \mathrm{succ} \circ \mathrm{p}(P)\})$, there is a term $\mathcal{R}^{N}_{P}(z, s) \in \mathrm{Tm}(\Gamma . N, P)$;

\item \textsc{($N$-Comp)} We have the equations
\begin{mathpar}
\mathcal{R}^{N}_{P}(z, s) \{ \mathrm{zero} \} = z \in \mathrm{Tm}(\Gamma, P\{\mathrm{zero}\}); \and
\mathcal{R}^{N}_{P}(z, s) \{ \mathrm{succ} \} = s \{ \langle \mathrm{id}_{\Gamma . N}, \mathcal{R}^{N}_{P}(z, s) \rangle_{P} \} \in \mathrm{Tm}(\Gamma . N, P\{\mathrm{succ}\});
\end{mathpar}

\item \textsc{($N$-Subst)} $N^{\Gamma}\{f\} = N^{\Delta} \in \mathrm{Ty}(\Delta)$;

\item \textsc{($\mathcal{R}^{N}$-Subst)} $\mathcal{R}^{N}_{P}(z, s) \{f^+\} = \mathcal{R}^{N}_{P\{f^+\}}(z\{f\}, s\{f^{++}\}) \in \mathrm{Tm}(\Delta . N, P\{f^+\})$, where
\begin{mathpar}
f^+ \colonequals \langle f \circ \mathrm{p}(N), \mathrm{v}_{N} \rangle_{N} : \Delta . N \to \Gamma . N \and
f^{++} \colonequals \langle f^+ \circ \mathrm{p}(P\{f^+\}), \mathrm{v}_{P\{f^+\}}\rangle_P : \Delta . N . P\{f^+\} \to \Gamma . N . P.
\end{mathpar}

\end{itemize}
\end{definition}
N-type is modeled in a CwF that supports natural numbers; see Appendix~\ref{NaturalNumberType} for the details.

We now propose our game semantics of N-type by basically employing the \emph{total} fragment of the standard game semantics of PCF \`{a} la McCusker \cite{abramsky1999game}:
\begin{theorem}[Realizability model of N-type \`{a} la game semantics]
The CwF $\mathbb{NPG}_{\mathrm{wrw}}^{\mathrm{wo}}$ supports natural numbers.
\end{theorem}
\begin{proof}
Let $\Gamma, \Delta \in \mathbb{NPG}_{\mathrm{wrw}}^{\mathrm{wo}}$ and $\uplus \phi \in \mathbb{NPG}_{\mathrm{wrw}}^{\mathrm{wo}}(\Delta, \Gamma)$, and assume that $\mathscr{R}_{\mathrm{wr}}^{\mathrm{cp}}(\Gamma) \neq \emptyset$ and $\mathscr{R}_{\mathrm{wr}}^{\mathrm{cp}}(\Delta) \neq \emptyset$ since the other cases are trivial.

\begin{itemize}
\item \textsc{($N$-Form)} Let $N^\Gamma$ be the constant dependent np-game on $\Gamma$ valued at $N$ (Example~\ref{ExamplesOfWellOpenedNonstandardPredicativeGames}), for which we write $\{ N \}_\Gamma$ or $\{ N \}$.

\item \textsc{($N$-Intro)} Let $\underline{0}_\Gamma \in \mathcal{TS}_{\mathrm{wr}}(\Pi_{\mathrm{wrw}}(\Gamma, \{ N \}))$ be $\uplus \{ \underline{0} \} :: \hat{\oc}\Gamma \rightarrowtriangle N$ up to `tags,' and $\mathrm{succ}_\Gamma : \Sigma_{\mathrm{wrw}}(\Gamma, \{ N^{[0]} \}) \rightarrow \Sigma_{\mathrm{wrw}}(\Gamma, \{ N^{[1]} \})$ the pairing $\langle \mathrm{p}(\{ N \}), \uplus \mathrm{s}^\Gamma \rangle$, where the superscripts $(\_)^{[i]}$ ($i = 0, 1$) are to distinguish the two copies of $N$, and the DoWRWLI $\uplus \mathrm{s}^\Gamma :: \Pi_{\mathrm{wrw}}(\Sigma_{\mathrm{wrw}}(\Gamma, \{ N^{[0]} \}), \{ N^{[1]} \})$ is defined by
\begin{equation*}
\mathrm{s}^\Gamma_{(\sigma^\dagger, e)} \colonequals \mathrm{Pref} (\{ q^{[1]} . q^{[0]} . n^{[0]} . (n+1)^{[1]} \}) \quad (\underline{n} \colonequals \mathrm{snd} \bullet \sigma)
\end{equation*}
up to `tags' for each $(\sigma^\dagger, e) \in \mathscr{R}_{\mathrm{wr}}^{\mathrm{cp}}(\hat{\oc} \Sigma_{\mathrm{wrw}}(\Gamma, \{ N^{[0]} \}))$.
Clearly, $\underline{0}_\Gamma \bullet \uplus \phi = \underline{0}_\Delta$ and $\uplus \mathrm{s}^\Gamma \bullet \langle \uplus \psi, \mathrm{v}_{\{N\}_\Delta} \rangle = \uplus \mathrm{s}^\Delta = \mathrm{v}_{\{N\}_\Delta}\{ \mathrm{succ}_\Delta \}$, where $\uplus \psi : \Sigma_{\mathrm{wrw}}(\Delta, \{ N \}_\Delta) \rightarrow \Gamma$ is any morphism in $\mathbb{NPG}_{\mathrm{wrw}}^{\mathrm{wo}}$, and therefore the required equations hold.

\item \textsc{($N$-Elim)} Given $P \in \mathscr{D}\mathbb{NPG}_{\mathrm{wrw}}^{\mathrm{wo}}(\Sigma_{\mathrm{wrw}}(\Gamma, \{ N \}))$, $\uplus \zeta \in \mathcal{TS}_{\mathrm{wr}}(\Pi_{\mathrm{wrw}}(\Gamma, P\{ \mathrm{zero} \}))$ and $\uplus \sigma \in \mathcal{TS}_{\mathrm{wr}}(\Pi_{\mathrm{wrw}}(\Sigma_{\mathrm{wrw}}(\Sigma_{\mathrm{wrw}}(\Gamma, \{ N \}), P), P\{ \mathrm{succ} \circ \mathrm{p}(P)\}))$, there are 
\begin{footnotesize}
\begin{mathpar}
\uplus \tilde{\zeta} \in \mathcal{TS}_{\mathrm{wr}}(\Pi_{\mathrm{wrw}}(\Sigma_{\mathrm{wrw}}(\Pi_{\mathrm{wrw}}(\Sigma_{\mathrm{wrw}}(\Gamma, \{N\}), P), \{ \Sigma_{\mathrm{wrw}}(\Gamma, \{ N \}) \}), P\{ \mathrm{zero} \bullet \mathrm{fst} \bullet \mathrm{snd} \})) \\
\uplus \tilde{\sigma} \in \mathcal{TS}_{\mathrm{wr}}(\Pi_{\mathrm{wrw}}(\Sigma_{\mathrm{wrw}}(\Pi_{\mathrm{wrw}}(\Sigma_{\mathrm{wrw}}(\Gamma, \{N\}), P), \{ \Sigma_{\mathrm{wrw}}(\Gamma, \{ N \}) \}), P\{ \mathrm{succ} \bullet \mathrm{pred} \bullet \mathrm{snd} \}))
\end{mathpar}
\end{footnotesize}defined respectively by
\begin{footnotesize}
\begin{mathpar}
\uplus \tilde{\zeta} : \Pi_{\mathrm{wrw}}(\Sigma_{\mathrm{wrw}}(\Gamma, \{ N \}), P) \mathbin{\&} \Sigma_{\mathrm{wrw}}(\Gamma, \{ N \}) \stackrel{\mathrm{snd}}{\longrightarrow} \Sigma_{\mathrm{wrw}}(\Gamma, \{ N \}) \stackrel{\mathrm{fst}}{\longrightarrow} \Gamma \stackrel{\uplus \zeta}{\longrightarrow} \textstyle \int P\{ \mathrm{zero} \}; \\
\uplus \tilde{\sigma} : \Pi_{\mathrm{wrw}}(\Sigma_{\mathrm{wrw}}(\Gamma, \{N\}), P) \mathbin{\&} \Sigma_{\mathrm{wrw}}(\Gamma, \{N\}) \stackrel{\langle \mathrm{pred} \bullet \mathrm{snd}, \mathrm{ev}_P \{ \langle \mathrm{fst}, \mathrm{pred} \bullet \mathrm{snd} \rangle \} \rangle}{\longrightarrow} \Sigma_{\mathrm{wrw}}(\Gamma, \{N\}) \mathbin{\&} \textstyle \int P \stackrel{\uplus \sigma}{\longrightarrow} \textstyle \int P\{ \mathrm{succ} \circ \mathrm{fst} \},
\end{mathpar}
\end{footnotesize}where $\mathrm{ev}_P \in \mathcal{TS}_{\mathrm{wr}}(\Pi_{\mathrm{wrw}}(\Sigma_{\mathrm{wrw}}(\Pi_{\mathrm{wrw}}(\Sigma_{\mathrm{wrw}}(\Gamma, \{N\}), P), \{ \Sigma_{\mathrm{wrw}}(\Gamma, \{N\}) \}), P\{\mathrm{snd}\}))$ is the \emph{evaluation} on $P$ \cite{abramsky1999game,mccusker1998games} given by $\mathrm{ev}_P \colonequals \lambda^{-1}(\uplus \mathrm{der}_{\Pi_{\mathrm{wrw}}(\Sigma_{\mathrm{wrw}}(\Gamma, \{ N \}), P)})$, and $\mathrm{pred} : \Sigma_{\mathrm{wrw}}(\Gamma, \{N\}) \to \Sigma_{\mathrm{wrw}}(\Gamma, \{N\})$ is the \emph{predecessor} \cite{abramsky1999game} given similarly to $\mathrm{succ}$ such that $\mathrm{pred} \bullet \mathrm{succ} = \uplus \mathrm{der}_{\Sigma_{\mathrm{wrw}}(\Gamma, \{ N \})}$ and $\mathrm{pred} \bullet \mathrm{zero} = \mathrm{zero}$.
 
In addition, let us define
\begin{mathpar}
P_{\mathrm{z}} \colonequals P\{ \mathrm{zero} \bullet \mathrm{p}(N) \}) \in \mathscr{D}\mathbb{NPG}_{\mathrm{wrw}}^{\mathrm{wo}}(\Sigma_{\mathrm{wrw}}(\Gamma, \{N\})); \\
P_{\mathrm{s}} \colonequals P\{ \mathrm{succ} \bullet \mathrm{pred} \bullet \mathrm{p}(P_{\mathrm{z}}) \} \in \mathscr{D}\mathbb{NPG}_{\mathrm{wrw}}^{\mathrm{wo}}(\Sigma_{\mathrm{wrw}}(\Sigma_{\mathrm{wrw}}(\Gamma, \{N\}), P_{\mathrm{z}})),
\end{mathpar}
and then we have
\begin{equation*}
\mathrm{cond}_P \in \mathcal{TS}_{\mathrm{wr}}(\Pi_{\mathrm{wrw}}(\Sigma_{\mathrm{wrw}}(\Sigma_{\mathrm{wrw}}(\Sigma_{\mathrm{wrw}}(\Gamma, \{ N \}), P_{\mathrm{z}}), P_{\mathrm{s}}), P\{ \mathrm{p}(P_{\mathrm{z}}) \bullet \mathrm{p}(P_{\mathrm{s}}) \})) 
\end{equation*}
that is the standard interpretation of \emph{conditionals} on $P$ \cite{abramsky1999game,mccusker1998games}: It first asks an input natural number in the component $N$ of the domain, and plays as the w.r.w. dereliction between $P_{\mathrm{z}}$ and $P\{ \mathrm{p}(P_{\mathrm{z}}) \bullet \mathrm{p}(P_{\mathrm{s}})\}$ if the answer is $\underline{0}$, and as the w.r.w. dereliction between $P_{\mathrm{s}}$ and $P\{ \mathrm{p}(P_{\mathrm{z}}) \bullet \mathrm{p}(P_{\mathrm{s}})\}$ otherwise. 
Then, we define $\mathcal{F}^{N}_{P}(\uplus \zeta, \uplus \sigma) : \Pi_{\mathrm{wrw}}(\Sigma_{\mathrm{wrw}}(\Gamma, \{N\}), P) \to \Pi_{\mathrm{wrw}}(\Sigma_{\mathrm{wrw}}(\Gamma, \{N\}), P)$ by
\begin{equation*}
\mathcal{F}^{N}_{P}(\uplus \zeta, \uplus \sigma) \colonequals \lambda_{\{ \Sigma_{\mathrm{wrw}}(\Gamma, \{ N \}) \}, \{ P\{ \mathrm{snd} \} \}} (\mathrm{cond}_P \{ \langle \langle \mathrm{snd}, \uplus \tilde{\zeta} \rangle, \uplus \tilde{\sigma} \rangle \}).
\end{equation*}

Finally, we define $\mathcal{R}^{N}_{P}(\uplus \zeta, \uplus \sigma) \in \mathcal{TS}_{\mathrm{wr}}(\Pi_{\mathrm{wrw}}(\Sigma_{\mathrm{wrw}}(\Gamma, \{N\}), P))$ to be the least upper bound of the chain $\big(\mathcal{R}^{N}_{P}(\uplus \zeta, \uplus \sigma)_n \in \mathcal{TS}_{\mathrm{wr}}(\Pi_{\mathrm{wrw}}(\Sigma_{\mathrm{wrw}}(\Gamma, \{N\}), P))\big)_{n \in \mathbb{N}}$:
\begin{mathpar}
\mathcal{R}^{N}_{P}(\uplus \zeta, \uplus \sigma)_0 \colonequals \top \quad \text{(up to `tags')}; \\
\mathcal{R}^{N}_{P}(\uplus \zeta, \uplus \sigma)_{n+1} \colonequals \mathcal{F}^{N}_{P}(\uplus \zeta, \uplus \sigma) \bullet \mathcal{R}^{N}_{P}(\uplus \zeta, \uplus \sigma)_n.
\end{mathpar}

\item \textsc{($N$-Comp)} By the definition of $\mathcal{R}^{N}_{P}(\uplus \zeta, \uplus \sigma)$, we calculate
\begin{mathpar}
\mathcal{R}^{N}_{P}(\uplus \zeta, \uplus \sigma) \{ \mathrm{zero} \} = \uplus \zeta; \\
\mathcal{R}^{N}_{P}(\uplus \zeta, \uplus \sigma) \{ \mathrm{succ} \} = \uplus \sigma \{ \langle \uplus \mathrm{der}_{\Sigma(\Gamma, \{N\})}, \mathcal{R}^{N}_{P}(\uplus \zeta, \uplus \sigma) \rangle \}.
\end{mathpar}

\item \textsc{($N$-Subst)} $\{N\}_{\Gamma}\{ \uplus \phi \} = \{N\}_{\Delta}$.

\item \textsc{($\mathcal{R}^{N}$-Subst)}
Finally, by the definition of $\mathcal{R}^{N}_{P}(\uplus \zeta, \uplus \sigma)$, we calculate
\begin{equation*}
\mathcal{R}^{N}_{P}(\uplus \zeta, \uplus \sigma) \{ \uplus \phi^+ \} = \mathcal{R}^{N}_{P\{ \uplus \phi^+ \}}(\uplus \zeta\{ \uplus \phi \}, \uplus \sigma\{ \uplus \phi^{++} \})
\end{equation*}
\end{itemize}
(or $\mathcal{R}^{N}_{P}(\uplus \zeta, \uplus \sigma)_n \{ \uplus \phi^+ \} = \mathcal{R}^{N}_{P\{ \uplus \phi^+ \}}(\uplus \zeta\{ \uplus \phi \}, \uplus \sigma\{ \uplus \phi^{++} \})_n$ for all $n \in \mathbb{N}$).
\end{proof}

We proceed to model \emph{identity (Id-) types} (Appendix~\ref{IdentityTypes}). 
Again, we first review the semantic type former for Id-types in an arbitrary CwF:
\begin{definition}[CwFs with Id-types \cite{hofmann1997syntax}]
A CwF $\mathcal{C}$ \emph{\bfseries supports Id} if
\begin{itemize}

\item \textsc{(Id-Form)} Given $\Gamma \in \mathcal{C}$ and $A \in \mathrm{Ty}(\Gamma)$, there is a type $\mathrm{Id}_A \in \mathrm{Ty}(\Gamma . A . A^+)$, where $A^+ \colonequals A\{\mathrm{p}(A)\} \in \mathrm{Ty}(\Gamma . A)$;

\item \textsc{(Id-Intro)} There is a morphism $\mathrm{Refl}_A : \Gamma . A \to \Gamma . A . A^+ . \mathrm{Id}_A$ in $\mathcal{C}$;

\item \textsc{(Id-Elim)} Given $B \in \mathrm{Ty}(\Gamma . A . A^+ . \mathrm{Id}_A)$ and $b \in \mathrm{Tm}(\Gamma . A, B\{\mathrm{Refl}_A\})$, there is a term $\mathcal{R}^{\mathrm{Id}}_{A,B}(b) \in \mathrm{Tm}(\Gamma . A . A^+ . \mathrm{Id}_A, B)$;

\item \textsc{(Id-Comp)} $\mathcal{R}^{\mathrm{Id}}_{A,B}(b)\{\mathrm{Refl}_A\} = b$;

\item \textsc{(Id-Subst)} $\mathrm{Id}_A\{f^{++}\} = \mathrm{Id}_{A\{f\}} \in \mathrm{Ty}(\Delta . A\{f\} . A\{f\}^+)$ for all $\Delta \in \mathcal{C}$ and $f : \Delta \to \Gamma$ in $\mathcal{C}$, where 
\begin{mathpar}
A\{f\}^+ \colonequals A\{f\}\{\mathrm{p}(A\{f\})\} \in \mathrm{Ty}(\Delta . A\{f\});
\and
f^+ \colonequals \langle f \circ \mathrm{p}(A\{ f \}), \mathrm{v}_{A\{f\}} \rangle_A : \Delta . A\{f\} \to \Gamma . A;
\and
f^{++} \colonequals \langle f^+ \circ \mathrm{p}(A^+\{f^+\}), \mathrm{v}_{A^+\{f^+\}} \rangle_{A^+} : \Delta . A\{f\} . A^+\{f^+\} \to \Gamma . A . A^+;
\end{mathpar}

\item \textsc{(Refl-Subst)} $\mathrm{Refl}_A \circ f^+ = f^{+++} \circ \mathrm{Refl}_{A\{f\}} : \Delta . A\{f\} \to \Gamma . A . A^+ . \mathrm{Id}_A$, where 
\begin{footnotesize}
\begin{mathpar}
f^{+++} \colonequals \langle f^{++} \circ \mathrm{p}(\mathrm{Id}_A\{f^{++}\}), \mathrm{v}_{\mathrm{Id}_A\{f^{++}\}} \rangle_{\mathrm{Id}_A} : \Delta . A\{f\} . A^+\{f^+\} . \mathrm{Id}_{A\{f\}} \to \Gamma . A . A^+ . \mathrm{Id}_A;
\end{mathpar}
\end{footnotesize}

\item \textsc{($\mathcal{R}^{\mathrm{Id}}$-Subst)} $\mathcal{R}^{\mathrm{Id}}_{A,B}(b)\{f^{+++}\} = \mathcal{R}^{\mathrm{Id}}_{A\{f\}, B\{f^{+++}\}}(b\{f^+\})$.

\end{itemize}
\end{definition}
Id-types are modeled in a CwF that supports Id; see Appendix~\ref{IdentityTypes} for the details.

We then equip the CwF $\mathbb{NPG}_{\mathrm{wrw}}^{\mathrm{wo}}$ with our game-semantic Id-types:
\begin{theorem}[Realizability model of Id-types \`{a} la game semantics]
\label{LemGameSemanticIdTypes}
The CwF $\mathbb{NPG}_{\mathrm{wrw}}^{\mathrm{wo}}$ supports Id.
\end{theorem}
\begin{proof}
Let $\Gamma, \Delta \in \mathbb{NPG}_{\mathrm{wrw}}^{\mathrm{wo}}$, $\uplus \phi \in \mathbb{NPG}_{\mathrm{wrw}}^{\mathrm{wo}}(\Delta, \Gamma)$ and $A \in \mathscr{D}\mathbb{NPG}_{\mathrm{wrw}}^{\mathrm{wo}}(\Gamma)$. 

\begin{itemize}

\item \textsc{(Id-Form)} We define $\mathrm{Id}_A \in \mathscr{D}\mathbb{NPG}_{\mathrm{wrw}}^{\mathrm{wo}}(\Sigma_{\mathrm{wrw}}(\Sigma_{\mathrm{wrw}}(\Gamma, A), A^+))$ by $\mathrm{Id}_A \colonequals (\mathrm{Id}_{A(\uplus \gamma)}(\uplus \alpha, \uplus \alpha'))_{\langle \langle \uplus \gamma, \uplus \alpha \rangle, \uplus \alpha' \rangle \in \mathbb{NPG}_{\mathrm{wrw}}^{\mathrm{wo}}(\Sigma_{\mathrm{wrw}}(\Sigma_{\mathrm{wrw}}(\Gamma, A), A^+))}$.

\item \textsc{(Id-Intro)} Let $\mathrm{Refl}_A : \Sigma_{\mathrm{wrw}}(\Gamma, A^{[1]}) \rightarrow \Sigma_{\mathrm{wrw}}(\Sigma_{\mathrm{wrw}}(\Sigma_{\mathrm{wrw}}(\Gamma, A^{[2]}), (A^{[3]})^+), \mathrm{Id}_A)$ be what plays by the w.r.w. dereliction between $\Sigma_{\mathrm{wrw}}(\Gamma, A^{[1]})$ and $\Sigma_{\mathrm{wrw}}(\Gamma, A^{[2]})$, between $\int_\Gamma A^{[1]}$ and $\int_\Gamma (A^{[3]})^+$, and trivially on $\Sigma_{\mathrm{wrw}}(\Gamma, A^{[1]}) \rightarrow \boldsymbol{1}$, where the superscripts $(\_)^{[i]}$ ($i = 1, 2, 3$) are to distinguish the three copies of $A$. 
The inverse $\mathrm{Refl}_A^{-1} : \Sigma_{\mathrm{wrw}}(\Sigma_{\mathrm{wrw}}(\Sigma_{\mathrm{wrw}}(\Gamma, A^{[2]}), (A^{[3]})^+), \mathrm{Id}_A) \to \Sigma_{\mathrm{wrw}}(\Gamma, A^{[1]})$ is the w.r.w. dereliction between $\Sigma_{\mathrm{wrw}}(\Gamma, A^{[2]})$ and $\Sigma_{\mathrm{wrw}}(\Gamma, A^{[1]})$ up to `tags.' 

\item \textsc{(Id-Elim)} Given $B \in \mathscr{D}\mathbb{NPG}_{\mathrm{wrw}}^{\mathrm{wo}}(\Sigma_{\mathrm{wrw}}(\Sigma_{\mathrm{wrw}}(\Sigma_{\mathrm{wrw}}(\Gamma, A^{[2]}), (A^{[3]})^+), \mathrm{Id}_A))$ and $\uplus \beta \in \mathcal{TS}_{\mathrm{wr}}(\Pi_{\mathrm{wrw}}(\Sigma_{\mathrm{wrw}}(\Gamma, A), B\{\mathrm{Refl}_A\}))$, we define
\begin{equation*}
\mathcal{R}^{\mathrm{Id}}_{A,B}(\uplus \beta) \colonequals \uplus \beta \bullet \mathrm{Refl}_A^{-1} \in \mathcal{TS}_{\mathrm{wr}}(\Pi_{\mathrm{wrw}}((\Sigma_{\mathrm{wrw}}(\Sigma_{\mathrm{wrw}}(\Sigma_{\mathrm{wrw}}(\Gamma, A), A^+), \mathrm{Id}_A), B)).
\end{equation*}

\item \textsc{(Id-Comp)} We then calculate 
\begin{align*}
\mathcal{R}^{\mathrm{Id}}_{A,B}(\uplus \beta)\{\mathrm{Refl}_A\} &= \mathcal{R}^{\mathrm{Id}}_{A,B}(\uplus \beta) \bullet \mathrm{Refl}_A \\
&= (\uplus \beta \bullet \mathrm{Refl}_A^{-1}) \bullet \mathrm{Refl}_A \\
&= \uplus \beta \bullet (\mathrm{Refl}_A^{-1} \bullet \mathrm{Refl}_A) \\
&= \uplus \beta \bullet \uplus \mathrm{der}_{\Sigma_{\mathrm{wrw}}(\Gamma, A)} \\
&= \uplus \beta.
\end{align*}

\item \textsc{(Id-Subst)} We calculate
\begin{align*} 
\mathrm{Id}_A\{\uplus \phi^{++}\} &= (\mathrm{Id}_{A(\uplus \gamma)}(\uplus \alpha, \uplus \alpha'))_{ \langle \langle \uplus \gamma, \uplus \alpha \rangle, \uplus \alpha' \rangle \in \mathbb{NPG}_{\mathrm{wrw}}^{\mathrm{wo}}(\Sigma_{\mathrm{wrw}}(\Sigma_{\mathrm{wrw}}(\Gamma, A), A^+))} \{ \uplus \phi^{++} \} \\
&= (\mathrm{Id}_A(\langle \langle \uplus \gamma, \uplus \alpha \rangle, \uplus \alpha' \rangle))_{\langle \langle \uplus \gamma, \uplus \alpha \rangle, \uplus \alpha' \rangle \in \mathbb{NPG}_{\mathrm{wrw}}^{\mathrm{wo}}(\Sigma_{\mathrm{wrw}}(\Sigma_{\mathrm{wrw}}(\Gamma, A), A^+)} \{\uplus \phi^{++}\} \\
&= (\mathrm{Id}_{A}(\uplus \phi^{++} \bullet \langle \langle \uplus \delta, \uplus \alpha \rangle, \uplus \alpha' \rangle))_{\langle \langle \uplus \delta, \uplus \alpha \rangle, \uplus \alpha' \rangle \in \mathbb{NPG}_{\mathrm{wrw}}^{\mathrm{wo}}(\Sigma_{\mathrm{wrw}}(\Sigma_{\mathrm{wrw}}(\Delta, A\{\uplus \phi\}), A\{\uplus \phi\}^+))} \\
&= (\mathrm{Id}_{A}(\langle \uplus \phi \bullet \uplus \delta, \uplus \alpha \rangle, \uplus \alpha' \rangle))_{\langle \langle \uplus \delta, \uplus \alpha \rangle, \uplus \alpha' \rangle \in \mathbb{NPG}_{\mathrm{wrw}}^{\mathrm{wo}}(\Sigma_{\mathrm{wrw}}(\Sigma_{\mathrm{wrw}}(\Delta, A\{\uplus \phi\}), A\{\uplus \phi\}^+))} \\
&= (\mathrm{Id}_{A(\uplus \phi \bullet \uplus \delta)}(\uplus \alpha, \uplus \alpha'))_{\langle \langle \uplus \delta, \uplus \alpha \rangle, \uplus \alpha' \rangle \in \mathbb{NPG}_{\mathrm{wrw}}^{\mathrm{wo}}(\Sigma_{\mathrm{wrw}}(\Sigma_{\mathrm{wrw}}(\Delta, A\{\uplus \phi\}), A\{\uplus \phi\}^+))} \\
&= (\mathrm{Id}_{A\{\uplus \phi\}(\uplus \delta)}(\uplus \alpha, \uplus \alpha'))_{\langle \langle \uplus \delta, \uplus \alpha \rangle, \uplus \alpha' \rangle \in \mathbb{NPG}_{\mathrm{wrw}}^{\mathrm{wo}}(\Sigma_{\mathrm{wrw}}(\Sigma_{\mathrm{wrw}}(\Delta, A\{\uplus \phi\}), A\{\uplus \phi\}^+))} \\
&= \mathrm{Id}_{A\{\uplus \phi\}}, 
\end{align*}
where the forth equation holds because
\begin{align*}
\uplus \phi^{++} \bullet \langle \langle \uplus \delta, \uplus \alpha \rangle, \uplus \alpha' \rangle &= \langle \uplus \phi^+ \bullet \mathrm{p}(A^+\{\uplus \phi^+\}), \mathrm{v}_{A^+\{\uplus \phi^+\}} \rangle \bullet \langle \langle \uplus \delta, \uplus \alpha \rangle, \uplus \alpha' \rangle \\
&= \langle \uplus \phi^+ \bullet \mathrm{p}(A^+\{\uplus \phi^+\}) \bullet \langle \langle \uplus \delta, \uplus \alpha \rangle, \uplus \alpha' \rangle, \mathrm{v}_{A^+\{\uplus \phi^+\}} \bullet \langle \langle \uplus \delta, \uplus \alpha \rangle, \uplus \alpha' \rangle \rangle \\
&= \langle \langle \uplus \phi \bullet \uplus \delta, \uplus \alpha \rangle, \uplus \alpha' \rangle.
\end{align*}

\item \textsc{(Refl-Subst)} We calculate
\begin{align*}
\mathrm{Refl}_A \bullet \uplus \phi^+ &= \mathrm{Refl}_A \bullet \langle \uplus \phi \bullet \mathrm{p}(A\{ \uplus \phi \}), \mathrm{v}_{A\{ \uplus \phi \}} \rangle \\
&= \langle \langle \langle \uplus \phi \bullet \mathrm{p}(A\{ \uplus \phi \}), \mathrm{v}_{A\{ \uplus \phi \}} \rangle \bullet \mathrm{p}(A^+\{ \uplus \phi^+ \}) \bullet \mathrm{p}(\mathrm{Id}_A\{ \uplus \phi^{++} \}), \\&\mathrm{v}_{A^+\{ \uplus \phi^+ \}} \bullet \mathrm{p}(\mathrm{Id}_A\{ \uplus \phi^{++} \}) \rangle, \mathrm{v}_{\mathrm{Id}_A\{ \uplus \phi^{++} \}} \rangle \bullet \mathrm{Refl}_{A\{ \uplus \phi \}} \quad \text{(by the definition of $\mathrm{Refl}$)} \\
&= \langle \langle \uplus \phi^+ \bullet \mathrm{p}(A^+\{ \uplus \phi^+ \}), \mathrm{v}_{A^+\{ \uplus \phi^+ \}} \rangle \bullet \mathrm{p}(\mathrm{Id}_A\{ \uplus \phi^{++} \}), \mathrm{v}_{\mathrm{Id}_A\{ \uplus \phi^{++} \}} \rangle \bullet \mathrm{Refl}_{A\{ \uplus \phi \}} \\
&= \langle \uplus \phi^{++} \bullet \mathrm{p}(\mathrm{Id}_A\{ \uplus \phi^{++} \}), \mathrm{v}_{\mathrm{Id}_A\{ \uplus \phi^{++} \}} \rangle \bullet \mathrm{Refl}_{A\{ \uplus \phi \}} \\
&= \uplus \phi^{+++} \bullet \mathrm{Refl}_{A\{ \uplus \phi \}}.
\end{align*}

\item \textsc{($\mathcal{R}^{\mathrm{Id}}$-Subst)} We calculate
\begin{align*}
\mathcal{R}^{\mathrm{Id}}_{A,B}(\uplus \beta)\{\uplus \phi^{+++}\} &= (\uplus \beta \bullet \mathrm{Refl}_A^{-1}) \bullet \uplus \phi^{+++} \\
&= \uplus \beta \bullet (\mathrm{Refl}_A^{-1} \bullet \uplus \phi^{+++}) \\
&= \uplus \beta \bullet (\uplus \phi^+ \bullet \mathrm{Refl}_{A\{\uplus \phi\}}^{-1}) \quad \text{(by Refl-Subst)} \\
&= (\uplus \beta \bullet \uplus \phi^+) \bullet \mathrm{Refl}_{A\{\uplus \phi\}}^{-1} \\
&= \mathcal{R}^{\mathrm{Id}}_{A\{\uplus \phi\}, B\{\uplus \phi^{+++}\}}(\uplus \beta \bullet \uplus \phi^+) \\
&= \mathcal{R}^{\mathrm{Id}}_{A\{\uplus \phi\}, B\{\uplus \phi^{+++}\}}(\uplus \beta \{\uplus \phi^+\}),
\end{align*}

\end{itemize}
which completes the proof.
\end{proof}

We next model \emph{unit-type} (Appendix~\ref{UnitType}). 
Again, let us first recall the semantic type former for unit-type:
\begin{definition}[CwFs with unit-type \cite{hofmann1997syntax}]
A CwF $\mathcal{C}$ \emph{\bfseries supports unit} if
\begin{itemize}

\item \textsc{(Unit-Form)} Given $\Gamma \in \mathcal{C}$, there is a type $\boldsymbol{1}^{\Gamma} \in \mathrm{Ty}(\Gamma)$, called the \emph{\bfseries unity type} in the context $\Gamma$;

\item \textsc{(Unit-Intro)} Given $\Gamma \in \mathcal{C}$, there is a term $\top_{\Gamma} \in \mathrm{Tm}(\Gamma, \boldsymbol{1}^\Gamma)$;

\item \textsc{(Unit-Elim)} Given $\Gamma \in \mathcal{C}$, $A \in \mathrm{Ty}(\Gamma . \boldsymbol{1}^\Gamma)$, $a \in \mathrm{Tm}(\Gamma, A\{ \overline{\top_\Gamma} \})$ and $t \in \mathrm{Tm}(\Gamma, \boldsymbol{1}^\Gamma)$, there is a term $\mathcal{R}^{\boldsymbol{1}}_A(a, t) \in \mathrm{Tm}(\Gamma, A\{\overline{t}\})$, where 
\begin{mathpar}
\overline{\top} \colonequals \langle \mathrm{id}_\Gamma, \top_\Gamma \rangle_{\boldsymbol{1}^\Gamma} : \Gamma \to \Gamma . \boldsymbol{1}^\Gamma 
\and
\overline{t} \colonequals \langle \mathrm{id}_\Gamma, t \rangle_{\boldsymbol{1}^\Gamma} : \Gamma \to \Gamma . \boldsymbol{1}^\Gamma;
\end{mathpar}

\item \textsc{(Unit-Comp)} $\mathcal{R}^{\boldsymbol{1}}_A(a, \top_\Gamma) = a$;

\item \textsc{(Unit-Subst)} Given $f : \Delta \to \Gamma$ in $\mathcal{C}$, $\boldsymbol{1}^{\Gamma}\{f\} = \boldsymbol{1}^{\Delta} \in \mathrm{Ty}(\Delta)$;

\item \textsc{($\top$-Subst)} $\top_{\Gamma}\{f\} = \top_{\Delta} \in \mathrm{Tm}(\Delta, \boldsymbol{1}^\Delta)$.

\end{itemize}

Moreover, $\mathcal{C}$ \emph{\bfseries supports one in the strict sense} if it additionally satisfies:
\begin{itemize}
\item \textsc{($\top$-Uniq)} 
$t = \top_\Gamma$ for all $t \in \mathrm{Tm}(\Gamma, \boldsymbol{1}^\Gamma)$.\footnote{Note that $\top$-Uniq implies Unit-Elim and Unit-Comp by defining $\mathcal{R}^{\boldsymbol{1}}_A(a, t) \colonequals a$.}
\end{itemize}
\end{definition}

Unit-type (with the $\eta$-rule) is modeled in a CwF that supports unit (in the strict sense); see Appendix~\ref{UnitType} for the details.

We now propose our game semantics of unit-type:
\begin{theorem}[Realizability model of unit-type \`{a} la game semantics]
The CwF $\mathbb{NPG}_{\mathrm{wrw}}^{\mathrm{wo}}$ supports unit in the strict sense.
\end{theorem}
\begin{proof} 
Let $\Gamma, \Delta \in \mathbb{NPG}_{\mathrm{wrw}}^{\mathrm{wo}}$ and $\uplus \phi \in \mathbb{NPG}_{\mathrm{wrw}}^{\mathrm{wo}}(\Delta, \Gamma)$, and assume that $\mathscr{R}_{\mathrm{wr}}^{\mathrm{cp}}(\Gamma) \neq \emptyset$ and $\mathscr{R}_{\mathrm{wr}}^{\mathrm{cp}}(\Delta) \neq \emptyset$ since the other cases are trivial. 
\begin{itemize}

\item \textsc{(Unit-Form)} Let $\boldsymbol{1}^{\Gamma}$ be the constant dependent np-game valued at the unit np-game $\boldsymbol{1}$ (Example~\ref{ExamplesOfWellOpenedNonstandardPredicativeGames}), for which we write $\{ \boldsymbol{1} \}_\Gamma$ or $\{ \boldsymbol{1} \}$.

\item \textsc{(Unit-Intro)} Let $\top_{\Gamma} \in \mathrm{Tm}(\Gamma, \{ \boldsymbol{1} \}_\Gamma)$ be the unique one $\uplus \{ \top \}$ up to `tags.'

\item \textsc{(Unit-Elim)} Given $A \in \mathscr{D}\mathbb{NPG}_{\mathrm{wrw}}^{\mathrm{wo}}(\Sigma_{\mathrm{wrw}}(\Gamma, \{ \boldsymbol{1} \}_\Gamma))$, $\uplus \alpha \in \mathcal{TS}_{\mathrm{wr}}(\Pi_{\mathrm{wrw}}(\Gamma, A \{ \overline{\top_\Gamma} \}))$ and $\uplus \tau \in \mathcal{TS}_{\mathrm{wr}}(\Pi_{\mathrm{wrw}}(\Gamma, \{ \boldsymbol{1} \}))$, the equation $\uplus \tau = \top_\Gamma$ clearly holds, i.e., $\top$-Uniq is satisfied, and therefore we define $\mathcal{R}^{\boldsymbol{1}}_A(\uplus \alpha, \uplus \tau) \colonequals \uplus \alpha$.

\item \textsc{(Unit-Comp)} Clearly, $\mathcal{R}^{\boldsymbol{1}}_A(\uplus \alpha, \top_\Gamma) = \uplus \alpha$.

\item \textsc{(Unit-Subst)} Clearly, $\{ \boldsymbol{1} \}_\Gamma\{ \uplus \phi \} = \{ \boldsymbol{1} \}_{\Delta}$.

\item \textsc{($\top$-Subst)} Clearly, $\top_{\Gamma} \bullet \uplus \phi = \top_{\Delta} : \Delta \to \boldsymbol{1}$,


\end{itemize}
which completes the proof.
\end{proof}

Finally, we model \emph{empty-type} (Appendix~\ref{EmptyType}). 
First, the semantic type former for empty-type is:
\begin{definition}[CwFs with empty-type \cite{hofmann1997syntax}]
A CwF $\mathcal{C}$ \emph{\bfseries supports empty} if
\begin{itemize}

\item \textsc{(Empty-Form)} Given $\Gamma \in \mathcal{C}$, there is a type $\boldsymbol{0}^{\Gamma} \in \mathrm{Ty}(\Gamma)$, called the \emph{\bfseries empty-type} in the context $\Gamma$;

\item \textsc{(Empty-Elim)} Given $\Gamma \in \mathcal{C}$, $A \in \mathrm{Ty}(\Gamma . \boldsymbol{0}^\Gamma)$ and $z \in \mathrm{Tm}(\Gamma, \boldsymbol{0}^\Gamma)$, there is a term $\mathcal{R}^{\boldsymbol{0}}_A(z) \in \mathrm{Tm}(\Gamma, A\{\overline{z}\})$, where $\overline{z} \colonequals \langle \mathrm{id}_\Gamma, z \rangle_{\boldsymbol{0}^\Gamma} : \Gamma \to \Gamma . \boldsymbol{0}^\Gamma$;

\item \textsc{(Empty-Subst)} Given $f : \Delta \to \Gamma$ in $\mathcal{C}$, $\boldsymbol{0}^{\Gamma}\{f\} = \boldsymbol{0}^{\Delta} \in \mathrm{Ty}(\Delta)$;

\item \textsc{($\mathcal{R}^0$-Subst)} $\mathcal{R}^{\boldsymbol{0}}_{A\{f^+\}}(z\{f\}) = \mathcal{R}^{\boldsymbol{0}}_A(z)\{ f \}$, where $f^+ \colonequals \langle f \bullet \mathrm{p}(\boldsymbol{0}^\Delta), \mathrm{v}_{\boldsymbol{0}^\Gamma} \rangle_{\boldsymbol{0}^\Gamma} : \Delta . \boldsymbol{0}^\Delta \to \Gamma . \boldsymbol{0}^\Gamma$.

\end{itemize}
\end{definition}

Empty-type is modeled in a CwF that supports empty; see Appendix~\ref{EmptyType} for the details.

Let us finally propose our game-semantic interpretation of empty-type:
\begin{theorem}[Realizability model of empty-type \`{a} la game semantics]
\label{ThmGameSemanticZeroType}
The CwF $\mathbb{NPG}_{\mathrm{wrw}}^{\mathrm{wo}}$ supports empty.
\end{theorem}
\begin{proof}

Let $\Gamma, \Delta \in \mathbb{NPG}_{\mathrm{wrw}}^{\mathrm{wo}}$ and $\uplus \phi \in \mathbb{NPG}_{\mathrm{wrw}}^{\mathrm{wo}}(\Delta, \Gamma)$. 

\begin{itemize}

\item \textsc{(Empty-Form)} $\boldsymbol{0}^\Gamma$ is the constant dependent np-game valued at the empty np-game $\boldsymbol{0}$ (Example~\ref{ExamplesOfWellOpenedNonstandardPredicativeGames}), for which we write $\{ \boldsymbol{0} \}_{\Gamma}$ or $\{ \boldsymbol{0} \}$.

\item \textsc{(Empty-Elim)} Let $A \in \mathscr{D}\mathbb{NPG}_{\mathrm{wrw}}^{\mathrm{wo}}(\Sigma_{\mathrm{wrw}}(\Gamma, \{ \boldsymbol{0} \}_\Gamma))$ and $\uplus \zeta \in \mathcal{TS}_{\mathrm{wr}}(\Pi_{\mathrm{wrw}}(\Gamma, \{ \boldsymbol{0} \}_\Gamma))$. 
Because $\mathscr{R}_{\mathrm{wr}}^{\mathrm{cp}}(\boldsymbol{0}) = \emptyset$, we have $\mathscr{R}_{\mathrm{wr}}^{\mathrm{cp}}(\Gamma) = \emptyset$ and $\zeta = \emptyset$.
Hence, it suffices to define $\mathcal{R}^{\boldsymbol{0}}_A(\uplus \zeta) \in \mathcal{TS}_{\mathrm{wr}}(\Pi_{\mathrm{wrw}}(\Gamma, A \{ \overline{\uplus \zeta} \}))$ to be the trivial one $\uplus \{ \top \}$. 

\item \textsc{(Empty-Subst)} Clearly, $\{ \boldsymbol{0} \}_{\Gamma}\{ \uplus \phi \} = \{ \boldsymbol{0} \}_{\Delta}$.

\item \textsc{($\mathcal{R}^{\boldsymbol{0}}$-Subst)} By the definition of the operation $\mathcal{R}^{\boldsymbol{0}}_{\_}(\_)$, we calculate
\begin{align*}
\mathcal{R}^{\boldsymbol{0}}_{A\{ \uplus \phi^+ \}}(\uplus \zeta\{ \uplus \phi \}) &= \uplus \{ \top \} \\
&= \uplus \{ \top \} \bullet \uplus \phi \quad \text{(n.b.,  $\mathscr{R}_{\mathrm{wr}}^{\mathrm{cp}}(\Delta) = \emptyset$ and $\phi = \emptyset$)} \\
&= \mathcal{R}^{\boldsymbol{0}}_{A\{ \uplus \phi^+ \}}(\uplus \zeta) \bullet \uplus \phi \\
&= \mathcal{R}^{\boldsymbol{0}}_A(\uplus \zeta)\{ \uplus \phi \},
\end{align*}
\end{itemize}
which completes the proof. 
\end{proof}

\subsection{Main result: consistency of intensional Martin-L\"{o}f type theory with formal Church's thesis and its constructivity}
\label{MainResult}
Finally, let us solve the targeted open problem, viz., consistency of MLTT with CT, \emph{in the affirmative}.
For precision, we implement the \emph{T-predicate} $T$ and the \emph{result-extracting function} $U$ (Section~\ref{IntroFormalChurch'sThesis}) explicitly in MLTT as follows. 

We first define a \emph{realizer} for a given partial function $f : \mathbb{N} \rightharpoonup \mathbb{N}$ to be a realizer for a t-skeleton $\uplus \phi :: \hat{\oc} N \rightarrowtriangle N$ whose \emph{extension} $\mathrm{ext}(\uplus \phi) : n \in \mathbb{N} \mapsto \mathrm{read}(\uplus \phi \bullet \underline{n}_{\boldsymbol{1}})$, where $\mathrm{read}(\uplus \{ \bot \}) \uparrow$ and $\mathrm{read}(\underline{m}_{\boldsymbol{1}}) \colonequals m$ for all $m \in \mathbb{N}$, coincides with $f$.

Then, we implement the T-predicate and the result-extracting function in MLTT with respect to the resulting realizers for total recursive functions:
\begin{itemize}

\item Let the T-predicate to be the type $\mathsf{x : N, y : N, z : N \vdash T(x, y, z) \ type}$ defined by $\mathsf{T(\underline{e}, \underline{n}, \underline{c}) \colonequals 1}$ if the triple $e, n, c \in \mathbb{N}$ satisfies the T-predicate with respect to the realizers just defined in the preceding paragraphs, for which we write $T(e, n, c)$, and $\mathsf{T(\underline{e}, \underline{n}, \underline{c}) \colonequals 0}$ otherwise;

\item Let the result-extracting function to be the term $\mathsf{x : N \vdash U(x) : N}$ defined by $\mathsf{U(\underline{c}) \colonequals \underline{n} : N}$ if $c \in \mathbb{N}$ encodes, with respect to the realizers just defined, the computational history of a terminating computation on $\mathbb{N}$ whose output is $n \in \mathbb{N}$, and $\mathsf{U(\underline{c}) \colonequals \underline{0} : N}$ otherwise. 

\end{itemize}

In practice, the T-predicate and the result-extracting function are given in MLTT by the elimination rule of N-type, which is possible because they are both primitive recursive \cite[Section~3.4.3]{troelstra1988constructivism}.
Thus, having modeled MLTT, including N-type, one may think that a model of the T-predicate and the result-extracting function has been obtained as well. 
However, it is not quite the case because we need a \emph{universe} to formalize the T-predicate in that way because it is a type, not a term.

For this point, we could extend our model of MLTT to a universe, but for brevity, let us skip modeling a universe in this article and instead define our interpretation of the T-predicate \emph{directly}: 
\begin{definition}[Realizability model of the T-predicate \`{a} la game semantics]
Let us interpret the T-predicate in the CwF $\mathbb{NPG}_{\mathrm{wrw}}^{\mathrm{wo}}$ by the dependent np-game $\mathcal{T}$ on the product $(N \mathbin{\&} N) \mathbin{\&} N$ defined by
\begin{equation*}
\mathcal{T} (\langle \langle \uplus \sigma, \uplus \tau \rangle, \uplus \gamma \rangle) \colonequals \begin{cases} \boldsymbol{1} &\text{if $T(\mathrm{read} (\uplus \sigma), \mathrm{read}(\uplus \tau), \mathrm{read} (\uplus \gamma))$;} \\ \boldsymbol{0} &\text{otherwise.} \end{cases}
\end{equation*}
\end{definition}

Clearly, our interpretation of the T-predicate is \emph{sound}, i.e., it respects the defining judgmental equality.
Having extended our model of MLTT in the CwF $\mathbb{NPG}_{\mathrm{wrw}}^{\mathrm{wo}}$ to the T-predicate, we are now ready to present our main theorem: 
\begin{theorem}[Realizability model of MLTT plus CT \`{a} la game semantics]
\label{ThmGameSemanticsOfMLTTPlusCT}
The CwF $\mathbb{NPG}_{\mathrm{wrw}}^{\mathrm{wo}}$ validates CT, and therefore it models MLTT equipped with CT.
\end{theorem}
\begin{proof}
We have given a model of MLTT equipped with the T-predicate and unit-, empty-, N-, sigma-, pi- and Id-types in $\mathbb{NPG}_{\mathrm{wrw}}^{\mathrm{wo}}$.
Hence, it remains to validate CT in $\mathbb{NPG}_{\mathrm{wrw}}^{\mathrm{wo}}$.
Let us write $\mathcal{CT} \in \mathbb{NPG}_{\mathrm{wrw}}^{\mathrm{wo}}$ for the interpretation of CT (\ref{CTinMLTT}) in $\mathbb{NPG}_{\mathrm{wrw}}^{\mathrm{wo}}$.

As sketched in Section~\ref{GameSemanticRealizability}, we provide the t-skeleton $\uplus \mathrm{ct} :: \mathcal{CT}$ whose component $\mathrm{ct}_{(\uplus \phi^\dagger, e)}$ for each $(\uplus \phi^\dagger, e) \in \mathscr{R}^{\mathrm{cp}}_{\mathrm{wr}}(\hat{\oc} (\hat{\oc} N \rightarrowtriangle N))$ is the pairing $\langle \uplus \{ \underline{e^\ddagger} \},  \mathrm{ct}' \rangle$ up to `tags,' where $e^\ddagger \in \pi_2(\mathscr{R}^{\mathrm{cp}}_{\mathrm{wr}}(\hat{\oc} N \rightarrowtriangle N))$ is the canonical realizer for $\uplus \phi$, and $\mathrm{ct}'$ is the evident DoWRWLI that essentially computes the code $c \in \mathbb{N}$ of the computational history of $\uplus \phi$ applied to a given t-skeleton $\underline{n}_{\boldsymbol{1}}$ ($n \in \mathbb{N}$), so that $T(e^\ddagger, n, c)$ and $\underline{U (c)}_{\boldsymbol{1}} = \uplus \phi \bullet \underline{n}_{\boldsymbol{1}}$ hold.
Finally, $\uplus \mathrm{ct}$ is clearly winning and recursive, which completes the proof. 
\end{proof}

Let us emphasize that our model of MLTT equipped with CT in the CwF $\mathbb{NPG}_{\mathrm{wrw}}^{\mathrm{wo}}$ is \emph{constructive} in the sense that every t-skeleton in $\mathbb{NPG}_{\mathrm{wrw}}^{\mathrm{wo}}$ is recursive. 
In other words, we have established a `constructive justification' of CT within MLTT.

At last, we are now ready to solve the open problem: 
\begin{corollary}[Consistency of MLTT with CT]
MLTT is consistent with CT.
\end{corollary}
\begin{proof}
The corollary follows from Theorem~\ref{ThmGameSemanticsOfMLTTPlusCT} because there is no total, let alone winning, t-skeleton on the empty np-game $\boldsymbol{0}$.
\end{proof}

\section{Conclusion and future work}
\label{ConclusionAndFutureWork}
We have proved consistency of MLTT with CT, which has been a long-standing open problem in constructive mathematics, by a novel realizability model of MLTT plus CT \`{a} la game semantics. 
Our main technical highlight is to resolve the dilemma between intensionality and extensionality in the consistency problem, which existing realizability models cannot overcome, by taking advantage of some distinguishing features of game semantics: the distinction and the asymmetry between O and P.

It is also worth noting that we have validated CT \emph{constructively} in the sense that an `effective' morphism inhabits the interpretation of CT.  

Methodologically, the present work has demonstrated that game semantics is a powerful semantic method not only for full abstraction/completeness problems but also for the meta-theoretic study on foundations of constructive mathematics. 

As future work, we plan to extend our game semantics of MLTT to \emph{homotopy type theory (HoTT)} \cite{hottbook} and apply the semantics to the study of HoTT similarly to the present work. 
We would also like to extend the semantics to other type constructions such as \emph{well-founded tree types} \cite{martin1998intuitionistic} and \emph{induction-recursion} \cite{dybjer2000general}.

\section*{Acknowledgements}
The author acknowleges the financial support from Funai Overseas Scholarship. 
He is also grateful to Thomas Streicher and Andrew Swan for fruitful discussions especially because the starting point of the present work was when the author leant from them that consistency of MLTT with CT has been an open problem.

\section*{Conflict of interest statement}
The author states that there is no conflict of interest.

\bibliographystyle{bmc-mathphys} 
\bibliography{CategoricalLogic,GamesAndStrategies,RecursionTheory,PCF,TypeTheoriesAndProgrammingLanguages,HoTT,GoI,LinearLogic}      


\begin{thebibliography}{55}
\ifx \bisbn   \undefined \def \bisbn  #1{ISBN #1}\fi
\ifx \binits  \undefined \def \binits#1{#1}\fi
\ifx \bauthor  \undefined \def \bauthor#1{#1}\fi
\ifx \batitle  \undefined \def \batitle#1{#1}\fi
\ifx \bjtitle  \undefined \def \bjtitle#1{#1}\fi
\ifx \bvolume  \undefined \def \bvolume#1{\textbf{#1}}\fi
\ifx \byear  \undefined \def \byear#1{#1}\fi
\ifx \bissue  \undefined \def \bissue#1{#1}\fi
\ifx \bfpage  \undefined \def \bfpage#1{#1}\fi
\ifx \blpage  \undefined \def \blpage #1{#1}\fi
\ifx \burl  \undefined \def \burl#1{\textsf{#1}}\fi
\ifx \doiurl  \undefined \def \doiurl#1{\textsf{#1}}\fi
\ifx \betal  \undefined \def \betal{\textit{et al.}}\fi
\ifx \binstitute  \undefined \def \binstitute#1{#1}\fi
\ifx \binstitutionaled  \undefined \def \binstitutionaled#1{#1}\fi
\ifx \bctitle  \undefined \def \bctitle#1{#1}\fi
\ifx \beditor  \undefined \def \beditor#1{#1}\fi
\ifx \bpublisher  \undefined \def \bpublisher#1{#1}\fi
\ifx \bbtitle  \undefined \def \bbtitle#1{#1}\fi
\ifx \bedition  \undefined \def \bedition#1{#1}\fi
\ifx \bseriesno  \undefined \def \bseriesno#1{#1}\fi
\ifx \blocation  \undefined \def \blocation#1{#1}\fi
\ifx \bsertitle  \undefined \def \bsertitle#1{#1}\fi
\ifx \bsnm \undefined \def \bsnm#1{#1}\fi
\ifx \bsuffix \undefined \def \bsuffix#1{#1}\fi
\ifx \bparticle \undefined \def \bparticle#1{#1}\fi
\ifx \barticle \undefined \def \barticle#1{#1}\fi
\ifx \bconfdate \undefined \def \bconfdate #1{#1}\fi
\ifx \botherref \undefined \def \botherref #1{#1}\fi
\ifx \url \undefined \def \url#1{\textsf{#1}}\fi
\ifx \bchapter \undefined \def \bchapter#1{#1}\fi
\ifx \bbook \undefined \def \bbook#1{#1}\fi
\ifx \bcomment \undefined \def \bcomment#1{#1}\fi
\ifx \oauthor \undefined \def \oauthor#1{#1}\fi
\ifx \citeauthoryear \undefined \def \citeauthoryear#1{#1}\fi
\ifx \endbibitem  \undefined \def \endbibitem {}\fi
\ifx \bconflocation  \undefined \def \bconflocation#1{#1}\fi
\ifx \arxivurl  \undefined \def \arxivurl#1{\textsf{#1}}\fi
\csname PreBibitemsHook\endcsname

\bibitem{ishihara2018consistency}
\begin{barticle}
\bauthor{\bsnm{Ishihara}, \binits{H.}},
\bauthor{\bsnm{Maietti}, \binits{M.E.}},
\bauthor{\bsnm{Maschio}, \binits{S.}},
\bauthor{\bsnm{Streicher}, \binits{T.}}:
\batitle{Consistency of the intensional level of the minimalist foundation with
  church’s thesis and axiom of choice}.
\bjtitle{Archive for Mathematical Logic}
\bvolume{57}(\bissue{7-8}),
\bfpage{873}--\blpage{888}
(\byear{2018})
\end{barticle}
\endbibitem

\bibitem{yamada2019game}
\begin{barticle}
\bauthor{\bsnm{Yamada}, \binits{N.}}:
\batitle{A game-semantic model of computation}.
\bjtitle{Research in the Mathematical Sciences}
\bvolume{6}(\bissue{1}),
\bfpage{3}
(\byear{2019})
\end{barticle}
\endbibitem

\bibitem{troelstra1988constructivism}
\begin{botherref}
\oauthor{\bsnm{Troelstra}, \binits{A.S.}},
\oauthor{\bparticle{van} \bsnm{Dalen}, \binits{D.}}:
Constructivism in mathematics. {V}ol. {I}, volume 121 of.
Studies in Logic and the Foundations of Mathematics,
26
(1988)
\end{botherref}
\endbibitem

\bibitem{troelstra2014constructivism}
\begin{bbook}
\bauthor{\bsnm{Troelstra}, \binits{A.S.}},
\bauthor{\bsnm{Van~Dalen}, \binits{D.}}:
\bbtitle{Constructivism in Mathematics}
vol. \bseriesno{2}.
\bpublisher{Elsevier},
\blocation{Amsterdam}
(\byear{2014})
\end{bbook}
\endbibitem

\bibitem{beeson2012foundations}
\begin{bbook}
\bauthor{\bsnm{Beeson}, \binits{M.J.}}:
\bbtitle{Foundations of Constructive Mathematics: Metamathematical Studies}
vol. \bseriesno{6}.
\bpublisher{Springer}, \blocation{???}
(\byear{2012})
\end{bbook}
\endbibitem

\bibitem{shoenfield1967mathematical}
\begin{bbook}
\bauthor{\bsnm{Shoenfield}, \binits{J.R.}}:
\bbtitle{Mathematical Logic}
vol. \bseriesno{21}.
\bpublisher{Addison-Wesley},
\blocation{Reading}
(\byear{1967})
\end{bbook}
\endbibitem

\bibitem{kleene1952introduction}
\begin{botherref}
\oauthor{\bsnm{Kleene}, \binits{S.C.}}:
Introduction to metamathematics
(1952)
\end{botherref}
\endbibitem

\bibitem{girard1989proofs}
\begin{bbook}
\bauthor{\bsnm{Girard}, \binits{J.-Y.}},
\bauthor{\bsnm{Taylor}, \binits{P.}},
\bauthor{\bsnm{Lafont}, \binits{Y.}}:
\bbtitle{Proofs and Types}
vol. \bseriesno{7}.
\bpublisher{Cambridge University Press Cambridge}, \blocation{???}
(\byear{1989})
\end{bbook}
\endbibitem

\bibitem{jacobs1999categorical}
\begin{bbook}
\bauthor{\bsnm{Jacobs}, \binits{B.}}:
\bbtitle{{Categorical Logic and Type Theory}}
vol. \bseriesno{141}.
\bpublisher{Elsevier}, \blocation{???}
(\byear{1999})
\end{bbook}
\endbibitem

\bibitem{sorensen2006lectures}
\begin{bbook}
\bauthor{\bsnm{S{\o}rensen}, \binits{M.H.}},
\bauthor{\bsnm{Urzyczyn}, \binits{P.}}:
\bbtitle{Lectures on the Curry-Howard Isomorphism}
vol. \bseriesno{149}.
\bpublisher{Elsevier}, \blocation{???}
(\byear{2006})
\end{bbook}
\endbibitem

\bibitem{gentzen1935untersuchungen}
\begin{barticle}
\bauthor{\bsnm{Gentzen}, \binits{G.}}:
\batitle{Untersuchungen {\"u}ber das logische schlie{\ss}en. i}.
\bjtitle{Mathematische zeitschrift}
\bvolume{39}(\bissue{1}),
\bfpage{176}--\blpage{210}
(\byear{1935})
\end{barticle}
\endbibitem

\bibitem{troelstra2000basic}
\begin{bbook}
\bauthor{\bsnm{Troelstra}, \binits{A.S.}},
\bauthor{\bsnm{Schwichtenberg}, \binits{H.}}:
\bbtitle{Basic Proof Theory}
vol. \bseriesno{43}.
\bpublisher{Cambridge University Press}, \blocation{???}
(\byear{2000})
\end{bbook}
\endbibitem

\bibitem{martin1982constructive}
\begin{barticle}
\bauthor{\bsnm{Martin-L{\"o}f}, \binits{P.}}:
\batitle{Constructive {M}athematics and {C}omputer {P}rogramming}.
\bjtitle{Studies in Logic and the Foundations of Mathematics}
\bvolume{104},
\bfpage{153}--\blpage{175}
(\byear{1982})
\end{barticle}
\endbibitem

\bibitem{martin1984intuitionistic}
\begin{bbook}
\bauthor{\bsnm{Martin-L{\"o}f}, \binits{P.}}:
\bbtitle{{Intuitionistic Type Theory: Notes by Giovanni Sambin of a Series of
  Lectures Given in Padova, June 1980}},
(\byear{1984})
\end{bbook}
\endbibitem

\bibitem{martin1998intuitionistic}
\begin{barticle}
\bauthor{\bsnm{Martin-L{\"o}f}, \binits{P.}}:
\batitle{{An Intuitionistic Theory of Types}}.
\bjtitle{Twenty-five years of constructive type theory}
\bvolume{36},
\bfpage{127}--\blpage{172}
(\byear{1998})
\end{barticle}
\endbibitem

\bibitem{enderton1977elements}
\begin{bbook}
\bauthor{\bsnm{Enderton}, \binits{H.B.}}:
\bbtitle{Elements of Set Theory}.
\bpublisher{Academic Press}, \blocation{???}
(\byear{1977})
\end{bbook}
\endbibitem

\bibitem{nordstrom1990programming}
\begin{botherref}
\oauthor{\bsnm{Nordstr{\"o}m}, \binits{B.}},
\oauthor{\bsnm{Petersson}, \binits{K.}},
\oauthor{\bsnm{Smith}, \binits{J.M.}}:
Programming in Martin-L{\"o}f’s type theory, volume 7 of International Series
  of Monographs on Computer Science.
Clarendon Press, Oxford
(1990)
\end{botherref}
\endbibitem

\bibitem{maietti2005toward}
\begin{barticle}
\bauthor{\bsnm{Maietti}, \binits{M.E.}},
\bauthor{\bsnm{Sambin}, \binits{G.}}:
\batitle{Toward a minimalist foundation for constructive mathematics}.
\bjtitle{From Sets and Types to Topology and Analysis: Practicable Foundations
  for Constructive Mathematics}
\bvolume{48},
\bfpage{91}--\blpage{114}
(\byear{2005})
\end{barticle}
\endbibitem

\bibitem{hottbook}
\begin{bbook}
\bauthor{\bsnm{{Univalent Foundations Program}}, \binits{T.}}:
\bbtitle{Homotopy Type Theory: Univalent Foundations of Mathematics}.
\bpublisher{\url{https://homotopytypetheory.org/book}},
\blocation{Institute for Advanced Study}
(\byear{2013})
\end{bbook}
\endbibitem

\bibitem{hofmann1997syntax}
\begin{bchapter}
\bauthor{\bsnm{Hofmann}, \binits{M.}}:
\bctitle{Syntax and {S}emantics of {D}ependent {T}ypes}.
In: \bbtitle{Extensional Constructs in Intensional Type Theory},
pp. \bfpage{13}--\blpage{54}.
\bpublisher{Springer}, \blocation{???}
(\byear{1997})
\end{bchapter}
\endbibitem

\bibitem{kreisel1970church}
\begin{bchapter}
\bauthor{\bsnm{Kreisel}, \binits{G.}}:
\bctitle{Church's thesis: a kind of reducibility axiom for constructive
  mathematics}.
In: \bbtitle{Studies in Logic and the Foundations of Mathematics}
vol. \bseriesno{60},
pp. \bfpage{121}--\blpage{150}.
\bpublisher{Elsevier}, \blocation{???}
(\byear{1970})
\end{bchapter}
\endbibitem

\bibitem{rogers1967theory}
\begin{bbook}
\bauthor{\bsnm{Rogers}, \binits{H.}},
\bauthor{\bsnm{Rogers}, \binits{H.}}:
\bbtitle{Theory of Recursive Functions and Effective Computability}
vol. \bseriesno{5}.
\bpublisher{McGraw-Hill},
\blocation{New York}
(\byear{1967})
\end{bbook}
\endbibitem

\bibitem{cutland1980computability}
\begin{bbook}
\bauthor{\bsnm{Cutland}, \binits{N.}}:
\bbtitle{Computability: An Introduction to Recursive Function Theory}.
\bpublisher{Cambridge University Press},
\blocation{Cambridge}
(\byear{1980})
\end{bbook}
\endbibitem

\bibitem{weihrauch2012computable}
\begin{bbook}
\bauthor{\bsnm{Weihrauch}, \binits{K.}}:
\bbtitle{Computable Analysis: An Introduction}.
\bpublisher{Springer},
\blocation{Berlin, Heidelberg}
(\byear{2012})
\end{bbook}
\endbibitem

\bibitem{rice1953classes}
\begin{barticle}
\bauthor{\bsnm{Rice}, \binits{H.G.}}:
\batitle{Classes of recursively enumerable sets and their decision problems}.
\bjtitle{Transactions of the American Mathematical Society}
\bvolume{74}(\bissue{2}),
\bfpage{358}--\blpage{366}
(\byear{1953})
\end{barticle}
\endbibitem

\bibitem{maietti2009minimalist}
\begin{barticle}
\bauthor{\bsnm{Maietti}, \binits{M.E.}}:
\batitle{A minimalist two-level foundation for constructive mathematics}.
\bjtitle{Annals of Pure and Applied Logic}
\bvolume{160}(\bissue{3}),
\bfpage{319}--\blpage{354}
(\byear{2009})
\end{barticle}
\endbibitem

\bibitem{reus1999realizability}
\begin{barticle}
\bauthor{\bsnm{Reus}, \binits{B.}}:
\batitle{{Realizability Models for Type Theories}}.
\bjtitle{Electronic Notes in Theoretical Computer Science}
\bvolume{23}(\bissue{1}),
\bfpage{128}--\blpage{158}
(\byear{1999})
\end{barticle}
\endbibitem

\bibitem{streicher2008realizability}
\begin{botherref}
\oauthor{\bsnm{Streicher}, \binits{T.}}:
Realizability
(2008)
\end{botherref}
\endbibitem

\bibitem{abramsky1997semantics}
\begin{barticle}
\bauthor{\bsnm{Abramsky}, \binits{S.}}, \betal:
\batitle{Semantics of interaction: An introduction to game semantics}.
\bjtitle{Semantics and Logics of Computation}
\bvolume{14},
\bfpage{1}--\blpage{31}
(\byear{1997})
\end{barticle}
\endbibitem

\bibitem{hyland1997game}
\begin{bchapter}
\bauthor{\bsnm{Hyland}, \binits{M.}}:
\bctitle{Game semantics}.
In: \bbtitle{Semantics and Logics of Computation}
vol. \bseriesno{14},
p. \bfpage{131}.
\bpublisher{Cambridge University Press},
\blocation{New York}
(\byear{1997})
\end{bchapter}
\endbibitem

\bibitem{amadio1998domains}
\begin{bbook}
\bauthor{\bsnm{Amadio}, \binits{R.M.}},
\bauthor{\bsnm{Curien}, \binits{P.-L.}}:
\bbtitle{Domains and Lambda-Calculi}
vol. \bseriesno{46}.
\bpublisher{Cambridge University Press},
\blocation{Cambridge}
(\byear{1998})
\end{bbook}
\endbibitem

\bibitem{abramsky2014intensionality}
\begin{bchapter}
\bauthor{\bsnm{Abramsky}, \binits{S.}}:
\bctitle{Intensionality, definability and computation}.
In: \bbtitle{Johan van Benthem on Logic and Information Dynamics},
pp. \bfpage{121}--\blpage{142}.
\bpublisher{Springer},
\blocation{Cham}
(\byear{2014})
\end{bchapter}
\endbibitem

\bibitem{abramsky1999game}
\begin{bchapter}
\bauthor{\bsnm{Abramsky}, \binits{S.}},
\bauthor{\bsnm{McCusker}, \binits{G.}}:
\bctitle{Game semantics}.
In: \bbtitle{Computational Logic: Proceedings of the 1997 Marktoberdorf Summer
  School},
pp. \bfpage{1}--\blpage{55}.
\bpublisher{Springer},
\blocation{Berlin, Heidelberg}
(\byear{1999})
\end{bchapter}
\endbibitem

\bibitem{abramsky2015games}
\begin{bchapter}
\bauthor{\bsnm{Abramsky}, \binits{S.}},
\bauthor{\bsnm{Jagadeesan}, \binits{R.}},
\bauthor{\bsnm{V{\'a}k{\'a}r}, \binits{M.}}:
\bctitle{Games for dependent types}.
In: \bbtitle{Automata, Languages, and Programming},
pp. \bfpage{31}--\blpage{43}.
\bpublisher{Springer},
\blocation{Berlin, Heidelberg}
(\byear{2015})
\end{bchapter}
\endbibitem

\bibitem{yamada2016game}
\begin{botherref}
\oauthor{\bsnm{Yamada}, \binits{N.}}:
Game semantics for {M}artin-{L}\"{o}f type theory.
arXiv preprint arXiv:1610.01669
(2016)
\end{botherref}
\endbibitem

\bibitem{mccusker1998games}
\begin{bbook}
\bauthor{\bsnm{McCusker}, \binits{G.}}:
\bbtitle{Games and Full Abstraction for a Functional Metalanguage with
  Recursive Types}.
\bpublisher{Springer},
\blocation{London}
(\byear{1998})
\end{bbook}
\endbibitem

\bibitem{yamada2016dynamic}
\begin{botherref}
\oauthor{\bsnm{Yamada}, \binits{N.}},
\oauthor{\bsnm{Abramsky}, \binits{S.}}:
Dynamic games and strategies.
arXiv preprint arXiv:1601.04147
(2016).
Accepted for publication by Mathematical Structures in Computer Science.
\end{botherref}
\endbibitem

\bibitem{girard1987linear}
\begin{barticle}
\bauthor{\bsnm{Girard}, \binits{J.-Y.}}:
\batitle{Linear logic}.
\bjtitle{Theoretical Computer Science}
\bvolume{50}(\bissue{1}),
\bfpage{1}--\blpage{101}
(\byear{1987})
\end{barticle}
\endbibitem

\bibitem{abramsky1994games}
\begin{barticle}
\bauthor{\bsnm{Abramsky}, \binits{S.}},
\bauthor{\bsnm{Jagadeesan}, \binits{R.}}:
\batitle{Games and full completeness for multiplicative linear logic}.
\bjtitle{The Journal of Symbolic Logic}
\bvolume{59}(\bissue{02}),
\bfpage{543}--\blpage{574}
(\byear{1994})
\end{barticle}
\endbibitem

\bibitem{lambek1988introduction}
\begin{bbook}
\bauthor{\bsnm{Lambek}, \binits{J.}},
\bauthor{\bsnm{Scott}, \binits{P.J.}}:
\bbtitle{Introduction to {H}igher-order {C}ategorical {L}ogic}
vol. \bseriesno{7}.
\bpublisher{Cambridge University Press}, \blocation{???}
(\byear{1988})
\end{bbook}
\endbibitem

\bibitem{abramsky2005game}
\begin{barticle}
\bauthor{\bsnm{Abramsky}, \binits{S.}},
\bauthor{\bsnm{Jagadeesan}, \binits{R.}}:
\batitle{A game semantics for generic polymorphism}.
\bjtitle{Annals of Pure and Applied Logic}
\bvolume{133}(\bissue{1}),
\bfpage{3}--\blpage{37}
(\byear{2005})
\end{barticle}
\endbibitem

\bibitem{hughes2000hypergame}
\begin{botherref}
\oauthor{\bsnm{Hughes}, \binits{D.}}:
Hypergame semantics: full completeness for system {F}.
PhD thesis,
D. Phil. thesis, Oxford University
(2000)
\end{botherref}
\endbibitem

\bibitem{abramsky2009game}
\begin{barticle}
\bauthor{\bsnm{Abramsky}, \binits{S.}},
\bauthor{\bsnm{Jagadeesan}, \binits{R.}}:
\batitle{Game semantics for access control}.
\bjtitle{Electronic Notes in Theoretical Computer Science}
\bvolume{249},
\bfpage{135}--\blpage{156}
(\byear{2009})
\end{barticle}
\endbibitem

\bibitem{curien2007definability}
\begin{barticle}
\bauthor{\bsnm{Curien}, \binits{P.-L.}}:
\batitle{Definability and full abstraction}.
\bjtitle{Electronic Notes in Theoretical Computer Science}
\bvolume{172},
\bfpage{301}--\blpage{310}
(\byear{2007})
\end{barticle}
\endbibitem

\bibitem{clairambault2010totality}
\begin{barticle}
\bauthor{\bsnm{Clairambault}, \binits{P.}},
\bauthor{\bsnm{Harmer}, \binits{R.}}:
\batitle{Totality in arena games}.
\bjtitle{Annals of pure and applied logic}
\bvolume{161}(\bissue{5}),
\bfpage{673}--\blpage{689}
(\byear{2010})
\end{barticle}
\endbibitem

\bibitem{curien2006notes}
\begin{botherref}
\oauthor{\bsnm{Curien}, \binits{P.-L.}}:
Notes on game semantics.
From the author?s web page
(2006)
\end{botherref}
\endbibitem

\bibitem{hyland2000full}
\begin{barticle}
\bauthor{\bsnm{Hyland}, \binits{J.M.E.}},
\bauthor{\bsnm{Ong}, \binits{C.-H.}}:
\batitle{On full abstraction for {PCF}: {I}, {II}, and {III}}.
\bjtitle{Information and Computation}
\bvolume{163}(\bissue{2}),
\bfpage{285}--\blpage{408}
(\byear{2000})
\end{barticle}
\endbibitem

\bibitem{curien1998abstract}
\begin{barticle}
\bauthor{\bsnm{Curien}, \binits{P.-L.}}:
\batitle{Abstract {B}{\"o}hm trees}.
\bjtitle{Mathematical Structures in Computer Science}
\bvolume{8}(\bissue{06}),
\bfpage{559}--\blpage{591}
(\byear{1998})
\end{barticle}
\endbibitem

\bibitem{abramsky2000full}
\begin{barticle}
\bauthor{\bsnm{Abramsky}, \binits{S.}},
\bauthor{\bsnm{Jagadeesan}, \binits{R.}},
\bauthor{\bsnm{Malacaria}, \binits{P.}}:
\batitle{Full abstraction for {PCF}}.
\bjtitle{Information and Computation}
\bvolume{163}(\bissue{2}),
\bfpage{409}--\blpage{470}
(\byear{2000})
\end{barticle}
\endbibitem

\bibitem{dybjer1996internal}
\begin{bchapter}
\bauthor{\bsnm{Dybjer}, \binits{P.}}:
\bctitle{{Internal Type Theory}}.
In: \bbtitle{Types for Proofs and Programs},
pp. \bfpage{120}--\blpage{134}.
\bpublisher{Springer}, \blocation{???}
(\byear{1996})
\end{bchapter}
\endbibitem

\bibitem{dybjer2000general}
\begin{botherref}
\oauthor{\bsnm{Dybjer}, \binits{P.}}:
A general formulation of simultaneous inductive-recursive definitions in type
  theory.
Journal of Symbolic Logic,
525--549
(2000)
\end{botherref}
\endbibitem

\bibitem{barendregt1984lambda}
\begin{bbook}
\bauthor{\bsnm{Barendregt}, \binits{H.P.}}, \betal:
\bbtitle{The Lambda Calculus}
vol. \bseriesno{3}.
\bpublisher{North-Holland},
\blocation{Amsterdam}
(\byear{1984})
\end{bbook}
\endbibitem

\bibitem{dybjer2016intuitionistic}
\begin{botherref}
\oauthor{\bsnm{Dybjer}, \binits{P.}},
\oauthor{\bsnm{Palmgren}, \binits{E.}}:
Intuitionistic type theory.
Stanford Encyclopedia of Philosophy
(2016)
\end{botherref}
\endbibitem

\bibitem{hankin1994lambda}
\begin{botherref}
\oauthor{\bsnm{Hankin}, \binits{C.}}:
Lambda calculi: A guide for the perplexed
(1994)
\end{botherref}
\endbibitem

\bibitem{peano1879arithmetices}
\begin{botherref}
\oauthor{\bsnm{Peano}, \binits{G.}}:
Arithmetices principia, nova methodo exposita, 1899.
English translation in [51],
83--97
(1879)
\end{botherref}
\endbibitem

\end{thebibliography}

\newcommand{\BMCxmlcomment}[1]{}

\BMCxmlcomment{

<refgrp>

<bibl id="B1">
  <title><p>Consistency of the intensional level of the Minimalist Foundation
  with Church’s Thesis and Axiom of Choice</p></title>
  <aug>
    <au><snm>Ishihara</snm><fnm>H</fnm></au>
    <au><snm>Maietti</snm><fnm>ME</fnm></au>
    <au><snm>Maschio</snm><fnm>S</fnm></au>
    <au><snm>Streicher</snm><fnm>T</fnm></au>
  </aug>
  <source>Archive for Mathematical Logic</source>
  <publisher>Springer</publisher>
  <pubdate>2018</pubdate>
  <volume>57</volume>
  <issue>7-8</issue>
  <fpage>873</fpage>
  <lpage>-888</lpage>
</bibl>

<bibl id="B2">
  <title><p>A game-semantic model of computation</p></title>
  <aug>
    <au><snm>Yamada</snm><fnm>N</fnm></au>
  </aug>
  <source>Research in the Mathematical Sciences</source>
  <publisher>Springer</publisher>
  <pubdate>2019</pubdate>
  <volume>6</volume>
  <issue>1</issue>
  <fpage>3</fpage>
</bibl>

<bibl id="B3">
  <title><p>Constructivism in Mathematics. {V}ol. {I}, volume 121
  of</p></title>
  <aug>
    <au><snm>Troelstra</snm><fnm>AS</fnm></au>
    <au><snm>Dalen</snm><fnm>D</fnm></au>
  </aug>
  <source>Studies in Logic and the Foundations of Mathematics</source>
  <pubdate>1988</pubdate>
  <fpage>26</fpage>
</bibl>

<bibl id="B4">
  <title><p>Constructivism in Mathematics</p></title>
  <aug>
    <au><snm>Troelstra</snm><fnm>AS</fnm></au>
    <au><snm>Van Dalen</snm><fnm>D</fnm></au>
  </aug>
  <publisher>Amsterdam: Elsevier</publisher>
  <pubdate>2014</pubdate>
  <volume>2</volume>
</bibl>

<bibl id="B5">
  <title><p>Foundations of constructive mathematics: Metamathematical
  studies</p></title>
  <aug>
    <au><snm>Beeson</snm><fnm>MJ</fnm></au>
  </aug>
  <publisher>Springer Science \& Business Media</publisher>
  <pubdate>2012</pubdate>
  <volume>6</volume>
</bibl>

<bibl id="B6">
  <title><p>Mathematical Logic</p></title>
  <aug>
    <au><snm>Shoenfield</snm><fnm>JR</fnm></au>
  </aug>
  <publisher>Reading: Addison-Wesley</publisher>
  <pubdate>1967</pubdate>
  <volume>21</volume>
</bibl>

<bibl id="B7">
  <title><p>Introduction to Metamathematics</p></title>
  <aug>
    <au><snm>Kleene</snm><fnm>SC</fnm></au>
  </aug>
  <pubdate>1952</pubdate>
</bibl>

<bibl id="B8">
  <title><p>Proofs and Types</p></title>
  <aug>
    <au><snm>Girard</snm><fnm>JY</fnm></au>
    <au><snm>Taylor</snm><fnm>P</fnm></au>
    <au><snm>Lafont</snm><fnm>Y</fnm></au>
  </aug>
  <publisher>Cambridge University Press Cambridge</publisher>
  <pubdate>1989</pubdate>
  <volume>7</volume>
</bibl>

<bibl id="B9">
  <title><p>{Categorical Logic and Type Theory}</p></title>
  <aug>
    <au><snm>Jacobs</snm><fnm>B</fnm></au>
  </aug>
  <publisher>Elsevier</publisher>
  <pubdate>1999</pubdate>
  <volume>141</volume>
</bibl>

<bibl id="B10">
  <title><p>Lectures on the Curry-Howard isomorphism</p></title>
  <aug>
    <au><snm>S{\o}rensen</snm><fnm>MH</fnm></au>
    <au><snm>Urzyczyn</snm><fnm>P</fnm></au>
  </aug>
  <publisher>Elsevier</publisher>
  <pubdate>2006</pubdate>
  <volume>149</volume>
</bibl>

<bibl id="B11">
  <title><p>Untersuchungen {\"u}ber das logische Schlie{\ss}en. I</p></title>
  <aug>
    <au><snm>Gentzen</snm><fnm>G</fnm></au>
  </aug>
  <source>Mathematische zeitschrift</source>
  <publisher>Springer</publisher>
  <pubdate>1935</pubdate>
  <volume>39</volume>
  <issue>1</issue>
  <fpage>176</fpage>
  <lpage>-210</lpage>
</bibl>

<bibl id="B12">
  <title><p>Basic proof theory</p></title>
  <aug>
    <au><snm>Troelstra</snm><fnm>AS</fnm></au>
    <au><snm>Schwichtenberg</snm><fnm>H</fnm></au>
  </aug>
  <publisher>Cambridge University Press</publisher>
  <pubdate>2000</pubdate>
  <issue>43</issue>
</bibl>

<bibl id="B13">
  <title><p>Constructive {M}athematics and {C}omputer {P}rogramming</p></title>
  <aug>
    <au><snm>Martin L{\"o}f</snm><fnm>P</fnm></au>
  </aug>
  <source>Studies in Logic and the Foundations of Mathematics</source>
  <publisher>Elsevier</publisher>
  <pubdate>1982</pubdate>
  <volume>104</volume>
  <fpage>153</fpage>
  <lpage>-175</lpage>
</bibl>

<bibl id="B14">
  <title><p>{Intuitionistic Type Theory: Notes by Giovanni Sambin of a series
  of lectures given in Padova, June 1980}</p></title>
  <aug>
    <au><snm>Martin L{\"o}f</snm><fnm>P</fnm></au>
  </aug>
  <pubdate>1984</pubdate>
</bibl>

<bibl id="B15">
  <title><p>{An Intuitionistic Theory of Types}</p></title>
  <aug>
    <au><snm>Martin L{\"o}f</snm><fnm>P</fnm></au>
  </aug>
  <source>Twenty-five years of constructive type theory</source>
  <publisher>Oxford University Press</publisher>
  <pubdate>1998</pubdate>
  <volume>36</volume>
  <fpage>127</fpage>
  <lpage>-172</lpage>
</bibl>

<bibl id="B16">
  <title><p>Elements of set theory</p></title>
  <aug>
    <au><snm>Enderton</snm><fnm>HB</fnm></au>
  </aug>
  <publisher>Academic Press</publisher>
  <pubdate>1977</pubdate>
</bibl>

<bibl id="B17">
  <title><p>Programming in Martin-L{\"o}f’s type theory, volume 7 of
  International Series of Monographs on Computer Science</p></title>
  <aug>
    <au><snm>Nordstr{\"o}m</snm><fnm>B</fnm></au>
    <au><snm>Petersson</snm><fnm>K</fnm></au>
    <au><snm>Smith</snm><fnm>JM</fnm></au>
  </aug>
  <publisher>Clarendon Press, Oxford</publisher>
  <pubdate>1990</pubdate>
</bibl>

<bibl id="B18">
  <title><p>Toward a minimalist foundation for constructive
  mathematics</p></title>
  <aug>
    <au><snm>Maietti</snm><fnm>ME</fnm></au>
    <au><snm>Sambin</snm><fnm>G</fnm></au>
  </aug>
  <source>From Sets and Types to Topology and Analysis: Practicable Foundations
  for Constructive Mathematics</source>
  <publisher>Oxford Logic Guides</publisher>
  <pubdate>2005</pubdate>
  <volume>48</volume>
  <fpage>91</fpage>
  <lpage>-114</lpage>
</bibl>

<bibl id="B19">
  <title><p>Homotopy Type Theory: Univalent Foundations of
  Mathematics</p></title>
  <aug>
    <au><snm>{Univalent Foundations Program}</snm><fnm>T</fnm></au>
  </aug>
  <publisher>Institute for Advanced Study:
  \url{https://homotopytypetheory.org/book}</publisher>
  <pubdate>2013</pubdate>
</bibl>

<bibl id="B20">
  <title><p>Syntax and {S}emantics of {D}ependent {T}ypes</p></title>
  <aug>
    <au><snm>Hofmann</snm><fnm>M</fnm></au>
  </aug>
  <source>Extensional Constructs in Intensional Type Theory</source>
  <publisher>Springer</publisher>
  <pubdate>1997</pubdate>
  <fpage>13</fpage>
  <lpage>-54</lpage>
</bibl>

<bibl id="B21">
  <title><p>Church's thesis: a kind of reducibility axiom for constructive
  mathematics</p></title>
  <aug>
    <au><snm>Kreisel</snm><fnm>G</fnm></au>
  </aug>
  <source>Studies in Logic and the Foundations of Mathematics</source>
  <publisher>Elsevier</publisher>
  <pubdate>1970</pubdate>
  <volume>60</volume>
  <fpage>121</fpage>
  <lpage>-150</lpage>
</bibl>

<bibl id="B22">
  <title><p>Theory of Recursive Functions and Effective
  Computability</p></title>
  <aug>
    <au><snm>Rogers</snm><fnm>H</fnm></au>
    <au><snm>Rogers</snm><fnm>H</fnm></au>
  </aug>
  <publisher>New York: McGraw-Hill</publisher>
  <pubdate>1967</pubdate>
  <volume>5</volume>
</bibl>

<bibl id="B23">
  <title><p>Computability: An Introduction to Recursive Function
  Theory</p></title>
  <aug>
    <au><snm>Cutland</snm><fnm>N</fnm></au>
  </aug>
  <publisher>Cambridge: Cambridge University Press</publisher>
  <pubdate>1980</pubdate>
</bibl>

<bibl id="B24">
  <title><p>Computable analysis: An introduction</p></title>
  <aug>
    <au><snm>Weihrauch</snm><fnm>K</fnm></au>
  </aug>
  <publisher>Berlin, Heidelberg: Springer Science \& Business Media</publisher>
  <pubdate>2012</pubdate>
</bibl>

<bibl id="B25">
  <title><p>Classes of recursively enumerable sets and their decision
  problems</p></title>
  <aug>
    <au><snm>Rice</snm><fnm>HG</fnm></au>
  </aug>
  <source>Transactions of the American Mathematical Society</source>
  <publisher>JSTOR</publisher>
  <pubdate>1953</pubdate>
  <volume>74</volume>
  <issue>2</issue>
  <fpage>358</fpage>
  <lpage>-366</lpage>
</bibl>

<bibl id="B26">
  <title><p>A minimalist two-level foundation for constructive
  mathematics</p></title>
  <aug>
    <au><snm>Maietti</snm><fnm>ME</fnm></au>
  </aug>
  <source>Annals of Pure and Applied Logic</source>
  <publisher>Elsevier</publisher>
  <pubdate>2009</pubdate>
  <volume>160</volume>
  <issue>3</issue>
  <fpage>319</fpage>
  <lpage>-354</lpage>
</bibl>

<bibl id="B27">
  <title><p>{Realizability Models for Type Theories}</p></title>
  <aug>
    <au><snm>Reus</snm><fnm>B</fnm></au>
  </aug>
  <source>Electronic Notes in Theoretical Computer Science</source>
  <publisher>Elsevier</publisher>
  <pubdate>1999</pubdate>
  <volume>23</volume>
  <issue>1</issue>
  <fpage>128</fpage>
  <lpage>-158</lpage>
</bibl>

<bibl id="B28">
  <title><p>Realizability</p></title>
  <aug>
    <au><snm>Streicher</snm><fnm>T</fnm></au>
  </aug>
  <publisher>Citeseer</publisher>
  <pubdate>2008</pubdate>
</bibl>

<bibl id="B29">
  <title><p>Semantics of Interaction: An Introduction to Game
  Semantics</p></title>
  <aug>
    <au><snm>Abramsky</snm><fnm>S</fnm></au>
    <au><cnm>others</cnm></au>
  </aug>
  <source>Semantics and Logics of Computation</source>
  <publisher>New York: Cambridge University Press</publisher>
  <pubdate>1997</pubdate>
  <volume>14</volume>
  <fpage>1</fpage>
  <lpage>-31</lpage>
</bibl>

<bibl id="B30">
  <title><p>Game Semantics</p></title>
  <aug>
    <au><snm>Hyland</snm><fnm>M</fnm></au>
  </aug>
  <source>Semantics and Logics of Computation</source>
  <publisher>New York: Cambridge University Press</publisher>
  <pubdate>1997</pubdate>
  <volume>14</volume>
  <fpage>131</fpage>
</bibl>

<bibl id="B31">
  <title><p>Domains and Lambda-Calculi</p></title>
  <aug>
    <au><snm>Amadio</snm><fnm>RM</fnm></au>
    <au><snm>Curien</snm><fnm>PL</fnm></au>
  </aug>
  <publisher>Cambridge: Cambridge University Press</publisher>
  <pubdate>1998</pubdate>
  <issue>46</issue>
</bibl>

<bibl id="B32">
  <title><p>Intensionality, Definability and Computation</p></title>
  <aug>
    <au><snm>Abramsky</snm><fnm>S</fnm></au>
  </aug>
  <source>Johan van Benthem on Logic and Information Dynamics</source>
  <publisher>Cham: Springer</publisher>
  <pubdate>2014</pubdate>
  <fpage>121</fpage>
  <lpage>-142</lpage>
</bibl>

<bibl id="B33">
  <title><p>Game Semantics</p></title>
  <aug>
    <au><snm>Abramsky</snm><fnm>S</fnm></au>
    <au><snm>McCusker</snm><fnm>G</fnm></au>
  </aug>
  <source>Computational Logic: Proceedings of the 1997 Marktoberdorf Summer
  School</source>
  <publisher>Berlin, Heidelberg: Springer</publisher>
  <pubdate>1999</pubdate>
  <fpage>1</fpage>
  <lpage>-55</lpage>
</bibl>

<bibl id="B34">
  <title><p>Games for Dependent Types</p></title>
  <aug>
    <au><snm>Abramsky</snm><fnm>S</fnm></au>
    <au><snm>Jagadeesan</snm><fnm>R</fnm></au>
    <au><snm>V{\'a}k{\'a}r</snm><fnm>M</fnm></au>
  </aug>
  <source>Automata, Languages, and Programming</source>
  <publisher>Berlin, Heidelberg: Springer</publisher>
  <pubdate>2015</pubdate>
  <fpage>31</fpage>
  <lpage>-43</lpage>
</bibl>

<bibl id="B35">
  <title><p>Game Semantics for {M}artin-{L}\"{o}f Type Theory</p></title>
  <aug>
    <au><snm>Yamada</snm><fnm>N</fnm></au>
  </aug>
  <source>arXiv preprint arXiv:1610.01669</source>
  <pubdate>2016</pubdate>
</bibl>

<bibl id="B36">
  <title><p>Games and Full Abstraction for a Functional Metalanguage with
  Recursive Types</p></title>
  <aug>
    <au><snm>McCusker</snm><fnm>G</fnm></au>
  </aug>
  <publisher>London: Springer Science \& Business Media</publisher>
  <pubdate>1998</pubdate>
</bibl>

<bibl id="B37">
  <title><p>Dynamic Games and Strategies</p></title>
  <aug>
    <au><snm>Yamada</snm><fnm>N</fnm></au>
    <au><snm>Abramsky</snm><fnm>S</fnm></au>
  </aug>
  <source>arXiv preprint arXiv:1601.04147</source>
  <pubdate>2016</pubdate>
  <note>Accepted for publication by Mathematical Structures in Computer
  Science.</note>
</bibl>

<bibl id="B38">
  <title><p>Linear logic</p></title>
  <aug>
    <au><snm>Girard</snm><fnm>JY</fnm></au>
  </aug>
  <source>Theoretical Computer Science</source>
  <publisher>Elsevier</publisher>
  <pubdate>1987</pubdate>
  <volume>50</volume>
  <issue>1</issue>
  <fpage>1</fpage>
  <lpage>-101</lpage>
</bibl>

<bibl id="B39">
  <title><p>Games and Full Completeness for Multiplicative Linear
  Logic</p></title>
  <aug>
    <au><snm>Abramsky</snm><fnm>S</fnm></au>
    <au><snm>Jagadeesan</snm><fnm>R</fnm></au>
  </aug>
  <source>The Journal of Symbolic Logic</source>
  <publisher>Cambridge Univ Press</publisher>
  <pubdate>1994</pubdate>
  <volume>59</volume>
  <issue>02</issue>
  <fpage>543</fpage>
  <lpage>-574</lpage>
</bibl>

<bibl id="B40">
  <title><p>Introduction to {H}igher-order {C}ategorical {L}ogic</p></title>
  <aug>
    <au><snm>Lambek</snm><fnm>J</fnm></au>
    <au><snm>Scott</snm><fnm>PJ</fnm></au>
  </aug>
  <publisher>Cambridge University Press</publisher>
  <pubdate>1988</pubdate>
  <volume>7</volume>
</bibl>

<bibl id="B41">
  <title><p>A Game Semantics for Generic Polymorphism</p></title>
  <aug>
    <au><snm>Abramsky</snm><fnm>S</fnm></au>
    <au><snm>Jagadeesan</snm><fnm>R</fnm></au>
  </aug>
  <source>Annals of Pure and Applied Logic</source>
  <publisher>Elsevier</publisher>
  <pubdate>2005</pubdate>
  <volume>133</volume>
  <issue>1</issue>
  <fpage>3</fpage>
  <lpage>-37</lpage>
</bibl>

<bibl id="B42">
  <title><p>Hypergame semantics: full completeness for system {F}</p></title>
  <aug>
    <au><snm>Hughes</snm><fnm>D</fnm></au>
  </aug>
  <source>PhD thesis</source>
  <publisher>D. Phil. thesis, Oxford University</publisher>
  <pubdate>2000</pubdate>
</bibl>

<bibl id="B43">
  <title><p>Game semantics for access control</p></title>
  <aug>
    <au><snm>Abramsky</snm><fnm>S</fnm></au>
    <au><snm>Jagadeesan</snm><fnm>R</fnm></au>
  </aug>
  <source>Electronic Notes in Theoretical Computer Science</source>
  <publisher>Elsevier</publisher>
  <pubdate>2009</pubdate>
  <volume>249</volume>
  <fpage>135</fpage>
  <lpage>-156</lpage>
</bibl>

<bibl id="B44">
  <title><p>Definability and full abstraction</p></title>
  <aug>
    <au><snm>Curien</snm><fnm>PL</fnm></au>
  </aug>
  <source>Electronic Notes in Theoretical Computer Science</source>
  <publisher>Elsevier</publisher>
  <pubdate>2007</pubdate>
  <volume>172</volume>
  <fpage>301</fpage>
  <lpage>-310</lpage>
</bibl>

<bibl id="B45">
  <title><p>Totality in Arena Games</p></title>
  <aug>
    <au><snm>Clairambault</snm><fnm>P</fnm></au>
    <au><snm>Harmer</snm><fnm>R</fnm></au>
  </aug>
  <source>Annals of pure and applied logic</source>
  <publisher>Elsevier</publisher>
  <pubdate>2010</pubdate>
  <volume>161</volume>
  <issue>5</issue>
  <fpage>673</fpage>
  <lpage>-689</lpage>
</bibl>

<bibl id="B46">
  <title><p>Notes on Game Semantics</p></title>
  <aug>
    <au><snm>Curien</snm><fnm>PL</fnm></au>
  </aug>
  <source>From the author?s web page</source>
  <publisher>Citeseer</publisher>
  <pubdate>2006</pubdate>
</bibl>

<bibl id="B47">
  <title><p>On Full Abstraction for {PCF}: {I}, {II}, and {III}</p></title>
  <aug>
    <au><snm>Hyland</snm><fnm>JME</fnm></au>
    <au><snm>Ong</snm><fnm>C HL</fnm></au>
  </aug>
  <source>Information and Computation</source>
  <publisher>Elsevier</publisher>
  <pubdate>2000</pubdate>
  <volume>163</volume>
  <issue>2</issue>
  <fpage>285</fpage>
  <lpage>-408</lpage>
</bibl>

<bibl id="B48">
  <title><p>Abstract {B}{\"o}hm Trees</p></title>
  <aug>
    <au><snm>Curien</snm><fnm>PL</fnm></au>
  </aug>
  <source>Mathematical Structures in Computer Science</source>
  <publisher>Cambridge Univ Press</publisher>
  <pubdate>1998</pubdate>
  <volume>8</volume>
  <issue>06</issue>
  <fpage>559</fpage>
  <lpage>-591</lpage>
</bibl>

<bibl id="B49">
  <title><p>Full Abstraction for {PCF}</p></title>
  <aug>
    <au><snm>Abramsky</snm><fnm>S</fnm></au>
    <au><snm>Jagadeesan</snm><fnm>R</fnm></au>
    <au><snm>Malacaria</snm><fnm>P</fnm></au>
  </aug>
  <source>Information and Computation</source>
  <publisher>Elsevier</publisher>
  <pubdate>2000</pubdate>
  <volume>163</volume>
  <issue>2</issue>
  <fpage>409</fpage>
  <lpage>-470</lpage>
</bibl>

<bibl id="B50">
  <title><p>{Internal Type Theory}</p></title>
  <aug>
    <au><snm>Dybjer</snm><fnm>P</fnm></au>
  </aug>
  <source>Types for Proofs and Programs</source>
  <publisher>Springer</publisher>
  <pubdate>1996</pubdate>
  <fpage>120</fpage>
  <lpage>-134</lpage>
</bibl>

<bibl id="B51">
  <title><p>A general formulation of simultaneous inductive-recursive
  definitions in type theory</p></title>
  <aug>
    <au><snm>Dybjer</snm><fnm>P</fnm></au>
  </aug>
  <source>Journal of Symbolic Logic</source>
  <publisher>JSTOR</publisher>
  <pubdate>2000</pubdate>
  <fpage>525</fpage>
  <lpage>-549</lpage>
</bibl>

<bibl id="B52">
  <title><p>The Lambda Calculus</p></title>
  <aug>
    <au><snm>Barendregt</snm><fnm>HP</fnm></au>
    <au><cnm>others</cnm></au>
  </aug>
  <publisher>Amsterdam: North-Holland</publisher>
  <pubdate>1984</pubdate>
  <volume>3</volume>
</bibl>

<bibl id="B53">
  <title><p>Intuitionistic Type Theory</p></title>
  <aug>
    <au><snm>Dybjer</snm><fnm>P</fnm></au>
    <au><snm>Palmgren</snm><fnm>E</fnm></au>
  </aug>
  <source>Stanford Encyclopedia of Philosophy</source>
  <pubdate>2016</pubdate>
</bibl>

<bibl id="B54">
  <title><p>Lambda Calculi: A Guide for the Perplexed</p></title>
  <aug>
    <au><snm>Hankin</snm><fnm>C</fnm></au>
  </aug>
  <pubdate>1994</pubdate>
</bibl>

<bibl id="B55">
  <title><p>Arithmetices principia, nova methodo exposita, 1899</p></title>
  <aug>
    <au><snm>Peano</snm><fnm>G</fnm></au>
  </aug>
  <source>English translation in [51]</source>
  <pubdate>1879</pubdate>
  <fpage>83</fpage>
  <lpage>-97</lpage>
</bibl>

</refgrp>
} 







\appendix
\section{Intensional Martin-L\"{o}f type theory}
\label{MLTT}
The difference between the \emph{intensional} and the \emph{extensional} variants of MLTT is that there is only the equality in the form of judgements or \emph{judgmental equality}\footnote{Judgmental equality is modulo \emph{$\alpha$-equivalence} or renaming of \emph{bound variables} \cite{jacobs1999categorical,hofmann1997syntax}.} in the extensional one, while the intensional one in addition has \emph{identity types} for another kind of equality, called \emph{propositional equality}, which is to be witnessed by terms.
See Appendix~\ref{IdentityTypes} for the details. 
It is easy to observe from the rules of identity types that judgmentally equal terms are also propositionally equal, but not vice versa. 
Formally, the extensional variant is the intensional one equipped with the axiom of \emph{equality reflection}, which derives judgmental equality from propositional one \cite{martin1982constructive} (n.b., the two kinds of equalities coincide in the extensional variant).

In this appendix, we briefly review the syntax of the intensional variant, which we call \emph{MLTT}.
We first recall \emph{contexts} in Section~\ref{Contexts} and \emph{structural rules} in Section~\ref{StructuralRules}. 
We then recall each type construction in Sections~\ref{UnitType}-\ref{IdentityTypes}.
Along the syntax, we also recall the \emph{interpretation} $\llbracket \_ \rrbracket$ of MLTT in an arbitrary CwF $\mathcal{C} = (\mathcal{C}, \mathrm{Ty}, \mathrm{Tm}, \_\{\_\}, T, \_.\_, \mathrm{p}, \mathrm{v}, \langle\_,\_\rangle_\_)$ \cite{hofmann1997syntax} fixed throughout this appendix.

Each type construction in MLTT is defined in terms of \emph{\bfseries formation}, \emph{\bfseries introduction}, \emph{\bfseries elimination} and \emph{\bfseries computation} rules.
Roughly, the formation rule stipulates how to form the type, and the introduction rule defines terms\footnote{Strictly speaking, the introduction rule defines \emph{canonical} terms of the type, which in turn defines terms of the type; see \cite{martin1984intuitionistic,dybjer2016intuitionistic,nordstrom1990programming} on this point.} of the type. 
On the other hand, the elimination and the computation rules describe how to consume the terms and the result of the consumption (in the form of equations), respectively, both of which are justified by the introduction rule.


\subsection{Contexts}
\label{Contexts}
A \emph{\bfseries context} is a finite sequence $\mathsf{x_1 : A_1, x_2 : A_2, \dots, x_n : A_n}$ of pairs $(\mathsf{x_i}, \mathsf{A_i})$ of a variable $\mathsf{x_i}$ and a type $\mathsf{A_i}$ such that the variables $\mathsf{x_1, x_2, \dots, x_n}$ are pairwise distinct. 
Let $\mathsf{\diamondsuit}$ represent the \emph{empty context}, i.e., the empty sequence $\bm{\epsilon}$. 
We often omit contexts (and sometimes the \emph{turnstile} $\vdash$) in judgements, especially the empty context $\mathsf{\diamondsuit}$.

We have the following axiom and rules for contexts:
\begin{mathpar}
\AxiomC{}
\LeftLabel{\textsc{(\textsc{Ctx-Emp})}}
\UnaryInfC{$\mathsf{\vdash \diamondsuit \ ctx}$}
\DisplayProof \and
\AxiomC{$\mathsf{\Gamma \vdash A \ type}$}
\LeftLabel{(\textsc{Ctx-Ext})}
\UnaryInfC{$\mathsf{\vdash \Gamma, x : A \ ctx}$}
\DisplayProof \and
\AxiomC{$\mathsf{\vdash \Gamma = \Delta \ ctx}$}
\AxiomC{$\mathsf{\Gamma \vdash A = B \ type \ }$}
\LeftLabel{(\textsc{Ctx-ExtEq})}
\BinaryInfC{$\mathsf{\vdash \Gamma, x : A = \Delta, y : B \ ctx \ }$}
\DisplayProof
\end{mathpar}
where $\mathsf{x}$ (resp. $\mathsf{y}$) does not occur in $\mathsf{\Gamma}$ (resp. $\mathsf{\Delta}$). 

The axiom Ctx-Emp and the rule Ctx-Ext define that contexts are exactly finite lists of pairs of a variable and a type.
On the other hand, the rule Ctx-ExtEq is an instance of a \emph{congruence rule} because it states that judgmental equality $=$ on contexts is preserved under `context extension' given by Ctx-Ext. 
Note also that we have $\mathsf{\vdash \diamondsuit = \diamondsuit \ ctx}$ by Ctx-Emp and the rule Ctx-EqRefl in the next section.

\begin{convention}
As in \cite{hofmann1997syntax}, let us skip writing congruence rules for other constructions.
\end{convention}

\subsection{Structural rules}
\label{StructuralRules}
Next, let us collect the rules applicable to all types as \emph{\bfseries structural rules}:
\begin{mathpar}
\AxiomC{$\mathsf{\vdash x_1 : A_1, x_2 : A_2, \dots, x_n : A_n \ ctx}$}
\LeftLabel{(\textsc{Var})}
\RightLabel{($j \in \{ 1, 2, \dots, n \}$)}
\UnaryInfC{$\mathsf{x_1 : A_1, x_2 : A_2, \dots, x_n : A_n \vdash x_j : A_j}$}
\DisplayProof \and
\AxiomC{$\mathsf{\vdash \Gamma \ ctx}$}
\LeftLabel{(\textsc{Ctx-EqRefl})}
\UnaryInfC{$\mathsf{\vdash \Gamma = \Gamma \ ctx}$}
\DisplayProof \and
\AxiomC{$\mathsf{\vdash \Gamma = \Delta \ ctx}$}
\LeftLabel{(\textsc{Ctx-EqSym})}
\UnaryInfC{$\mathsf{\vdash \Delta = \Gamma \ ctx}$}
\DisplayProof \and
\AxiomC{$\mathsf{\vdash \Gamma = \Delta \ ctx}$}
\AxiomC{$\mathsf{\vdash \Delta = \Theta \ ctx}$}
\LeftLabel{(\textsc{Ctx-EqTrans})}
\BinaryInfC{$\mathsf{\vdash \Gamma = \Theta \ ctx}$}
\DisplayProof \and
\AxiomC{$\mathsf{\Gamma \vdash A \ type}$}
\LeftLabel{(\textsc{Ty-EqRefl})}
\UnaryInfC{$\mathsf{\Gamma \vdash A = A \ type}$}
\DisplayProof \and
\AxiomC{$\mathsf{\Gamma \vdash A = B \ type}$}
\LeftLabel{(\textsc{Ty-EqSym})}
\UnaryInfC{$\mathsf{\Gamma \vdash B = A \ type}$}
\DisplayProof \and
\AxiomC{$\mathsf{\Gamma \vdash A = B \ type}$}
\AxiomC{$\mathsf{\Gamma \vdash B = C \ type}$}
\LeftLabel{(\textsc{Ty-EqTrans})}
\BinaryInfC{$\mathsf{\Gamma \vdash A = C \ type}$}
\DisplayProof \and
\AxiomC{$\mathsf{\Gamma \vdash a : A}$}
\LeftLabel{(\textsc{Tm-EqRefl})}
\UnaryInfC{$\mathsf{\Gamma \vdash a = a : A}$}
\DisplayProof \and
\AxiomC{$\mathsf{\Gamma \vdash a = a' : A}$}
\LeftLabel{(\textsc{Tm-EqSym})}
\UnaryInfC{$\mathsf{\Gamma \vdash a' = a : A}$} 
\DisplayProof \and
\AxiomC{$\mathsf{\Gamma \vdash a = a' : A}$}
\AxiomC{$\mathsf{\Gamma \vdash a' = a'' : A}$}
\LeftLabel{(\textsc{Tm-EqTrans})}
\BinaryInfC{$\mathsf{\Gamma \vdash a = a'' : A}$}
\DisplayProof \and
\AxiomC{$\mathsf{\vdash \Gamma = \Delta \ ctx}$}
\AxiomC{$\mathsf{\Gamma \vdash A \ type}$}
\LeftLabel{(\textsc{Ty-Conv})}
\BinaryInfC{$\mathsf{\Delta \vdash A \ type}$}
\DisplayProof \and
\AxiomC{$\mathsf{\Gamma \vdash a : A}$}
\AxiomC{$\mathsf{\vdash \Gamma = \Delta \ ctx}$}
\AxiomC{$\mathsf{\Gamma \vdash A = B \ type}$}
\LeftLabel{(\textsc{Tm-Conv})}
\TrinaryInfC{$\mathsf{\Delta \vdash a : B}$}
\DisplayProof
\end{mathpar}

The rule Var states the reasonable idea that we may give an element $\mathsf{x_j : A_j}$ if it occurs in the context just by `copy-catting' it.
The next nine rules stipulate that every judgmental equality $=$ is an equivalence relation.
Finally, the rules Ty-Conv and Tm-Conv formalize the natural phenomenon that judgements are preserved under the exchange of judgmentally equal contexts and/or types. 

It is then easy to see that the following \emph{\bfseries weakening} and \emph{\bfseries substitution} rules are \emph{admissible} in MLTT, but it is convenient to present them explicitly:
\begin{mathpar}
\AxiomC{$\mathsf{\Gamma, \Delta \vdash \mathsf{J}}$}
\AxiomC{$\mathsf{\Gamma \vdash A \ type}$}
\LeftLabel{(\textsc{Weak})}
\BinaryInfC{$\mathsf{\Gamma, x : A, \Delta \vdash \mathsf{J}}$}
\DisplayProof \and
\AxiomC{$\mathsf{\Gamma, x : A, \Delta \vdash \mathsf{J}}$}
\AxiomC{$\mathsf{\Gamma \vdash a : A}$}
\LeftLabel{(\textsc{Subst})}
\BinaryInfC{$\mathsf{\Gamma, \Delta[a/x] \vdash \mathsf{J}[a/x]}$}
\DisplayProof
\end{mathpar}
where $\mathsf{x}$ does not occur in $\mathsf{\Gamma}$ or $\mathsf{\Delta}$ for Weak, and not in $\mathsf{\Gamma}$ for Subst, and $\mathsf{\mathsf{J}[a/x]}$ (resp. $\mathsf{\Delta[a/x]}$) is the \emph{capture-free substitution} \cite{hankin1994lambda} of $\mathsf{a}$ for $\mathsf{x}$ in $\mathsf{J}$\footnote{Here, $\mathsf{J}$ denotes the RHS of the turnstile $\vdash$ in an arbitrary judgement.} (resp. $\Delta$). 

Note that a priori we cannot define an interpretation of MLTT by induction on deductions since a derivation of a judgement in MLTT is not unique in the presence of the rules Ty-Con and Tm-Con \cite{hofmann1997syntax}. 
For this point, a standard approach is to define an interpretation $\llbracket \_ \rrbracket$ on \emph{pre-syntax}, which is \emph{partial}, by induction on the length of pre-syntax, and show that it is well-defined on every `valid pre-syntax,' i.e., judgement, and preserves judgmental equality by the semantic equality.
By this \emph{soundness theorem} \cite{hofmann1997syntax}, a posteriori we may describe the interpretation $\llbracket \_ \rrbracket$ of the syntax by induction on derivation of judgements:
\begin{definition}[Interpretation of contexts and structural rules in CwFs \cite{hofmann1997syntax}]
\label{DefMLTTInCwFs}
The interpretation $\llbracket \_ \rrbracket$ of contexts and structural rules in a CwF $\mathcal{C}$ is defined by:
\begin{itemize}

\item (\textsc{Ct-Emp}) $\llbracket \mathsf{\vdash \diamondsuit \ ctx} \rrbracket \colonequals T$;

\item (\textsc{Ct-Ext}) $\llbracket \mathsf{\vdash \Gamma, x : A \ ctx} \rrbracket \colonequals \llbracket \mathsf{\vdash \Gamma \ ctx} \rrbracket . \llbracket \mathsf{\Gamma \vdash A \ type} \rrbracket$;

\item (\textsc{Var}) $\llbracket \mathsf{\Gamma, x: A \vdash x : A} \rrbracket \colonequals \mathrm{v}_{\llbracket \mathsf{A} \rrbracket}$; \\ $\llbracket \mathsf{\Gamma, x: A, \Delta, y : B \vdash x : A} \rrbracket \colonequals \llbracket \mathsf{\Gamma, x: A, \Delta \vdash x : A} \rrbracket \{ \mathrm{p}(\llbracket \mathsf{\Gamma, x : A, \Delta \vdash B \ type} \rrbracket) \}_{\llbracket \mathsf{A} \rrbracket}$;

\item (\textsc{Ty-Con}) $\llbracket \mathsf{\Delta \vdash A \ type} \rrbracket \colonequals \llbracket \mathsf{\Gamma \vdash A \ type} \rrbracket$;

\item (\textsc{Tm-Con}) $\llbracket \mathsf{\Delta \vdash a : B} \rrbracket \colonequals \llbracket \mathsf{\Gamma \vdash a : A} \rrbracket$.

\end{itemize}
\end{definition}

In the rest of this appendix, we recall specific type constructions in MLTT and the interpretation of them in an arbitrary CwF.

\subsection{Unit-type}
\label{UnitType}
Let us begin with the simplest type, called \emph{\bfseries unit-type} (or \emph{\bfseries one-type}) $\mathsf{1}$, which is the type that has just one term $\mathsf{\top}$.\footnote{Strictly speaking, $\mathsf{1}$ has just one \emph{canonical} term $\mathsf{\top}$. However, for simplicity, let us be casual about the distinction between canonical and non-canonical terms in the present work, and we usually call canonical/non-canonical terms just \emph{terms}.}
Therefore, from the logical point of view, it represents the `simplest true formula.'

Rules of unit-type are the following:
\begin{mathpar}
\AxiomC{$\mathsf{\vdash \Gamma \ ctx}$}
\LeftLabel{(\textsc{$\mathsf{1}$-Form})}
\UnaryInfC{$\mathsf{\Gamma \vdash 1 \ type}$}
\DisplayProof \and
\AxiomC{$\mathsf{\vdash \Gamma \ ctx}$}
\LeftLabel{(\textsc{$\mathsf{1}$-Intro})}
\UnaryInfC{$\mathsf{\Gamma \vdash \top : \bm{1}}$}
\DisplayProof \and
\AxiomC{$\mathsf{\Gamma \vdash t : 1}$}
\LeftLabel{(\textsc{$\mathsf{1}$-Uniq})}
\UnaryInfC{$\mathsf{\Gamma \vdash t = \top : 1}$}
\DisplayProof \and
\AxiomC{$\mathsf{\Gamma, x : \bm{1} \vdash C \ type}$}
\AxiomC{$\mathsf{\Gamma \vdash c : C[\top / x]}$}
\AxiomC{$\mathsf{\Gamma \vdash t : \bm{1}}$}
\LeftLabel{(\textsc{$\mathsf{1}$-Elim})}
\TrinaryInfC{$\mathsf{\Gamma \vdash R^{\bm{1}}(C, c, t) : C[t/x]}$}
\DisplayProof \and
\AxiomC{$\mathsf{\Gamma, x : \bm{1} \vdash C \ type}$}
\AxiomC{$\mathsf{\Gamma \vdash c : C[\top/x]}$}
\LeftLabel{(\textsc{$\mathsf{1}$-Comp})}
\BinaryInfC{$\mathsf{\Gamma \vdash R^{\bm{1}}(C, c, \top) = c : C[\top/x]}$}
\DisplayProof
\end{mathpar}
Note that $\mathsf{1}$-Uniq implies $\mathsf{1}$-Elim and $\mathsf{1}$-Comp if we define $\mathsf{R^{1}(C, c, t) \colonequals c}$.

The formation rule $\mathsf{1}$-Form states that unit-type is \emph{atomic}, i.e., we may form it without assuming any other types.
The introduction rule $\mathsf{1}$-Intro defines that it has just one term, viz., $\mathsf{\top}$.
Then, the uniqueness rule $\mathsf{1}$-Uniq makes sense, from which the remaining rules $\mathsf{1}$-Elim and $\mathsf{1}$-Comp follow.

\begin{definition}[Interpretation of unit-type in CwFs]
The interpretation $\llbracket \_ \rrbracket$ of unit-type in a CwF $\mathcal{C}$ that supports unit is given by:

\begin{itemize} 

\item (\textsc{$\mathsf{1}$-Form}) $\llbracket \mathsf{\Gamma \vdash 1 \ type} \rrbracket \colonequals \boldsymbol{1}^{\llbracket \mathsf{\Gamma} \rrbracket}$;

\item (\textsc{$\mathsf{1}$-Intro}) $\llbracket \mathsf{\Gamma \vdash \top : 1} \rrbracket \colonequals \top_{\llbracket \mathsf{\Gamma} \rrbracket}$;

\item (\textsc{$\mathsf{1}$-Elim}) $\llbracket \mathsf{\Gamma \vdash R^{1}(C, c, t) : C[t/x]} \rrbracket \colonequals \mathcal{R}^{\boldsymbol{1}}_{\llbracket \mathsf{C} \rrbracket}(\llbracket \mathsf{c} \rrbracket, \llbracket \mathsf{t} \rrbracket)$.

\end{itemize}
\end{definition}

\subsection{Empty-type}
\label{EmptyType}
Next, let us recall \emph{\bfseries empty-type} (or \emph{\bfseries zero-type}) $\mathsf{0}$, which is the type that has no terms.
Thus, it corresponds in logic to the `simplest false formula.'

Rules of empty-type are the following:
\begin{mathpar}
\AxiomC{$\mathsf{\vdash \Gamma \ ctx}$}
\LeftLabel{(\textsc{$\mathsf{0}$-Form})}
\UnaryInfC{$\mathsf{\Gamma \vdash \bm{0} \ type}$}
\DisplayProof \and
\AxiomC{$\mathsf{\Gamma, x : \bm{0} \vdash C \ type}$}
\AxiomC{$\mathsf{\Gamma \vdash a : \bm{0}}$}
\LeftLabel{(\textsc{$\mathsf{0}$-Elim})}
\BinaryInfC{$\mathsf{\Gamma \vdash R^{\bm{0}}(C, a) : C[a/x]}$}
\DisplayProof
\end{mathpar}

The formation rule $\mathsf{0}$-Form is similar to $\mathsf{1}$-Form, and the elimination rule $\mathsf{0}$-Elim corresponds in logic to \emph{ex falso}, i.e., `anything follows from a contradiction.'
Empty-type has neither introduction nor computation rule since it has no terms.

\begin{definition}[Interpretation of empty-type in CwFs]
The interpretation $\llbracket \_ \rrbracket$ of empty-type in a CwF $\mathcal{C}$ that supports empty is given by:

\begin{itemize} 

\item (\textsc{$\mathsf{0}$-Form}) $\llbracket \mathsf{\Gamma \vdash 0 \ type} \rrbracket \colonequals \boldsymbol{0}^{\llbracket \mathsf{\Gamma} \rrbracket}$;

\item (\textsc{$\mathsf{0}$-Elim}) $\llbracket \mathsf{\Gamma \vdash R^{0}(C, a) : C[a/x]} \rrbracket \colonequals \mathcal{R}^{\boldsymbol{0}}_{\llbracket \mathsf{C} \rrbracket}(\llbracket \mathsf{a} \rrbracket)$.

\end{itemize}
\end{definition}

\subsection{N-type}
\label{NaturalNumberType}
We proceed to recall an atomic type of computational significance, \emph{\bfseries natural number type} (or \emph{\bfseries N-type}) $\mathsf{N}$, which is a type of natural numbers.

Rules of N-type are the following:
\begin{footnotesize}
\begin{mathpar}
\AxiomC{$\mathsf{\vdash \Gamma \ ctx}$}
\LeftLabel{(\textsc{$\mathsf{N}$-Form})}
\UnaryInfC{$\mathsf{\Gamma \vdash N \ type}$}
\DisplayProof \and
\AxiomC{$\mathsf{\vdash \Gamma \ ctx}$}
\LeftLabel{(\textsc{$\mathsf{N}$-IntroZ})}
\UnaryInfC{$\mathsf{\Gamma \vdash zero : N}$}
\DisplayProof \and
\AxiomC{$\mathsf{\Gamma \vdash n : N}$}
\LeftLabel{(\textsc{$\mathsf{N}$-IntroS})}
\UnaryInfC{$\mathsf{\Gamma \vdash  succ(n) : N}$}
\DisplayProof \and
\AxiomC{$\mathsf{\Gamma, x : N \vdash C \ type}$}
\AxiomC{$\mathsf{\Gamma \vdash c_z : C[zero/x]}$}
\AxiomC{$\mathsf{\Gamma, x : N, y : C \vdash c_s : C[succ(x)/x]}$}
\AxiomC{$\mathsf{\Gamma \vdash n : N}$}
\LeftLabel{(\textsc{$\mathsf{N}$-Elim})}
\QuaternaryInfC{$\mathsf{\Gamma \vdash R^N(C, c_z, c_s, n) : C[n/x]}$}
\DisplayProof \and
\AxiomC{$\mathsf{\Gamma, x : N \vdash C \ type}$}
\AxiomC{$\mathsf{\Gamma \vdash c_z : C[zero/x]}$}
\AxiomC{$\mathsf{\Gamma, x : N, y : C \vdash c_s : C[succ(x)/x]}$}
\LeftLabel{(\textsc{$\mathsf{N}$-CompZ})}
\TrinaryInfC{$\mathsf{\Gamma \vdash R^N(C, c_z, c_s, zero) = c_z : C[zero/x]}$}
\DisplayProof \and
\AxiomC{$\mathsf{\Gamma, x : N \vdash C \ type}$}
\AxiomC{$\mathsf{\Gamma \vdash c_z : C[zero/x]}$}
\AxiomC{$\mathsf{\Gamma, x : N, y : C \vdash c_s : C[succ(x)/x]}$}
\AxiomC{$\mathsf{\Gamma \vdash n : N}$}
\LeftLabel{(\textsc{$\mathsf{N}$-CompS})}
\QuaternaryInfC{$\mathsf{\Gamma \vdash R^N(C, c_z, c_s, succ(n)) = c_s[n/x, R^N(C, c_z, c_s, n)/y] : C[succ(n)/x]}$}
\DisplayProof
\end{mathpar}
\end{footnotesize}

Again, the formation rule $\mathsf{N}$-Form says that N-type is atomic.
The introduction rules $\mathsf{N}$-IntroZ and $\mathsf{N}$-IntroZ inductively define terms of N-type: $\mathsf{zero}$ (for $0 \in \mathbb{N}$) and $\mathsf{succ(n)}$ if so is $\mathsf{n}$ (for $n \in \mathbb{N} \Rightarrow n+1 \in \mathbb{N}$).
The elimination rule $\mathsf{N}$-Elim represents both \emph{mathematical induction} and \emph{primitive recursion}: To show a predicate $\mathsf{C}$ over $\mathsf{N}$, it suffices to prove $\mathsf{C(zero)}$ and $\mathsf{C(n)} \Rightarrow \mathsf{C(succ(n))}$, or equivalently under the CHIs, to define a (dependent) function $\mathsf{f}$ from $\mathsf{N}$ to $\mathsf{C}$, it suffices to define its outputs $\mathsf{f(zero)}$ on $\mathsf{zero}$ and $\mathsf{f(succ(n))}$ on $\mathsf{succ(n)}$ in terms of $\mathsf{f(n)}$ and $\mathsf{n}$.
The elimination rule makes sense by the introduction rule, i.e., for terms of N-type are only zero and successors.
Finally, the computation rules $\mathsf{N}$-CompZ and $\mathsf{N}$-CompS stipulate the expected behavior of computations given by $\mathsf{N}$-Elim.

\begin{notation}
Given a context $\mathsf{\vdash \Gamma \ ctx}$ and a natural number $n \in \mathbb{N}$, we define the term $\mathsf{\Gamma \vdash \underline{n} : N}$, called the \emph{\bfseries $\bm{n}^{\text{th}}$ numeral} (in the context $\Gamma$), by induction on $n$: $\mathsf{\underline{0} \colonequals zero}$ and $\mathsf{\underline{n+1} \colonequals succ(\underline{n})}$.
That is, the $n^{\text{th}}$ numeral is to represent the number $n$.
\end{notation}

\begin{definition}[Interpretation of N-type in CwFs]
The interpretation $\llbracket \_ \rrbracket$ of N-type in a CwF $\mathcal{C}$ that supports N is given by:

\begin{itemize} 

\item (\textsc{$\mathsf{N}$-Form}) $\llbracket \mathsf{\Gamma \vdash N \ type} \rrbracket \colonequals N^{\llbracket \mathsf{\Gamma} \rrbracket}$;

\item (\textsc{$\mathsf{N}$-IntroZ}) $\llbracket \mathsf{\Gamma \vdash zero : N} \rrbracket \colonequals \underline{0}_{\llbracket \mathsf{\Gamma} \rrbracket}$;

\item (\textsc{$\mathsf{N}$-IntroS}) $\llbracket \mathsf{\Gamma \vdash succ(n) : N} \rrbracket \colonequals \mathrm{v}_N \{ \mathrm{succ}_{\llbracket \mathsf{\Gamma} \rrbracket} \circ \langle \mathrm{id}_{\llbracket \mathsf{\Gamma} \rrbracket}, \llbracket \mathsf{\Gamma \vdash n : N} \rrbracket \rangle \}$;

\item (\textsc{$\mathsf{N}$-Elim}) $\llbracket \mathsf{\Gamma \vdash R^{N}(C, c_z, c_s, n) : C[n/x]} \rrbracket \colonequals \mathcal{R}^{N}_{\llbracket \mathsf{C} \rrbracket}(\llbracket \mathsf{c_z} \rrbracket, \llbracket \mathsf{c_s} \rrbracket) \{ \langle \mathrm{id}_{\llbracket \mathsf{\Gamma} \rrbracket}, \llbracket \mathsf{n} \rrbracket \rangle_N \}$.

\end{itemize}
\end{definition}

It is easy to see by (the meta-theoretic) mathematical induction that the equation $\llbracket \mathsf{\Gamma \vdash \underline{n} : N} \rrbracket = \underline{n}_{\llbracket \mathsf{\Gamma} \rrbracket}$ holds for any context $\mathsf{\vdash \Gamma \ ctx}$ and natural number $n \in \mathbb{N}$.

\subsection{Pi-types}
\label{DependentFunctionTypes}
Now, let us recall the \emph{\bfseries dependent function types} (or \emph{\bfseries pi-types}) construction $\mathsf{\Pi}$.
In terms of a set-theoretic analogy, the pi-type $\mathsf{\Pi_{x:A}B(x)}$ is something like the set of all (total) functions $\mathsf{f}$ from $\mathsf{A}$ to $\bigcup_{\mathsf{x : A}} \mathsf{B(x)}$ such that $\mathsf{f(a) : B(a)}$ for all $\mathsf{a : A}$, called \emph{dependent functions} from $\mathsf{A}$ to $\mathsf{B}$, where we informally interpret simple types $\mathsf{A}$ and terms $\mathsf{a : A}$ as sets and elements $\mathsf{a \in A}$ of sets, respectively, and dependent types $\mathsf{B}$ on $\mathsf{A}$ as families $\mathsf{(B(x))_{x : A}}$ of sets $\mathsf{B(x)}$.

Rules of pi-types are the following:
\begin{small}
\begin{mathpar}
\AxiomC{$\mathsf{\Gamma \vdash A \ type}$}
\AxiomC{$\mathsf{\Gamma, x : A \vdash B \ type}$}
\LeftLabel{(\textsc{$\mathsf{\Pi}$-Form})}
\BinaryInfC{$\mathsf{\Gamma \vdash \Pi_{x : A}B \ type}$}
\DisplayProof \and
\AxiomC{$\mathsf{\Gamma, x : A \vdash b : B}$}
\LeftLabel{(\textsc{$\mathsf{\Pi}$-Intro})}
\UnaryInfC{$\mathsf{\Gamma \vdash \lambda x^A . b : \Pi_{x : A} B}$}
\DisplayProof \and
\AxiomC{$\mathsf{\Gamma \vdash f : \Pi_{x : A}B}$}
\AxiomC{$\mathsf{\Gamma \vdash a : A}$}
\LeftLabel{(\textsc{$\mathsf{\Pi}$-Elim})}
\BinaryInfC{$\mathsf{\Gamma \vdash f(a) : B[a/x]}$}
\DisplayProof \and
\AxiomC{$\mathsf{\Gamma, x : A \vdash b : B}$}
\AxiomC{$\mathsf{\Gamma \vdash a : A}$}
\LeftLabel{(\textsc{$\mathsf{\Pi}$-Comp})}
\BinaryInfC{$\mathsf{\Gamma \vdash (\lambda x^A . b)(a) = b[a/x] : B[a/x]}$}
\DisplayProof \and
\AxiomC{$\mathsf{\Gamma \vdash f : \Pi_{x : A}B}$}
\LeftLabel{(\textsc{$\mathsf{\Pi}$-Uniq})}
\UnaryInfC{$\mathsf{\Gamma \vdash \lambda x^A . f(x) = f : \Pi_{x : A} B}$}
\DisplayProof
\end{mathpar}
\end{small}
where in \textsc{$\mathsf{\Pi}$-Uniq} the variable $\mathsf{x}$ does not occur free in $\mathsf{f}$.

The formation rule $\mathsf{\Pi}$-Form states that we may form the pi-type $\mathsf{\Pi_{x:A}B}$ from types $\mathsf{A}$ and $\mathsf{B}$, where $\mathsf{B}$ depends on $\mathsf{A}$.
The introduction rule $\mathsf{\Pi}$-Intro defines how to construct terms of $\mathsf{\Pi_{x:A}B}$; it is the ordinary \emph{currying} yet generalized to dependent types.
Then, the elimination and the computation rules $\mathsf{\Pi}$-Elim and $\mathsf{\Pi}$-Comp make sense by the introduction rule.
Finally, the uniqueness rule $\mathsf{\Pi}$-Uniq stipulates that (canonical) terms of pi-types are only $\lambda$-abstractions.


\begin{definition}[Interpretation of pi-types in CwFs \cite{hofmann1997syntax}]
The interpretation $\llbracket \_ \rrbracket$ of pi-types in a CwF $\mathcal{C}$ that supports pi is given by:

\begin{itemize} 

\item (\textsc{$\mathsf{\Pi}$-Form}) $\llbracket \mathsf{\Gamma \vdash \Pi_{x : A}B \ type} \rrbracket \colonequals \Pi (\llbracket \mathsf{\Gamma \vdash A \ type} \rrbracket, \llbracket \mathsf{\Gamma, x : A \vdash B \ type} \rrbracket)$;

\item (\textsc{$\mathsf{\Pi}$-Intro}) $\llbracket \mathsf{\Gamma \vdash \lambda x . \ \! b : \Pi_{x : A}B} \rrbracket \colonequals \lambda_{\llbracket \mathsf{A} \rrbracket, \llbracket \mathsf{B} \rrbracket} (\llbracket \mathsf{\Gamma, x : A \vdash b : B} \rrbracket)$;

\item (\textsc{$\mathsf{\Pi}$-Elim}) $\llbracket \mathsf{\Gamma \vdash f(a) : B[a/x]} \rrbracket \colonequals \mathrm{App}_{\llbracket \mathsf{A} \rrbracket, \llbracket \mathsf{B} \rrbracket} (\llbracket \mathsf{\Gamma \vdash f : \Pi_{x : A}B} \rrbracket, \llbracket \mathsf{\Gamma \vdash a : A} \rrbracket)$.

\end{itemize}
\end{definition}

\subsection{Sigma-types}
\label{DependentPairTypes}
Another important type construction is the \emph{\bfseries dependent sum types} (or \emph{\bfseries sigma-types}) construction $\mathsf{\Sigma}$. 
In terms of the set-theoretic analogy again, the sigma-type $\mathsf{\Sigma_{x:A}B(x)}$ represents the set of all pairs $\mathsf{\langle a, b \rangle}$ such that $\mathsf{a : A}$ and $\mathsf{b : B(a)}$, called \emph{dependent pairs} of $\mathsf{A}$ and $\mathsf{B}$.

Rules of sigma-types are the following:
\begin{small}
\begin{mathpar}
\AxiomC{$\mathsf{\Gamma \vdash A \ type}$}
\AxiomC{$\mathsf{\Gamma, x : A \vdash B \ type}$}
\LeftLabel{(\textsc{$\mathsf{\Sigma}$-Form})}
\BinaryInfC{$\mathsf{\Gamma \vdash \Sigma_{x : A}B \ type}$}
\DisplayProof \and
\AxiomC{$\mathsf{\Gamma, x : A \vdash B \ type}$}
\AxiomC{$\mathsf{\Gamma \vdash a : A}$}
\AxiomC{$\mathsf{\Gamma \vdash b : B[a/x]}$}
\LeftLabel{(\textsc{$\mathsf{\Sigma}$-Intro})}
\TrinaryInfC{$\mathsf{\Gamma \vdash \langle a, b \rangle : \Sigma_{x : A} B}$}
\DisplayProof \and
\AxiomC{$\mathsf{\Gamma, z : \Sigma_{x : A}B \vdash C \ type}$}
\AxiomC{$\mathsf{\Gamma, x : A, y : B \vdash g : C[\langle x, y \rangle/z]}$}
\AxiomC{$\mathsf{\Gamma \vdash p : \Sigma_{x : A}B}$}
\LeftLabel{(\textsc{$\mathsf{\Sigma}$-Elim})}
\TrinaryInfC{$\mathsf{\Gamma \vdash R^{\Sigma}([z : \Sigma_{x : A}B]C, [x:A, y:B]g, p) : C[p/z]}$}
\DisplayProof \and
\AxiomC{$\mathsf{\Gamma, z : \Sigma_{x : A}B \vdash C \ type}$}
\AxiomC{$\mathsf{\Gamma, x : A, y : B \vdash g : C[\langle x, y \rangle/z]}$}
\AxiomC{$\mathsf{\Gamma \vdash a : A}$}
\AxiomC{$\mathsf{\Gamma \vdash b : B[a/x]}$}
\LeftLabel{(\textsc{$\mathsf{\Sigma}$-Comp})}
\QuaternaryInfC{$\mathsf{\Gamma \vdash R^{\Sigma}([z : \Sigma_{x : A}B]C, [x:A, y:B]g, \langle a, b \rangle) = g[a/x, b/y] : C[\langle a, b \rangle/z]}$}
\DisplayProof \and
\AxiomC{$\mathsf{\Gamma \vdash p : \Sigma_{x:A}B}$}
\LeftLabel{(\textsc{$\mathsf{\Sigma}$-Uniq})}
\UnaryInfC{$\mathsf{\Gamma \vdash \langle \pi^{A, B}_1(p), \pi^{A, B}_2(p) \rangle = p : \Sigma_{x:A}B}$}
\DisplayProof
\end{mathpar}
\end{small}
where 
\begin{align*}
&\mathsf{\Gamma \vdash \pi^{A, B}_1(p) \stackrel{\mathrm{df. }}{\equiv} R^\Sigma ([z : \Sigma_{x : A}B]A, [x:A, y:B]x, p) : A} \\
&\mathsf{\Gamma \vdash \pi^{A, B}_2(p) \stackrel{\mathrm{df. }}{\equiv} R^\Sigma([z : \Sigma_{x : A}B]B[\pi^{A, B}_1(z)/x], [x:A, y:B]y, p]) : B[\pi^{A, B}_1(p)/x]}
\end{align*}
are \emph{projections} constructed by \textsc{$\mathsf{\Sigma}$-Elim}.

The formation rule $\mathsf{\Sigma}$-Form is the same as that of pi-types.
The introduction rule $\mathsf{\Sigma}$-Intro specifies that terms of a sigma-type $\mathsf{\Sigma_{x : A} B}$ are dependent pairs $\mathsf{\langle a, b \rangle : \Sigma_{x:A}B}$ of terms $\mathsf{a : A}$ and $\mathsf{b : B[a/x]}$.
Again, the elimination and the computation rules $\mathsf{\Sigma}$-Elim and $\mathsf{\Sigma}$-Comp make sense by the introduction rule.
Finally, the uniqueness rule $\mathsf{\Sigma}$-Uniq stipulates that (canonical) terms of sigma-types are only dependent pairs.

\begin{definition}[Interpretation of sigma-types in CwFs \cite{hofmann1997syntax}]
The interpretation $\llbracket \_ \rrbracket$ of sigma-types in a CwF $\mathcal{C}$ that supports sigma is given by:

\begin{itemize} 

\item (\textsc{$\Sigma$-Form}) $\llbracket \mathsf{\Gamma \vdash \Sigma_{x : A}B \ type} \rrbracket \colonequals \Sigma (\llbracket \mathsf{\Gamma \vdash A \ type} \rrbracket, \llbracket \mathsf{\Gamma, x : A \vdash B \ type} \rrbracket)$;

\item (\textsc{$\Sigma$-Intro}) $\llbracket \mathsf{\Gamma \vdash (a, b) : \Sigma_{x : A}B} \rrbracket \colonequals \mathrm{Pair}_{\llbracket \mathsf{A} \rrbracket, \llbracket \mathsf{B} \rrbracket} \circ \langle \overline{\llbracket \mathsf{\Gamma \vdash a : A} \rrbracket}, \llbracket \mathsf{\Gamma \vdash b : B[a/x] \rrbracket} \rangle_{\llbracket \mathsf{B} \rrbracket}$;

\item (\textsc{$\Sigma$-Elim}) $\llbracket \mathsf{\Gamma \vdash R^{\Sigma}(C, g, p) : C[p/z]} \rrbracket \colonequals \mathcal{R}^{\Sigma}_{\llbracket \mathsf{A} \rrbracket, \llbracket \mathsf{B} \rrbracket, \llbracket \mathsf{C} \rrbracket} (\llbracket \mathsf{\Gamma, x : A, y : B \vdash g : C[(x, y)/z]} \rrbracket) \circ \overline{\llbracket \mathsf{\Gamma \vdash p : \Sigma_{x : A} B} \rrbracket}$,

\end{itemize}
where $\overline{\llbracket \mathsf{\Gamma \vdash a : A} \rrbracket} \colonequals \langle \mathit{id}_{\llbracket \mathsf{\Gamma} \rrbracket}, \llbracket \mathsf{a} \rrbracket \rangle : \llbracket \mathsf{\Gamma} \rrbracket \to \llbracket \mathsf{\Gamma} \rrbracket . \llbracket \mathsf{A} \rrbracket$ and $\overline{\llbracket \mathsf{\Gamma \vdash p : \Sigma_{x : A} B} \rrbracket} \colonequals \langle \mathit{id}_{\llbracket \mathsf{\Gamma} \rrbracket}, \llbracket \mathsf{p} \rrbracket \rangle : \llbracket \mathsf{\Gamma} \rrbracket \to \llbracket \mathsf{\Gamma} \rrbracket . \llbracket \mathsf{\Sigma_{x:A}B} \rrbracket$.
\end{definition}

\subsection{Id-types}
\label{IdentityTypes}
Note that a judgmental equality $\mathsf{\Gamma \vdash a = a' : A}$ is a judgement, not a formula, and therefore it cannot be used in a context or derived by an induction principle such as $\mathsf{N}$-Elim.
To overcome this deficiency, the \emph{\bfseries identity types} (or \emph{\bfseries Id-types}) construction $\mathrm{Id}$ has been introduced.\footnote{We can then, e.g., formulate and prove \emph{Peano axioms} \cite{peano1879arithmetices} by Id-types.}
Informally, the Id-type $\mathsf{Id_A(a, a')}$ represents the set of all \emph{(identity) proofs} that `witnesses' equality between $\mathsf{a}$ and $\mathsf{a'}$.

Rules of Id-types are the following: 
\begin{small}
\begin{mathpar}
\AxiomC{$\mathsf{\Gamma \vdash A \ type}$}
\AxiomC{$\mathsf{\Gamma \vdash a : A}$}
\AxiomC{$\mathsf{\Gamma \vdash a' : A}$}
\LeftLabel{(\textsc{$\mathrm{Id}$-Form})}
\TrinaryInfC{$\mathsf{\Gamma \vdash Id_A(a, a') \ type}$}
\DisplayProof \and
\AxiomC{$\mathsf{\Gamma \vdash A \ type}$}
\AxiomC{$\mathsf{\Gamma \vdash a : A}$}
\LeftLabel{(\textsc{$\mathrm{Id}$-Intro})}
\BinaryInfC{$\mathsf{\Gamma \vdash refl_a : Id_A(a, a)}$}
\DisplayProof \and
\AxiomC{$\mathsf{\Gamma, x : A, y : A, p : Id_A(x, y) \vdash C \ type} $}
\AxiomC{$\mathsf{\Gamma, z : A \vdash c : C[z/x, z/y, refl_z/p]}$}
\AxiomC{$\mathsf{\Gamma \vdash q : Id_A(a, a')}$}
\LeftLabel{(\textsc{$\mathrm{Id}$-Elim})}
\TrinaryInfC{$\mathsf{\Gamma \vdash R^=(C, c, a, a', q) : C[a/x, a'/y, q/p]}$}
\DisplayProof \and
\AxiomC{$\mathsf{\Gamma, x : A, y : A, p : Id_A(x, y) \vdash C \ type}$}
\AxiomC{$\mathsf{\Gamma, z : A \vdash c : C[z/x, z/y, refl_z/p]}$}
\AxiomC{$\mathsf{\Gamma \vdash a : A}$}
\LeftLabel{(\textsc{$\mathrm{Id}$-Comp})}
\TrinaryInfC{$\mathsf{\Gamma \vdash R^=(C, c, a, a, refl_a) = c[a/z] : C[a/x, a/y, refl_a/p]}$}
\DisplayProof
\end{mathpar}
\end{small}

The formation rule $\mathrm{Id}$-Form states that we may form the Id-type $\mathsf{Id_A(a, a')}$ from a type $\mathsf{A}$ and terms $\mathsf{a, a' : A}$.
The introduction rule $\mathrm{Id}$-Intro defines that there is just one term $\mathsf{refl_a}$ of the Id-type $\mathsf{Id_A(a, a)}$.
Again, the elimination and the computation rules $\mathrm{Id}$-Elim and $\mathrm{Id}$-Comp make sense by the introduction rule.

\begin{definition}[Interpretation of Id-types in CwFs \cite{hofmann1997syntax}]
The interpretation $\llbracket \_ \rrbracket$ of Id-types in a CwF $\mathcal{C}$ that supports Id is given by:

\begin{itemize} 

\item (\textsc{$=$-Form}) $\llbracket \mathsf{\Gamma \vdash a =_A a' \ type} \rrbracket \colonequals \mathrm{Id}_{\llbracket \mathsf{A} \rrbracket} \{ \langle \overline{\llbracket \mathsf{\Gamma \vdash a : A} \rrbracket}, \llbracket \mathsf{\Gamma \vdash a' : A} \rrbracket \rangle_{\llbracket \mathsf{A} \rrbracket} \}$;

\item (\textsc{$=$-Intro}) $\llbracket \mathsf{\Gamma \vdash refl_A : a =_A a} \rrbracket \colonequals \mathrm{v}_{\mathrm{Id}_{\llbracket \mathsf{A} \rrbracket}} \{ \mathrm{Refl}_{\llbracket \mathsf{A} \rrbracket} \circ \overline{\llbracket \mathsf{\Gamma \vdash a : A} \rrbracket} \}$;

\item (\textsc{$=$-Elim}) $\llbracket \mathsf{\Gamma \vdash R^{=}(C, c, a, a', q) : C[a/x, a'/y, q/p]} \rrbracket \colonequals \mathcal{R}_{\llbracket \mathsf{A}, \mathsf{C} \rrbracket}^{\mathrm{Id}}(\llbracket \mathsf{c} \rrbracket) \{ \langle \langle \overline{\llbracket \mathsf{a} \rrbracket}, \llbracket \mathsf{a'} \rrbracket \rangle_{\llbracket \mathsf{A} \rrbracket}, \llbracket \mathsf{q} \rrbracket \rangle_{\llbracket \mathsf{a =_A a'} \rrbracket} \}$.

\end{itemize}
\end{definition}

\end{document}